\documentclass[11pt]{article}
\usepackage{amsmath}
\usepackage{graphicx}
\usepackage{psfrag}
\usepackage{epsf}
\usepackage{enumerate}

\usepackage[round,longnamesfirst]{natbib}
\usepackage{etoolbox}

\makeatletter
\newcommand\bibstyle@comma{\bibpunct(),a,,}
\newcommand\bibstyle@semicolon{\bibpunct();a,,}
\makeatother

\pretocmd\cite{\citestyle{comma}}\relax\relax
\pretocmd\citep{\citestyle{semicolon}}\relax\relax

\usepackage{url} % not crucial - just used below for the URL
\usepackage{amsthm}
\usepackage{xpatch}
\usepackage{subcaption}
\usepackage{bm}
\usepackage{mathdots}

\usepackage{longtable}
\usepackage{bm}
\usepackage{mathrsfs}
\usepackage{tikz}
\usepackage{color} 
\usepackage{amssymb}
\usepackage{latexsym}
\usepackage{xr}
\usepackage{epstopdf}
\usepackage{float}
\usepackage{multirow}
\usepackage{dsfont}
\usepackage[ruled]{algorithm2e}
\usepackage{lscape}
\usepackage{rotating}
\numberwithin{equation}{section}

\usepackage[hyperfootnotes = TRUE]{hyperref}
\def\RR{\mathbb R}
\def\ZZ{\mathbb Z}

\newcommand{\rank}{\operatorname{rk}} 
\newcommand{\vecop}{\operatorname{vec}}

\newcommand{\diag}{\operatorname{diag}} 
\newcommand{\sign}{\operatorname{sign}} 
\newcommand{\tr}{\operatorname{tr}} 
\newcommand{\Prob}{\mathbb{P}} 
\newcommand{\E}{\mathbb{E}} 
\newcommand{\VPsi}{\operatorname{\Psi}}

\newcommand{\distr}{\operatorname{d}} 
\newcommand{\prob}{\operatorname{p}} 

\DeclareMathOperator*{\esssup}{ess\,sup}

\newcommand {\argmin}[1]{\underset{#1}{\rm{argmin}}}

\newcommand{\vertiii}[1]{{\left\vert\kern-0.25ex\left\vert\kern-0.25ex\left\vert #1 
    \right\vert\kern-0.25ex\right\vert\kern-0.25ex\right\vert}}

\newtheorem{theorem}{Theorem}[section]
\newtheorem{proposition}{Proposition}[section]
\newtheorem{lemma}{Lemma}[section]
\newtheorem{corollary}{Corollary}[section]

\theoremstyle{definition}
\newtheorem{remark}{Remark}[section]

\newtheorem{assumption}{Assumption}

\makeatletter
\xpatchcmd{\proof}{\@addpunct{.}}{\@addpunct{:}}{}{}
\makeatother

\DeclareFontFamily{U}{mathx}{\hyphenchar\font45}
\DeclareFontShape{U}{mathx}{m}{n}{<-> mathx10}{}
\DeclareSymbolFont{mathx}{U}{mathx}{m}{n}
\DeclareMathAccent{\widebar}{0}{mathx}{"73}
\newcommand{\mockalph}[1]{}

\allowdisplaybreaks

\begin{document}

\def\spacingset#1{\renewcommand{\baselinestretch}%
{#1}\small\normalsize} \spacingset{1}

%%%%%%%%%%%%%%%%%%%%%%%%%%%%%%%
\newtheorem*{assumptionBIC*}{\assumptionnumber}
\providecommand{\assumptionnumber}{}
\makeatletter
\newenvironment{assumptionBIC}[2]
 {%
  \renewcommand{\assumptionnumber}{Assumption #1#2}%
  \begin{assumptionBIC*}%
  \protected@edef\@currentlabel{#1#2}%
 }
 {%
  \end{assumptionBIC*}
 }
\makeatother
 %%%%%%%%%%%%%%%%%%%%%%%%%%%%%%%

%%%%%%%%%%%%%%%%%%%%%%%%%%%%%%%%%%%%%%%%%%%%%%%%%%%%%%%%%%%%%%%%%%%%%%%%%%%%%%

%High dimensional spectral density estimation by a LASSO type local Whittle estimator
%High dimensional local Whittle estimation under long-range dependence
%LASSO type graphical local Whittle estimation and Thresholding under long-range dependence
\title{
Local Whittle estimation of high-dimensional \\
long-run variance and precision matrices
%Thresholding and graphical local Whittle estimation
\footnote{AMS subject classification. Primary: 62M15, 62H12. Secondary: 62H20.}
\footnote{Keywords and phrases: High-dimensional time series, frequency domain, short- and long-range dependence, spectral density estimation, shrinkage estimation, local Whittle estimation.}
\footnote{This work was carried out during stays of the first and second authors in the Department of Statistics and Operation Research at the University of North Carolina, Chapel Hill. The first and second authors thank the department, in particular, Vladas Pipiras for their hospitality. The first author was supported by the National Research Foundation of Korea (NRF-2019R1F1A1057104, NRF-2022R1F1A1066209). The second author was supported by the DFG (RTG 2131) and the NSF grant 1934985. The third author was supported in part by the NSF grant DMS-1712966.}}

\author{
Changryong Baek\\Sungkyunkwan University \and
Marie-Christine D\"uker\\ Cornell University \and
Vladas Pipiras\\ University of North Carolina}
%\date{\today}

\maketitle

\bigskip

\begin{abstract}
\noindent 
%The spectral density matrix at zero frequency carries information about temporal and spatial dependence of multivariate time series. 
This work develops non-asymptotic theory for estimation of the long-run variance matrix and its inverse, the so-called precision matrix, for high-dimensional time series under general assumptions on the dependence structure including long-range dependence. 
The estimation involves shrinkage techniques which are thresholding and penalizing versions of the classical multivariate local Whittle estimator. 
The results ensure consistent estimation in a double asymptotic regime where the number of component time series is allowed to grow with the sample size as long as the true model parameters are sparse. 
The key technical result is a concentration inequality of the local Whittle estimator for the long-run variance matrix around the true model parameters. In particular, it handles simultaneously the estimation of the memory parameters which enter the underlying model.
Novel algorithms for the considered procedures are proposed, and a simulation study and a data application are also provided.
\end{abstract}

\section{Introduction} \label{se:intro}
%%%%%%%%%%%%%%%%%%%%%%%%%%%%%%%%%%%%%%%%%%%%%%%
Spectral density matrices characterize the component and temporal dependence of multivariate time series, and its estimation is of interest in many areas, including economics and neuroscience.
The long-run variance and precision matrices give, respectively, information about correlations and partial correlations between different component series around zero frequency; see \cite{dahlhaus2000graphical}. 
Their estimation are the frequency domain analogues of covariance and inverse covariance estimation; see \cite{fan2016overview} for a survey on large (inverse) covariance matrix estimation.
Obtaining an estimate for the spectral density matrix can become particularly challenging in a high-dimensional regime when the number of component series becomes relatively large compared to the length of the time series.
In this regime, estimation often employs different shrinkage methods.
The development, theoretical verification and application of different shrinkage methods has been an active research area, along with a growing interest in non-asymptotic theory in high-dimensional statistics; see \cite{wainwright_2019} for a survey on non-asymptotic theory.
\par
This paper develops non-asymptotic theory for estimation of the long-run variance and precision matrices of a stationary multivariate time series around zero frequency while allowing for short- and long-range dependence. 
The local Whittle estimation is used with thresholding and LASSO-type penalizations.
Our non-asymptotic theory allows to infer consistency results on the estimators around the true parameters in a high-dimensional regime where the number of component series can be large compared to the number of observations. We also note that our non-asymptotic results are new even for the one-dimensional case.
\par
Our setting is as follows. Consider a $p$-dimensional second-order stationary time series
$X_{n}=(X_{1,n}, \dots , X_{p,n})'$, $n \in \ZZ$, with zero mean and autocovariance matrix function $\Gamma_{X}(k)=\E X_{n+k}X_{n}'$, $k \in \ZZ$. Suppose that its spectral density matrix $f_{X}(\lambda)$, $\lambda \in (-\pi, \pi)$, related to the autocovariance matrix through $\Gamma_{X}(k)=\int_{-\pi}^{\pi} e^{ik \lambda} f_{X}(\lambda) d \lambda $, satisfies
\begin{equation} \label{eq:f}
f_{X}(\lambda) = \lambda^{-D_{0}}G(\lambda)\lambda^{-D_{0}}, \hspace{ 0.2cm }
G(\lambda) \sim G_{0}, \hspace{ 0.2cm } \text{ as } \lambda \to 0^{+},
\end{equation}
where $\sim$ denotes componentwise asymptotic equivalence, $D_{0}=\diag(d_{0,1}, \dots ,d_{0,p})$ with $d_{0,r} \in (-1/2,1/2)$, $r=1, \dots , p$, $\lambda^{-D_{0}} = \diag(\lambda^{-d_{0,1}}, \dots, \lambda^{-d_{0,p}}) $ and $G_{0}=(G_{0,rs})_{r,s=1, \dots ,p}$ is Hermitian symmetric and positive definite. 
Each individual component series $X_{r,n}$, $n \in \ZZ$, satisfies \eqref{eq:f} with memory parameter $d_{0,r}$. The case $d_{0,r} = 0$ is associated with short-range dependence, the case $d_{0,r}>0$ with long-range dependence and $d_{0,r}<0$ with antipersistence.
We refer to \cite{beran2013long,PipirasTaqqu} for more details on univariate short- and long-range dependence and \cite{KechagiasPipiras} for a discussion on multivariate long-range dependence.
The matrix $G_{0}$ is the long-run variance matrix and $P_{0}=G_{0}^{-1}$ is the precision matrix. They are the focus of this work.
\par
In the presence of long-range dependence, local Whittle estimation is commonly used to estimate the parameters $(D_{0},G_{0})$ of the model \eqref{eq:f}. We introduce here the classical multivariate local Whittle estimators for $(D_{0},G_{0})$ and refer to Section \ref{s:EstAss} for a detailed explanation of the used shrinkage techniques. In particular, we aim to utilize a thresholding technique to estimate the long-run variance matrix $G_{0}$ sparsely (see \eqref{eq:thresholding} below) and a LASSO-type estimator to estimate the precision matrix $P_{0}$ sparsely (see \eqref{eq:Prholasso} below). 
\par
The local Whittle estimators $(\widehat{D},\widehat{G})$ introduced in \cite{Robinson2008} are given by
\begin{equation} \label{eq:LW}
(\widehat{D},\widehat{G})=
\argmin{(D,G)} \ \ell(D,G)
\end{equation}
for the negative log-likelihood
\begin{equation} \label{eq:neglog}
\ell(D,G) = 
\frac{1}{m} \sum_{j=1}^{m} (\log|\lambda_{j}^{-D} G \lambda_{j}^{-D}|+
\tr( I_{X}(\lambda_{j}) \lambda_{j}^{D} G^{-1}  \lambda_{j}^{D} )),
\end{equation}
where $| \cdot |:=\det(\cdot)$ and $\tr(\cdot)$ denote the determinant and the trace of a matrix, 
\begin{equation} \label{eq:periodogram}
I_{X}(\lambda)
= 
\frac{1}{2 \pi N} \Big( \sum_{n=1}^{N} X_{n} e^{i n \lambda} \Big) \Big( \sum_{n=1}^{N} X_{n} e^{i n\lambda} \Big)^{*}
\end{equation}
is the periodogram for sample size $N$
and $m$ is the number of Fourier frequencies $\lambda_{j}=2\pi j /N$ used in estimation.
The optimization problem \eqref{eq:neglog} can be reduced to
\begin{equation} \label{eq:R(D)}
\widehat{D}
=
\argmin{D} \ R(D)
\hspace{0.2cm}
\text{ with }
\hspace{0.2cm}
R(D)=\frac{1}{m} \sum_{j=1}^{m} \log|\lambda_{j}^{-D} \widehat{G}(D)\lambda_{j}^{-D}|,
\end{equation}
where
\begin{equation} \label{eq:G(D)}
\widehat{G}(D) = \frac{1}{m} \sum_{j=1}^{m} \lambda_{j}^{D} I_{X}(\lambda_{j}) \lambda_{j}^{D}.
\end{equation}
\par
Local Whittle estimation was studied by multiple authors.
\cite{robinson1995gaussian} showed consistency and asymptotic normality of the univariate local Whittle estimators.
In the bivariate case $p=2$, the asymptotic normality of the local Whittle estimators of memory parameters $d_{0,1},d_{0,2}$ was established in \cite{Robinson2008}, and that of all model parameters in \cite{BaekKechagiasPipiras2020}. Asymptotic normality results in special cases of \eqref{eq:f} but general fixed dimension $p$ appear in \cite{SHIMOTSU2007:Gaussian,Nielsen2011:LW}. In \cite{DuekerPipiras2019:multiLW}, an asymptotic normality result for the local Whittle estimators \eqref{eq:LW} of all model parameters and general fixed $p$ was obtained. 
\par
The graphical local Whittle estimator is an $l_{1}$-penalized version of the negative log-likelihood in \eqref{eq:neglog}, as proposed in \cite{BKP2017multi} with the focus on its good numerical performance. It can also be written as a function of the precision matrix $P_{0}$.
We refer to Section \ref{s:EstAss} and equation \eqref{eq:lassoyyy} below for a detailed description of the penalized objective function.
The objective function coincides with that used in estimating covariance matrices sparsely; see \cite{Bien2011:Sparse}. On the other hand, for a penalization of the respective inverse, it coincides with the graphical LASSO estimator; see \cite{friedman2008:Sparse}. 
\cite{DuekerPipiras2019:multiLW} derived asymptotic results for the estimators of the long-run variance matrix and the precision matrix in a ``fixed $p$, large $N$" regime under the discussed $l_{1}$-penalization.
\par
Sparse covariance and its precision matrix estimation were studied by numerous authors. LASSO-type estimators were investigated by \cite{rothman2008sparse}, \cite{cai2011constrained} and \cite{shu2019estimation}. See also \cite{cai2016estimating} for a review of recent developments.
%(\cite{basu2015})
Thresholding based strategies were pursued by \cite{bickel2008,bickel2008covariance}, \cite{rothman2009generalized} and \cite{cai2011adaptive}.
In contrast, there is less research work for high-dimensional spectral density matrix estimation.
\par
The works of \cite{shu2019estimation}, \cite{Sun2018:LargeSpectral} and \cite{fiecas2019} are probably the closest to our work.
\cite{shu2019estimation} considered the estimation of large covariance and its precision matrices from high-dimensional sub-Gaussian or heavier-tailed observations with slowly decaying temporal dependence.
\cite{Sun2018:LargeSpectral} developed a non-asymptotic theory for estimation of the spectral density matrix of multivariate time series under short-range dependence, that is, when $D_{0} \equiv 0$ in \eqref{eq:f}.
The work in \cite{fiecas2019} developed some non-asymptotic theory for estimation of the spectral density matrix and its inverse for a class of time series exhibiting short-range dependence under a mixing condition.
In contrast, we allow for a quite general dependence structure including short- and long-range dependence and antipersistence. 
%\cite{fiecas2014}
\par
From a practical perspective estimating the spectral density matrix and its inverse have applications in many fields including signal processing \citep{schneider2016partial}, neuroscience \citep{fiecas2011generalized,bowyer2016coherence,bordier2017graph} and economics \citep{granger1969investigating,hansen1983dimensionality,politis2011higher,plagborg2021local,cavicchioli2022goodness}. 
The spectral density matrix captures contemporaneous correlation and correlation across different lags. It therefore provides a richer description of the dependence structure in a multivariate time series than the covariance matrix. 
\par
The literature review shows that there is a gap in theoretical results concerning high-dimensional spectral density estimation for time series possibly exhibiting long-range dependence.
We are the first to provide non-asymptotic theoretical results for thresholding and graphical local Whittle estimation which allow to infer consistency in a possibly double asymptotic regime of large $p$ and $N$. The presence of long-range dependence and the simultaneous estimation of the memory parameters $D_{0}$, make it particularly challenging to derive non-asymptotic results. We overcome those challenges by using a uniform concentration inequality and controlling the difference between the sample and the population version of the matrix $D_{0}$ simultaneously. Our theoretical results turn out to be useful not only for thresholding and graphical local Whittle estimators but can be applied to derive consistency for other kinds of penalized estimators. We demonstrate that by deriving consistency results for estimators based on the coherence matrix and a constrained $l_{1}$-minimization (CLIME). We also address the question of consistent model selection by adopting different thresholding procedures to the spectral setting.
Additionally, we introduce novel algorithms to compute the thresholded and penalized local Whittle estimators.
The results are accompanied by a simulation study which assesses the numerical performance of the suggested algorithms and estimators. 
\par
The rest of the paper is organized as follows. In Section \ref{s:EstAss}, we discuss our estimation procedure and present some assumptions required for our theoretical analysis. In Section \ref{s:mainResults}, we present an outline of the proof, our main results and some discussions of those results. Appendix \ref{se:tablesResults} provides more technical details for the statements of our results, allowing to keep the notation in the paper's main body shorter.
In Section \ref{s:algorithm}, we introduce two algorithms to compute the penalized graphical local Whittle estimators. The performance of those algorithms is analyzed in a simulation study conducted in Section \ref{s:sim} with complementary results in Appendix \ref{app:table}. An application can be found in Section \ref{s:realdata}. We conclude with Section \ref{s:conclusions}.
The proofs can be found in Appendix \ref{s:proofsA}. In Appendices \ref{s:proofsB} and \ref{s:appSTR}, we provide some technical results and their proofs. Finally, Appendix \ref{se:linearprocesses} provides the proofs for an extension to linear processes.
\par
\textit{Notation:}
For the reader's convenience, we give a collection of notation used throughout the paper.
We denote the maximum and minimum eigenvalues of a symmetric or Hermitian matrix $A$ by $\lambda_{\max}(A)$ and $\lambda_{\min}(A)$, respectively.
To indicate that a matrix $A$ is positive (semi-)definite, we write $A \succ 0$ $(A \succcurlyeq 0)$. 
We use a range of different matrix norms, namely, the maximum norm, the spectral norm and the Frobenius norm, defined respectively as $\| A \|_{\max}=\max_{1 \leq r,s \leq p} |A_{rs}|$, $\| A \|=\sqrt{\lambda_{\max}(A'A)}$ and $\| A \|_{F}=\sqrt{\tr(A'A)}$ for a matrix $A$. 
We use $e_{r}$ to denote the $r$th unit vector in $\RR^{p}$ for $r=1,\dots,p$.
For the vectorized version of a matrix $A$, we write $\vecop(A)$. The vec operator transforms a
matrix into a vector by stacking its columns one underneath the other.
For a $p \times N$ matrix, composed of $N$ $p$-dimensional vectors $v_{1},\dots,v_{N}$, we write $[v_{1} : \dots : v_{N}]$.
We let $L^2(0,1)$ be the space of square-integrable functions on $(0,1)$ with respect to the Lebesgue measure. If $A$ is an integral operator on $L^2(0,1)$ of the form $(Af)(x)=\int_{0}^{1} k(x,y) f(y) dy$, then $A$ is called Hilbert-Schmidt if and only if
\begin{equation} \label{eq:HSnorm}
\int_{0}^{1} \int_{0}^{1} |k(x,y)|^2 dx dy < \infty,
\end{equation}
where the double integral in \eqref{eq:HSnorm} is denoted as $\| A \|_{2}^{2}$ and called the Hilbert-Schmidt norm.
Let further $A:V \to W$ be a linear operator with normed spaces $V,W$. We write $\| A \|_{op} = \sup_{x \neq 0} \|Ax\|_{W}/\|x\|_{V}$
for the operator norm of $A$, where $\|\cdot\|_{W}$ denotes the norm on $W$.
As a further convention we write $a \succsim b$ if there exists a universal constant $c$ such that $a \geq cb$. 
We further use the notation $\frac{\partial}{\partial x}$ to denote the partial derivative with respect to $x$ and 
$\nabla$ to denote the gradient $\nabla f = \sum_{r=1}^{p}e_{r}\frac{\partial}{\partial x_{r}}f$ of a function $f : \RR^{p} \to \RR$.

\section{Estimation methods and assumptions}
\label{s:EstAss}
In this section, we formulate the long-run variance and precision matrix estimation through the thresholding and graphical local Whittle estimators, respectively. Furthermore, we give the required assumptions to prove non-asymptotic bounds which ensure consistency results for the estimators in a double asymptotic regime of large $p$ and $N$.
\par
The thresholding and graphical local Whittle estimators require estimation of the memory parameters $D_{0} = \diag(d_{0,1}, \dots, d_{0,p})$. We propose here to estimate each $d_{0,r}, r=1, \dots , p$, by the univariate local Whittle estimator; see Remark \ref{re:univd} below for a discussion of this.
We introduce a notation different from that used for the multivariate local Whittle estimators in \eqref{eq:LW} to emphasize the use of the univariate version of the local Whittle estimator. For a multivariate time series satisfying \eqref{eq:f}, each individual, univariate time series $\{X_{r,n}\}_{n \in \ZZ}$, $r = 1,\dots, p$, satisfies
\begin{equation*}
f_{X,rr}(\lambda) = \lambda^{-2d_{0,r}} g_{r}(\lambda), 
\hspace{ 0.2cm }
g_{r}(\lambda) \sim g_{0,r}, \hspace{ 0.2cm } \text{ as } \lambda \to 0^{+},
\end{equation*}
where $g_{r}(\lambda)=G_{rr}(\lambda)$, $g_{0,r}=G_{0,rr}$, and $f_{X,rs}(\lambda)$ and $G_{rs}(\lambda)$ denote the $(r,s)$th entry of $f_{X}(\lambda)$ and $G(\lambda)$, respectively. Then, the univariate local Whittle estimator for $d_{0,r}$ is given by
\begin{equation} \label{eq:univLWd}
\widehat{d}_{r} = \argmin{ d \in \Theta } \ R_{r}(d)
\hspace{0.2cm}
\text{ with }
\hspace{0.2cm}
R_{r}(d)=\frac{1}{m} \sum_{j=1}^{m} \log (\lambda_{j}^{-2d} \widehat{g}_{r}(d)),
\end{equation}
where the set of admissible estimates is defined as $\Theta = \{ d ~|~ \Delta_{1} \leq d \leq \Delta_{2} \}$ with $ -\frac{1}{2} < \Delta_{1} < \Delta_{2} < \frac{1}{2}$ and
\begin{equation*} 
%\label{eq:hatgr}
\widehat{g}_{r}(d) = \frac{1}{m} \sum_{j=1}^{m} \lambda_{j}^{2d} I_{X,rr}(\lambda_{j}),
\end{equation*}
where $I_{X,rs}(\lambda)$ denotes the $(r,s)$th entry of the periodogram $I_{X}(\lambda)$ in \eqref{eq:periodogram}.
After estimating each individual memory parameter $d_{0,r}$ by \eqref{eq:univLWd}, we want to estimate the long-run variance matrix $G_{0}$ and the precision matrix $P_{0}$ sparsely by thresholding and graphical local Whittle estimation, respectively.
\par
\textit{Thresholding local Whittle:} We propose to use hard thresholding to estimate the long-run variance matrix sparsely. (Soft or adaptive thresholding could also be used.) That is, 
\begin{equation} \label{eq:thresholding}
T_{\rho}(\widehat{G}_{rs}(\widehat{D}))=
\begin{cases}
\widehat{G}_{rs}(\widehat{D}),	& \hspace{0.2cm} \text{ if } |\widehat{G}_{rs}(\widehat{D})| \geq \rho, \\
0, 						& \hspace{0.2cm} \text{ if } |\widehat{G}_{rs}(\widehat{D})| < \rho,
\end{cases}
\end{equation}
where $\rho>0$ is a threshold and $T_{\rho}(\cdot)$ is a thresholding operator applied to $\widehat{G}_{rs}(\widehat{D})$, the $(r,s)$th entry of the estimator for the long-run variance matrix \eqref{eq:G(D)} and the components of $\widehat{D}$ are estimated univariately by \eqref{eq:univLWd}.
\par
\textit{Graphical local Whittle:} The precision matrix $P_{0}=G_{0}^{-1}$ can be estimated sparsely by the graphical local Whittle estimator, a penalized version of the negative log-likelihood function \eqref{eq:neglog}.
The penalized estimator $\widehat{P}_{\rho}$ is given by
\begin{equation} \label{eq:Prholasso}
\widehat{P}_{\rho}=\argmin{P \succ 0} \ \ell_{\rho}(\widehat{D},P),
\end{equation}
where $\widehat{D}$ is estimated univariately by \eqref{eq:univLWd} and
\begin{equation} \label{eq:lassoyyy}
\begin{aligned}
\ell_{\rho}(D,P)=-\frac{1}{m} \sum_{j=1}^{m} &\log|\lambda_{j}^{D} P \lambda_{j}^{D}|
+\tr( \widehat{G}(D)P) + \rho \| P \|_{1,off}
\end{aligned}
\end{equation}
with a penalty parameter $\rho>0$ and the $l_{1}$-norm $\| \cdot \|_{1,off}$ excluding the diagonal elements. 
\par
Next, we give some assumptions, required to establish our theoretical results. Other assumptions appear in the statements of our results. Subsequently, we discuss those assumptions in several remarks.

\begin{assumption} \label{ass:f0}
Suppose that
\begin{equation} \label{eq:af}
f_{X}(\lambda) = \lambda^{-D_{0}} G(\lambda) \lambda^{-D_{0}}, \hspace{ 0.2cm } 
G(\lambda) \sim G_{0},
\end{equation}
where $\sim$ denotes componentwise asymptotic equivalence, $G_{0}=(G_{0,rs})_{r,s=1,\dots,p}$ is Hermitian symmetric and positive definite and $D_{0} \in \{ D \in \mathcal{M}_{\diag} ~|~ \Delta_{1}I_{p} \preccurlyeq D \preccurlyeq \Delta_{2}I_{p} \}$, where $\mathcal{M}_{\diag}$ denotes the set of all real-valued diagonal matrices.
We further suppose that the positive eigenvalues of $G_{0}$ can be bounded from below as
\begin{equation} \label{as:eigenvalueG0}
\lambda_{\min}(G_{0}) \geq \underline{k} > 0.
\end{equation}
\end{assumption}

\begin{assumption} \label{ass:G-G}
For some $q \in (0,1]$, the spectral density matrix satisfies 
\begin{equation} \label{eq:ass:G-G}
|f_{X,rs}(\lambda_{j})-\lambda_{j}^{-d_{0,r}-d_{0,s}} G_{0,rs}| \leq \bm{c}_{G,1}\lambda_{j}^{2 q-d_{0,r}-d_{0,s}}
\end{equation}
for $j = 1, \dots, m$ and some $\bm{c}_{G,1}>0$.
\end{assumption}

\begin{assumption} \label{ass:derivative}
The function $f_{X,rs}(\lambda)$ is differentiable on $\lambda \in (-\pi,\pi)\backslash \{0\}$ and there is a constant $\bm{c}_{G,2} >0$ such that 
\begin{equation*}
\Big| \frac{\partial}{\partial \lambda} f_{X,rs}(\lambda) \Big| \leq \bm{c}_{G,2} \lambda^{-1-d_{0,r}-d_{0,s}}.
\end{equation*}
\end{assumption}

Besides assumptions on the spectral density matrix of the underlying process, we will also impose some mild assumptions on the process itself. In particular, our results are valid not only for Gaussian time series but also for a large class of non-Gaussian processes.
Our assumption will be formulated in terms of sub-Gaussian random variables, that is, their distribution is dominated by a centered Gaussian distribution.
More precisely, we call a random variable $X$ sub-Gaussian if there is a constant $c$ such that
\begin{equation*}
\E( | X |^{r})^{\frac{1}{r}} \leq  c r^{\frac{1}{2}} \hspace{0.2cm} \text{ for all } r \geq1.
\end{equation*}
We further denote $\| X \|_{\phi} = \sup_{ r \geq 1} r^{-\frac{1}{2}} \E( | X |^{r})^{\frac{1}{r}}$, the sub-Gaussian norm of a real-valued random variable $X$. Gaussian random variables belong to the class of sub-Gaussian random variables.
We refer to \cite{vershynin2010introduction} for more details on sub-Gaussian random variables.

\begin{assumption} \label{ass:GnG}
The time series $\{ X_{n} \}$ is assumed to be either Gaussian or to have a linear representation
$X_{n} = \sum_{j\in \ZZ} \VPsi_{j}\varepsilon_{n-j}$ with $\sum_{j\in \ZZ} \| \VPsi_{j} \|_{F}^{2} < \infty$ and independent mean $0$ innovations $\{ \varepsilon_{j} \}_{j \geq 1}$, where each component $\varepsilon_{r,j}$, $r =1,\dots,p$ of the random vector $\varepsilon_{j}$ is assumed to be sub-Gaussian, satisfying
\begin{equation} \label{eq:ACsubGgamma}
\| \varepsilon_{r,j} \|_{\phi} \leq \gamma
\end{equation}
for some constant $\gamma \in (0, \infty)$.
\end{assumption}

\begin{assumption} \label{ass:antip}
The number of frequencies $m=m(N)$ used in estimation and the lower bound of the interval of admissible estimates $\Delta_{1}$ satisfy
\begin{equation*}
m \succsim N^{-2\Delta_{1}}.
\end{equation*}
\end{assumption}

Our work intends to provide non-asymptotic results. However, we impose some mild assumptions on our choices of the number of frequencies $m$ and the sample size $N$ to simplify some of our bounds. Throughout the paper we suppose that the number of frequencies and the sample size satisfy $m, N > 2$. Those assumptions allow us to use $\log(m)+1 \leq 2\log(m)$ and $\log(m)+1 \leq (\log(m)+1)^{2}$, and the same for $N$. Another assumption we impose is $m \leq \frac{N}{2}-1$ which ensures that the bound on the bias term of the periodogram is finite. 

We use different measures of sparsity for the long-run variance and the precision matrices. Both are commonly used in the respective literatures.
\par
In the context of thresholding, a commonly used measure of sparsity for the long-run variance matrix $G_{0}$ is given by 
\begin{equation} \label{eq:sparsityconditionG0}
\| G_{0} \|_{a}^{a}= \max_{r=1,\dots,p} \sum_{s=1}^{p} |G_{0,rs}|^{a}
\end{equation}
for $a \in [0,1)$. This measure was proposed in
\cite{bickel2008covariance} and shown to capture a variety of sparsity patterns. It was further applied in the context of spectral density estimation in a non-asymptotic regime in \cite{Sun2018:LargeSpectral}.
\par
For the precision matrix, we define the set
\begin{equation} \label{as:cardianlityS}
\operatorname{S}=\{ (r,s) ~|~P_{0,rs} \neq 0, r \neq s\}
\end{equation}
and bound its cardinality with $|\operatorname{S}| \leq \operatorname{s}$.
\par
We will also use
\begin{equation*}
\vertiii{G} = \esssup_{\lambda \in (-\pi,\pi)} \|G(\lambda)\| = \esssup_{\lambda \in (-\pi,\pi)} \| \lambda^{D_{0}} f_{X} (\lambda) \lambda^{D_{0}} \|
\end{equation*}
as a measure of stability of the time series $\{X_{n}\}$. This follows \cite{basu2015} and \cite{Sun2018:LargeSpectral} who considered the case $G(\lambda) \equiv f_{X}(\lambda)$ and $D_{0} \equiv 0$, which is associated with short-range dependence of the underlying time series. See also the second paragraph of Section 2 in \cite{Sun2018:LargeSpectral} for a discussion on how $\vertiii{G}$ acts as a measure of stability.
\par
The following remarks comment on the model, estimation procedure and on the assumptions above. Remark \ref{re:paramP} comments on the model and Remark \ref{re:univd} concerns estimating the memory parameters univariately,
Remark \ref{re:discussA1A2} is on Assumptions \ref{ass:f0} and \ref{ass:G-G}, and Remarks \ref{re:diff}, \ref{re:model} and \ref{re:antipe} are on Assumptions \ref{ass:derivative}, \ref{ass:GnG} and \ref{ass:antip}, respectively.

\begin{remark} \label{re:paramP}
In this work, we assume that the matrix $G_{0}$ in \eqref{eq:f} is possibly complex valued Hermitian symmetric. In order to achieve sparsity, both real and imaginary part need to be zero. Related literature has also studied an alternative way to parametrize the matrix $G_{0}$. Proposed by \cite{Robinson2008} and further studied in \cite{DuekerPipiras2019:multiLW} and \cite{BaekKechagiasPipiras2020}, one can write $G_{0}$ in terms of polar coordinates, that is, 
\begin{equation*}
%\label{eq:GparamP}
G_{0}=(\omega_{kl}e^{\sign(k-l) i \phi_{kl}})_{k,l=1, \dots ,p}
\end{equation*}
with the so-called phase parameter $\phi_{kl} \in (-\pi/2, \pi/2)$ and $\omega_{kl} \in \RR$.
In this parametrization, one cannot test for uncorrelatedness between component series ($\omega_{kl}e^{\sign(k-l) i \phi_{kl}} = 0$), since the respective phase parameter $\phi_{kl}$ is not identifiable for $\omega_{kl}=0$, $k \neq l$; see \cite{DuekerPipiras2019:multiLW} and \cite{BaekKechagiasPipiras2020} for a related discussion. 
\end{remark}

\begin{remark} \label{re:univd}
Our proposed estimation procedure involves estimating the memory parameters $d_{0,1}, \dots , d_{0,p}$ by the univariate local Whittle estimators \eqref{eq:univLWd} rather than using the multivariate estimator of the matrix $D_{0}=\diag(d_{0,1}, \dots , d_{0,p})$ in \eqref{eq:R(D)}. The reasons are twofold, one is theoretical, the other computational. 

The theoretical reason is that getting a concentration inequality on $\| \widehat{D}-D_{0} \|_{\max}$ for the multivariate estimates of $D_{0}$ involves a concentration inequality on $| \log|\widehat{G}(\widehat{D})|-\log|G_{0}| |$. 
In the asymptotic regime $N \to \infty$ and for fixed dimension $p$, a consistency result for $| \log|\widehat{G}(\widehat{D})|-\log|G_{0}| |$ can be achieved easily by combining the continuous mapping theorem and a consistency result on $\| \widehat{G}(\widehat{D})-G_{0} \|_{\max}$.
However, our non-asymptotic setting involves an inequality of the form $| \log|\widehat{G}(\widehat{D})|-\log|G_{0}| | \leq C p \| \widehat{G}(\widehat{D})-G_{0} \|$, with a generic constant $C$. The additional $p$ weakens the results in the sense that $p$ has to grow much slower than $m$ in order to achieve consistency.
Bounding further as $\| \widehat{G}(\widehat{D})-G_{0} \| \leq p\| \widehat{G}(\widehat{D})-G_{0} \|_{\max}$ results in an additional $p$. A potential way to avoid the second $p$ one gets through bounding the operator norm might be to impose a sparsity assumption on $G_{0}$ and threshold $\widehat{G}$ in the objective function \eqref{eq:R(D)}.
That is, estimating $D_{0}$ as in \eqref{eq:R(D)} involves an estimator for $G_{0}$. However, the estimator $\widehat{D}$ for $D_{0}$ is not based on a thresholded version of $\widehat{G}$. One possibility to address this issue is to introduce a shrinkage on $\widehat{G}$ in \eqref{eq:R(D)} by using a thresholded version of $\widehat{G}$. 

On the other hand, computationally, it is faster to minimize $p$ univariate functions as opposed to optimizing a matrix function over a certain set of diagonal matrices.
In a simulation study in Appendix \ref{appF:univsmulti}, we show that the difference between the multivariate and univariate estimates is negligible.
\end{remark}

\begin{remark} \label{re:discussA1A2}
Assumption \ref{ass:f0} with \eqref{eq:af} coincides with the basic model \eqref{eq:f} and supposes additionally that the true memory parameters $D_{0}$ are contained in the interval of admissible estimates. Besides assuming that the matrix $G_{0}$ is positive definite, we suppose in \eqref{as:eigenvalueG0} that the eigenvalues of $G_{0}$ are bounded from below. This is a typical assumption in sparse covariance estimation; see \cite{rothman2008sparse}. Assumption \ref{ass:G-G} is a smoothness condition and controls the second order terms of the spectral density matrix. 
Usually, the componentwise relation $f_{X}(\lambda) = \lambda^{-D_{0}} G_{0} \lambda^{-D_{0}}(1 + O(\lambda^{2q}))$, as $\lambda \to 0^{+}$, is imposed to derive asymptotic results in the context of spectral density estimation. However, we require a slightly stronger assumption \eqref{eq:ass:G-G} in order to control the bias terms to derive non-asymptotic results.
\end{remark}

\begin{remark} \label{re:diff}  
Assumption \ref{ass:derivative} is required to ensure that the bias term is asymptotically negligible.
This kind of assumption appears in the asymptotic literature regarding local Whittle estimation as well; see Assumption A.2 in \cite{robinson1995gaussian} and Assumption A.1 in \cite{Robinson2008}. However, those assumptions typically only require differentiability in an epsilon region around the origin. We need to impose differentiability in a region which includes all frequencies used in estimation 
and allows for all choices of $m =1, \dots , \frac{N}{2}$. 
\end{remark}

\begin{remark} \label{re:model}
Our main assumptions on the underlying process are the parametrization of the spectral density in terms of the matrices $(D_{0},G_{0})$ as formalized in Assumption \ref{ass:f0}, and Assumption \ref{ass:GnG} which ensures that the series is either Gaussian or follows a linear representation. Though, we require the innovations of the linear representation to be sub-Gaussian, our results can be used to derive statements for sub-exponential innovations or assuming finite fourth moments; see Remark \ref{re:different_innovations} for a more detailed discussion. 
Our assumptions allow for quite general long- and short-range dependent linear time series. For long-range dependence, examples are multivariate FARIMA series as defined in \cite{KechagiasPipiras}. For short-range dependence, examples are multivariate ARMA models.
\end{remark}

\begin{remark} \label{re:antipe}
Assumption \ref{ass:antip} is satisfied, in particular, when $\Delta_{1} = 0$, that is, when the underlying time series exhibits only short- or long-range dependence. In other words, the assumption is only needed when the true memory parameters $d_{0,r}$, $r=1,\dots,p$, are known to take also values in $(-\frac{1}{2},0)$. The case $d_{0,r} < 0$ contributes to the non-asymptotic bounds in our main results in form of two terms $m N^{2\Delta_{1}}$ and $m \left(\sum_{j=1}^{m}\lambda_{j}^{4 \Delta_{1}}\right)^{-1}$. 
Assumption \ref{ass:antip} is necessary to ensure that our non-asymptotic bounds prove consistency, which is the case as long as both quantities go to infinity while the sample size increases.
Assumption \ref{ass:antip} not only controls $m N^{2\Delta_{1}}$ but is also sufficient to control the second quantity, since
\begin{equation*}
\sum_{j=1}^{m}\lambda_{j}^{4 \Delta_{1}}
\leq c N^{-4\Delta_{1}} \sum_{j=1}^{m} j^{4 \Delta_{1}} 
\leq c N^{-4\Delta_{1}} 
\begin{cases}
\frac{1}{4 \Delta_{1} + 1} m^{4 \Delta_{1} + 1}, & \text{ if } \Delta_{1} \in (-\frac{1}{4},0), \\ 
\log(m)+1, & \text{ if } \Delta_{1} \leq -\frac{1}{4}.
\end{cases}
\end{equation*}
\end{remark}

\section{Main results}
\label{s:mainResults}
In this section, we present our main results. Section \ref{se:idea} provides a roadmap for our proofs which reveals what kind of results are necessary to prove consistency for both the thresholding and graphical local Whittle estimation. This includes in particular a consistency result on the maximum norm of $\widehat{G}(\widehat{D})-G_{0}$. Subsequently, we formally state our main results in Section \ref{se:statements}, that is, consistency results for the thresholding and graphical local Whittle estimators. 
Section \ref{se:alternativeestimators} discusses alternative estimators for precision matrix estimation and their convergence rates. We provide results on consistent model selection in Section \ref{se:modelselection}.
In Section \ref{se:comparison}, we discuss how our results compare to existing results in the literature.

\subsection{Proof idea} \label{se:idea}
In contrast to the spectral density estimation under short-range dependence, allowing for long-range dependence and antipersistence requires estimation of two different kinds of model parameters, the matrix $G_{0}$ and the memory parameters $d_{0,r}$, $r=1, \dots,p$. For this reason, deriving a concentration inequality becomes particularly challenging. Results for the graphical and thresholding local Whittle estimators require a concentration inequality on the event 
\begin{equation} \label{eq:mainevent}
\{\| \widehat{G}(\widehat{D})-G_{0} \|_{\max} > \delta \} .
\end{equation}
For this, our theoretical analysis reveals that the memory parameters $d_{0,r}$, $r=1, \dots,p$, have to be controlled simultaneously, and we propose to derive a concentration inequality on the event \eqref{eq:mainevent} by incorporating the event $\{\| \widehat{D}-D_{0}\|_{\max} \leq \varepsilon\}$, and then use multiple bounds of the probability of the event of interest \eqref{eq:mainevent} in terms of events which are representable as quadratic forms of i.i.d.\ sub-Gaussian random vectors. 
A key tool in our analysis is a uniform concentration inequality introduced by \cite{dicker2017}. For completeness, we present a slightly modified version of their result in Appendix \ref{s:uniformCI}.
\par
Before we present the main results, we introduce some further notation. We write the population analogue of $\widehat{G}(D)$ in \eqref{eq:G(D)} as
\begin{equation} \label{eq:Gtilde}
\widetilde{G}(D) = \frac{1}{m} \sum_{j=1}^{m} \lambda_{j}^{D-D_{0}} G_{0} \lambda_{j}^{D-D_{0}}
\end{equation}
and the respective univariate counterpart as $\widetilde{g}_{r}(d_{r})=\widetilde{G}_{rr}(D)$.
Furthermore, we write
\begin{equation} \label{eq:Tl}
L(d)=\int_{0}^{1} x^{2 d} dx = \frac{1}{2d+1}
\text{ for } d > -1/2.
\end{equation}
\par
We now present the aforementioned inequalities on the probability of the event \eqref{eq:mainevent}, which give insights into what kind of concentration inequalities are required to prove the desired consistency result on $\| \widehat{G}(\widehat{D})-G_{0} \|_{\max}$. A detailed analysis can be found in the proof of Proposition \ref{prop:Op} below.
For some $\delta, \varepsilon>0$,
\begin{align}
&
\Prob( \| \widehat{G}(\widehat{D})-G_{0} \|_{\max} > \delta ) \nonumber
\\ & =
\Prob \bigg( 
\{ \| \widehat{G}(\widehat{D})-G_{0} \|_{\max} > \delta \} \cap
\Big(\{ \| \widehat{D}-D_{0} \|_{\max} \leq \varepsilon \} \cup \{ \| \widehat{D}-D_{0} \|_{\max} > \varepsilon \} \Big)
\bigg) \nonumber
\\ & \leq
\Prob( \{
\| \widehat{G}(\widehat{D})-\widetilde{G}(\widehat{D}) \|_{\max} > \delta/2 \} \cap \{ \| \widehat{D}-D_{0} \|_{\max} \leq \varepsilon \} ) \nonumber
\\ & \hspace{1cm} +
\Prob\bigg( \Big(\{
\| \widetilde{G}(\widehat{D})-G_{0} \|_{\max} > \delta/2 \} \cap \{ \| \widehat{D}-D_{0} \|_{\max} \leq \varepsilon \} \Big)
\cup
\{ \| \widehat{D}-D_{0} \|_{\max} > \varepsilon \} \bigg) \nonumber
\\ & \leq
\Prob( \{
\| \widehat{G}(\widehat{D})-\widetilde{G}(\widehat{D}) \|_{\max} > \delta/2 \} \cap \{ \| \widehat{D}-D_{0} \|_{\max} \leq \varepsilon \} )
+ 
\Prob( \| \widehat{D}-D_{0} \|_{\max} > \eta )
\label{eq:ineq1}
\\ & \leq
\sum_{r,s=1}^{p} \Prob( \sup_{ D \in \Omega(\varepsilon) } 
| \widehat{G}_{rs}(D)-\widetilde{G}_{rs}(D) | > \delta/2 )
+ 
\Prob( \| \widehat{D}-D_{0} \|_{\max} > \eta ) \label{eq:ineq2}
\end{align}
with
\begin{equation*} 
%\label{eq:omegaeps}
\Omega(\varepsilon)= \{ D \in \mathcal{M}_{\diag} | \Delta_{1} I_{p} \preccurlyeq D \preccurlyeq \Delta_{2} I_{p} \text{ and } \| D - D_{0} \|_{\max} \leq \varepsilon\}
\end{equation*}
and $\eta=\min\{\varepsilon, \frac{\delta}{4} (\|G_{0}\| \log(N) \lambda_{m}^{-2\varepsilon} 
L(-\varepsilon))^{-1} \}$. The relation \eqref{eq:ineq1} will follow from Lemma \ref{le:FrobeniusGtoD}.
This way, the problem reduces to finding a uniform concentration inequality on $| \widehat{G}_{rs}(D)-\widetilde{G}_{rs}(D) |$ and a concentration inequality on $\| \widehat{D}-D_{0} \|_{\max}$. Instead of considering the maximum, we bound the probability of $\{\| \widehat{D}-D_{0} \|_{\max} > \eta\}$ componentwise for each $r=1,\dots,p$ as
\begin{equation} \label{eq:ineq3}
\begin{aligned}
\Prob( | \widehat{d}_{r}-d_{0,r}| > \eta )
& \leq
\Prob( | \widehat{g}_{r}(d_{0,r})-\widetilde{g}_{r}(d_{0,r}) | > \eta_{1} ) 
\\& \hspace{1cm} +
\Prob( \sup_{d \in \Theta_{1} } | \frac{1}{m} \sum_{j=1}^{m} \Big(\frac{j}{m}\Big)^{2d-2d_{0,r}} (\lambda_{j}^{2d_{0,r}}I_{X,rr}(\lambda_{j})-g_{0,r}) | > \eta_{2} )
\\& \hspace{2cm} +
\Prob( | \frac{1}{m} \sum_{j=1}^{m} (l_{j}-1) ( \lambda_{j}^{2d_{0,r}}I_{X,rr}(\lambda_{j})-g_{0,r}) | \mathds{1}_{\{d_{0,r} \geq \Delta_{1} + \frac{1}{2}\}} 
> \eta_{3} )
\end{aligned}
\end{equation}
with $\Theta_{1}$ as in \eqref{eq:defTheta1}; see Remark \ref{re:nonuniformly}. Furthermore, with $\ell= \exp(\frac{1}{m} \sum_{j=1}^{m} \log(j))$ and $\Delta$ as in Remark \ref{re:nonuniformly}, 
\begin{align} \label{eq:lj}
l_{j}=
\begin{cases}
\Big(\frac{j}{\ell}\Big)^{2(-\frac{1}{2}+\Delta)}, \hspace{0.2cm} &1 \leq j \leq \ell, \\
\Big(\frac{j}{\ell}\Big)^{2(\Delta_{1}-d_{0,r})}, \hspace{0.2cm} &\ell < j \leq m.
\end{cases}
\end{align}
For the sake of simplicity, $\eta_{1},\eta_{2},\eta_{3}>0$ in \eqref{eq:ineq3} and the arguments for \eqref{eq:ineq3} are not further specified; we refer to the proof of Proposition \ref{prop:Op} for more details.
The inequalities \eqref{eq:ineq2} and \eqref{eq:ineq3} reveal that it is enough to prove a uniform concentration inequality for an object of the form
\begin{equation} \label{eq:crucialterms}
\sup_{ D \in \Omega } 
| \widehat{H}_{rs}(D) - \widetilde{H}_{rs}(D) |
%\sup_{ D \in \Omega } 
%| A^{(1)}_{rs}(D) - \widetilde{A}^{(1)}_{rs}(D) |
\end{equation}
with $\widehat{H}_{rs}(D), \widetilde{H}_{rs}(D)$ denoting the $(r,s)$th elements of
\begin{equation} \label{eq:crucialterms1}
\widehat{H}(D) = \frac{1}{m} \sum_{j=1}^{m} t_{j}(D) I_{X}(\lambda_{j}) t_{j}(D)
\hspace{0.2cm}
\text{ and }
\hspace{0.2cm}
\widetilde{H}(D) = \frac{1}{m} \sum_{j=1}^{m} t_{j}(D) \lambda_{j}^{-D_{0}} G_{0} \lambda_{j}^{-D_{0}} t_{j}(D),
\end{equation}
where $t_{j}(D)=\diag(t_{j,1}(d_{1}),\dots,t_{j,p}(d_{p})) \in \mathcal{M}_{\diag}$ consist of suitable functions of $d$'s. The set $\Omega$ is of the form
\begin{equation} \label{eq:OmegaAB}
\Omega = \{ D \in \mathcal{M}_{\diag} ~|~ A \preccurlyeq D \preccurlyeq B \}
\end{equation}
with
\begin{equation} \label{eq:AB}
A=\diag(a_{1},\dots,a_{p}),
\hspace{0.2cm}
B=\diag(b_{1},\dots,b_{p}).
\end{equation}
The functions $t_{j,r}:[\Delta_{1},\Delta_{2}] \to [0,\infty)$ are assumed to be differentiable on $(\Delta_{1},\Delta_{2})$ with bounded derivatives.

From here on, \eqref{eq:crucialterms} can be separated into a probabilistic and a deterministic part as
\begin{equation} \label{eq:results_GhatGtilde}
\begin{aligned}
\sup_{ D \in \Omega } 
| \widehat{H}_{rs}(D) - \widetilde{H}_{rs}(D) | 
\leq
\sup_{ D \in \Omega} 
| \widehat{H}_{rs}(D) - \E \widehat{H}_{rs}(D) | +
\sup_{ D \in \Omega } 
| \E \widehat{H}_{rs}(D) - \widetilde{H}_{rs}(D) |.
\end{aligned}
\end{equation}
We treat both terms separately. For the first summand, we need a high probability upper bound stated in Lemma \ref{prop:supGhat}. 
On the other hand, the second term in \eqref{eq:results_GhatGtilde} is deterministic and an upper bound is given in Lemma \ref{prop:bias}.
Lemmas \ref{prop:supGhat} and \ref{prop:bias} are stated in Appendix \ref{s:proofsA}. Both are crucial to infer upper bounds on the 
probabilities in \eqref{eq:ineq2} and \eqref{eq:ineq3}. Those results are stated in Propositions \ref{prop:supGhat1}--\ref{prop:supGhat4} below and its proofs can be found in Appendix \ref{s:proofsA}.

\textit{Technical contributions:}
The following points highlight our main technical contributions and give some orientation of how the different appendices contribute. 

The statement for a probabilistic bound on the first summand of \eqref{eq:results_GhatGtilde} is stated in Lemma \ref{prop:supGhat} in Appendix \ref{s:proofsA}.
The key tool to handle the first summand of \eqref{eq:results_GhatGtilde} is a uniform concentration bound of \cite{dicker2017}, slightly reformulated to serve better our needs in Appendix \ref{s:uniformCI}. Our arguments above show how the incorporation of the event $\{ \| \widehat{D}-D_{0} \|_{\max} \leq \varepsilon \}$ allows one to get to the setting where that bound could potentially be applicable. Making the bound workable for the probabilities in \eqref{eq:ineq2} and \eqref{eq:ineq3} was another major challenge and can be found in Appendix \ref{se:B1}. The difficulties involved allowing for general dependence structure (short- and long-range dependence, and antipersistence) and, more importantly, developing non-asymptotic theory in terms of any dimension $p$ and sample size $N$. Though there are certainly many works on local Whittle estimation (and we adapt some of their techniques), non-asymptotic results are not available even for the one-dimensional case $p=1$. Key to our non-asymptotic developments are bounds of independent interest on various autocovariance matrices under general dependence assumptions in Appendix \ref{se:B3}. Those results are derived by replacing the matrices by integral operators. We believe that those inequalities are crucial in order to derive non-asymptotic theory in any context involving
high-dimensional long-range dependence.

To deal with the bounds on the second summand of \eqref{eq:results_GhatGtilde} and deterministic part of the quantities in \eqref{eq:ineq2} and \eqref{eq:ineq3}, we state Lemma \ref{prop:bias}. The proof involves the derivation of a series of non-asymptotic bounds for the periodogram in Appendix \ref{se:B2}.

\subsection{Statements} \label{se:statements}

In this section we formally state our main results. In order to keep the statements as simple as possible, we moved the expressions of some quantities to Appendix \ref{se:tablesResults}.

We further introduce the following quantity which will allow us to express our bounds in a simplified way and emphasize the necessary distinction between different ranges of the  memory parameters as will become clearer in the proofs
\begin{equation} \label{eq:delta-0-1/4}
\begin{gathered}
\widebar{\Delta}_{N} =  \max_{r=1,\dots,p} (\mathds{1}_{\{d_{0,r} \leq 0 \}} + \log(N) \mathds{1}_{\{d_{0,r} > 0\}}).
\end{gathered}
\end{equation}

\begin{proposition} \label{prop:supGhat1}
Let $\{X_{n}\}_{n \in \ZZ}$ be a $p$-dimensional, stationary, centered time series with spectral density $f_{X}$ and suppose Assumptions \ref{ass:f0}--\ref{ass:antip} are satisfied. Then, for any $\varepsilon \in (0,\min\{\Delta_2,-\Delta_1,q\})$ with $q$ as in Assumption \ref{ass:G-G}, there are positive constants $c_{1},c_{2}$ such that for any $\bm{C} \geq 1$,
\begin{equation*}
\Prob \Big( \sup_{D \in \Omega(\varepsilon) }
| \widehat{G}_{rs}(D) - \widetilde{G}_{rs}(D) | > \nu \Big)
%| \widehat{G}_{rs}(D) - \E \widehat{G}_{rs}(D) | > \nu \Big)
\leq 
c_{1} p^{-c_{2}\bm{C}} \label{eq:cor1}
\end{equation*}
for
\begin{equation} \label{eq:nuinprop1}
\nu= \bm{C} \vertiii{G}
\sqrt{\frac{\log(p)}{\mathcal{R}_{1}}} 
+
\mathcal{T}_{1}(\varepsilon),
\end{equation}
where $\mathcal{R}_{1}= \min\{ (\log(N)\widebar{\Delta}_{N})^{-1} \mathcal{R}_{11}, (\log(N)\widebar{\Delta}_{N})^{-2} \mathcal{R}_{12} \}$ with $\widebar{\Delta}_{N}$ as in  \eqref{eq:delta-0-1/4} and $\mathcal{R}_{1i},\ i=1,2$ are characterized in Table \ref{label 1}. A representation of $\mathcal{T}_{1}(\varepsilon)$ can be found in Table \ref{label 2}.
\end{proposition}

The parameter $\varepsilon$ controls the deviation of the estimated memory parameters $\widehat{D}$ around $D_{0}$, and will be chosen appropriately in Proposition \ref{prop:Op} below.
To ensure that the deviation of $\widehat{G}_{rs}(D)$ around $\widetilde{G}_{rs}(D)$ can be controlled, 
the quantities which characterize $\nu$ in \eqref{eq:nuinprop1} have to satisfy $\mathcal{R}_{1} \succsim \vertiii{G}^2 \log(p)$ and 
$mN^{-2\varepsilon} \succsim \big(1 + \frac{1}{2\varepsilon} \big) Q_{m}$ with $Q_{m}$ defined in \eqref{eq:mR:Qrs} in Appendix \ref{se:tablesResults}. 
The quantities in $\nu$ can be expected to satisfy those assumptions since our bounds are sharp enough to get $\mathcal{R}_{1} \to \infty$ and $\mathcal{T}_{1}(\varepsilon) \to 0$ as $N \to \infty$. Those asymptotics are crucial in order to achieve consistency which entails a negligible bias. This observation can be made not only for Proposition \ref{prop:supGhat1} but as well in the similarly structured Propositions \ref{prop:supGhat2}--\ref{prop:supGhat4} below. 

\begin{remark}
Proposition \ref{prop:supGhat1} and subsequent results provide non-asymptotic bounds when estimating quantities of interest. In the considered setting, $\{ X_{n} \}_{n \in \ZZ}$ is a stationary series with a spectral density $f_{X}(\lambda)$ and observed for $n=1,\dots,N$, and of fixed dimension $p$. But note that our non-asymptotic bounds are expressed in terms of $p,N,\vertiii{G}$ and other quantities. When $p$ changes, the dependence structure of $\{ X_{n} \}_{n \in \ZZ}$ also changes and our bounds adjust through changing $\vertiii{G}$ and those other quantities (e.g. $\bm{c}_{G,1}, \bm{c}_{G,2}$ in Assumptions \ref{ass:G-G} and \ref{ass:derivative}).
The same with changing $N$. We note that because of the term $\log(p)/\mathcal{R}_{1}$, in \eqref{eq:nuinprop1}, the obtained bounds suggest consistent estimation in a typical high-dimensional regime where $p$ is much larger than $N$, but $\log(p)$ is much smaller than $N$ (or the power of $N$). We also note that because of the constants $\bm{c}_{G,1}, \bm{c}_{G,2}$ in Proposition \ref{prop:supGhat1}
and similar subsequent results are absolute in the sense that they do not depend on $p,N$ and the underlying stationary series; the dependence on the latter is captured through the other quantities in the bounds.
\end{remark}

The following three propositions give upper bounds on the probabilities in \eqref{eq:ineq3}. Those are the probabilities required to control the estimates for the memory parameters $D_{0}$.

\begin{proposition} \label{prop:supGhat2}
Let $\{X_{n}\}_{n \in \ZZ}$ be a $p$-dimensional, stationary, centered time series with spectral density $f_{X}$ and suppose Assumptions \ref{ass:f0}--\ref{ass:antip} are satisfied. Then, there are positive constants $c_{1},c_{2}$ such that for any $\bm{C} \geq 1$,
\begin{equation*} 
%\label{eq:cor2}
\Prob \Big( 
| \widehat{g}_{r}(d_{0,r}) - g_{0,r} | > \nu_{1} \Big)
%| \widehat{g}_{r}(d_{0,r}) - \E \widehat{g}_{r}(d_{0,r}) | > \nu_{1} \Big)
\leq 
c_{1} p^{-c_{2}\bm{C}}
\end{equation*}
for
\begin{equation} \label{eq:nu1}
\nu_{1}= \bm{C}
\vertiii{G} \sqrt{\frac{\log(p)}{\mathcal{R}_{2}}} 
+
\mathcal{T}_{2},
\end{equation}
where $\mathcal{R}_{2}=\min\{ \widebar{\Delta}^{-1}_{N} \mathcal{R}_{21}, \widebar{\Delta}^{-2}_{N} \mathcal{R}_{22} \}$ with $\widebar{\Delta}_{N}$ as in  \eqref{eq:delta-0-1/4} and $\mathcal{R}_{2i},\ i=1,2$ are characterized in Table \ref{label 1}. A representation of $\mathcal{T}_{2}$ can be found in Table \ref{label 2}.
\end{proposition}

In order to ensure meaningful estimation, the quantities which characterize $\nu_{1}$ in \eqref{eq:nu1} have to satisfy $\mathcal{R}_{2} \succsim \vertiii{G}^2 \log(p)$ and $m \succsim \log(m) Q_{m}$ with $Q_{m}$ defined in \eqref{eq:mR:Qrs} in Appendix \ref{se:tablesResults}.
\par
For the following proposition, we use the set $\Theta_{1}$ in \eqref{eq:defTheta1}, which is characterized by some $\Delta > 0$.

\begin{proposition} \label{prop:supGhat3}
Let $\{X_{n}\}_{n \in \ZZ}$ be a $p$-dimensional, stationary, centered time series with spectral density $f_{X}$ and suppose Assumptions \ref{ass:f0}--\ref{ass:antip} are satisfied. Then, there are positive constants $c_{1},c_{2}$ such that for any $\bm{C} \geq 1$,
\begin{equation*} 
%\label{eq:cor3}
\Prob \Big( \sup_{d \in \Theta_{1} } 
| \frac{1}{m} \sum_{j=1}^{m} \Big(\frac{j}{m}\Big)^{2d-2d_{0,r}} \lambda_{j}^{2d_{0,r}} (I_{X,rr}(\lambda_{j})-\lambda_{j}^{-2d_{0,r}}g_{0,r}) |
> \nu_{2} \Big)
\leq 
c_{1} p^{-c_{2}\bm{C}}
\end{equation*}
for
\begin{equation} \label{eq:nu2}
\nu_{2}= 
\bm{C} \vertiii{G} \sqrt{\frac{\log(p)}{\mathcal{R}_{3}}} 
+
\mathcal{T}_{3},
\end{equation}
where $\mathcal{R}_{3}= \min\{ (\log(m) \widebar{\Delta}_{N})^{-1} \mathcal{R}_{31}, (\log(m) \widebar{\Delta}_{N})^{-2} \mathcal{R}_{32} \}$ with $\widebar{\Delta}_{N}$ as in  \eqref{eq:delta-0-1/4} and $\mathcal{R}_{3i},\ i=1,2$ are characterized in Table \ref{label 1}. A representation of $\mathcal{T}_{3}$ can be found in Table \ref{label 2}.
\end{proposition}

In order to ensure meaningful estimation, the quantities which characterize $\nu_{2}$ in \eqref{eq:nu2} have to satisfy $\mathcal{R}_{3} \succsim \vertiii{G}^2 \log(p)$ and $m^{2\widetilde{\Delta}_{r}}  \succsim Q_{m}$ with $Q_{m}$ defined in \eqref{eq:mR:Qrs} in Appendix \ref{se:tablesResults}. 
\par
The next proposition gives a bound on the third probability in \eqref{eq:ineq3}. Recall the definitions of $\ell$ and $l_{j}$ given in \eqref{eq:lj}.

\begin{proposition} \label{prop:supGhat4}
Let $\{X_{n}\}_{n \in \ZZ}$ be a $p$-dimensional, stationary, centered time series with spectral density $f_{X}$ and suppose Assumptions \ref{ass:f0}--\ref{ass:antip} are satisfied. Then, there are positive constants $c_{1},c_{2}$ such that for any $\bm{C} \geq 1$,
\begin{equation*} 
%\label{eq:cor4}
\Prob \Big( | \frac{1}{m} \sum_{j=1}^{m} (l_{j}-1) \lambda_{j}^{2d_{0,r}} (I_{X,rr}(\lambda_{j})-\lambda_{j}^{-2d_{0,r}}g_{0,r})| \mathds{1}_{\{d_{0,r} \geq \Delta_{1} + \frac{1}{2} \}}
> \nu_{3} \Big)
\leq 
c_{1} p^{-c_{2}\bm{C}}
\end{equation*}
for
\begin{equation} \label{eq:nu3}
\nu_{3}= \bm{C}
\vertiii{G} \sqrt{\frac{\log(p)}{\mathcal{R}_{4}}} 
+
\mathcal{T}_{4},
\end{equation}
where $\mathcal{R}_{4}=\min\{ \widebar{\Delta}_{N}^{-1} \mathcal{R}_{41}, \widebar{\Delta}_{N}^{-2} \mathcal{R}_{42} \}$ with $\widebar{\Delta}_{N}$ as in  \eqref{eq:delta-0-1/4} and $\mathcal{R}_{4i},\ i=1,2$ are characterized in Table \ref{label 1}. A representation of $\mathcal{T}_{4}$ can be found in Table \ref{label 2}.
\end{proposition}

In order to ensure meaningful estimation, the quantities which characterize $\nu_3$ in \eqref{eq:nu3} have to satisfy $\mathcal{R}_{4} \succsim \vertiii{G}^2 \log(p)$, $m \succsim \ell^{1-2\Delta} Q_{m}$ and $m \succsim \log(m) Q_{m}$ with $Q_{m}$ defined in \eqref{eq:mR:Qrs} in Appendix \ref{se:tablesResults}.
\par
Propositions \ref{prop:supGhat1}--\ref{prop:supGhat4} combined together enable us to obtain a consistency result for \eqref{eq:mainevent}, which is stated in the following proposition. Recall the definition of $\underline{k}$ given in \eqref{as:eigenvalueG0} and of the function $L$ in \eqref{eq:Tl}.

\begin{proposition} \label{prop:Op}
Suppose that the assumptions in Propositions \ref{prop:supGhat1}--\ref{prop:supGhat4} hold.
Then, there are positive constants $c_{1},c_{2}$ such that for any $\bm{C} \geq 1$,
\begin{equation} \label{eq:consFrob}
\Prob( \| \widehat{G}(\widehat{D})-G_{0} \|_{\max} > \delta )
\leq
c_{1} p^{2-c_{2}\bm{C}}
\end{equation}
for
\begin{equation} \label{eq:deltaepsilon}
\begin{aligned}
\delta=
\max \{ 
2 \nu,  \varepsilon 4 \|G_{0}\| \log(N) \lambda_{m}^{-2\varepsilon} L(-\varepsilon) \},
\hspace{0.2cm}
\varepsilon= \max_{i=1,2,3} \eta_{i},
\end{aligned}
\end{equation}
with $\nu$ as in \eqref{eq:nuinprop1},
\begin{equation} \label{eq:eta1eta2eta3}
\begin{aligned}
\eta_{1}^{2} = 
8(\mathcal{V}_{1}(m) \underline{k} )^{-1} \nu_{1},
\hspace{0.2cm}
\eta_{2}^{2} =
8(\mathcal{V}_{1}(m)L(\Delta_{2}-\Delta_{1}) \underline{k} )^{-1} \nu_{2},
\hspace{0.2cm}
\eta_{3}^{2} =
(\mathcal{V}_{2}(m) \underline{k} )^{-1} \nu_{3},
\end{aligned}
\end{equation}
where $\nu_{i}$, $i=1,2,3$ are as in \eqref{eq:nu1}, \eqref{eq:nu2} and \eqref{eq:nu3} and 
\begin{equation} \label{eq:nu1nu2}
\mathcal{V}_{1}(m)=\frac{1}{3}\frac{1}{m^4} \sum_{i,j=1}^{m} (i-j)^2,
\hspace{0.2cm}
\mathcal{V}_{2}(m)=\frac{1}{m} \sum_{j=1}^{m} (l_{j}-1),
\end{equation}
and it is assumed that $\varepsilon \in (0,\frac{1}{2})$.
\end{proposition}
In order to ensure meaningful estimation, the term $\mathcal{V}_{2}(m)$ needs to be positive. This is proven in Lemma \ref{le:V2positive}.

\begin{remark}
In Corollary \ref{cor:Op}, we formally state an analogous result to Proposition \ref{prop:Op} under the assumption that the underlying process is either short- or long-range dependent. Note that in contrast to Proposition \ref{prop:Op}, Corollary \ref{cor:Op} does not require Assumption \ref{ass:antip}. In particular, one can infer an asymptotic result without requiring any further assumptions on the relation between $m$ and $N$ besides $\frac{1}{m} + \frac{m}{N} \to \infty$ for $N \to \infty$ which coincides with Assumption 4 in \cite{robinson1995gaussian}.
A numeric illustration of our non-asymptotic results is given in Appendix \ref{appF:rate} in terms of Corollary \ref{cor:Op}.
\end{remark}

The following propositions give non-asymptotic consistency results for the graphical and thresholding local Whittle estimators, respectively.

\begin{proposition} \label{prop:Opspectral}
Let $\{X_{n}\}_{n \in \ZZ}$ be a $p$-dimensional, stationary, centered time series with spectral density $f_{X}$ and suppose Assumptions \ref{ass:f0}--\ref{ass:antip} are satisfied. 
Then, there are positive constants $c_{1},c_{2}$, such that choosing a threshold $\rho = \delta$ as in \eqref{eq:deltaepsilon} yields, for any $\bm{C} \geq 1$, 
\begin{equation*} 
\begin{gathered}
\Prob(\| T_{\rho}(\widehat{G}(\widehat{D}))-G_{0} \|^2_{F} > 13 p \| G_{0} \|_{a}^{a} \rho^{2-a} )
\leq
c_{1} p^{2-c_{2}\bm{C}},
\\
\Prob(\| T_{\rho}(\widehat{G}(\widehat{D}))-G_{0} \| > 7 \| G_{0} \|_{a}^{a} \rho^{1-a} )
\leq
c_{1} p^{2-c_{2}\bm{C}}
\end{gathered}
\end{equation*}
for any $a \in [0,1)$ and $T_{\rho}$ as in \eqref{eq:thresholding}.
\end{proposition}

\begin{proposition} \label{prop:OpPrecision}
Let $\{X_{n}\}_{n \in \ZZ}$ be a $p$-dimensional, stationary, centered time series with spectral density $f_{X}$ and suppose Assumptions \ref{ass:f0}--\ref{ass:antip} are satisfied. 
Then, there are positive constants $c_{1},c_{2}$, such that choosing a penalty parameter $\rho = \delta$ as in \eqref{eq:deltaepsilon} yields, for any $\bm{C} \geq 1$, 
\begin{equation*} 
\Prob( \| \widehat{P}_{\rho}(\widehat{D})-P_{0} \|^{2}_{F} >  \frac{16^2}{\underline{k}^4} (p+\operatorname{s}) \rho^{2} )
\leq
c_{1} p^{2-c_{2}\bm{C}}
\end{equation*}
with $\operatorname{s}$ as in \eqref{as:cardianlityS} and $\frac{16}{\underline{k}^2} \sqrt{p+\operatorname{s}}  \rho \leq \| P_{0} \|$.
\end{proposition}

Propositions \ref{prop:Opspectral} and \ref{prop:OpPrecision} can be expressed in terms of the quantities in Corollary \ref{cor:Op} when allowing only for long- and short-range dependence. That is, $\delta$ in Propositions \ref{prop:Opspectral} and \ref{prop:OpPrecision} can be chosen as in Corollary \ref{cor:Op} defined in terms of \eqref{eq:cor:Rs}.

In contrast to Proposition \ref{prop:Opspectral}, Proposition \ref{prop:OpPrecision} states only a result for the Frobenius norm. A result on the operator norm can be inferred based on the inequality $\| \cdot \| \leq \| \cdot \|_{F}$. However, the operator norm provides the same convergence rate as the Frobenius norm. In contrast, for thresholding estimators, the convergence rates differ by one $p$; see Proposition \ref{prop:Opspectral}. 
The literature on covariance estimation has addressed this problem by considering alternative estimators. Section \ref{se:alternativeestimators} below presents a modified graphical local Whittle estimator based on estimating the coherence matrix and a constrained $l_{1}$-minimization for inverse matrix estimation
(CLIME) version.

\begin{remark}
The probability bounds in Propositions \ref{prop:supGhat1}--\ref{prop:OpPrecision} are given in terms of $p$ which might suggest consistency only in the limit of $p \to \infty$, as long as $\bm{C}$ is large enough. Note, however, that $\bm{C}=\bm{C}_{m}$ may depend on $m$ and $N$, and enters in our choices of $\nu$ in \eqref{eq:nuinprop1} and $\nu_{i}$, $i=1,2,3$, in \eqref{eq:nu1}, \eqref{eq:nu2} and \eqref{eq:nu3}. Under suitable assumptions, one can in fact have 
\begin{equation*}
\bm{C}_{m} \sqrt{ \frac{\log(p)}{\mathcal{R}_{i}}} \to 0
\hspace{0.2cm} \text{ and } \hspace{0.2cm}
\bm{C}_{m} \to \infty
\hspace{0.2cm} \text{ as } \hspace{0.2cm}
m \to \infty
\hspace{0.2cm} \text{ for } \hspace{0.2cm}
i =1, \dots,4, 
\end{equation*}
in Propositions \ref{prop:supGhat1}--\ref{prop:supGhat4}, so that the resulting probability bounds are small even for fixed low-dimensional $p$. In the latter case, even when $p=1$, these results also provide new non-asymptotic exponential bounds, for example, on the probability of $\{ 
\| \widehat{G}(\widehat{D}) - G_{0} \|_{\max} > \delta \}$.
\end{remark}

\begin{remark} \label{re:nonuniformly}
The work by \cite{robinson1995gaussian} concerns consistency results for the univariate local Whittle estimators for $G_{0}$ and $D_{0}$. The consistency of $\widehat{D}$ is stated in Theorem 1 in \cite{robinson1995gaussian}. The proof in the asymptotic regime mentions that the function $R(D)$ behaves nonuniformly around $D=d_{0}-\frac{1}{2}$. For this reason, one has to consider the cases $d_{0}-\frac{1}{2} < \Delta_{1}$ and $d_{0}-\frac{1}{2} \geq \Delta_{1}$ to separate the set $\Theta$ as $\Theta = \Theta_{1} \cup \Theta_{2}$ with
\begin{equation} \label{eq:defTheta1}
\Theta_{1}=
\begin{cases}
\{ d ~|~ d_{0} - \frac{1}{2} + \Delta \leq d \leq \Delta_{2} \}, &\hspace{0.2cm} \text{ if } d_{0} \geq \Delta_{1} + \frac{1}{2}, \\
\{ d ~|~ \Delta_{1} \leq d \leq \Delta_{2} \}, &\hspace{0.2cm} \text{ if } d_{0} < \Delta_{1} + \frac{1}{2},
\end{cases}
\end{equation}
\begin{equation} \label{eq:defTheta2}
\Theta_{2}=
\begin{cases}
\{ d ~|~ \Delta_{1} \leq d < d_{0} - \frac{1}{2} + \Delta \}, &\hspace{0.2cm} \text{ if } d_{0} \geq \Delta_{1} + \frac{1}{2}, \\
\emptyset, &\hspace{0.2cm} \text{ if } d_{0} < \Delta_{1} + \frac{1}{2},
\end{cases}
\end{equation}
where $\Delta \in (0,\Delta_{2})$.
As displayed in \eqref{eq:defTheta1}, separating $\Theta$ is only necessary if $d_{0} < \Delta_{1} + \frac{1}{2}$. The case $d_{0} < \Delta_{1} + \frac{1}{2}$ includes $d_{0} \leq 0$ since $-\frac{1}{2} < \Delta_{1}$. Note that $d_{0} \leq 0$ coincides with the prior knowledge that the observed time series is not long-range dependent. The sets $\Theta_{1}$ and $\Theta_{2}$ are used to determine the range of admissible estimates $d$ for an individual component series $X_{r,n}$. Then, $\Theta_{1}$ and $\Theta_{2}$ depend on $d_{0,r}$ instead of $d_{0}$. Throughout the paper we do not reflect the dependence on $r$ in the notation of $\Theta_{1}$ and $\Theta_{2}$.
\end{remark}

\subsection{Alternative estimators for precision matrix} \label{se:alternativeestimators}
We study here a modified graphical local Whittle estimator based on estimating the coherence matrix and a CLIME version; see Sections \ref{se:Spice} and \ref{se:Clime}. The section and in particular the corresponding proofs in Appendix \ref{s:proofsA} also emphasize the value of our Proposition \ref{prop:Op} since it can be used to infer consistency results even for modified versions of penalized local Whittle estimators.

\subsubsection{Modified graphical local Whittle} \label{se:Spice}
As known for covariance matrices, the rate of convergence can be improved for the operator norm by considering the correlation matrix instead; see \cite{rothman2008sparse} and \cite{shu2019estimation} for temporally correlated data. Analogously, we can consider the coherence matrix. Let $G_{0} = W_{0} \Gamma_{0} W_{0}$, where $W_{0}=\diag(G_{0,11}^{1/2},\dots, G_{0,pp}^{1/2})$ and $\Gamma_{0}$ is the true coherence matrix. Then, the precision matrix satisfies $P_{0} = W_{0}^{-1} \Gamma_{0}^{-1} W_{0}^{-1}$ and therefore $K_{0} := \Gamma_{0}^{-1} = W_{0} P_{0} W_{0}$. We write $\widehat{\Gamma} = \widehat{\Gamma}(\widehat{D})$ and $\widehat{W} = \widehat{W}(\widehat{D})$ for their sample counterparts, and indicate their dependence on the matrix $D$.
The matrix $K_{0}$ can then be estimated as
\begin{equation} \label{eq:Khatrho}
\widehat{K}_{\rho} = \argmin{K \succ 0} \ \ell^{\Gamma}_{\rho}(\widehat{D},K),
\end{equation}
where $\widehat{D}$ is estimated univariately by \eqref{eq:univLWd} and
\begin{equation} \label{eq:onjfunctionGamma}
\begin{aligned}
\ell^{\Gamma}_{\rho}(D,K)=-\frac{1}{m} \sum_{j=1}^{m} &\log|\lambda_{j}^{D} K \lambda_{j}^{D}|
+\tr( \widehat{\Gamma}(D)K) + \rho \| K \|_{1,off}.
\end{aligned}
\end{equation}
Then, we can define a modified coherence-based graphical local Whittle estimator
\begin{equation} \label{eq:modGLW}
\widehat{P}^{M}_{\rho} = \widehat{W}^{-1} \widehat{K}_{\rho} \widehat{W}^{-1},
\end{equation}
where $\widehat{W}(\widehat{D}) = \diag(\widehat{G}_{11}^{1/2}(\widehat{D}),\dots, \widehat{G}_{pp}^{1/2}(\widehat{D}))$.
The following statement gives a non-asymptotic consistency result on the spectral norm.
\begin{proposition} \label{prop:OperatorPrecisionmodified}
Let $\{X_{n}\}_{n \in \ZZ}$ be a $p$-dimensional, stationary, centered time series with spectral density $f_{X}$ and suppose Assumptions \ref{ass:f0}--\ref{ass:antip} are satisfied. 
Then, there are positive constants $c_{1},c_{2}$, such that choosing a penalty parameter $\rho = \delta$ as in \eqref{eq:deltaepsilon} yields, for any $\bm{C} \geq 1$, 
\begin{equation*} 
\Prob( \| \widehat{P}^{M}_{\rho}(\widehat{D})-P_{0} \| >  30 \times 48 \max\{1, 1 / \underline{k}^9 \} \max\{1, \| K_{0} \| \} \sqrt{\operatorname{s}} \ \rho )
\leq
c_{1} p^{2-c_{2}\bm{C}}
\end{equation*}
with $\operatorname{s}$ as in \eqref{as:cardianlityS} and $48 \max\{1, 1/\underline{k}^4 \} \sqrt{\operatorname{s}}  \rho \leq \| K_{0} \|$.
\end{proposition}

In contrast to Proposition \ref{prop:OpPrecision}, Proposition \ref{prop:OperatorPrecisionmodified} provides a convergence rate for the spectral norm which shows that the modified graphical local Whittle estimator \eqref{eq:modGLW} can achieve the same convergence rate as the thresholding local Whittle estimator.

A handy result to prove Proposition \ref{prop:OperatorPrecisionmodified} and of independent interest is the following lemma which gives a consistency result for the coherence matrix estimator in \eqref{eq:Khatrho}.

\begin{lemma} \label{col:OperatorPrecisionmodified}
Let $\{X_{n}\}_{n \in \ZZ}$ be a $p$-dimensional, stationary, centered time series with spectral density $f_{X}$ and suppose Assumptions \ref{ass:f0}--\ref{ass:antip} are satisfied. 
Then, there are positive constants $c_{1},c_{2}$, such that choosing a penalty parameter $\rho = \delta$ as in \eqref{eq:deltaepsilon} with $N, p$ such that $\delta < 1$ yields, for any $\bm{C} \geq 1$, 
\begin{equation*} 
\Prob( \| \widehat{K}_{\rho}(\widehat{D}) - K_{0} \|^2_{F} >  (48 \max\{1, 1/\underline{k}^4 \} )^2 \operatorname{s} \rho^{2} )
\leq
c_{1} p^{2-c_{2}\bm{C}}
\end{equation*}
with $\operatorname{s}$ as in \eqref{as:cardianlityS} and $48 \max\{1, 1/\underline{k}^4 \}  \sqrt{\operatorname{s}}  \rho \leq \| K_{0} \|$. 
\end{lemma}
Note that we assume $N, p$ such that $\delta < 1$ only to achieve a simplified representation of the result. In general, it is possible to state the result for any fixed $N,p$.

\subsubsection{CLIME estimation} \label{se:Clime}
CLIME estimation for i.i.d. samples was introduced in \cite{cai2011constrained} and further studied in \cite{shu2019estimation} allowing for temporal correlation. We adopt their approach to the spectral domain and set $\widehat{\Theta} = (\widehat{\theta}_{rs})_{r,s =1,\dots,p}$ to be the solution of the minimization problem
\begin{equation*}
\min \| \Theta \|_{1} 
\hspace{0.2cm}
\text{ subject to }
\hspace{0.2cm}
\| \widehat{G}(\widehat{D}) \Theta - I_{p} \|_{\max} \leq \rho,
\end{equation*}
where $\rho$ is a tuning parameter. Then, the CLIME estimator is defined as
\begin{equation} \label{eq:CLIMEestimator}
\widehat{P}^{C}_{\rho} = (\widehat{\theta}^{C}_{rs})_{r,s =1,\dots,p}
\hspace{0.2cm}
\text{ with }
\hspace{0.2cm}
\widehat{\theta}_{rs}^{C} = \widehat{\theta}_{sr}^{C}
= 
\widehat{\theta}_{rs} \mathds{1}_{ \{ |\widehat{\theta}_{rs}| \leq |\widehat{\theta}_{sr}| \} }
+
\widehat{\theta}_{sr} \mathds{1}_{ \{ |\widehat{\theta}_{rs}| > |\widehat{\theta}_{sr}| \} }.
\end{equation}

We impose the sparsity assumption \eqref{eq:sparsityconditionG0} on $P_{0}$.

\begin{proposition} \label{prop:OperatorCLIME}
Let $\{X_{n}\}_{n =1,\dots,N}$ be a $p$-dimensional, stationary, centered time series with spectral density $f_{X}$ and suppose Assumptions \ref{ass:f0}--\ref{ass:antip} are satisfied. 
Then, there are positive constants $c_{1},c_{2}$, such that choosing a penalty parameter $\rho = \delta$ as in \eqref{eq:deltaepsilon} yields, for any $\bm{C} \geq 1$, 
\begin{equation*} 
\Prob( \| \widehat{P}^{C}_{\rho}(\widehat{D}) -P_{0} \|^{2} >  24 \| P_{0} \|_{1}^{1-a} \| P_{0} \|^{a}_{a} \rho^{1-a} )
\leq
c_{1} p^{2-c_{2}\bm{C}}
\end{equation*}
for any $a \in [0,1)$.
\end{proposition}

\subsection{Graphical model selection consistency}
\label{se:modelselection}
We give here results on consistent recovery of the sparsity pattern and sign consistency for our estimators. For the long-run variance matrix estimation, we focus on the thresholding local Whittle estimator. For the precision matrix, we consider a thresholded CLIME estimator and conclude with a discussion on consistent graph recovery for other estimators. The proofs of the statements in this section can be found in Appendix \ref{s:proofsA}.

The following proposition gives a non-asymptotic result for consistent graph recovery of the thresholding local Whittle.
\cite{rothman2009generalized} consider covariance matrix estimation for i.i.d. $p$-dimensional random vectors in a high-dimensional regime. In particular, their Theorem 2 states that the thresholding operator consistently recovers the sparsity pattern. The proof of Theorem 2 in \cite{rothman2009generalized} is generic and proves, combined with our Proposition \ref{prop:Op}, consistent recovery of the sign and sparsity pattern.
\begin{proposition} \label{prop:revoverythreshold}
Let $\{X_{n}\}_{n =1,\dots,N}$ be a $p$-dimensional, stationary, centered time series with spectral density $f_{X}$ and suppose Assumptions \ref{ass:f0}--\ref{ass:antip} are satisfied. 
Then, there are positive constants $c_{1},c_{2}$, such that choosing a threshold $\rho = \delta$ as in \eqref{eq:deltaepsilon} yields, for any $\bm{C} \geq 1$, 
\begin{equation} \label{eq:GMSC:TH1}
\Prob\left( T_{\rho}(\widehat{G}_{rs}(\widehat{D})) = 0 \text{ for all } r,s \text{ such that } G_{0,rs} = 0 \right)
\geq
1-c_{1} p^{2-c_{2}\bm{C}}.
\end{equation}
If we additionally assume that all non-zero elements of $G_{0}$ satisfy $|G_{0,rs}| > \tau$, where $\tau$ is of the same order as $\rho$, we have
\begin{equation} \label{eq:GMSC:TH2}
\Prob\left( \sign( T_{\rho}(\widehat{G}_{rs}(\widehat{D})) G_{0,rs} ) = 1  \text{ for all } r,s \text{ such that } G_{0,rs} \neq 0 \right) \geq
1-c_{1} p^{2-c_{2}\bm{C}}.
\end{equation}
\end{proposition}

The CLIME estimator in \eqref{eq:CLIMEestimator} can be modified to recover the support of the precision matrix. More precisely, we conduct an additional thresholding step by applying \eqref{eq:thresholding} to $\widehat{P}^{C}_{\rho}(\widehat{D})$ in \eqref{eq:CLIMEestimator}. This procedure follows Section 4 in \cite{cai2011constrained} who considered inverse covariance estimation for i.i.d.\ data. Subsequently, \cite{shu2019estimation} used a thresholded CLIME for inverse covariance estimation under temporal dependence; see Theorem 5 in \cite{shu2019estimation} for their result on consistent sparsity and sign recovery. Following their arguments, we can state the following proposition.

\begin{proposition} \label{prop:revoveryCLIME}
Let $\{X_{n}\}_{n \in \ZZ}$ be a $p$-dimensional, stationary, centered time series with spectral density $f_{X}$ and suppose Assumptions \ref{ass:f0}--\ref{ass:antip} are satisfied. 
Then, there are positive constants $c_{1},c_{2}$, such that choosing a penalty parameter $\rho = \delta$ as in \eqref{eq:deltaepsilon} yields, for any $\bm{C} \geq 1$, 
\begin{equation} \label{eq:GMSC:CLIME1}
\Prob\left( T_{\tau}(\widehat{P}^{C}_{\rho}(\widehat{D})) = 0 \text{ for all } r,s \text{ such that } P_{0,rs} = 0 \right)
\geq
1-c_{1} p^{2-c_{2}\bm{C}}.
\end{equation}
If we additionally assume that all non-zero elements of $G_{0}$ satisfy $|G_{0,rs}| > \tau$, where $\tau$ is of the same order as $\rho$, we have
\begin{equation} \label{eq:GMSC:CLIME2}
\Prob\left( \sign( T_{\tau}(\widehat{\theta}^{C}_{rs}) P_{0,rs} ) = 1  \text{ for all } r,s \text{ such that } P_{0,rs} \neq 0 \right) \geq
1-c_{1} p^{2-c_{2}\bm{C}}.
\end{equation}
\end{proposition}

For the graphical local Whittle estimation, consistent recovery of the sparsity pattern is not expected. This is discussed in \cite{zou2006adaptive} for the classical LASSO in a regression context.
There have been several approaches to modify the classical LASSO to achieve consistent graph recovery. For instance the consideration of an adaptive version \citep{zou2006adaptive} and a thresholded LASSO \citep{zhou2010thresholded,ravikumar2011high,wang2021thresholded}. 

An adaptive version is based on a weighted penalty, where the weights are data driven using a preliminary estimator. Though our results are expected to be helpful to prove consistency for an adaptive version, a detailed investigation goes beyond the scope of this work. A related discussion on the difficulties of proving consistent recovery of the sparsity pattern with help of an extended Bayesian information criterion can be found in Remark \ref{re:eBIC}.

A thresholding graphical local Whittle estimator is expected to consistently recover the sparsity pattern. It is similar in flavor to the thresholded CLIME and a consistency result can be inferred with help of Proposition \ref{prop:Op}.

\subsection{Comparison to existing results} \label{se:comparison}
In this section, we compare our results to related work. Existing results are either for short-range dependent time series or, if they allow for stronger temporal and spatial dependence, the dependence measure is characterized by an unknown and unestimated quantity. In particular, results for data with stronger dependence structure have only been derived in the time domain.

\textit{\cite{Sun2018:LargeSpectral}:} 
We recover recently proven results on spectral density estimation at frequency zero for short-range dependent time series. \cite{Sun2018:LargeSpectral} supposed that $D_{0} \equiv 0$ and could prove results in a $\log(p)/m \to 0$ regime. We get the same result by setting $\Delta_{1}= \Delta_{2}= \varepsilon=0$ in \eqref{eq:deltaepsilon}. 
Strictly speaking, we do not allow for $\varepsilon =0$ since $\mathcal{T}_{1}(\varepsilon)$ in Table \ref{label 2} involves $(1 + \frac{1}{2\varepsilon})$. However, a look into the proof of Proposition \ref{prop:supGhat1} reveals that for $\varepsilon =0$, one can bound the respective term by $\log(m)$ instead. We refrained from incorporating the case $\varepsilon =0$ explicitly for simplicity.
We note also that our result includes an additional $\log(N)$. However, this is an artifact of using a slightly simplified notation to make reading easier. More precisely, to prove Proposition \ref{prop:Op}, we apply a uniform concentration inequality in Lemma \ref{prop:supGhat} which involves the supremum of a partial derivative; see \eqref{eq:TandboldT}. The supremum is taken over a closed set. In particular, the set is not empty. For this reason, even when the set contains only one point as in the short-range dependent case ($D_{0} \equiv 0$), the derivative is included in the respective bounds. The logarithm appears because of the derivative in our uniform concentration inequality. This can be easily avoided by considering the supremum over a half-open interval. However, it would require careful distinction through all our proofs between whether the set is empty or not. In order to avoid over complicated notation, we refrained from incorporating this case.

\textit{\cite{shu2019estimation}:} This related work focusses on the estimation of the covariance matrix and its inverse, with the results in the time domain. However, \cite{shu2019estimation} also allow for long-range dependence. They assume that the correlation $\rho^{i}_{kn} = \frac{\sigma^{i}_{kn}}{\sigma^{i}_{kk} \sigma^{i}_{nn}} $ with $\sigma^{i}_{kn} = \E X_{i,k} X_{i,n} $ in the component series $X_{i,n}$ satisfies
\begin{equation}
\max_{i =1, \dots, p} | \rho^{i}_{kn} | \leq C | k-n|^{2\alpha-1}
\hspace{0.2cm}
\text{ for }
\hspace{0.2cm}
k \neq n.
\end{equation}
For $\alpha \in (0,\frac{1}{2})$ the individual time series can thus be long-range dependent in the sense that the correlation sequences are not absolutely summable. To make a fair comparison, we will assume a known memory parameters $D_{0}$ and, for simplicity, further ignore the bounds on the deterministic part and only consider the case when the time series is short- or long-range dependent. Then, based on Proposition \ref{prop:supGhat1} for $\varepsilon=0$ and $\Delta_{1}=0$, our convergence rate simply reduces to
\begin{equation} \label{eq:forcomparison}
\nu = \bm{C}
\vertiii{G} \sqrt{\frac{\log(p)}{\mathcal{R}_{1}}},
\hspace{0.2cm}
\mathcal{R}_{1}= \min\{ (\log(N)\widebar{\Delta}_{N})^{-1} \mathcal{R}_{11}, (\log(N)\widebar{\Delta}_{N})^{-2} \mathcal{R}_{12} \}
\end{equation}
with $\widebar{\Delta}_{N}$ as in \eqref{eq:delta-0-1/4} and
\begin{equation*}
\mathcal{R}_{11}=
m^{1-2\Delta_{2}},
\hspace{0.2cm}
\mathcal{R}_{12}=
\min\Big\{
m^{2-4\widebar{\Delta}_{u}},
m^{\frac{3}{2} - 2\widebar{\Delta}_{u} }
\Big\}.
\end{equation*}
In \cite{shu2019estimation}, the result analogous to our Proposition 3.1 is Lemma A.2., (i) with rates given in Remark 2. 
Based on Remark 2 in \cite{shu2019estimation}, the quantity analogous to $\mathcal{R}_{1}$ in \eqref{eq:forcomparison} is given by
\begin{equation}
\mathcal{R} = 
\begin{cases}
N^{1-2\alpha}, \hspace{0.2cm} & \alpha \in (\frac{1}{4}, \frac{1}{2}), \\
\min \Big\{ N^{1-2\alpha} , N^{\frac{1}{2}} \Big\}, \hspace{0.2cm} & \alpha \in (0,\frac{1}{4}), \\
\min \Big\{ (\log(N))^{-1} N ,  N^{\frac{1}{2}} \Big\}, \hspace{0.2cm} & \alpha = \frac{1}{4}. \\
\end{cases}
\end{equation}
Since the results are in the time domain, the convergence rates are in terms of the sample size $N$ rather than the number of frequencies used in estimation $m$ as in our results and \cite{Sun2018:LargeSpectral}. See also Remark \ref{re:discussionlog}, for a discussion on why our bounds include $\log(N)$. Otherwise, our bounds coincide with those in \cite{shu2019estimation} for $\Delta_{2} = \alpha$. In contrast to our statements, the results in \cite{shu2019estimation} are not non-asymptotic. Furthermore, they do not estimate $\alpha$.

\section{The choice of the shrinkage parameters and algorithms} \label{s:algorithm}
Both thresholding \eqref{eq:thresholding} and graphical local Whittle estimation \eqref{eq:Prholasso} depend respectively on the threshold and penalization parameter $\rho$. In this section, we discuss how to select $\rho$. This choice plays a critical role in finite sample performance. We propose to use cross-validation for the thresholding parameter and an extended Bayesian information criterion (eBIC) for graphical local Whittle estimation.

\subsection{Thresholding local Whittle estimation} \label{ss:tuning-threshold}

Cross-validation is generally suggested to select the penalization parameter for thresholding covariance matrix estimation; see \cite{bickel2008covariance}. \cite{Sun2018:LargeSpectral} modified it to sampling over Fourier frequencies in order to account for temporal dependence in spectral density matrix estimation. Similarly, we propose to select the optimal thresholding parameter $\rho$ by cross-validating over the periodogram \eqref{eq:periodogram}. More precisely, we split the sequence of the periodogram evaluated at different frequencies into two groups. Then, we apply the thresholding estimator of the long-run variance matrix \eqref{eq:thresholding} to the first group. The long-run variance matrix estimator from the latter group is used as a reference. The optimal thresholding parameter is selected by minimizing the average Frobenius norm (squared) between the thresholding estimators and the reference estimators of the long-run variance matrix. The detailed procedure can be found in Algorithm \ref{al:tuning-threshold}. As in \cite{Sun2018:LargeSpectral}, this procedure remains to be justified in theory even for short-range dependent series.

\begin{algorithm} \caption{Threshold selection by cross-validation over periodogram} \label{al:tuning-threshold}
	\KwIn{$I_X(\lambda_j)$, $j=1, \ldots, m$, $\widehat D$. Range $\mathcal{L}$ of $\rho$. Number $N_{1}$ of validation sets.}
	\KwOut{Optimal threshold parameter $\widehat \rho := \argmin{\rho \in \mathcal{L}}~ \widehat R_\rho $.}
	\For{ $\rho \in \mathcal{L}$}{
	\For{ $k=1, \ldots, N_1$}{
		1. Randomly divide $\{1, \ldots, m \}$ into two sets $J_1$ and $J_2$ of sizes $m_1 = [m/2]$ and $m_{2} = m-[m/2]$, respectively.\\
		2. Calculate $\widehat G_{1,k} = m_1^{-1} \sum_{j \in J_1} \lambda_j^{\widehat D} I_X(\lambda_j) \lambda_j^{\widehat D}$ and $\widehat G_{2,k} = m_2^{-1} \sum_{j \in J_2} \lambda_j^{\widehat D} I_X(\lambda_j) \lambda_j^{\widehat D}$.
	}
	Obtain $\widehat R_\rho = {N_1}^{-1} \sum_{k=1}^{N_1} \| T_\rho (\widehat G_{1,k}(\widehat D)) - \widehat G_{2,k} (\widehat D) \|_F^2$. 
}
\end{algorithm}

\subsection{Graphical local Whittle estimation} \label{ss:tuning-Graphical}
For the penalization parameter in graphical local Whittle estimation, we suggest to use an extended Bayesian information criterion (eBIC). Tuning parameter selection for penalized likelihood estimation by an eBIC has been studied by multiple authors allowing the dimension to grow with the sample size. The eBIC was proposed by \cite{foygel2010extended} for Gaussian graphical models and further used in \cite{foygel2011bayesian} for model selection in sparse generalized linear models.
Later, \cite{gao2012tuning} proved that using the eBIC to select the tuning parameter in penalized likelihood estimation with the so-called SCAD penalty (\cite{fan2001variable}) can lead to consistent graphical model selection.
See also \cite{chen2012extended}.
%and prove that eBIC selects consistently the true model. 
To be more specific, we use the following criterion
\begin{equation} \label{e:eBIC}
\widehat{P}_{\rm eBIC} = \argmin{P \in \mathbb{G}}~ {\rm eBIC}_\gamma (P), \hspace{0.2cm}
{\rm eBIC}_\gamma(P) = {\rm tr}(\widehat{G}(\widehat{D}) P) - \log |P| + \| P \|_0 \frac{1}{N} (\log N + 4 \gamma \log p),
\end{equation}
where $\mathbb{G}$ is the set of all $\widehat{P}_{\rho}$ estimated by \eqref{eq:Prholasso} over a range of $\rho$ and $\|P\|_0$ denotes the norm counting the number of non-zero elements in $P$. Furthermore, the criterion is indexed by a parameter $\gamma \in [0,1]$; see \cite{foygel2010extended} and \cite{chen2008extended} for the Bayesian interpretation of $\gamma$.
Then, $\widehat{P}_{\rm eBIC}$ gives us a data-driven estimate of $\widehat{P}_{\rho}$ and the associated penalty $\rho$.

In order to determine $\widehat{P}_{\rho}$ over a range of $\rho$ and the set $\mathbb{G}$, we consider an algorithm to compute the graphical local Whittle estimator for a given penalty parameter $\rho$. This algorithm is a natural extension of graphical LASSO algorithms for real symmetric covariance matrices to complex-valued Hermitian matrices. The generalization to complex-valued Hermitian matrices is necessary since the true precision matrix $P_{0}$ is possibly complex-valued.
We propose a complex-valued alternating linearization method (ALM) which is a variation of alternating direction method of multipliers (ADMM) proposed in \cite{scheinberg2010sparse}. Our limited simulation study shows that the proposed method converges faster than the na\"ive ADMM algorithm. We also note that the complex-valued ADMM for SRD series was considered by \cite{jung2015graphical}.

The ALM (Algorithm \ref{al:ALM}) solves the problem 
\begin{equation*}
\argmin{P, Y} \left\{ -\log|P| + {\rm tr}(\widehat G P) +  \rho \| Y \|_{1,off}  \right\}
\end{equation*}
subject to $ P = Y$ being positive definite. It invokes the ADMM algorithm to find sparse positive definite estimates of $P_{0}$ by introducing augmented Lagrangian. It furthermore carefully selects augmented Lagrangian penalty parameter $\mu_k$ so that the positive definiteness is achieved throughout the iterations. For the initial estimator $\widetilde{P}$ in lower dimensions, we take $(\widehat{G}(\widehat{D}))^{-1}$. The shrinkage operator in Step 4 of Algorithm \ref{al:ALM} is defined as ${\rm shrink}(M,\nu) = \sign(M_{rs}) \max( |M_{rs}| - \nu, 0) $ for a matrix $M=(M_{rs})_{r,s =1, \dots, p}$ and some $\nu \geq 0$. 
Step 5 of Algorithm \ref{al:ALM} requires to update $\mu_{k}$. 
It is reduced by a constant factor $\eta_{\mu}$ on every $N_{\mu}$ iteration till a lower bound is achieved by following the idea of
\cite{scheinberg2010sparse}. That is,  set $\mu$ as $\max\{\mu \eta_{\mu}, \widebar{\mu} \}$ to reduce $\mu$ by a constant factor $\eta_{\mu}$ on every $N_{\mu}$ iteration. In this paper, we choose   $\mu_{0} = .01$, $N_{\mu} = 10$,  $\widebar{\mu} = 10^{-3}$ and $\eta_{\mu} = 1/4$. Finally, we terminate the ALM algorithm following the stopping rules in (20) in \cite{scheinberg2010sparse} except that the first condition is replaced by stopping after 1000 iterations.

\begin{algorithm} \caption{Alternating linearization method algorithm for graphical local Whittle}
	\KwIn{$Y^0 = {\rm diag}(\widetilde P)$ with initial estimator $\widetilde P$ of $P_{0}$, $\Lambda^0 = 0$, $\mu_0$, $\rho$.} \label{al:ALM}
	\KwOut{Sparse estimation of $P_{0}$.}
	{\bf Repeat until convergence:} \\
	\For{ $k=0, 1, \ldots$}{
		1. Let $W^{k+1} = Y^k + \mu_k (\Lambda^k - \widehat G)$ and perform the singular value decomposition $W = U {\rm diag}(\eta_1,\dots,\eta_p) V^*$. \\
		2. $X^{k+1} = U {\rm diag}(\gamma_1, \dots, \gamma_p) V^*$, where $\gamma_i = .5 (\eta_i + \sqrt{\eta_i^2 + 4\mu_k})$, $i=1, \ldots, p$.\\
		3. $Y^{k+1} = {\rm shrink}(X^{k+1} - \mu_{k}(\widehat G - (X^{k+1})^{-1}), \mu_{k} \rho)$. \\
		4. $\Lambda^{k+1} = \widehat G - (X^{k+1})^{-1} + (X^{k+1} - Y^{k+1})/\mu_{k}$. \\
		5. Pick $\mu_{k+1} \le \mu_k$.
	}
\end{algorithm}

We conclude with a discussion on the criterion \eqref{e:eBIC}.
\begin{remark} \label{re:eBIC}
The criterion \eqref{e:eBIC} is adapted from \cite{foygel2010extended}. The log-likelihood function in equation (2) in \cite{foygel2010extended} which gives an estimate for the inverse covariance matrix of a Gaussian model is replaced by the negative log-likelihood function $\ell$ in \eqref{eq:neglog} in terms of the matrix $P_{0}=G_{0}^{-1}$. The criterion seems to perform well in our simulation study. However,
from a theoretical perspective, a consistency result can only be established when the penalization term in \eqref{e:eBIC} is chosen in dependence of the rate of convergence of the $(r,s)$th component of $\widehat{G}(\widehat{D})$ around the true $G_{0,rs}$. 
In their theoretical results, \cite{foygel2010extended} considered independent and identically distributed Gaussian random vectors. In this case, the penalization depends 
on the convergence rate of the deviation of the sample covariance matrix around the true covariance matrix, that is $\sqrt{\frac{\log(p)}{N}}$. \cite{foygel2010extended} proved that the eBIC selects the correct model consistently in a high-dimensional regime $p,N\to \infty$; see Theorem 5 in \cite{foygel2010extended}.
To establish an analogous consistency result in our setting, we suggest to replace the penalization $\frac{\log(p)}{N}$ in the eBIC objective function by $\delta^{2}$ in \eqref{eq:deltaepsilon}, since $\delta$ gives the convergence rate of the maximum norm of $\widehat{G}(\widehat{D})$ around the true $G_{0}$. The detailed proof of such a consistency result goes beyond the scope of this work and, from a practical perspective, the usage of \eqref{e:eBIC} seems to be more natural, since including $\delta$ in the penalization would involve a number of unknown parameters.
\end{remark}

\section{Simulation study}
\label{s:sim}
In this section, we examine the proposed methods through simulations. Our two-stage approach first estimates the memory parameters $D_{0}$ and non-sparse long-run variance matrix $G_{0}$ based on the local Whittle estimator \eqref{eq:G(D)}. Then, we apply either thresholding or graphical local Whittle estimation to get sparse estimators. Several tuning parameters need to be selected for our methods. We comment first on the number of frequencies used in local Whittle estimation. Details about the selection of tuning parameters for sparse estimation are provided in the subsequent sections. 

The selection of the number of frequencies $m$ is important in practice and should be balanced: be small enough to capture long-range dependence and large enough to get reliable estimates. In univariate local Whittle estimation, asymptotic theory suggests $m=O(N^{.8})$; see \cite{robinson1995gaussian}.
There are several papers studying data dependent bandwidth. In this regard, the most influential paper is \cite{henry2001robust}. \cite{henry2001robust} suggests a bandwidth minimizing the mean squared error of the univariate local Whittle estimator. A visual approach to ensure the balance between capturing long-range dependence and getting reliable estimates is proposed in \cite{BaekKechagiasPipiras2020}. \cite{BaekKechagiasPipiras2020} used the so-called local Whittle plots which present estimates of the memory parameters as function of the tuning parameter $m$ supplemented with confidence intervals. 

In our simulation study we base our choice on the asymptotic theory in \cite{robinson1995gaussian} suggesting $m = [N^{.8}]$, where $[x]$ is the largest integer less than or equal to $x$.

\subsection{Thresholding local Whittle estimation} \label{s:sub2Sim}

We consider the following three data generating processes (DGPs) to evaluate the finite sample performance of the thresholding local Whittle method. The one-sided and two-sided VARFIMA(0, $D$, 0) (Vector Autoregressive Fractionally Integrated Moving Average) models are used to generate multivariate long-range dependent time series. 
See \cite{KechagiasPipiras, kechagias2015identification} for definitions of these models.
We consider dimensions $p=20, 40, 60$ with sample sizes $N=200, 400, 1000$. Long-range dependent parameters $D_{0}$ are selected at random from .1 to .45, if not specified otherwise. We use the notation $G(r,s)$ to denote the $(r,s)$th entry of $G_{0}$.
To be more precise, the DGP's are given as follows:

\begin{itemize}
	\item[]
	\begin{itemize}
		\item[(thDGP1)] One-sided VARFIMA($0, D, 0$) with $G_{0} = (I_{[p/20]} \otimes G_1)$, where the diagonal entries of $G_1$ are .159 except $G_1(1,1) = G_1(6,6) = .312$, $G_1(11,11)= G_1(14,14) = .212$, $G_1(3,3) = G_1(20, 20) = .189$, and $G_1(1,6) = -.208 -.064i$, $G_1(3,20) = -.074+.015i$, and $G_1(11,14) = -.105-.013i$.
		
		\item[(thDGP2)] Two-sided VARFIMA($0, D, 0$) with $G_{0} = (I_{[p/20]} \otimes G_2)$, where the diagonal entries of $G_2$ are 1 and $G_2(3,9) = .5 +.2i$, $G_2(5, 14) = .4+ .2i$.
		
		\item[(thDGP3)] Two-sided VARFIMA($0, D, 0$) with banded $G_{0}$ matrix given by  $G_{0} = (I_{[p/20]} \otimes G_3)$,
		where the diagonal entries of $G_3$ are 1 and $G_3(r,r+1) = .4 + .2i$, $r=1, \ldots, 19$.  
	\end{itemize} 
\end{itemize}
For the reader's convenience, the sparsity patterns imposed on $G_{k}$, $k=1,2,3,$ are depicted in Figure \ref{f:graphical-DGP}. 

The thresholding parameter $\rho$ is selected based on cross-validation introduced in Section \ref{ss:tuning-threshold}. We set $\mathcal{L}$ to be the smallest and the largest value of $|\widehat G(\widehat D)|$ in \eqref{eq:G(D)}. We evaluated the performance of the thresholding local Whittle estimator using the mean squared error of the  parameters $\widehat D$, total number of misspecified coefficients
$ \sum_{r \ge s} ( \mathds{1}_{\{ T_{\rho}(\widehat G_{rs}(\widehat D) = 0\}} - \mathds{1}_{\{G_{0,rs} = 0\}})^2$, 
the Frobenius norm $\| T_{\rho}(\widehat G(\widehat D)) - G_{0} \|_F$ and the spectral norm $\| T_{\rho}(\widehat G(\widehat D)) - G_{0} \|$. The performance measures are calculated based on 1000 iterations.

\begin{figure}
	\centering
	\includegraphics[width=5 in, height=1.6 in]{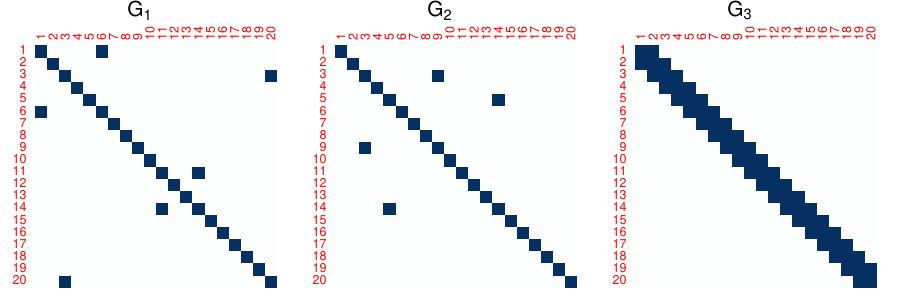}
	\vspace{-2 mm}
	\caption{Sparsity patterns of DGPs.} \label{f:graphical-DGP}
\end{figure}

\begin{figure}
	\centering
	\includegraphics[width=6 in, height=1.85 in]{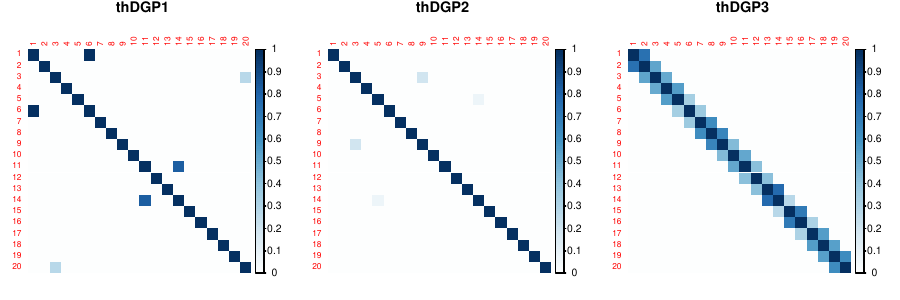}
	\vspace{-2 mm}
	\caption{The number of times having non-zero coefficients using thresholding local Whittle estimation with cross-validation tuning parameter selection where $p=20$, $N=200$.} \label{f:thDGP-bic}
\end{figure}

Figure \ref{f:thDGP-bic} shows the proportion of times each component is estimated to be non-zero using our thresholding local Whittle approach with $p=20$ and $N=200$. The estimation is close to the true sparsity pattern though some locations are more difficult to be estimated correctly. For example, in thDGP1, $G_1(3, 20)$ is detected as non-zero less frequently, but this is natural since $|G_1(3, 20)| = .075$ is smaller than the other coefficients. However, such misspecification vanishes as sample size increases. Table \ref{tab:threshold} in Appendix \ref{app:table} shows the performance measures calculated for $G_{0}$. It can be observed that all performance measures are decreasing as the sample size increases. 

\subsection{Graphical local Whittle estimation} \label{s:sub1Sim}
We consider the following three DGPs to see the performance of the graphical local Whittle estimation. We introduce the matrices $P_{k}$, $k=1,2,3$, to define the sparsity pattern of $P_{0}$. The notation $P_{k}(r,s)$, $k=1,2,3$, denotes the $(r,s)$th entry of $P_{k}$.

\begin{itemize}
	\item[]
\begin{itemize}
	\item[(DGP1)] One-sided VARFIMA($0, D, 0$) with $P_{0} = (I_{[p/20]} \otimes P_1)$, where $P_1(r,r) = 6.28$, $P_1(1,6) = 4.20+1.29i$, $P_1(3,20) = 2.46 - .51i$ and $P_1(11,14) = 3.12+.39i$. 
	 
	\item[(DGP2)] Two-sided VARFIMA($0, D, 0$) with $P_{0} = (I_{[p/20]} \otimes P_2)$, where the diagonal entries of $P_{2}$ are 1 and $P_2(3,9) = .5 + .2i$, $P_2(5,14) = .4 + .2i$. 

 	\item[(DGP3)] Two-sided VARFIMA($0, D, 0$) with banded $P_{0}$ matrix given by  $P_{0} = (I_{[p/20]} \otimes P_3)$, where the diagonal entries of $P_3$ are 1 and $P_3(r,r+1) = .2 + .1i$, $r=1, \ldots, 19$.  
 	
\end{itemize} 
\end{itemize}

The penalty parameter $\rho$ is chosen by the eBIC in \eqref{e:eBIC} with $\gamma = 1$. 
Also, the ALM algorithm requires the Lagrangian penalty parameter, which as noted above is set to decrease by $\frac{1}{4}$ on every $10$th iteration with $\mu_0 = 10^{-2}$ but it is taken no smaller than $10^{-6}$. The same performance measures are used as those for thresholding local Whittle estimation in Section \ref{s:sub2Sim}, but the inverse of long-run variance $P_{0}$ is used instead of $G_{0}$.

\begin{figure}[t!]
	\centering
	\includegraphics[width=6 in, height=2 in]{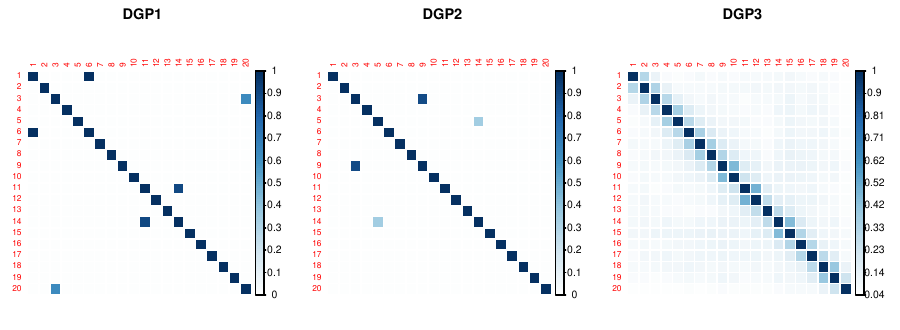}
	\vspace{-2 mm}
	\caption{The number of times having non-zero coefficients using graphical local Whittle estimation with eBIC tuning parameter selection where $p=20$, $N=200$.} \label{f:DGP1}
\end{figure}

Figure \ref{f:DGP1} shows  the proportions of estimated non-zero coefficients in $P_k$, $k=1,2,3$, when the dimension is $p=20$ and the sample size is $N=200$. Note that the sparsity patterns of $P_k$, $k=1,2,3$, are the same as in Figure \ref{f:graphical-DGP}. It can be observed that the graphical local Whittle estimator recovers the sparsity pattern of the underlying model. For example, the non-zero coefficient $P_1(11, 14)$ is found to be non-zero about 90\% of times by our proposed method. More detailed performance measures can be found in 
Table \ref{tab:graphicalLW}. Table \ref{tab:graphicalLW} suggests that for all considered models our performance measures tend to improve as the sample size increases for fixed dimension. The other way around, for fixed sample size $N$, the performance measures are getting worse as the dimension increases. We can conclude that the graphical local Whittle estimator correctly identifies zero coefficients and yields estimates close to the true values as the sample size increases.

\section{Real data application}
\label{s:realdata}

\begin{figure}
	\centering
	\includegraphics[width=6 in, height=3 in]{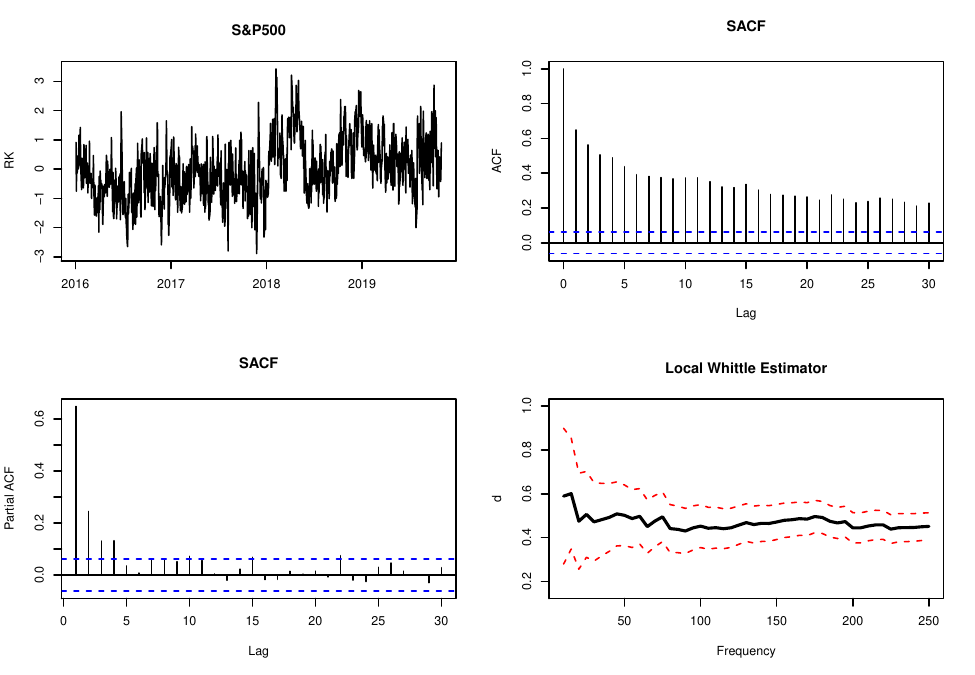}
	\vspace{-2 mm}
	\caption{Time plot (top left), sample ACF plot (top right), sample PACF plot (bottom left) and the local Whittle estimators (bottom right) for S\&P 500.} \label{f:sp500}
\end{figure}

In this section, we apply our proposed methods to 31 realized volatilities obtained by aggregating the 5-min within-day returns taken from the Oxford Man Institute of Quantitative Finance (\url {http://www.oxford-man.ox.ac.uk}). We adjusted the different opening days over the stock markets by applying linear interpolation and log-transforming the data. Furthermore, we removed the possible mean changes in the data by following the proposed procedure in \cite{baek2014distinguishing}. We also studentized each series to have zero-mean and unit variance in order to focus on volatility linkage.
The total number of observations is 1001 dating from Jan 4, 2016 to Oct 31, 2019. The 31 global stock indices are 
AEX index (AEX), 
All Ordinaries (AORD),
Bell 20 index (BFX),  
S\&P BSE Sensex (BSESN), 
PSI All Shares Gross Return Index (BVLG),
BVSP BOVESPA Index (BVSP), 
Dow Jones Industrial Average (DJI),
CAC 40 (FCHI),
FTSE MIB (FTMIB), 
FTSE 100 (FTSE),
DAX (GDAXI),
S\&P/TSX Composite index (GSPTSE),
HANG SENG Index (HSI), 
IBEX 35 Index (IBEX), 
Nasdaq 100 (IXIC),
Korea Composite Stock Price Index (KS11), 
Karachi SE 100 Index (KSE),
IPC Mexico (MXX), 
Nikkei 225 (N225),
NIFTY 50 (NSEI),
OMX Copenhagen 20 Index (OMXC20),
OMX Helsinki All Share Index (OMXHPI),
OMX Stockholm All Share Index (OMXSPI),
Oslo Exchange All-share Index (OSEAX),
Russel 2000 (RUT),
Madrid General Index (SMSI),
S\&P 500 Index (SPX),
Shanghai Composite Index (SSEC),
Swiss Stock Market Index (SSMI),
Straits Times Index (STI), 
EURO STOXX 50 (STOXX50E).

Figure \ref{f:sp500} shows some exploratory plots for S\&P 500 such as time plot (top left), sample autocorrelation plot (top right), sample partial autocorrelation plot (bottom left) and long-range dependent parameter estimates over the number of frequencies used (bottom right). It shows typical features of long-range dependent time series: non-cyclical trends, slow decay of autocorrelations and  memory parameters close to .5. This suggests that multivariate long-range dependence modeling is meaningful and we applied our methods to estimate long-run variance and precision matrices. 

\begin{figure}[t!]
	\centering
	\includegraphics[width=6.5 in, height=2.2 in]{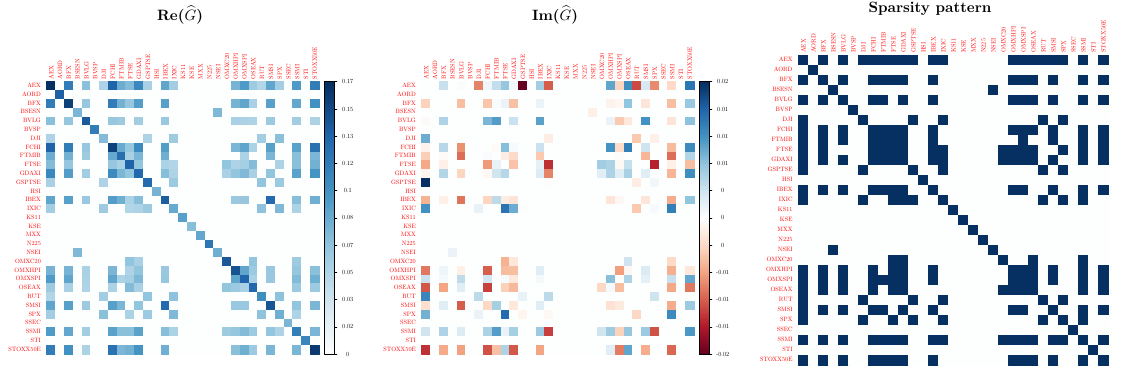}
	\vspace{-2 mm}
	\caption{Thresholding local Whittle estimation.} \label{f:G}
\end{figure}

\begin{figure}[t!]
	\centering
	\includegraphics[width=6.5 in, height=2.2 in]{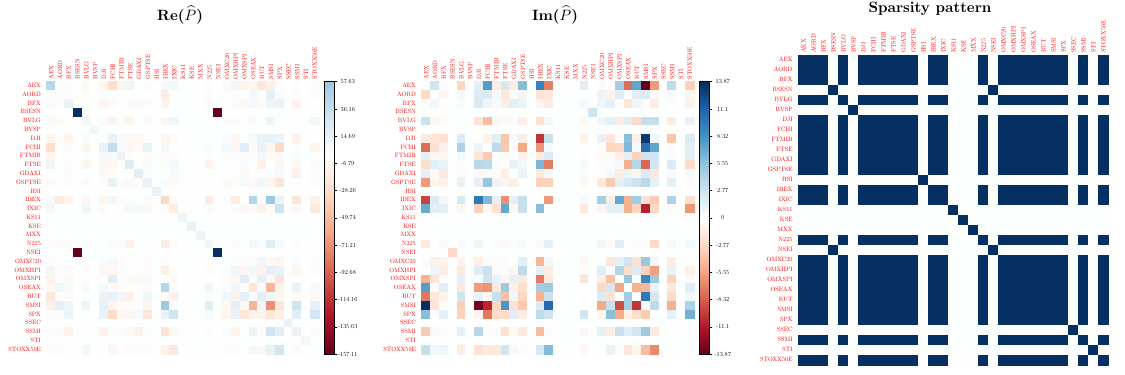}
	\vspace{-2 mm}
	\caption{Graphical local Whittle estimation.} \label{f:Ginv}
\end{figure}

\begin{figure}[t!]
	\centering
	\includegraphics[width=5.5 in, height=3 in]{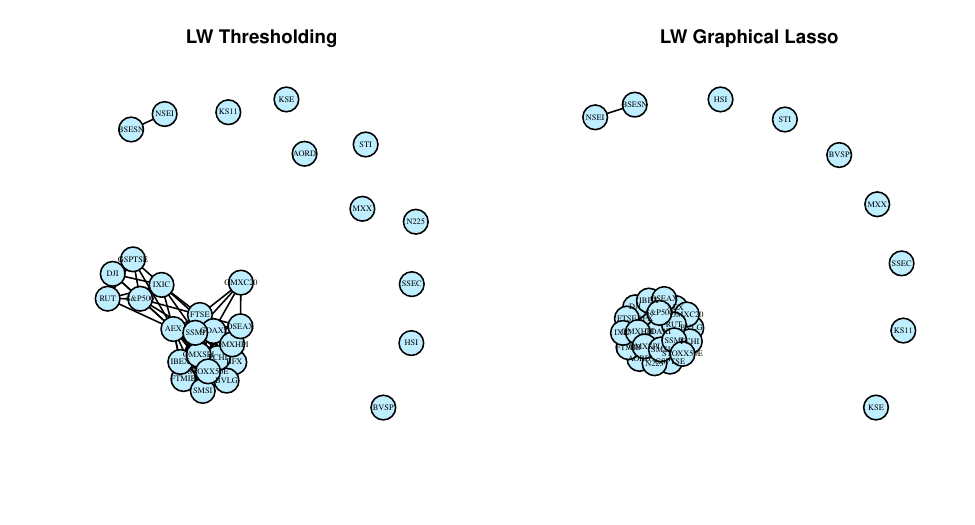}
	\vspace{-10 mm}
	\caption{Network graph representation.} \label{f:graph}
\end{figure}

Sparse long-run variance matrix estimation by thresholding is presented in Figure \ref{f:G}. The left panel shows the real parts of $\widehat G$ and the imaginary parts can be found in the middle panel. The right panel presents the sparsity pattern, that is, the locations of the non-zero coefficients are colored dark blue. 
Figure \ref{f:Ginv} follows the same structure as Figure \ref{f:G} but shows the estimated sparse precision matrix $\widehat P$ using graphical local Whittle estimation. The penalty parameter is selected by using eBIC in \eqref{e:eBIC} with $\gamma=1$. 
BSESN and NSEI, both Indian market indices, showed the largest real coefficients in absolute term. 

Observe that the thresholding method gives sparser estimation. We also observe the clustering of stock market indices for both long-run variance and precision matrices. It is seen more clearly from a network representation of linkages as in Figure \ref{f:graph}. More interestingly, both methods give similar clusterings. The isolated nodes based on $\widehat P$ correspond
to India (BSESN/NSEI), China (SSEC), South Korea (KS11), Hong Kong (HSI), Singapore (STI), Mexico (MXX), Portugal (BSVP), Pakistan (KSE). The sparse long-run variance $\widehat G$ using thresholding adds Japan (N225) and Australia (AORL). That is, it seems that our empirical analysis tells that stock market indices can be roughly divided into the US-European market and somewhat independent markets from the rest of the world including Asia, India, Australia and Mexico in late 2010s. The thresholding method seems to further distinguish US and European markets. In fact, all US stock indices showed larger values of memory parameter estimates by having more than .4 which is not the case for European market indices. 
It is particularly interesting to see the grouping of multinational realized volatilities according to regional or spatial dependence.

Finally, we note but do not include the results here, that the analysis of the data assuming short-range dependence led to highly non-sparse patterns for the considered connectivity matrices. These findings were consistent with the scenario where we simulated data from LRD models as in Section \ref{s:sim} but worked as if they were SRD.

\section{Conclusions}
\label{s:conclusions}

In this work, we derived consistency results for the long-run variance and precision matrices
in a non-asymptotic regime allowing the underlying time series to admit a general dependence structure including long-range dependence. The results are derived under mild assumptions on the underlying time series which is allowed to be either Gaussian or have a linear representation. The shrinkage techniques are thresholding and graphical local Whittle estimation. Our non-asymptotic results can be used to infer consistency in a high-dimensional regime where the number of component series can be large compared to the sample size. 

The key technical contribution is the incorporation of the memory parameter matrix which carries information about the dependence structure of the underlying time series. Our results allow estimating those memory parameters simultaneously while estimating the long-run variance and the precision matrices sparsely. 

We see the proposed proof techniques as a basis to study other questions concerning high-dimensional long-range dependence.
Possible future directions include the use of other shrinkage methods, for example, adaptive penalizations; sparse estimation of fractionally (co)integrated vector autoregressive (VAR) models and sparse estimation of linear regression with long-range dependent errors.

\appendix

\setcounter{table}{0}
\renewcommand{\thetable}{\Alph{section}.\arabic{table}}

\section{Quantities in main results and special cases}
\label{se:tablesResults}

In this section we provide the expressions for several quantities appearing in Propositions \ref{prop:supGhat1}--\ref{prop:supGhat4}. In addition, we will discuss the case when the underlying time series admits only short- or long-range dependence.
\par
Recall that the non-asymptotic bounds in Propositions \ref{prop:supGhat1}--\ref{prop:supGhat4} are all of the form
\begin{equation} \label{eq:genrep}
\bm{C} \vertiii{G}
\sqrt{\frac{\log(p)}{\mathcal{R}_{i}}} 
+
\mathcal{T}_{i}
\hspace{0.2cm}
\text{ with }
\hspace{0.2cm}
\mathcal{R}_{i} = \min\{s_{i,N} \mathcal{R}_{i1}, s^{2}_{i,N} \mathcal{R}_{i2}\}.
\end{equation}
Here, $\mathcal{R}_{i}$'s arise in bounds on the probabilistic parts and $\mathcal{T}_{i}$'s on deterministic parts. The sequences $s_{i,N}$ can be found in Propositions \ref{prop:supGhat1}--\ref{prop:supGhat4}.
The quantities $\mathcal{R}_{i1},\mathcal{R}_{i2}$ and $\mathcal{T}_{i}$ are respectively given in Tables \ref{label 1} and \ref{label 2} for Propositions \ref{prop:supGhat1}--\ref{prop:supGhat4}.

Recall that $\Delta_{1},\Delta_{2}$ determine the interval of admissible estimates of the memory parameters and the quantity $\widebar{\Delta}_{N}$ is defined in \eqref{eq:delta-0-1/4}. We further introduce a few quantities which will allow us to express our bounds in a simplified way and emphasize the necessary distinction between different ranges of the  memory parameters as will become clearer in the proofs. Let 
\begin{equation} \label{eq:delta-u-l}
\begin{gathered}
\widebar{\Delta}_{u} = \max_{r =1,\dots,p}\{ \Delta_{2}, (\frac{1}{4} + d_{0,r})\mathds{1}_{\{d_{0,r} < \frac{1}{4}\}}\},
\\
\widebar{\Delta}_{l,1} = \max_{r=1,\dots,p} d_{0,r} \mathds{1}_{\{d_{0,r} < \Delta_{1} + \frac{1}{2}\}} ,
\hspace{0.2cm}
\widebar{\Delta}_{l,2} = \max_{r =1,\dots,p} (\frac{1}{4} \mathds{1}_{\{ d_{0,r} < \min\{\Delta_{1} + \frac{1}{2}, \frac{1}{4}\}\}} + d_{0,r} \mathds{1}_{\{d_{0,r} < \Delta_{1} + \frac{1}{2}\}}),
\end{gathered}
\end{equation}
where the subscripts $u$ and $l$ allude to the dependence on $\Delta_{2}$ and $\Delta_{1}$, respectively.
\par
For the bounds on the deterministic parts we will use
\begin{equation}
\label{eq:mR:Qrs}
Q_{m} = 
\vertiii{G} 
\frac{72 (\cos(\lambda_{m}/2))^{-2} }{\pi(1+2\min\{\Delta_{1},-\Delta_{2}\})}
+ \bm{c}_{G,2} 4(2+ \log(m))
\end{equation}
and $\widetilde{\Delta}_{r} = (d_{0,r}-\frac{1}{2}+\Delta) \mathds{1}_{\{d_{0,r} \geq \Delta_{1} +\frac{1}{2}\}} + \Delta_{1} \mathds{1}_{\{d_{0,r} < \Delta_{1} +\frac{1}{2}\}}$. Recall further that $q$ appears in Assumption \ref{ass:G-G}.

\begin{table}[h]
\centering
\renewcommand{\arraystretch}{1.7}
\begin{tabular}{c | c}
\hline
Prop. & $\mathcal{R}_{i1}$ and $\mathcal{R}_{i2}$ in \eqref{eq:genrep} (probabilistic parts) \\
\hline \hline
\ref{prop:supGhat1} & 
$\begin{gathered}
\vspace{0.02cm}\\
\mathcal{R}_{11}=
\min\Big\{
m^{1-2\Delta_{2}+2\varepsilon} N^{-2\varepsilon},
m^{1 - \Delta_{2} - \Delta_{1}+\varepsilon} N^{\Delta_{1}-\varepsilon} , m N^{\Delta_{1}-\varepsilon}
\Big\},
\\
\mathcal{R}_{12}=
\min\Big\{
m^{2-4\widebar{\Delta}_{u}+4\varepsilon} N^{-4\varepsilon},
m^{\frac{3}{2} - 2\widebar{\Delta}_{u} - 2\Delta_{1}+2\varepsilon} N^{2\Delta_{1}-2\varepsilon} , mN^{2\Delta_{1}-2\varepsilon},
m^{2}\Big( \sum_{j=1}^{m}\lambda_{j}^{4 \Delta_{1}} \Big)^{-1}
\Big\}.
\end{gathered}$
\\
\ref{prop:supGhat2}
& 
$\begin{gathered}
\vspace{0.02cm}\\
\mathcal{R}_{21}
=
\min \Big\{
m^{1-2 \Delta_{2}}, mN^{2\Delta_{1}} \Big\},
\hspace{0.2cm}
\mathcal{R}_{22}
=
\min \Big\{
m^{2-4\bar{\Delta}_{u}}, 
m^{2} \Big( \sum_{j=1}^{m} \lambda_{j}^{4\Delta_{1}} \Big)^{-1} \Big\}.
\end{gathered}$ 
 \\ 
\ref{prop:supGhat3} & 
$\begin{gathered}
\vspace{0.02cm}\\
\mathcal{R}_{31}= \min \Big\{
m^{1-2\Delta_{2}}, m^{2 \Delta}, m^{1-2\widebar{\Delta}_{l,1}+2\Delta_{1}}, mN^{2 \Delta_{1}}  
\Big\},
\\
\mathcal{R}_{32}= \min \Big\{
m^{2-4\widebar{\Delta}_{u}}, m^{4 \Delta}, m^{4\Delta} N^{1+4\Delta_{1}}, m^{2-4\widebar{\Delta}_{l,2}+4\Delta_{1}}, m^{2}N^{4 \Delta_{1}}, m^{2} \Big( \sum_{j=1}^{m} \lambda_{j}^{4\Delta_{1}} \Big)^{-1}  
\Big\}.
\end{gathered}$
 \\ 
\ref{prop:supGhat4} & 
$\begin{gathered}
\vspace{0.02cm}\\
\mathcal{R}_{41}
= \min\{m \ell^{-1+2\Delta}, m^{1-2 \Delta_{2}}\},
\hspace{0.2cm}
\mathcal{R}_{42}
= \min\{m \ell^{-2+4\Delta}, m^{2-4\bar{\Delta}_{u}}\}.
\end{gathered}$ 
\end{tabular}
\caption{Expressions for $\mathcal{R}_{i1}$ and $\mathcal{R}_{i2}$ in \eqref{eq:genrep}.}
\label{label 1}
\end{table}

\begin{table}[h]
\centering
\renewcommand{\arraystretch}{1.7}
\begin{tabular}{c | c}
\hline
Prop. & $\mathcal{T}_{i}$ in \eqref{eq:genrep} (deterministic parts) \\
\hline \hline
\ref{prop:supGhat1}  &
$\begin{gathered}
\vspace{0.02cm}\\
\mathcal{T}_{1}(\varepsilon)
=
\bm{c}_{G,1}
\frac{1}{2\pi} \lambda_{m}^{2q-2\varepsilon}
+
\frac{1}{2 \pi m} \Big(1 + \frac{1}{2\varepsilon} \Big)
\lambda_{1}^{-2\varepsilon} Q_{m}
\end{gathered}$
\\
\ref{prop:supGhat2}
&
$\begin{gathered}
\vspace{0.02cm}\\
\mathcal{T}_{2} =
\bm{c}_{G,1}
\frac{1}{2 \pi}  \lambda_{m}^{2 q}
+
\frac{1}{\pi m} \log(m)Q_{m}
\end{gathered}$ 
 \\ 
\ref{prop:supGhat3} & 
$\begin{gathered}
\vspace{0.02cm}\\
\mathcal{T}_{3}
=
\max_{r = 1,\dots,p} \mathcal{T}_{3}(\widetilde{\Delta}_{r}) 
\\= 
\max_{r = 1,\dots,p} \left(
\bm{c}_{G,1}
\frac{1}{2 \pi} \lambda_{m}^{2q} \frac{1}{2\widetilde{\Delta}_{r}} 
+
\frac{1}{2 \pi} m^{-2\widetilde{\Delta}_{r}} \Big(1+\frac{1}{1-2\widetilde{\Delta}_{r}} \Big) Q_{m} \right)
\end{gathered}$ 
 \\ 
\ref{prop:supGhat4} & 
$\begin{gathered}
\vspace{0.02cm}\\
\mathcal{T}_{4} = 
\bm{c}_{G,1}
\frac{1}{2 \pi} \lambda_{m}^{2q} \Big( \frac{1}{2\Delta} + 1 \Big)
+
\frac{1}{2 \pi m} \Big( \ell^{1-2\Delta} \Big(1+ \frac{1}{1-2\Delta} \Big)
+ \log(m) \Big) Q_{m}
\end{gathered}$ 
\end{tabular}
\caption{Expressions for $\mathcal{T}_{i}$ in \eqref{eq:genrep}.}
\label{label 2}
\end{table}

As a corollary of Proposition \ref{prop:Op}, we give the result where the underlying process is known to admit only short- or long-range dependence, that is, the true memory parameters satisfy $D_{0} \succcurlyeq 0$. 
\begin{corollary} \label{cor:Op}
Let $\{X_{n}\}_{n \in \ZZ}$ be a $p$-dimensional, stationary, centered time series with spectral density $f_{X}$ and suppose Assumptions \ref{ass:f0}--\ref{ass:GnG} are satisfied. Then, there are positive constants $c_{1},c_{2}$ such that for any $\bm{C} \geq 1$,
\begin{equation*} 
%\label{eq:consFroblong}
\Prob( \| \widehat{G}(\widehat{D})-G_{0} \|_{\max} > \delta )
\leq
c_{1} p^{2-c_{2}\bm{C}}
\end{equation*}
with $\delta$ as in \eqref{eq:deltaepsilon} and $\nu, \nu_{i}$, $i=1,2,3$, are as in \eqref{eq:nuinprop1}, \eqref{eq:nu1}, \eqref{eq:nu2} and \eqref{eq:nu3} with
\begin{equation} \label{eq:cor:Rs}
\begin{gathered}
\mathcal{R}_{11}=
m^{1-2\Delta_{2}+2\varepsilon} N^{-2\varepsilon}, 
\
\mathcal{R}_{12}=
\min\{
m^{2-4\widebar{\Delta}_{u}+4\varepsilon} N^{-4\varepsilon},
m^{\frac{3}{2} - 2\widebar{\Delta}_{u}+2\varepsilon} N^{-2\varepsilon}
\},
\\
\mathcal{R}_{21}
=
m^{1-2 \Delta_{2}},
\
\mathcal{R}_{22}
=
m^{2-4\bar{\Delta}_{u}},
\\
\mathcal{R}_{31}= \min \Big\{
m^{1-2\Delta_{2}}, m^{2 \Delta}, m^{1-2\widebar{\Delta}_{l,1}} 
\Big\},
\
\mathcal{R}_{32}= \min \Big\{
m^{2-4\widebar{\Delta}_{u}}, m^{4 \Delta}, m^{2-4\widebar{\Delta}_{l,2}}
\Big\},
\end{gathered}
\end{equation}
and it is assumed that $\varepsilon \in (0,\frac{1}{2})$.
\end{corollary}
The corollary is a simple consequence of setting $\Delta_{1} = 0$ in Proposition \ref{prop:Op}.

\section{Proofs of the main results}
\label{s:proofsA}

In this section we will give bounds on the probabilistic and deterministic terms in \eqref{eq:results_GhatGtilde}, stated respectively in Lemmas \ref{eq:prop:supGhatLs} and \ref{prop:bias}. Lemma \ref{eq:prop:supGhatLs} focusses on results under the assumption that the underlying process is Gaussian. Its analogue for linear processes can be found in Appendix \ref{se:linearprocesses}. Up to a constant the bounds are the same as for the Gaussian case. For this reason all proceeding results remain true.
Note that Lemma \ref{prop:bias} only relies on assumptions on the spectral density and remains valid as well.

The following Lemma \ref{prop:supGhat} gives multiple upper bounds on the probabilistic term in \eqref{eq:results_GhatGtilde}. Those different bounds are used later depending on whether the true parameters $d_{0,r}$ and $d_{0,s}$ are positive or negative. Recall in particular the notation $t_{j,r}(d)$ below \eqref{eq:crucialterms1} and $a_{r}, b_{r}$ in \eqref{eq:OmegaAB}--\eqref{eq:AB}.
Let
\begin{equation} \label{eq:prop:supGhatLs}
\begin{aligned}
L_{rs,i}
&= c_{r,i,N} \widetilde{c}_{s,i,N} 
\left( \| T(a_{r},a_{s}) \| +
\sup_{(d_{r},d_{s}) \in [a_{r},b_{r}] \times [a_{s},b_{s}]} \| \bm{T}(d_{r},d_{s}) \| \right),
\hspace{0.2cm}
i=1,\dots,4,
\\
L_{rs,5}
&= N^{\max\{d_{0,r},0\}+\max\{d_{0,s},0\}} 
\left( \| T(a_{r},a_{s}) \|_{F} +
\sup_{(d_{r},d_{s}) \in [a_{r},b_{r}] \times [a_{s},b_{s}]}  \| \bm{T}(d_{r},d_{s}) \|_{F} \right)
\end{aligned}
\end{equation}
with 
\begin{equation} \label{eq:TandboldT}
\begin{aligned}
T(d_{r},d_{s}) &= \diag( t_{1,r}(d_{r})t_{1,s}(d_{s}), \dots, t_{m,r}(d_{r})t_{m,s}(d_{s})),
\\
\bm{T}(d_{r},d_{s}) &= \diag( \| \nabla t_{1,r}(d_{r})t_{1,s}(d_{s}) \|_{F}, \dots, \| \nabla t_{m,r}(d_{r})t_{m,s}(d_{s}) \|_{F})
\end{aligned}
\end{equation} 
and
\begin{equation*} 
\begin{gathered}
c_{r,1,N}=N^{\max\{d_{0,r},0\}},
\widetilde{c}_{s,1,N}=N^{\max\{d_{0,s},0\}};
\hspace{0.2cm}
c_{r,2,N}= N^{\max\{d_{0,r},\frac{1}{4}\}},
\widetilde{c}_{s,2,N}= N^{\max\{d_{0,s},\frac{1}{4}\}};
\\
c_{r,3,N}= N^{\max\{d_{0,r},\frac{1}{4}\}},
\widetilde{c}_{s,3,N}=m^{\frac{1}{4}} N^{\max\{d_{0,s},0\}};
\hspace{0.2cm}
c_{r,4,N}=m^{\frac{1}{4}}N^{\max\{d_{0,r},0\}},
\widetilde{c}_{s,4,N}=m^{\frac{1}{4}}N^{\max\{d_{0,s},0\}}.
\end{gathered}
\end{equation*}
Note that $c_{r,i,N}$ is different from $\widetilde{c}_{r,i,N}$ only for $i=3$. This notation though will allow writing our arguments in a more unified way.

\begin{lemma} \label{prop:supGhat}
Let $\{X_{n}\}_{n \in \ZZ}$ be a $p$-dimensional stationary, centered, Gaussian time series with spectral density $f_{X}$ as in \eqref{eq:f}.
Then, there are positive constants $c_{1},c_{2}$ such that 
\begin{equation} \label{eq:concreal1}
\begin{aligned}
&\Prob \Big( \sup_{D \in \Omega } 
| \widehat{H}_{rs}(D) - \E \widehat{H}_{rs}(D) | > 
\vertiii{G} \nu \Big) \leq 
\mathcal{B}(r,s,i),
\hspace{0.2cm} i =2,\dots,5,
\end{aligned}
\end{equation}
for $\nu^2 \geq \gamma^4 L^2_{rs,i}/ (m^2 c_{2} )$, where
\begin{equation} \label{eq:mathcalB}
\mathcal{B}(r,s,i)=
c_{1} \exp\Bigg( -c_{2} \min 
\Bigg\{ 
\frac{ \nu m }{ \gamma^2 \widebar{\Delta}_{N} L_{rs,1} },
\frac{ \nu^2 m^2 }{ \gamma^4 \widebar{\Delta}_{N}^{2} L_{rs,i}^2 }
\Bigg\} \Bigg) 
\end{equation}
with $r \neq s$ if $i=3$ and $\Omega$ is given in \eqref{eq:OmegaAB}. The constant $\gamma = \frac{\sqrt{2}}{\Gamma(\frac{1}{2})}$ bounds the sub-Gaussian norm of a standard normal random variable as in \eqref{eq:ACsubGgamma} and $\widebar{\Delta}_{N}$ is defined in \eqref{eq:delta-0-1/4}. 
\end{lemma}

\begin{proof}[Proof]
To prove the desired concentration inequality, we first introduce some general notation.
For a fixed frequency, the periodogram in \eqref{eq:periodogram} can be written as
\begin{equation} \label{eq:PeriodogramMatrix}
I_{X}(\lambda_{j}) = \frac{1}{2 \pi} \Big(
\mathcal{X}' (C_{j}C_{j}' + S_{j}S_{j}') \mathcal{X} + i \mathcal{X}' (C_{j}S_{j}' - S_{j}C_{j}') \mathcal{X} \Big)
\end{equation}
with $\mathcal{X}' = [ X_{1} : \dots : X_{N} ]$ and
\begin{equation}
\begin{aligned}
C_{j}' &= \frac{1}{\sqrt{N}}(\cos(\lambda_{j}) , \dots, \cos( (N-1) \lambda_{j}) , 1), \\
S_{j}' &= \frac{1}{\sqrt{N}}(\sin(\lambda_{j}) , \dots, \sin( (N-1) \lambda_{j}), 0);
\end{aligned}
\end{equation}
see equation (A.6) in \cite{Sun2018:LargeSpectral}. 
The event of interest can be separated into real and imaginary parts as
\begin{equation} \label{eq:realim}
\begin{aligned}
&\Prob \Big( \sup_{ D \in \Omega }  | \widehat{H}_{rs}(D) - \E \widehat{H}_{rs}(D) | > 
\vertiii{G} \nu \Big) \\
& \leq
\Prob \Big( \sup_{ D \in \Omega }  | \Re( \widehat{H}_{rs}(D) - \E \widehat{H}_{rs}(D) ) | > \frac{1}{2} \vertiii{G} \nu \Big) 
\\ & \hspace{1cm} + 
\Prob \Big( \sup_{ D \in \Omega }  | \Im( \widehat{H}_{rs}(D) - \E \widehat{H}_{rs}(D) ) | > \frac{1}{2} \vertiii{G} \nu \Big) .
\end{aligned}
\end{equation}
Note that the imaginary part of the diagonal elements is zero, that is, $\Im( e_{r}' I_{X}(\lambda_{j}) e_{r} ) = 0$. 
We consider the diagonal elements first $(r=s)$ and then distinguish the real and imaginary parts in \eqref{eq:realim} for the off-diagonal elements $(r \neq s)$.
\par
\textit{Diagonal elements:}
To rewrite the real-valued diagonal elements $\widehat{H}_{rr}(D)$ as a quadratic form, we define the $N \times 2m$ matrix
\begin{equation} \label{eq:Rm}
R_{m}'=
\Big[ C_{1} : S_{1} : \dots : C_{m} : S_{m} \Big]
\end{equation}
and the matrix-valued function
\begin{equation} \label{eq:fctT2}
T_{r}(d)=\diag( t_{1,r}(d), t_{1,r}(d),\dots, t_{m,r}(d), t_{m,r}(d)).
%T(d)=\diag( \lambda_{1}^{d}, \lambda_{1}^{d},\dots, \lambda_{m}^{d}, \lambda_{m}^{d}).
\end{equation}
Then, in view of \eqref{eq:PeriodogramMatrix}, \eqref{eq:Rm}, \eqref{eq:fctT2} and since $e_{r}'t_{j}(D) = e_{r}'t_{j,r}(d_{r})$, $\widehat{H}_{rr}(D)$ can be written  as
\begin{equation} \label{eq:Hrr_rewritten}
\begin{aligned}
\frac{1}{m} e_{r}' \sum_{j=1}^{m} t_{j}(D) I_{X}(\lambda_{j}) t_{j}(D) e_{r}
&=
\frac{1}{2 \pi m} e_{r}' \sum_{j=1}^{m} t_{j}(D) \mathcal{X}' (C_{j}C_{j}' + S_{j}S_{j}') \mathcal{X} t_{j}(D) e_{r}
\\&=
\frac{1}{2 \pi m} \sum_{j=1}^{m} \mathcal{E}_{r}' \Sigma_{rr}^{\frac{1}{2}} t_{j,r}(d_{r}) 
(C_{j}C_{j}' + S_{j}S_{j}')
t_{j,r}(d_{r}) \Sigma_{rr}^{\frac{1}{2}} \mathcal{E}_{r}
\\&=
\frac{1}{2 \pi m} \mathcal{E}_{r}' \Sigma_{rr}^{\frac{1}{2}} 
R_{m}' T^2_{r}(d_{r}) R_{m}
\Sigma_{rr}^{\frac{1}{2}} \mathcal{E}_{r}
\end{aligned}
\end{equation}
with $e_{r}'\mathcal{X}'=\mathcal{E}_{r}' \Sigma_{rr}^{\frac{1}{2}}$, where $\mathcal{E}_{r}=(\varepsilon_{r,1},\dots,\varepsilon_{r,N})'$ is Gaussian with $\E ( \mathcal{E}_{r} \mathcal{E}_{r}' )=I_{N}$ and 
$\Sigma_{rr} = (\Sigma_{rr}(n-k))_{n,k=1,\dots,N} = \E (\mathcal{X} e_{r} (\mathcal{X} e_{r})') $. In \eqref{eq:realCI1} below, we apply Theorem \ref{th:Dicker_Th1} with $K=1$ and $R=b_{r}-a_{r} \leq 1$ to obtain
\begin{align}
&
\Prob \Big( \sup_{ D \in \Omega }  
| \widehat{H}_{rr}(D) - \E \widehat{H}_{rr}(D) | 
> \vertiii{G} \nu \Big) \nonumber
\\&=
\Prob \Big( \sup_{d \in [a_{r},b_{r}] } 
| \mathcal{E}_{r}' \Sigma_{rr}^{\frac{1}{2}} R_{m}' T^2_{r}(d) R_{m} \Sigma_{rr}^{\frac{1}{2}} \mathcal{E}_{r} %\nonumber
%\\& \phantom{=\Prob \Big( \sup_{d \in [a_{r},b_{b}] } | }
- 
\E (\mathcal{E}_{r}' \Sigma_{rr}^{\frac{1}{2}} R_{m}' T^2_{r}(d) R_{m} \Sigma_{rr}^{\frac{1}{2}} \mathcal{E}_{r}) | >
2 \pi m \vertiii{G} \nu \Big) 
\nonumber
\\&=
\Prob \Big( \sup_{d \in [a_{r},b_{r}] } 
| \mathcal{E}_{r}' R(d) \mathcal{E}_{r} \nonumber
- 
\E (\mathcal{E}_{r}' R(d) \mathcal{E}_{r}) | >
2 \pi m \vertiii{G} \nu \Big) 
\nonumber
\\&\leq
c_{1} \exp\Bigg( -c_{2} \min 
\Bigg\{ 
\frac{ \nu m \vertiii{G}}{ \gamma^2 \widetilde{\mathcal{T}}_{1} },
\frac{ \nu^2 m^2 \vertiii{G}^{2}}{ \gamma^4 \widetilde{\mathcal{T}}_{i}^2}
\Bigg\} \Bigg) \label{eq:realCI1}
\\&\leq
c_{1} \exp\Bigg( -c_{2} \min 
\Bigg\{ 
\frac{ \nu m }{ \gamma^2 \widebar{\Delta}_{N} \widetilde{L}_{1} },
\frac{ \nu^2 m^2 }{ \gamma^4 \widebar{\Delta}_{N}^{2} \widetilde{L}_{i}^2 }
\Bigg\} \Bigg) \label{eq:realCI2}
\end{align}
for $ \nu^2 \geq \gamma^4 \widetilde{\mathcal{T}}^2_{i}/(c_{2} m^2 \vertiii{G}^{2} )$
with $i=2,3,4$ and $R(d)=\Sigma_{rr}^{\frac{1}{2}} R_{m}' T^2_{r}(d) R_{m} \Sigma_{rr}^{\frac{1}{2}}$ and
\begin{equation} \label{eq:realmathcalT}
\begin{aligned}
\widetilde{\mathcal{T}}_{1}&=\| A_{m} \| \left( 
\| T^{2}_{r}(a_{r})\| +
\sup_{d \in [a_{r},b_{r}]} \| \frac{\partial}{\partial d} T^{2}_{r}(d)\| \right),
\\
\widetilde{\mathcal{T}}_{2}=\widetilde{\mathcal{T}}_{3}&=\| A_{m} \|_{F} \left(
\| T^{2}_{r}(a_{r})\| +
\sup_{d \in [a_{r},b_{r}]} \| \frac{\partial}{\partial d} T^{2}_{r}(d)\| \right),
\\
\widetilde{\mathcal{T}}_{4}&=\| A_{m} \| \left(
\| T^{2}_{r}(a_{r})\|_{F} +
\sup_{d \in [a_{r},b_{r}]} \| \frac{\partial}{\partial d} T^2_{r}(d) \|_{F} \right),
\end{aligned}
\end{equation}
where $A_{m}=R_{m} \Sigma_{rr} R_{m}'  $. Though $\widetilde{\mathcal{T}}_{2}=\widetilde{\mathcal{T}}_{3}$, we bound them differently in \eqref{eq:realCI2}.
Furthermore, 
\begin{equation} \label{eq:Ltilde}
\widetilde{L}_{1}
=L_{rr,1},
\hspace{0.2cm}
\widetilde{L}_{2}
=L_{rr,2},
\hspace{0.2cm}
\widetilde{L}_{3}
=L_{rr,4},
\hspace{0.2cm}
\widetilde{L}_{4}
=L_{rr,5}
\end{equation}
with $L_{rr,1}, L_{rr,2}, L_{rr,4}$ and $L_{rr,5}$ as in \eqref{eq:prop:supGhatLs}.
We get \eqref{eq:realCI2} by bounding the quantities in \eqref{eq:realmathcalT}. Note that
\begin{align}
\Vert A_{m} \Vert
&=
\Vert R_{m} \Sigma_{rr} R_{m}' \Vert
\leq
\Vert R_{m} \Vert^2 \Vert \Sigma_{rr} \Vert 
\leq
c \vertiii{G} N^{\max\{2d_{0,r},0\}} \widebar{\Delta}_{N},  \label{eq:specRS} \\
\Vert A_{m} \Vert_{F}
&=
\Vert R_{m} \Sigma_{rr} R_{m}' \Vert_{F}
\leq
\Vert R_{m} \Vert^2 \Vert \Sigma_{rr} \Vert _{F}
\leq
c \vertiii{G} N^{\max\{2d_{0,r},\frac{1}{2}\}} \log(N)^{\frac{1}{2}} \widebar{\Delta}_{N}, \label{eq:specRS1}
\\
\Vert A_{m} \Vert_{F}
&\leq
\sqrt{2m} \Vert R_{m} \Sigma_{rr} R_{m}' \Vert
\leq
c \vertiii{G} \sqrt{2m} N^{\max\{2d_{0,r},0\}} \widebar{\Delta}_{N}, \label{eq:specRS2}
\end{align}
where we used the submultiplicativity of the spectral norm, Lemma \ref{le:normineq1} and the fact that $\| A \|_{F} \leq \sqrt{\rank(A)} \| A \|$ for a matrix $A$ for the first inequalities in \eqref{eq:specRS}--\eqref{eq:specRS2}, respectively. The last inequalities in \eqref{eq:specRS} and \eqref{eq:specRS1} for the spectral and Frobenius norms of $\Sigma_{rr}$ follow by Lemmas \ref{le:ineqSpecSigmashort} and \ref{le:ineqSpecFrobSigma}. 
Furthermore, we used the fact that the spectral norm of $R_{m}$ can be calculated as
\begin{equation*} 
%\label{eq:specnormRm}
\Vert R_{m} \Vert= 1;
\end{equation*}
see Lemma C.4 in \cite{Sun2018:LargeSpectral}.
\par
\textit{Real part (off-diagonal):}
The real part of the off-diagonal elements $\Re(\widehat{H}_{rs}(D))$ can be written in terms of \eqref{eq:PeriodogramMatrix} as
\begin{equation} \label{eq:off_diag_rew}
\frac{1}{m}e_{r}' \sum_{j=1}^{m} t_{j}(D) \Re(I_{X}(\lambda_{j})) t_{j}(D) e_{s}
=
\frac{1}{2 \pi m} e_{r}' \sum_{j=1}^{m} t_{j}(D) \mathcal{X}'
(C_{j}C_{j}' + S_{j}S_{j}')
\mathcal{X} t_{j}(D) e_{s}.
\end{equation}
As in \eqref{eq:Hrr_rewritten}, we will write \eqref{eq:off_diag_rew} as a quadratic form but now using
$(e_{r}'\mathcal{X}' ~ e_{s}'\mathcal{X}' )=\mathcal{E}' \Sigma^{\frac{1}{2}}$, where $\mathcal{E}' = (\mathcal{E}_{r}' ~ \mathcal{E}_{s}')$ is a Gaussian vector with $\E(\mathcal{E} \mathcal{E}')=I_{2N}$ and 
\begin{equation} \label{eq:Sigma_rs}
\Sigma = 
\begin{pmatrix}
\Sigma_{rr} & \Sigma_{rs} \\
\Sigma_{sr} & \Sigma_{ss} 
\end{pmatrix}
\hspace{0.2cm}
\text{ with }
\hspace{0.2cm}
\Sigma_{rs} = (\Sigma_{rs}(n-k))_{n,k=1,\dots,N} = \E (\mathcal{X} e_{r} (\mathcal{X} e_{s})') .
\end{equation}
For this, define the $4m \times 2N$ matrix
\begin{equation*} 
\label{eq:Rtildem}
\widetilde{R}_{m} = 
\begin{pmatrix}
R_{m} & 0_{2m \times N} \\
0_{2m \times N} & R_{m}
\end{pmatrix}
\end{equation*}
with $R_{m}$ as in \eqref{eq:Rm}. 
Furthermore, define the matrix $F(d_{r},d_{s})=\diag(T_{r}(d_{r}),T_{s}(d_{s}))$ with $T_{r}$ as in \eqref{eq:fctT2} and 
\begin{equation} \label{eq:fctT}
M_{m} =
\begin{pmatrix}
0_{m,m} &  I_{m} \\
0_{m,m} & 0_{m,m}
\end{pmatrix}.
\end{equation}
Write
\begin{equation*}
\begin{aligned}
e_{r}' \sum_{j=1}^{m} t_{j}(D) \mathcal{X}'
(C_{j}C_{j}' + S_{j}S_{j}')
\mathcal{X} t_{j}(D) e_{s}
&=
\mathcal{E}' \Sigma^{\frac{1}{2}} \widetilde{R}_{m}' F(d_{r},d_{s}) M_{2m} 
F(d_{r},d_{s}) \widetilde{R}_{m} \Sigma^{\frac{1}{2}} \mathcal{E}
\\&=
\mathcal{E}' R(d_{r},d_{s}) \mathcal{E}
\end{aligned}
\end{equation*}
with
\begin{equation*}
R(d_{r},d_{s})=\Sigma^{\frac{1}{2}} \widetilde{R}_{m}' F(d_{r},d_{s}) M_{2m} F(d_{r},d_{s}) \widetilde{R}_{m} \Sigma^{\frac{1}{2}}.
\end{equation*}
In order to apply Theorem \ref{th:Dicker_Th1}, we further write
\begin{align}
&\Prob \Big( \sup_{ D \in \Omega }   | 
\sum_{j=1}^{m} e_{r}' (t_{j}(D) \mathcal{X}'
(C_{j}C_{j}' + S_{j}S_{j}')
\mathcal{X} t_{j}(D) - 
\E (t_{j}(D) \mathcal{X}'
(C_{j}C_{j}' + S_{j}S_{j}')
\mathcal{X} t_{j}(D) ) ) e_{s} | > \pi m \vertiii{G} \nu \Big) \nonumber
\\&=
\Prob \Big( \sup_{(d_{r},d_{s}) \in 
[a_{r},b_{r}] \times [a_{s},b_{s}]} | \mathcal{E}' R(d_{r},d_{s}) \mathcal{E}  - 
\E (\mathcal{E}' R(d_{r},d_{s}) \mathcal{E} ) | > \pi m \vertiii{G} \nu \Big). \label{eq:real_formthc1}
\end{align}
Note that the matrix $R(d_{r},d_{s})$ can be rewritten as
\begin{equation} \label{eq:Rrewrtitten}
\begin{aligned}
R(d_{r},d_{s})
&=
\Sigma^{\frac{1}{2}} \mathcal{A}_{N,i}^{-1} \widetilde{R}_{m}' \mathcal{A}_{2m,i} F(d_{r},d_{s}) M_{2m} F(d_{r},d_{s})  \mathcal{A}_{2m,i} \widetilde{R}_{m} \mathcal{A}_{N,i}^{-1} \Sigma^{\frac{1}{2}}
\end{aligned}
\end{equation}
for $i=1,\dots,4$ with
\begin{equation} \label{eq:mathcalAi}
\begin{aligned}
\mathcal{A}_{m,1}
&=
\diag(c_{r,1,N}I_{m},\widetilde{c}_{s,1,N}I_{m}),
\hspace{0.2cm}
\mathcal{A}_{m,2}
=
\diag(c_{r,2,N}I_{m},\widetilde{c}_{s,2,N}I_{m}),
\\
\mathcal{A}_{m,3}
&=
\diag(c_{r,3,N}I_{m},\widetilde{c}_{s,3,N}I_{m}),
\hspace{0.2cm}
\mathcal{A}_{m,4}
=
m^{\frac{1}{4}} \mathcal{A}_{m,1}.
\end{aligned}
\end{equation}
The matrices $\mathcal{A}_{N,i}$ in \eqref{eq:Rrewrtitten} are defined by replacing $I_{m}$'s in \eqref{eq:mathcalAi} by $I_{N}$. These matrices are to normalize $\Sigma^{\frac{1}{2}}$ in \eqref{eq:Rrewrtitten}.
\par
We continue to bound \eqref{eq:real_formthc1} by applying Theorem \ref{th:Dicker_Th1}. In order to verify the applicability of Theorem \ref{th:Dicker_Th1}, note that the matrix $\mathcal{A}_{2m,i} F(d_{r},d_{s}) M_{2m} F(d_{r},d_{s})  \mathcal{A}_{2m,i} $ in \eqref{eq:Rrewrtitten}  is not diagonal but can be represented as a unitary transformation of a diagonal matrix
\begin{equation} \label{eq:real_SVDTMT}
\begin{aligned}
&
\mathcal{A}_{2m,i} F(d_{r},d_{s})' M_{2m} F(d_{r},d_{s}) \mathcal{A}_{2m,i}
\\&=
\mathcal{A}_{2m,i}
\begin{pmatrix}
0_{2m,2m} & T_{r}(d_{r})T_{s}(d_{s})  \\
0_{2m,2m} & 0_{2m,2m}
\end{pmatrix}
\mathcal{A}_{2m,i}
\\&=
\begin{pmatrix}
0_{2m,2m} & c_{r,i,N}\widetilde{c}_{s,i,N}T_{r}(d_{r})T_{s}(d_{s}) \\
0_{2m,2m} & 0_{2m,2m}
\end{pmatrix}
\\&=
\begin{pmatrix}
I_{2m} & 0_{2m,2m} \\
0_{2m,2m}  &I_{2m}
\end{pmatrix}
\begin{pmatrix}
c_{r,i,N}\widetilde{c}_{s,i,N}T_{r}(d_{r})T_{s}(d_{s}) & 0_{2m,2m} \\
0_{2m,2m} & 0_{2m,2m}
\end{pmatrix}
\begin{pmatrix}
0_{2m,2m}  &I_{2m}  \\
I_{2m} & 0_{2m,2m}
\end{pmatrix}.
\end{aligned}
\end{equation}
For this reason, Theorem \ref{th:Dicker_Th1} remains applicable due to Remark \ref{re:Th1_diag}.
In \eqref{eq:real_imCI1} below, we apply Theorem \ref{th:Dicker_Th1} with $K=2$ and $R = b_{r} - a_{r} \leq 1$ to obtain
\begin{align}
&\Prob \Big( \sup_{(d_{r},d_{s}) \in 
[a_{r},b_{r}] \times [a_{s},b_{s}]} | \mathcal{E}' R(d_{r},d_{s}) \mathcal{E}  - 
\E (\mathcal{E}' R(d_{r},d_{s}) \mathcal{E} ) | > \pi m \vertiii{G} \nu \Big) \nonumber
\\&\leq
c_{1} \exp\Bigg( -c_{2} \min 
\Bigg\{ 
\frac{ \nu m \vertiii{G} }{ \gamma^2 \mathcal{T}_{1} },
\frac{ \nu^2 m^2 \vertiii{G}^2 }{ \gamma^4 \mathcal{T}^2_{i} }
\Bigg\} \Bigg) \label{eq:real_imCI1}
\\&\leq
c_{1} \exp\Bigg( -c_{2} \min 
\Bigg\{ 
\frac{ \nu m }{ \gamma^2 \widebar{\Delta}_{N} L_{rs,1} },
\frac{ \nu^2 m^2 }{ \gamma^4 \widebar{\Delta}_{N}^{2} L^2_{rs,i} }
\Bigg\} \Bigg) \label{eq:real_imCI2}
\end{align}
for $ \nu^2 \geq \gamma^4 \mathcal{T}^2_{i}/(c_{2} m^2 \vertiii{G}^2 )$ and $i=2,\dots,5$ with
\begin{equation} \label{eq:real_imTi}
\mathcal{T}_{1}=\| B_{m,1} \| L_{rs,1},
\hspace{0.2cm}
\mathcal{T}_{i}=\| B_{m,i} \|_{F} L_{rs,i},
\hspace{0.2cm}
\mathcal{T}_{5}=\| B_{m,1} \| L_{rs,5}
\end{equation}
for $i=2,\dots,4$, where the $L_{rs,i}$'s are given in \eqref{eq:prop:supGhatLs} and 
\begin{equation} \label{eq:real_Bmi}
B_{m,i}= \widetilde{R}_{m} \mathcal{A}_{N,i}^{-1} \Sigma \mathcal{A}_{N,i}^{-1} \widetilde{R}_{m}'
\end{equation}
for $i=1,\dots,4$. The $L_{rs,i}$'s in \eqref{eq:real_imTi} can indeed be represented as in \eqref{eq:prop:supGhatLs} due to \eqref{eq:real_SVDTMT} and Remark \ref{re:Th1_diag}, and since $\| T_{r}(d_{r})T_{s}(d_{s}) \| = \| T(d_{r},d_{s})\|$ and $\| T_{r}(d_{r})T_{s}(d_{s}) \|_{F} = \sqrt{2}\| T(d_{r},d_{s})\|_{F}$ and similarly with $\nabla$, where $T(d_{r},d_{s})$ is in \eqref{eq:TandboldT}. We will now discuss bounds on the Frobenius and spectral norms of $B_{m,i}$ to get \eqref{eq:real_imCI2}. 

The relation \eqref{eq:real_imCI2} is a consequence of bounding the quantities in \eqref{eq:real_imTi} as follows.
Let $c$ denote a generic constant which might differ from line to line.
Then, with the explanations given below,
\begin{equation} \label{eq:real_Bm1}
\begin{aligned}
\| B_{m,1} \| 
&=
\| \widetilde{R}_{m} \mathcal{A}_{N,1}^{-1} \Sigma \mathcal{A}_{N,1}^{-1} \widetilde{R}_{m}' \|
\leq
\| \widetilde{R}_{m} \|^2 \| \mathcal{A}_{N,1}^{-1} \Sigma \mathcal{A}_{N,1}^{-1} \| 
\\&\leq 
\| c_{r,1,N}^{-2} \Sigma_{rr} \| + \| \widetilde{c}_{s,1,N}^{\ -2} \Sigma_{ss} \|
\\&= 
N^{-\max\{2d_{0,r},0\}} \| \Sigma_{rr} \| + N^{-\max\{2d_{0,s},0\}} \| \Sigma_{ss} \|
\leq 
c \widebar{\Delta}_{N} \vertiii{G},
\end{aligned}
\end{equation}
\begin{equation} \label{eq:real_Bm2}
\begin{aligned}
\| B_{m,2} \|_{F}
&=
\| \widetilde{R}_{m} \mathcal{A}_{N,2}^{-1} \Sigma \mathcal{A}_{N,2}^{-1} \widetilde{R}_{m}' \|_{F}
\leq
\|  \widetilde{R}_{m} \|^2 \| \mathcal{A}_{N,2}^{-1} \Sigma \mathcal{A}_{N,2}^{-1} \| _{F}
\\&\leq
\| c_{r,2,N}^{-2} \Sigma_{rr} \|_{F} + \| \widetilde{c}_{s,2,N}^{\ -2} \Sigma_{ss} \|_{F}
\\&=
N^{-\max\{2d_{0,r},\frac{1}{2}\}} \| \Sigma_{rr} \|_{F} + N^{-\max\{2d_{0,s},\frac{1}{2}\}} \| \Sigma_{ss} \|_{F}
\leq 
c \widebar{\Delta}_{N} \vertiii{G},
\end{aligned}
\end{equation}
%\textit{For $\Delta_{1}\leq d_{0,r} \leq \Delta_{2}$ and $d_{0,s} <0$:} We get
\begin{equation} \label{eq:real_Bm3}
\begin{aligned}
\| B_{m,3} \|_{F}
&=
\| \widetilde{R}_{m} \mathcal{A}_{N,3}^{-1} \Sigma \mathcal{A}_{N,3}^{-1} \widetilde{R}_{m}' \|_{F}
\\&\leq
\| R_{m}' c_{r,3,N}^{-1} \Sigma_{rr} c_{r,3,N}^{-1} R_{m} \|_{F}
+
\| R_{m}' \widetilde{c}_{s,3,N}^{\ -1} \Sigma_{ss} \widetilde{c}_{s,3,N}^{\ -1} R_{m} \|_{F}
\\&\leq
\| R_{m}' c_{r,3,N}^{-1} \Sigma_{rr} c_{r,3,N}^{-1} R_{m} \|_{F}
+
\sqrt{2m} \| R_{m}' \widetilde{c}_{s,3,N}^{\ -1} \Sigma_{ss} \widetilde{c}_{s,3,N}^{\ -1} R_{m} \|
\\&\leq
\| R_{m} \|^2 \| c_{r,3,N}^{-2} \Sigma_{rr} \|_{F}
+
\sqrt{2m} \| R_{m} \|^2 \| \widetilde{c}_{s,3,N}^{\ -2} \Sigma_{ss} \|
%\leq
%\| Q_{m} \|^2 \| \mathcal{A}_{N,3}^{-1} \Sigma \mathcal{A}_{N,3}^{-1} \|_{F}
\\&=
N^{-\max\{2d_{0,r},\frac{1}{2}\}} \| \Sigma_{rr} \| _{F} + N^{-\max\{2d_{0,s},0\}} \| \Sigma_{ss}\|
\leq 
c \widebar{\Delta}_{N} \vertiii{G},
\end{aligned}
\end{equation}
\begin{equation} \label{eq:real_Bm4}
\begin{aligned}
\| B_{m,4} \|_{F}
&=
\| \widetilde{R}_{m} \mathcal{A}_{N,4}^{-1} \Sigma \mathcal{A}_{N,4}^{-1} \widetilde{R}_{m}' \|_{F}
\\&\leq
\sqrt{4m}\| \widetilde{R}_{m}' \mathcal{A}_{N,4}^{-1} \Sigma \mathcal{A}_{N,4}^{-1} \widetilde{R}_{m} \|
=
\sqrt{4} \| \widetilde{R}_{m}' \mathcal{A}_{N,1}^{-1} \Sigma \mathcal{A}_{N,1}^{-1} \widetilde{R}_{m} \|
\leq 
c \widebar{\Delta}_{N} \vertiii{G}.
\end{aligned}
\end{equation}
The first bounds in \eqref{eq:real_Bm1} and \eqref{eq:real_Bm2} follow from the submultiplicativity of the spectral norm and Lemma \ref{le:normineq1}, respectively. 
Then, we use Lemma \ref{le:normineq2} to eliminate the off-diagonal matrix blocks, and Lemmas \ref{le:ineqSpecSigmashort} and \ref{le:ineqSpecFrobSigma} are used to bound the spectral and Frobenius norms of $\Sigma_{rr}$. 
In \eqref{eq:real_Bm3}, we first apply Lemma \ref{le:normineq2} and then the fact that $\| A \|_{F} \leq \sqrt{\rank(A)} \| A \|$ for a matrix $A$. Note also that $\rank(AB) \leq \min\{ \rank(A), \rank(B) \}$ for two matrices $A,B$.
Similarly, we applied $\| A \|_{F} \leq \sqrt{\rank(A)} \| A \|$ in \eqref{eq:real_Bm4} to use \eqref{eq:real_Bm1}.
Note also that $\| \widetilde{R}_{m} \|^2 = \| R_{m} \|^{2} =1$, where the first equality is due to the block diagonal structure of $\widetilde{R}_{m}$ and the second equality follows by Lemma C.4 in \cite{Sun2018:LargeSpectral}.
\par
\textit{Imaginary part (off-diagonal):}
The imaginary part of the off-diagonal elements $\Im(\widehat{H}_{rs}(D))$ can be written in terms of \eqref{eq:PeriodogramMatrix} as
\begin{equation*}
\frac{1}{m}e_{r}' \sum_{j=1}^{m} t_{j}(D) \Im(I_{X}(\lambda_{j})) t_{j}(D) e_{s}
=
\frac{1}{2 \pi m} e_{r}' \sum_{j=1}^{m} t_{j}(D) \mathcal{X}'
(C_{j}S_{j}' - S_{j}C_{j}')
\mathcal{X} t_{j}(D) e_{s}.
\end{equation*}
Then,
\begin{equation} \label{eq:twoboundsimaginarypart}
\begin{aligned}
& \Prob \Big( \sup_{ D \in \Omega }  | \Im(\widehat{H}_{rs}(D) - \E \widehat{H}_{rs}(D)) | > 
\frac{1}{2} \vertiii{G} \nu \Big) \\
& \leq
\Prob \Big( \sup_{ D \in \Omega }  | \sum_{j=1}^{m} e_{r}' (t_{j}(D) \mathcal{X}'
C_{j}S_{j}'
\mathcal{X} t_{j}(D) - 
\E (t_{j}(D) \mathcal{X}'
C_{j}S_{j}'
\mathcal{X} t_{j}(D) )) e_{s} | > \frac{\pi m}{2} \vertiii{G} \nu \Big)
\\ & \hspace{1cm} + 
\Prob \Big( \sup_{ D \in \Omega }  | \sum_{j=1}^{m} e_{r}' (t_{j}(D) \mathcal{X}'
S_{j}C_{j}'
\mathcal{X} t_{j}(D) - 
\E (t_{j}(D) \mathcal{X}'
S_{j}C_{j}'
\mathcal{X} t_{j}(D) )) e_{s} | > \frac{\pi m}{2} \vertiii{G} \nu \Big).
\end{aligned}
\end{equation}
We focus on the first probability term in the bound, since the second can be dealt with analogously. 
Define the $2m \times 2N$ matrix
\begin{equation*} 
%\label{eq:Qm}
Q_{m} = 
\begin{pmatrix}
Q_{1,m} & 0_{m \times N} \\
0_{m \times N} & Q_{2,m}
\end{pmatrix}
\end{equation*}
with
\begin{equation*}
Q_{1,m}'
=
\Big[C_{1} : \dots : C_{m} \Big],
\hspace{0.2cm}
Q_{2,m}'
=
\Big[S_{1} : \dots : S_{m} \Big].
\end{equation*}
Furthermore, define the matrix $\widetilde{F}(d_{r},d_{s})=\diag(\widetilde{T}_{r}(d_{r}),\widetilde{T}_{s}(d_{s}))$ in terms of the function $\widetilde{T}_{r}(d)=\diag( t_{1,r}(d),\dots,t_{m,r}(d))$ and recall $M_{m}$ in \eqref{eq:fctT}
to write
\begin{equation*}
e_{r}' \sum_{j=1}^{m} t_{j}(D) \mathcal{X}'
C_{j}S_{j}' 
\mathcal{X} t_{j}(D) e_{s}
=
\mathcal{E}' \Sigma^{\frac{1}{2}} Q_{m}' \widetilde{F}(d_{r},d_{s}) M_{m} 
\widetilde{F}(d_{r},d_{s}) Q_{m} \Sigma^{\frac{1}{2}} \mathcal{E}
=
\mathcal{E}' Q(d_{r},d_{s}) \mathcal{E}
\end{equation*}
with
\begin{equation*}
Q(d_{r},d_{s})=\Sigma^{\frac{1}{2}} Q_{m}' \widetilde{F}(d_{r},d_{s}) M_{m} \widetilde{F}(d_{r},d_{s}) Q_{m} \Sigma^{\frac{1}{2}}
\end{equation*}
and $\Sigma^{\frac{1}{2}} \mathcal{E}$ characterized as in \eqref{eq:Sigma_rs}.
%with $(e_{r}'\mathcal{X}' ~ e_{s}'\mathcal{X}' )=\mathcal{E}' \Sigma^{\frac{1}{2}}$, where $\mathcal{E}' = (\mathcal{E}_{r}' ~ \mathcal{E}_{s}')$ is a Gaussian vector with $\E(\mathcal{E} \mathcal{E}')=I_{2N}$ and 
%\begin{equation}
%\Sigma = 
%\begin{pmatrix}
%\Sigma_{rr} & \Sigma_{rs} \\
%\Sigma_{sr} & \Sigma_{ss} 
%\end{pmatrix}
%\hspace{0.2cm}
%\text{ with }
%\hspace{0.2cm}
%\Sigma_{rs} = (\Sigma_{rs}(n-k))_{n,k=1,\dots,N} = \E (\mathcal{X} e_{r} (\mathcal{X} e_{s})') .
%\end{equation}
In order to apply Theorem \ref{th:Dicker_Th1}, we further write
\begin{align}
&\Prob \Big( \sup_{ D \in \Omega }   | 
\sum_{j=1}^{m} e_{r}' (t_{j}(D) \mathcal{X}'
C_{j}S_{j}' 
\mathcal{X} t_{j}(D) - 
\E (t_{j}(D) \mathcal{X}'
C_{j}S_{j}' 
\mathcal{X} t_{j}(D) )) e_{s} | > \frac{\pi m}{2} \vertiii{G} \nu \Big) \nonumber
\\&=
\Prob \Big( \sup_{(d_{r},d_{s}) \in 
[a_{r},b_{r}] \times [a_{s},b_{s}]} | \mathcal{E}' Q(d_{r},d_{s}) \mathcal{E}  - 
\E (\mathcal{E}' Q(d_{r},d_{s}) \mathcal{E} ) | > \frac{\pi m}{2} \vertiii{G} \nu \Big). \label{eq:formthc1}
\end{align}
Note that the matrix $Q(d_{r},d_{s})$ can be rewritten as
\begin{equation*} 
%\label{eq:Qrewrtitten}
\begin{aligned}
Q(d_{r},d_{s})
&=
\Sigma^{\frac{1}{2}} \mathcal{A}_{N,i}^{-1} Q_{m}' \mathcal{A}_{m,i} \widetilde{F}(d_{r},d_{s}) M_{m} 
\widetilde{F}(d_{r},d_{s}) \mathcal{A}_{m,i} Q_{m} \mathcal{A}_{N,i}^{-1} \Sigma^{\frac{1}{2}}
\end{aligned}
\end{equation*}
for $i=1,\dots,4$ with $\mathcal{A}_{m,1}$ and $\mathcal{A}_{N,i}$, $i = 1, \dots, 4$ defined as in \eqref{eq:mathcalAi}. We continue to bound \eqref{eq:formthc1}. As for the real parts of the off-diagonal elements, note that $\mathcal{A}_{m,i} \widetilde{F}(d_{r},d_{s})' M_{m} \widetilde{F}(d_{r},d_{s}) \mathcal{A}_{m,i}$ is not a diagonal matrix. However, it can be rewritten as a unitary transformation of a diagonal matrix
\begin{equation} \label{eq:SVDTMT}
\begin{aligned}
&
\mathcal{A}_{m,i} \widetilde{F}(d_{r},d_{s})' M_{m} \widetilde{F}(d_{r},d_{s}) \mathcal{A}_{m,i}
\\&=
\begin{pmatrix}
I_{m} & 0_{m,m} \\
0_{m,m} & I_{m} 
\end{pmatrix}
\begin{pmatrix}
c_{r,i,N}\widetilde{c}_{s,i,N}\widetilde{T}_{r}(d_{r})\widetilde{T}_{s}(d_{s}) & 0_{m,m} \\
0_{m,m} & 0_{m,m}
\end{pmatrix}
\begin{pmatrix}
0_{m,m} & I_{m} \\
I_{m} & 0_{m,m}
\end{pmatrix},
\end{aligned}
\end{equation}
where we used the same calculations as in \eqref{eq:real_SVDTMT}. Following Remark \ref{re:Th1_diag}, Theorem \ref{th:Dicker_Th1} is applicable.
In \eqref{eq:imCI1} below, we apply Theorem \ref{th:Dicker_Th1} with $K=2$ and $R = b_{r} - a_{r} \leq 1$ to obtain
\begin{align}
&\Prob \Big( \sup_{(d_{r},d_{s}) \in 
[a_{r},b_{r}] \times [a_{s},b_{s}]} | \mathcal{E}' Q(d_{r},d_{s}) \mathcal{E}  - 
\E (\mathcal{E}' Q(d_{r},d_{s}) \mathcal{E} ) | > \frac{2 \pi m}{4} \vertiii{G} \nu \Big) \nonumber
\\&\leq
c_{1} \exp\Bigg( -c_{2} \min 
\Bigg\{ 
\frac{ \nu m \vertiii{G} }{ \gamma^2 \mathcal{T}_{1} },
\frac{ \nu^2 m^2 \vertiii{G}^2 }{ \gamma^4 \mathcal{T}^2_{i} }
\Bigg\} \Bigg) \label{eq:imCI1}
\\&\leq
c_{1} \exp\Bigg( -c_{2} \min 
\Bigg\{ 
\frac{ \nu m }{ \gamma^2 \widebar{\Delta}_{N} L_{rs,1} },
\frac{ \nu^2 m^2 }{ \gamma^4 \widebar{\Delta}_{N}^{2} L^2_{rs,i} }
\Bigg\} \Bigg) \label{eq:imCI2}
\end{align}
for $\nu^2 \geq \gamma^4 \mathcal{T}^2_{i} / ( c_{2} m^2 \vertiii{G}^2 )$ and $i=2,\dots,5$ with
\begin{equation} \label{eq:imTi}
\mathcal{T}_{1}=\| C_{m,1} \| L_{rs,1},
\hspace{0.2cm}
\mathcal{T}_{i}=\| C_{m,i} \|_{F} L_{rs,i},
\hspace{0.2cm}
\mathcal{T}_{5}=\| C_{m,1} \| L_{rs,5}
\end{equation}
for $i=2,\dots,4$, where the $L_{rs,i}$'s are given in \eqref{eq:prop:supGhatLs} and 
\begin{equation} \label{eq:im_Bmi}
C_{m,i}= Q_{m} \mathcal{A}_{N,i}^{-1} \Sigma \mathcal{A}_{N,i}^{-1} Q_{m}'
\end{equation}
for $i=1,\dots,4$. Given \eqref{eq:SVDTMT}, the $L_{rs,i}$'s in \eqref{eq:imTi} can indeed be represented as in \eqref{eq:prop:supGhatLs}.

Since the quantities \eqref{eq:im_Bmi} are equal to those in \eqref{eq:real_Bmi} by replacing $Q_{m}$ with $\widetilde{R}_{m}$, the norms 
$\| C_{m,1} \|, \| C_{m,2} \|_{F}, \| C_{m,3} \|_{F}$ and $\| C_{m,4} \|_{F}$ can be bounded as the norms of $B_{m,i}$, $i=1,\dots,4$, in \eqref{eq:real_Bm1}--\eqref{eq:real_Bm4}. In order to deal with $Q_{m}$, note that $\| Q_{m} \|^2 = \max\{ \| Q_{1,m} \|^{2}, \| Q_{2,m} \|^{2} \} \leq 1$, where the equality is due to the block diagonal structure of $Q_{m}$ and the inequality follows by Lemma C.4 in \cite{Sun2018:LargeSpectral}.
\par
In order to get the statements \eqref{eq:concreal1}--\eqref{eq:mathcalB} of the lemma, we shall get an upper bound on the probability in \eqref{eq:realim}. For the diagonal terms ($r=s$), the imaginary part is zero. As proved in \eqref{eq:realCI2}, for $r=s$, \eqref{eq:realim} can be bounded by $\mathcal{B}(r,r,i)$ for $i =2,4,5$; see \eqref{eq:Ltilde}.
For the off-diagonal elements, one needs to get a bound on the two probabilities in \eqref{eq:realim}.
For the real part of the off-diagonal elements, \eqref{eq:real_imCI2} gives the upper bound $\mathcal{B}(r,s,i)$, for $i =2,\dots,5$. For the imaginary parts of the off-diagonal elements, which was further bounded in \eqref{eq:twoboundsimaginarypart}, \eqref{eq:imCI2} provides the upper bound $\mathcal{B}(r,s,i)$, for $i =2,\dots,5$.
Combining the different bounds leads to the desired result.
\end{proof}

The next lemma gives a bound on the deterministic term in \eqref{eq:results_GhatGtilde}.
\begin{lemma} \label{prop:bias}
Suppose Assumptions \ref{ass:f0}--\ref{ass:derivative}. Then, the deterministic term in 
\eqref{eq:results_GhatGtilde} can be bounded as
\begin{equation} \label{eq:prop:biasbound}
\begin{aligned}
\sup_{D \in \Omega} 
| \E \widehat{H}_{rs}(D) - \widetilde{H}_{rs}(D) | 
&\leq 
\sup_{(d_{r},d_{s}) \in [a_{r},b_{r}] \times [a_{s},b_{s}]} \Bigg(
\bm{c}_{G,1} \frac{1}{2 \pi m} \sum_{j=1}^{m} t_{j,r}(d_{r}) t_{j,s}(d_{s}) \lambda_{j}^{2q-d_{0,r}-d_{0,s}}
\\& \hspace{1cm}+
\frac{1}{2 \pi m} \sum_{j=1}^{m} t_{j,r}(d_{r}) t_{j,s}(d_{s}) j^{-1} \lambda_{j}^{-d_{0,r}-d_{0,s}} Q_{m}\Bigg) 
\end{aligned}
\end{equation}
with
\begin{equation} \label{eq:Om}
Q_{m} = 
\vertiii{G} 
\frac{72 (\cos(\lambda_{m}/2))^{-2} }{\pi(1+2\min\{\Delta_{1},-\Delta_{2}\})}
+ \bm{c}_{G,2} 4(2+ \log(m)).
\end{equation}
\end{lemma}

\begin{proof}[Proof]
The deviation of the expected value around the population quantity can be bounded as
\begin{align}
&
| \widetilde{H}_{rs}(D) - \E \widehat{H}_{rs}(D) | \nonumber \\
& =
\Big| \frac{1}{m} \sum_{j=1}^{m} t_{j,r}(d_{r}) t_{j,s}(d_{s}) \lambda_{j}^{-d_{0,r}-d_{0,s}} G_{0,rs} - 
\E \Big(\frac{1}{m} \sum_{j=1}^{m} t_{j,r}(d_{r}) t_{j,s}(d_{s}) I_{X,rs}(\lambda_{j}) \Big) \Big| \nonumber \\
& \leq
\Big| \frac{1}{m} \sum_{j=1}^{m} t_{j,r}(d_{r}) t_{j,s}(d_{s}) \lambda_{j}^{-d_{0,r}-d_{0,s}} G_{0,rs} -
 \frac{1}{m} \sum_{j=1}^{m} t_{j,r}(d_{r}) t_{j,s}(d_{s}) \lambda_{j}^{-d_{0,r}-d_{0,s}} G_{0,rs}(\lambda_{j}) \Big| \nonumber
 \\& \hspace{1cm}+
\Big| \frac{1}{m} \sum_{j=1}^{m} t_{j,r}(d_{r}) t_{j,s}(d_{s}) \lambda_{j}^{-d_{0,r}-d_{0,s}} G_{0,rs}(\lambda_{j}) - 
\E \Big(\frac{1}{m} \sum_{j=1}^{m} t_{j,r}(d_{r}) t_{j,s}(d_{s}) I_{X,rs}(\lambda_{j}) \Big) \Big| \nonumber 
\\ & =
\Big| \frac{1}{m} \sum_{j=1}^{m} t_{j,r}(d_{r}) t_{j,s}(d_{s}) \lambda_{j}^{-d_{0,r}-d_{0,s}} (G_{0,rs} - G_{0,rs}(\lambda_{j}) ) \Big| \nonumber
 \\& \hspace{1cm}+
\Big| \frac{1}{m} \sum_{j=1}^{m} t_{j,r}(d_{r}) t_{j,s}(d_{s}) \lambda_{j}^{-d_{0,r}-d_{0,s}} G_{0,rs}(\lambda_{j}) - 
\frac{1}{m} \sum_{j=1}^{m} t_{j,r}(d_{r}) t_{j,s}(d_{s}) \E (I_{X,rs}(\lambda_{j}) ) \Big| \nonumber 
\\& \leq
\Big| 
\frac{1}{m} \sum_{j=1}^{m} t_{j,r}(d_{r}) t_{j,s}(d_{s}) \lambda_{j}^{-d_{0,r}-d_{0,s}} (G_{0,rs} - G_{0,rs}(\lambda_{j}) ) 
\Big| 
\nonumber
 \\& \hspace{1cm}+
\Big| \frac{1}{m} \sum_{j=1}^{m} t_{j,r}(d_{r}) t_{j,s}(d_{s}) 
\Big(
\lambda_{j}^{-d_{0,r}-d_{0,s}} G_{0,rs}(\lambda_{j}) 
- \E (I_{X,rs}(\lambda_{j}) ) \Big) \Big| \label{eq:propbias1}
\\& \leq
\bm{c}_{G,1}
\frac{1}{2 \pi m} \sum_{j=1}^{m} t_{j,r}(d_{r}) t_{j,s}(d_{s}) \lambda_{j}^{2q-d_{0,r}-d_{0,s}} 
+
\frac{1}{2 \pi m} \sum_{j=1}^{m} t_{j,r}(d_{r}) t_{j,s}(d_{s}) j^{-1} \lambda_{j}^{-d_{0,r}-d_{0,s}} Q_{m}, \label{eq:propbias2}
\end{align}
where $Q_{m}$ is in \eqref{eq:Om}.
We consider the two different summands in \eqref{eq:propbias1} separately to prove \eqref{eq:propbias2}.
For the first summand in \eqref{eq:propbias1}, the upper bound is a consequence of Assumption \ref{ass:G-G}.

The second summand in \eqref{eq:propbias1} can be represented and bounded as 
\begin{align}
&
\Big| \frac{1}{2 \pi m} \sum_{j=1}^{m} t_{j,r}(d_{r}) t_{j,s}(d_{s}) 
\Big(
\lambda_{j}^{-d_{0,r}-d_{0,s}} G_{0,rs}(\lambda_{j}) 
- \E (I_{X,rs}(\lambda_{j}) ) \Big) \Big| \nonumber
\\ & \leq
\frac{1}{2 \pi m} \sum_{j=1}^{m} t_{j,r}(d_{r}) t_{j,s}(d_{s}) 
\Big| 
f_{rs}(\lambda_{j}) 
- \E (I_{X,rs}(\lambda_{j}) ) \Big| \nonumber
\\ & \leq
\frac{1}{2 \pi m} \sum_{j=1}^{m} t_{j,r}(d_{r}) t_{j,s}(d_{s}) 
N^{-1} \lambda_{j}^{ -1-d_{0,r}-d_{0,s} }
Q_{m}, \label{eq:www}
\end{align}
with $Q_{m}$ in \eqref{eq:Om} and the last inequality \eqref{eq:www} is a consequence of Lemma \ref{le:biasappB1} given Assumptions \ref{ass:f0} and \ref{ass:derivative}.
\end{proof}

We next prove Propositions \ref{prop:supGhat1}--\ref{prop:supGhat4}. The proofs are all consequences of Lemmas \ref{prop:supGhat} and \ref{prop:bias}, and structured in the same way. We first choose a function $t_{j}(D)$ in \eqref{eq:crucialterms}--\eqref{eq:crucialterms1} and then apply Lemma \ref{prop:supGhat} to the respective probabilistic part and Lemma \ref{prop:bias} to the respective deterministic part in the bound \eqref{eq:results_GhatGtilde}. 

For the probabilistic parts in the proofs of Propositions \ref{prop:supGhat1}--\ref{prop:supGhat4}, we note that Lemma \ref{prop:supGhat} requires 
$ \widetilde{\nu}^2  \geq \gamma^4 L^2_{rs,i}/(m^2 c_{2} )$, $i=2,\dots,5$ with $L_{rs,i}$ as in \eqref{eq:prop:supGhatLs}. We do not verify this condition in the proofs since it is automatically satisfied by the bounds we get on the $L_{rs,i}$'s and under the assumptions in Propositions \ref{prop:supGhat1}--\ref{prop:supGhat4}. This is due to our choices of $\nu$, $\nu_{1}$, $\nu_{2}$ and $\nu_{3}$ in \eqref{eq:nuinprop1}, \eqref{eq:nu1}, \eqref{eq:nu2} and \eqref{eq:nu3}, which are always of the form $\nu = \widetilde{\nu} + \mathcal{T}$, where $\mathcal{T}$ accounts for the respective bias terms and $\widetilde{\nu}$ is always chosen as $m^{-1} L_{rs,i} \precsim \widetilde{\nu}$.

\begin{proof}[Proof of Proposition \ref{prop:supGhat1}]
To apply Lemmas \ref{prop:supGhat} and \ref{prop:bias}, we take
\begin{equation} \label{eq:Ghat1fctOmega}
t_{j}(D) = \lambda_{j}^{D}
\hspace{0.2cm}
\text{ and }
\hspace{0.2cm}
\Omega=\Omega(\varepsilon)=\{ D \in \mathcal{M}_{\diag} | \Delta_{1} I_{p} \preccurlyeq D \preccurlyeq \Delta_{2} I_{p} \text{ and } \| D - D_{0} \|_{\max} \leq \varepsilon\}
\end{equation}
in \eqref{eq:crucialterms}--\eqref{eq:crucialterms1}. 
We consider the probabilistic and deterministic parts separately as in \eqref{eq:results_GhatGtilde}.

\textit{Probabilistic part:}
We distinguish three cases depending on whether the true memory parameters $d_{0,r}, d_{0,s}$ are positive or negative. Therefore, we write
\begin{equation} \label{eq:sepProp1Ai}
\begin{aligned}
\Prob \Big( \sup_{D \in \Omega(\varepsilon) }
| \widehat{G}_{rs}(D) - \E(\widehat{G}_{rs}(D)) | > \nu \Big)
&\leq
\sum_{k=1}^{3} \Prob \Big( \sup_{D \in \Omega(\varepsilon) }
| \widehat{G}_{rs}(D) - \E(\widehat{G}_{rs}(D)) |\mathds{1}_{A_{k}} > \nu \Big),
\end{aligned}
\end{equation}
where
\begin{equation*}
\begin{aligned}
A_{1}=\{d_{0,r}\leq 0,d_{0,s} \leq 0\},
\hspace{0.2cm}
A_{2}&=\{ d_{0,r} > 0, d_{0,s} > 0\},
\hspace{0.2cm}
A_{3}=\{d_{0,r} > 0, d_{0,s} \leq 0\} 
\end{aligned}
\end{equation*}
and
\begin{equation*}
\begin{aligned}
\Omega(\varepsilon)
=
([d_{0,1}-\varepsilon,d_{0,1}+\varepsilon] \cap [\Delta_{1},\Delta_{2}]) \times \cdots \times 
([d_{0,p}-\varepsilon,d_{0,p}+\varepsilon] \cap [\Delta_{1},\Delta_{2}]).
%\subset 
%[\Delta_{1},\Delta_{2}]^{p}
\end{aligned}
\end{equation*}
For each case, we apply Lemma \ref{prop:supGhat} and bound the respective quantities $L_{rs,i}$, $i=1,\dots,5$ in \eqref{eq:prop:supGhatLs} that then yield the desired result. We will show that, for $k=1,2,3$,
\begin{equation} \label{eq:LA1applied_mathcalLK}
\begin{aligned}
\Prob \Big( \sup_{D \in \Omega(\varepsilon) }
| \widehat{G}_{rs}(D) - \E(\widehat{G}_{rs}(D)) |\mathds{1}_{A_{k}} > \vertiii{G} \widetilde{\nu} \Big)
\leq
c_{1} \exp\Bigg( -c_{2} \min 
\Bigg\{ 
\frac{ \widetilde{\nu} m }{ \gamma^2 \mathcal{L}_{k1} },
\frac{ \widetilde{\nu}^2 m^2 }{ \gamma^4 \mathcal{L}_{k2}^2 }
\Bigg\} \Bigg).
\end{aligned}
\end{equation}

\textit{Case $d_{0,r} \leq 0,d_{0,s} \leq 0$:}
We get
\begin{align}
L_{rs,1}
&\leq 
\max_{j=1,\dots,m}  |\lambda_{j}^{2\Delta_{1}}| 
+
\sup_{(d_{r},d_{s}) \in [\Delta_{1},\Delta_{2}]^{2}} 
\max_{j=1,\dots,m} \sqrt{2} |\log(\lambda_{j}) \lambda_{j}^{d_{r}+d_{s}}| \nonumber
\\&\leq 
c (1+\log(N)) N^{-2\Delta_{1}} 
\leq 
c \log(N) N^{-2\Delta_{1}} =: \mathcal{L}_{11} \label{eq:1L11},
\\
L_{rs,5}
&\leq
\Big( \sum_{j=1}^{m}\lambda_{j}^{4 \Delta_{1}} \Big)^{\frac{1}{2}}
+
\sup_{(d_{r},d_{s}) \in [\Delta_{1},\Delta_{2}]^{2}} 
\sqrt{2} \Big( \sum_{j=1}^{m}\log(\lambda_{j})^2 \lambda_{j}^{2d_{r}+2d_{s}} \Big)^{\frac{1}{2}} \nonumber
\\&\leq 
c (1+\log(N)) \Big( \sum_{j=1}^{m}\lambda_{j}^{4 \Delta_{1}} \Big)^{\frac{1}{2}} 
\leq 
c \log(N) \Big( \sum_{j=1}^{m}\lambda_{j}^{4 \Delta_{1}} \Big)^{\frac{1}{2}} 
=:\mathcal{L}_{12} 
\label{eq:1L21}.
\end{align}
Since $ \mathcal{L}_{11}$ and $ \mathcal{L}_{21}$ do not depend on $r,s$, applying Lemma \ref{prop:supGhat} with $i = 5$ in \eqref{eq:mathcalB} gives \eqref{eq:LA1applied_mathcalLK} with $k=1$.

\textit{Case $d_{0,r}>0,d_{0,s}>0$}: 
We get
\begin{align}
L_{rs,1}
&=
N^{d_{0,r}+d_{0,s}} \left(
\max_{j=1,\dots,m} | \lambda_{j}^{d_{0,r}+d_{0,s}-2\varepsilon}| +
\sup_{D \in \Omega(\varepsilon)} 
\max_{j=1,\dots,m} \sqrt{2} |\log(\lambda_{j}) \lambda_{j}^{d_{r}+d_{s}}| \right) \nonumber
\\&\leq 
c (1+\log(N)) N^{d_{0,r}+d_{0,s}} (
\lambda_{m}^{d_{0,r}+d_{0,s}-2\varepsilon} \mathds{1}_{\{d_{0,r}+d_{0,s}-2\varepsilon \geq 0\}} +
\lambda_{1}^{d_{0,r}+d_{0,s}-2\varepsilon} \mathds{1}_{\{d_{0,r}+d_{0,s}-2\varepsilon < 0\}} )
\label{eq:to_explain}
\\ & \leq 
c \log(N) \max\{ m^{d_{0,r}+d_{0,s}-2\varepsilon} N^{2 \varepsilon}, N^{2\varepsilon} \} \nonumber
\\ & \leq 
c \log(N) m^{2\Delta_{2} -2\varepsilon} N^{2\varepsilon} =:\mathcal{L}_{21}.
\label{eq:1L12}
\end{align}
We pause here to draw the attention to inequality \eqref{eq:to_explain}, since the argument will be used not only here but also in the proofs of Proposition \ref{prop:supGhat2}--\ref{prop:supGhat4}. It can be assumed that the frequencies satisfy $\lambda_{j} \leq 1$. Then, the function $f(d) = \lambda_{j}^{2d}$ is monotonically decreasing in $d$ for all $j=1,\dots,m$. For a non-negative exponent $d > 0$, $f$ reaches its maximum for $j=m$, for a negative exponent $d \leq 0$, for $j=1$. Due to the monotonicity of $f$, the function reaches its supremum for the smallest possible values $d$ can take.

We proceed with the case $d_{0,r}>0,d_{0,s}>0$, 
\begin{align}
L_{rs,i}
&=
c_{r,i,N} \widetilde{c}_{s,i,N}  \left(
\max_{j=1,\dots,m} | \lambda_{j}^{d_{0,r}+d_{0,s}-2\varepsilon}| +
\sup_{D \in \Omega(\varepsilon)} 
\max_{j=1,\dots,m} \sqrt{2} |\log(\lambda_{j}) \lambda_{j}^{d_{r}+d_{s}}| \right) \nonumber
\\&\leq 
c (1+\log(N)) c_{r,i,N} \widetilde{c}_{s,i,N}  (
\lambda_{m}^{d_{0,r}+d_{0,s}-2\varepsilon} \mathds{1}_{\{d_{0,r}+d_{0,s}-2\varepsilon \geq 0\}} +
\lambda_{1}^{d_{0,r}+d_{0,s}-2\varepsilon} \mathds{1}_{\{d_{0,r}+d_{0,s}-2\varepsilon < 0\}} ) \nonumber
\\&\leq 
c \log(N)
\begin{cases}
\max\{m^{d_{0,r}+d_{0,s}-2\varepsilon} N^{2\varepsilon}, N^{2\varepsilon} \},
& \hspace{0.2cm} \text{ if } d_{0,r},d_{0,s} > \frac{1}{4}, \ i=2, \\
\max\{ m^{\frac{1}{4}+d_{0,r}+d_{0,s}-2\varepsilon}N^{2\varepsilon} , m^{\frac{1}{4}}N^{2\varepsilon} \},
& \hspace{0.2cm} \text{ if } d_{0,r} > \frac{1}{4},d_{0,s} \leq \frac{1}{4}, \ i=3,  \\
m^{\frac{1}{2}} N^{d_{0,r} + d_{0,s}} \lambda_{m}^{d_{0,r}+d_{0,s}-2\varepsilon},
& \hspace{0.2cm} \text{ if } d_{0,r},d_{0,s} \leq \frac{1}{4}, d_{0,r}+d_{0,s}-2\varepsilon \geq 0, \ i=4, \\
m^{\frac{1}{2}} N^{d_{0,r} + d_{0,s}} \lambda_{1}^{d_{0,r}+d_{0,s}-2\varepsilon},
& \hspace{0.2cm} \text{ if } d_{0,r},d_{0,s} \leq \frac{1}{4}, d_{0,r}+d_{0,s}-2\varepsilon < 0, \ i=4, 
\end{cases} \nonumber
\\&\leq 
c \log(N)
\begin{cases}
m^{2\Delta_{2}-2\varepsilon} N^{2\varepsilon},
& \hspace{0.2cm} \text{ if } d_{0,r},d_{0,s} > \frac{1}{4}, \ i=2, \\
m^{\frac{1}{4}+\Delta_{2}+d_{0,s}-2\varepsilon}N^{2\varepsilon},
& \hspace{0.2cm} \text{ if } d_{0,r} > \frac{1}{4},d_{0,s} \leq \frac{1}{4}, \ i=3,  \\
m^{\frac{1}{2} + d_{0,r} + d_{0,s}-2\varepsilon} N^{2\varepsilon},
& \hspace{0.2cm} \text{ if } d_{0,r},d_{0,s} \leq \frac{1}{4}, d_{0,r}+d_{0,s}-2\varepsilon \geq 0, \ i=4, \\
m^{\frac{1}{2}} N^{2\varepsilon},
& \hspace{0.2cm} \text{ if } d_{0,r},d_{0,s} \leq \frac{1}{4}, d_{0,r}+d_{0,s}-2\varepsilon < 0, \ i=4,
\end{cases} \nonumber
\\&\leq 
c \log(N)
\max\{m^{2\widebar{\Delta}_{u}-2\varepsilon} N^{2\varepsilon} , m^{\frac{1}{2}} N^{2\varepsilon}\}
=: \mathcal{L}_{22}
\label{eq:1L22}
\end{align}
with $\widebar{\Delta}_{u}$ as in \eqref{eq:delta-u-l}.
Since $ \mathcal{L}_{21}$ and $ \mathcal{L}_{22}$ do not depend on $r,s$, applying Lemma \ref{prop:supGhat} with $i = 2,3,4$, gives
\eqref{eq:LA1applied_mathcalLK} with $k=2$.

\textit{Case $d_{0,r}>0, d_{0,s} \leq 0$:}
We get
\begin{align}
L_{rs,1}
&=
N^{d_{0,r}} \left(
\max_{j=1,\dots,m} |\lambda_{j}^{d_{0,r}-\varepsilon+\max\{d_{0,s}-\varepsilon,\Delta_{1}\}}| +
\sup_{D \in \Omega(\varepsilon)}
\max_{j=1,\dots,m} \sqrt{2} |\log(\lambda_{j}) \lambda_{j}^{d_{r}+d_{s}}| \right) \nonumber
\\&\leq 
N^{d_{0,r}} 
\max_{j=1,\dots,m} |\lambda_{j}^{d_{0,r}-\varepsilon+\max\{d_{0,s}-\varepsilon,\Delta_{1}\}}| \nonumber
\\ & \hspace{1cm} +
c\log(N) N^{d_{0,r}}
(
\lambda_{m}^{d_{0,r}-\varepsilon + \Delta_{1}} \mathds{1}_{\{d_{0,r}-\varepsilon + \Delta_{1} \geq 0\}} +
\lambda_{1}^{d_{0,r}-\varepsilon+\Delta_{1}} \mathds{1}_{\{d_{0,r}-\varepsilon + \Delta_{1} < 0\}} ) \nonumber
\\&\leq 
c (1+\log(N)) (
m^{\Delta_{2} + \Delta_{1} - \varepsilon} N^{-\Delta_{1}+\varepsilon}  \mathds{1}_{\{d_{0,r}-\varepsilon + \Delta_{1} \geq 0\}} +
N^{-\Delta_{1}+\varepsilon} \mathds{1}_{\{d_{0,r}-\varepsilon + \Delta_{1} < 0\}} ) \nonumber
\\&\leq 
c \log(N) \max\{m^{\Delta_{2} + \Delta_{1} - \varepsilon},1\}N^{-\Delta_{1}+\varepsilon} =:\mathcal{L}_{31}
\label{eq:1L13}
\end{align}
since $\varepsilon < -\Delta_{1}$, 
and
\begin{align}
L_{rs,i}
&=
c_{r,i,N} \widetilde{c}_{s,i,N} \left(
\max_{j=1,\dots,m} |\lambda_{j}^{d_{0,r}-\varepsilon+\max\{d_{0,s}-\varepsilon,\Delta_{1}\}}| +
\sup_{D \in \Omega(\varepsilon)}
\max_{j=1,\dots,m} \sqrt{2} |\log(\lambda_{j}) \lambda_{j}^{d_{r}+d_{s}}| \right) \nonumber
\\&\leq 
c (1+\log(N)) c_{r,i,N} \widetilde{c}_{s,i,N}  (
\lambda_{m}^{d_{0,r}-\varepsilon + \Delta_{1}} \mathds{1}_{\{d_{0,r}-\varepsilon + \Delta_{1} \geq 0\}} +
\lambda_{1}^{d_{0,r}-\varepsilon+\Delta_{1}} \mathds{1}_{\{d_{0,r}-\varepsilon + \Delta_{1} < 0\}} ) \nonumber
\\&\leq 
c \log(N)
\begin{cases}
m^{\frac{1}{4}} N^{d_{0,r}} \max\{ m^{d_{0,r}+\Delta_{1}-\varepsilon}, 1 \} N^{-d_{0,r}-\Delta_{1}+\varepsilon} ,
& \hspace{0.2cm} \text{ if } d_{0,r} > \frac{1}{4}, d_{0,s} \leq 0, \ i=3, \\
m^{\frac{1}{2}} N^{d_{0,r}} \max\{ m^{d_{0,r}+\Delta_{1}-\varepsilon}, 1 \} N^{-d_{0,r}-\Delta_{1}+\varepsilon},
& \hspace{0.2cm} \text{ if } \frac{1}{4} \geq d_{0,r}>0, d_{0,s} \leq 0, \ i=4, 
\end{cases} \nonumber
\\&\leq 
c \log(N)
\begin{cases}
\max\{ m^{\frac{1}{4} + \Delta_{2} + \Delta_{1}-\varepsilon}, m^{\frac{1}{4}} \} N^{-\Delta_{1}+\varepsilon} ,
& \hspace{0.2cm} \text{ if } d_{0,r} > \frac{1}{4}, d_{0,s} \leq 0, \ i=3, \\
\max\{ m^{\frac{1}{2} + d_{0,r} + \Delta_{1}-\varepsilon}, m^{\frac{1}{2}} \} N^{-\Delta_{1}+\varepsilon},
& \hspace{0.2cm} \text{ if } \frac{1}{4} \geq d_{0,r}>0, d_{0,s} \leq 0, \ i=4, 
\end{cases} \nonumber
\\&\leq 
c \log(N) \max\{ m^{\frac{1}{4} + \widebar{\Delta}_{u} + \Delta_{1}-\varepsilon}, m^{\frac{1}{2}} \} N^{-\Delta_{1}+\varepsilon}
=: \mathcal{L}_{32} .
\label{eq:1L23}
\end{align}
Applying Lemma \ref{prop:supGhat} with $i=3,4$ gives \eqref{eq:LA1applied_mathcalLK} with $k=3$.

\textit{Deterministic part:} Using  Lemma \ref{prop:bias} with \eqref{eq:Ghat1fctOmega}, we get
\begin{equation} \label{eq:ineqprop1det}
\begin{aligned}
&\sup_{ D \in \Omega(\varepsilon) } 
| \E \widehat{G}_{rs}(D) - \widetilde{G}_{rs}(D) |
\leq
\bm{c}_{G,1} \frac{1}{2\pi} \lambda_{m}^{2q-2\varepsilon}
+
c \frac{1}{2 \pi m} \Big(1 + \frac{1}{2\varepsilon} \Big)
N^{2\varepsilon} Q_{m} 
= \mathcal{T}_{1}(\varepsilon),
\end{aligned}
\end{equation}
where $\mathcal{T}_{1}(\varepsilon)$ is in Table \ref{label 2} and the inequality can be obtained by bounding the two terms in \eqref{eq:prop:biasbound} given \eqref{eq:Ghat1fctOmega} as follows. The first summand in \eqref{eq:prop:biasbound} can be bounded as
\begin{equation*}
\begin{aligned}
\sup_{ D \in \Omega(\varepsilon) } 
\frac{1}{2\pi m} \sum_{j=1}^{m} \lambda_{j}^{d_{r}+d_{s}} \lambda_{j}^{2q-d_{0,r}-d_{0,s}}
\leq
\frac{1}{2\pi m} \sum_{j=1}^{m} \lambda_{j}^{2q-2\varepsilon}
\leq
\frac{1}{2\pi} \lambda_{m}^{2q-2\varepsilon}
\end{aligned}
\end{equation*}
for $q-\varepsilon>0$, and the second summand as
\begin{equation*}
\begin{aligned}
\sup_{ D \in \Omega(\varepsilon) } 
\frac{1}{2 \pi m} \sum_{j=1}^{m} \lambda_{j}^{d_{r}+d_{s}} j^{-1} \lambda_{j}^{-d_{0,r}-d_{0,s}} 
&\leq
\frac{1}{2 \pi m} \sum_{j=1}^{m} j^{-1} \lambda_{j}^{-2\varepsilon}
\\& \leq 
\frac{1}{2 \pi m} \sum_{j=1}^{m} j^{-1-2\varepsilon} \lambda_{1}^{-2\varepsilon}
\leq 
%c \frac{1}{2 \pi m} (1 + \log(m))
%N^{2\varepsilon}.
\frac{1}{2 \pi m} \Big(1 + \frac{1}{2\varepsilon} \Big) \lambda_{1}^{-2\varepsilon}.
\end{aligned}
\end{equation*}
%with $L(\cdot)$ in \eqref{eq:Tl}.

Finally, we combine our results on the probabilistic and deterministic terms to obtain the statement of the proposition.
With our choice of $\nu$ in \eqref{eq:nuinprop1} and for any $\bm{C} \geq 1$, observe that, with explanations given below, 
\begin{align}
&\Prob \Big( \sup_{D \in \Omega(\varepsilon) }
| \widehat{G}_{rs}(D) - \widetilde{G}_{rs}(D) | > \nu \Big) \nonumber
\\&\leq
\Prob \Big( \sup_{D \in \Omega(\varepsilon) }
| \widehat{G}_{rs}(D) - \E(\widehat{G}_{rs}(D)) | + \sup_{D \in \Omega(\varepsilon) }
| \E(\widehat{G}_{rs}(D)) - \widetilde{G}_{rs}(D) | > \bm{C} \vertiii{G}
\sqrt{\frac{\log(p)}{\mathcal{R}_{1}}} 
+
\mathcal{T}_{1}(\varepsilon) \Big) \nonumber
%\label{eq:prop1end1}
\\&\leq
\Prob \Big( \sup_{D \in \Omega(\varepsilon) }
| \widehat{G}_{rs}(D) - \E(\widehat{G}_{rs}(D)) | > \bm{C} \vertiii{G}\sqrt{\frac{\log(p)}{\mathcal{R}_{1}}} 
\Big) \label{eq:prop1end2}
\\&\leq
\sum_{k=1}^{3} \Prob \Big( \sup_{D \in \Omega(\varepsilon) }
| \widehat{G}_{rs}(D) - \E(\widetilde{G}_{rs}(D)) |\mathds{1}_{A_{k}} > \bm{C} \vertiii{G}\sqrt{\frac{\log(p)}{\mathcal{R}_{1}}}  \Big) \label{eq:prop1end3}
\\&\leq
\sum_{k=1}^{3} c_{1} \exp\Bigg( -c_{2} \bm{C} \min 
\Bigg\{ 
\frac{ \sqrt{\frac{\log(p)}{\log(N)^{-1} \widebar{\Delta}_{N}^{-1} \mathcal{R}_{11}}} m }{ \gamma^2 \widebar{\Delta}_{N} \mathcal{L}_{k1} },
\frac{ \frac{\log(p)}{\log(N)^{-2} \widebar{\Delta}_{N}^{-2} \mathcal{R}_{12}} m^2 }{ \gamma^4 \widebar{\Delta}_{N}^{2} \mathcal{L}_{k2}^2 }
\Bigg\} \Bigg) \label{eq:prop1end4}
\\&\leq
c_{1} \exp\Bigg( -c_{2} \bm{C} \min 
\Bigg\{ 
\frac{ \sqrt{\frac{\log(p)}{\log(N)^{-1} \widebar{\Delta}^{-1}_{N} \mathcal{R}_{11}}} m }{ \gamma^2 \widebar{\Delta}_{N} \max_{k=1,2,3}\mathcal{L}_{k1} },
\frac{ \frac{\log(p)}{\log(N)^{-2} \widebar{\Delta}_{N}^{-2} \mathcal{R}_{12}} m^2 }{ \gamma^4 \widebar{\Delta}_{N}^{2} \max_{k=1,2,3}\mathcal{L}_{k2}^2 }
\Bigg\} \Bigg) \nonumber
%\label{eq:prop1end5}
\\&\leq
c_{1} \exp\Bigg( -c_{2} \bm{C} \min 
\Bigg\{ \sqrt{ \frac{\log(p)}{\widebar{\Delta}_{N} \log(N)} \mathcal{R}_{11} }, \log(p) 
\Bigg\} \Bigg) \label{eq:prop1end6}
\\&\leq
c_{1} p^{ -c_{2} \bm{C} } \label{eq:prop1end7}
\end{align}
with $\mathcal{L}_{k1},\mathcal{L}_{k2}$, $k=1,2,3$ defined in \eqref{eq:1L11}, \eqref{eq:1L21}, \eqref{eq:1L12}, \eqref{eq:1L22} and \eqref{eq:1L13}, \eqref{eq:1L23}. The constants $c_{1},c_{2}$ are generic and might differ from line to line.
Indeed, the inequality \eqref{eq:prop1end2} is due to \eqref{eq:ineqprop1det} with $\mathcal{T}_{1}(\varepsilon)$ as in Table \ref{label 2}. In \eqref{eq:prop1end3}, the probabilistic part is bounded as in \eqref{eq:sepProp1Ai}. Applying \eqref{eq:LA1applied_mathcalLK} yields \eqref{eq:prop1end4}. 
For the inequality \eqref{eq:prop1end6}, note that $\mathcal{R}_{11}$ and $\mathcal{R}_{12}$ are chosen in Table \ref{label 1} such that
\begin{equation*}
\log(N)^{-1}\mathcal{R}_{11} \precsim m ( \max_{k=1,2,3} \mathcal{L}_{k1})^{-1}
\hspace{0.2cm} \text{ and } \hspace{0.2cm}
\log(N)^{-2}\mathcal{R}_{12} \precsim m^{2} ( \max_{k=1,2,3} \mathcal{L}_{k2})^{-2}.
\end{equation*}
The inequality \eqref{eq:prop1end7} follows since we work under the assumption $\mathcal{R}_{11} \succsim \widebar{\Delta}_{N} \log(N)\log(p)$; see the discussion following Proposition \ref{prop:supGhat1}.
\end{proof}

\begin{proof}[Proof of Proposition \ref{prop:supGhat2}]
To apply Lemmas \ref{prop:supGhat} and \ref{prop:bias}, we choose
\begin{equation} \label{eq:Ghat2fctOmega}
t_{j,r}(d_{0,r}) = \lambda_{j}^{d_{0,r}}
\end{equation}
in \eqref{eq:crucialterms}--\eqref{eq:crucialterms1}. Here, we do not need a uniform bound. For this reason, it is not necessary to take the supremum over all admissible estimates of $D$ in Lemmas \ref{prop:supGhat} and \ref{prop:bias}. In particular, Lemma \ref{prop:supGhat} simplifies, since the derivatives in the respective bounds $L_{rr,i}$ in \eqref{eq:prop:supGhatLs} become zero with the choice \eqref{eq:Ghat2fctOmega}. 

\textit{Probabilistic part:}
It is enough to distinguish two cases
\begin{equation} \label{eq:sepProp2Ai}
\begin{aligned}
\Prob \Big( 
| \widehat{g}_{r}(d_{0,r}) - \E(\widehat{g}_{r}(d_{0,r})) | > \vertiii{G} \widetilde{\nu}_{1} \Big)
&\leq
\sum_{k=1}^{2} \Prob \Big( 
| \widehat{g}_{r}(d_{0,r}) - \E(\widehat{g}_{r}(d_{0,r})) |\mathds{1}_{A_{k}} > \vertiii{G} \widetilde{\nu}_{1} \Big) 
\end{aligned}
\end{equation}
with $A_{1}=\{d_{0,r}\leq 0 \}$ and $A_{2}=\{d_{0,r} > 0 \}$.
For both cases, we apply Lemma \ref{prop:supGhat} and bound the respective quantities $L_{rr,i}$, $i=1,\dots,5$ in \eqref{eq:prop:supGhatLs} that then yield the desired result. We will show that, for $k=1,2$,
\begin{equation} \label{eq:2LA1applied_mathcalL1k} 
\begin{aligned}
\Prob \Big( 
| \widehat{g}_{r}(d_{0,r}) - \E(\widehat{g}_{r}(d_{0,r})) |\mathds{1}_{A_{k}} > \vertiii{G} \widetilde{\nu}_{1} \Big)
\leq
c_{1} \exp\Bigg( -c_{2} \min 
\Bigg\{ 
\frac{ \widetilde{\nu}_{1} m }{ \gamma^2 \mathcal{L}_{k1} },
\frac{ \widetilde{\nu}_{1}^2 m^2 }{ \gamma^4 \mathcal{L}_{k2}^2 }
\Bigg\} \Bigg).
\end{aligned}
\end{equation}

\textit{Case $d_{0,r} \leq 0$:} We get
\begin{equation}\label{eq:2L21}
\begin{aligned}
L_{rr,1}
&= 
N^{\max\{2d_{0,r},0\}} \max_{j=1,\dots,m} \lambda_{j}^{2 d_{0,r}}
%\Big( \sum_{j=1}^{m}\lambda_{j}^{4 d_{0,r}} \Big)^{\frac{1}{2}}
%\\&\leq c\max\{ m^{2 d_{0,r}}, N^{-2d_{0,r}} \}
\leq c N^{-2\Delta_{1}} =: \mathcal{L}_{11},
%c \Big( \sum_{j=1}^{m}\lambda_{j}^{4 \Delta_{1}} \Big)^{\frac{1}{2}},
\\
L_{rr,5}
%&= \| \mathcal{A}_{m,1} \widetilde{T}^{2}(d_{0,r},d_{0,r}) \mathcal{A}_{m,1} \|_{F}
&= \Big( \sum_{j=1}^{m}\lambda_{j}^{4d_{0,r}} \Big)^{\frac{1}{2}}
\leq \Big( \sum_{j=1}^{m}\lambda_{j}^{4 \Delta_{1}} \Big)^{\frac{1}{2}} =: \mathcal{L}_{12}.
\end{aligned}
\end{equation}
Then, applying Lemma \ref{prop:supGhat} with $i = 5$ gives \eqref{eq:2LA1applied_mathcalL1k} with $k=1$. 

\textit{Case $d_{0,r} > 0$:} We get
\begin{align}
L_{rr,1}
&=
N^{\max\{2d_{0,r},0\}} \max_{j = 1, \dots, m} \lambda_{j}^{2d_{0,r}} 
\leq c m^{2 \Delta_{2}}=: \mathcal{L}_{21}, \label{eq:2L12}
\\
L_{rr,i}
&= 
c^{2}_{r,1,N} \lambda_{m}^{2d_{0,r}} \nonumber
\\&= 
\begin{cases}
N^{\max\{2d_{0,r},\frac{1}{2}\}} \lambda_{m}^{2d_{0,r}},
\hspace{0.2cm}& \text{ if } d_{0,r} > \frac{1}{4}, \ i=2,\\
m^{\frac{1}{2}} N^{\max\{2d_{0,r},0\}} \lambda_{m}^{2d_{0,r}},
\hspace{0.2cm}& \text{ if } d_{0,r} \leq \frac{1}{4}, \ i=4,
\end{cases} \nonumber
\\&\leq
c 
\begin{cases}
m^{2d_{0,r}},
\hspace{0.2cm}& \text{ if } d_{0,r} > \frac{1}{4}, \ i=2,\\
m^{\frac{1}{2}+2d_{0,r}},
\hspace{0.2cm}& \text{ if } d_{0,r} \leq \frac{1}{4}, \ i=4,
\end{cases} \nonumber
\\& \leq c m^{2\widebar{\Delta}_{u}}
=:
\mathcal{L}_{22}. \label{eq:2L22}
\end{align}
Then, applying Lemma \ref{prop:supGhat} with $i=2$ gives \eqref{eq:2LA1applied_mathcalL1k} with $k=2$. 

\textit{Deterministic part:} Using  Lemma \ref{prop:bias} with \eqref{eq:Ghat2fctOmega}, we get
\begin{equation} \label{eq:ineqprop2det}
\begin{aligned}
| \E (\widehat{g}_{r}(d_{0,r})) - g_{0,r} |
\leq
\bm{c}_{G,1}
\frac{1}{2 \pi}  \lambda_{m}^{q}
+
\frac{1}{\pi m} \log(m)Q_{m} = \mathcal{T}_{2},
\end{aligned}
\end{equation}
where $\mathcal{T}_{2}$ is in Table \ref{label 2} and the inequality can be obtained  by bounding the two terms in \eqref{eq:prop:biasbound} given \eqref{eq:Ghat2fctOmega} as follows. The first summand in \eqref{eq:prop:biasbound} can be bounded as
\begin{equation*}
\begin{aligned}
\frac{1}{2 \pi m} \sum_{j=1}^{m} \lambda_{j}^{2d_{0,r}} \lambda_{j}^{2q-2d_{0,r}}
&=
\frac{1}{2 \pi m} \sum_{j=1}^{m} \lambda_{j}^{2q}
&\leq
\frac{1}{2 \pi} \lambda_{m}^{2q},
\end{aligned}
\end{equation*}
the second summand as
\begin{equation*}
\begin{aligned}
\frac{1}{2 \pi m} \sum_{j=1}^{m} \lambda_{j}^{2d_{0,r}} j^{-1} \lambda_{j}^{-2d_{0,r}}
=
\frac{1}{2 \pi m} \sum_{j=1}^{m} j^{-1}
\leq
\frac{1}{2 \pi m} (1+\log(m))
\leq
\frac{1}{\pi m} \log(m).
\end{aligned}
\end{equation*}

Finally, we combine our results on the probabilistic and deterministic terms. With our choice of $\nu_{1}$ in \eqref{eq:nu1} and for any $\bm{C}\geq 1$, observe that, with explanations given below, 
\begin{align}
&\Prob \Big( 
| \widehat{g}_{r}(d_{0,r}) - g_{0,r} | > \nu_{1} \Big) \nonumber
\\&\leq
\Prob \Big(
| \widehat{g}_{r}(d_{0,r}) - \E(\widehat{g}_{r}(d_{0,r})) | + 
| \E(\widehat{g}_{r}(d_{0,r})) - g_{0,r} | > \bm{C} \vertiii{G}
\sqrt{\frac{\log(p)}{\mathcal{R}_{2}}} 
+
\mathcal{T}_{2} \Big) \nonumber
%\label{eq:prop2end1}
\\&\leq
\Prob \Big( 
| \widehat{g}_{r}(d_{0,r}) - \E(\widehat{g}_{r}(d_{0,r})) | > \bm{C} \vertiii{G}\sqrt{\frac{\log(p)}{\mathcal{R}_{2}}} 
\Big) \label{eq:prop2end2}
\\&\leq
\sum_{k=1}^{2} \Prob \Big( 
| \widehat{g}_{r}(d_{0,r}) - \E(\widehat{g}_{r}(d_{0,r})) |\mathds{1}_{A_{k}} > \bm{C} \vertiii{G}\sqrt{\frac{\log(p)}{\mathcal{R}_{2}}}  \Big) \label{eq:prop2end3}
\\&\leq
\sum_{k=1}^{2} c_{1} \exp\Bigg( -c_{2} \bm{C} \min 
\Bigg\{ 
\frac{ \sqrt{\frac{\log(p)}{\widebar{\Delta}_{N}^{-1} \mathcal{R}_{21} }} m }{ \gamma^2 \widebar{\Delta}_{N} \mathcal{L}_{k1} },
\frac{ \frac{\log(p)}{\widebar{\Delta}_{N}^{-2} \mathcal{R}_{22}} m^2 }{ \gamma^4 \widebar{\Delta}_{N}^{2} \mathcal{L}_{k2}^2 }
\Bigg\} \Bigg) \label{eq:prop2end4}
\\&\leq
c_{1} \exp\Bigg( -c_{2} \bm{C} \min 
\Bigg\{ 
\frac{ \sqrt{\frac{\log(p)}{\widebar{\Delta}_{N}^{-1} \mathcal{R}_{21}}} m }{ \gamma^2 \widebar{\Delta}_{N} \max_{k=1,2} \mathcal{L}_{k1} },
\frac{ \frac{\log(p)}{\widebar{\Delta}_{N}^{-2} \mathcal{R}_{22}} m^2 }{ \gamma^4 \widebar{\Delta}_{N}^{2} \max_{k=1,2} \mathcal{L}_{k2}^2 }
\Bigg\} \Bigg) \nonumber
%\label{eq:prop2end5}
\\&\leq
c_{1} \exp\Bigg( -c_{2} \bm{C} \min 
\Bigg\{ \sqrt{ \frac{\log(p)}{\widebar{\Delta}_{N}} \mathcal{R}_{21} }, \log(p) 
\Bigg\} \Bigg) \label{eq:prop2end7}
\\&\leq
c_{1} p^{ -c_{2} \bm{C} } \label{eq:prop2end8}
\end{align}
with $\mathcal{L}_{k1},\mathcal{L}_{k2}$, $k=1,2$ defined in \eqref{eq:2L21}, \eqref{eq:2L12} and \eqref{eq:2L22}. The constants $c_{1},c_{2}$ are generic and might differ from line to line.
Indeed, the inequality \eqref{eq:prop2end2} is due to \eqref{eq:ineqprop2det}. In \eqref{eq:prop2end3}, the probabilistic part is bounded as in \eqref{eq:sepProp2Ai}. 
Applying \eqref{eq:2LA1applied_mathcalL1k} yields \eqref{eq:prop2end4}. 
For the inequality \eqref{eq:prop2end7}, note that $\mathcal{R}_{21}$ and $\mathcal{R}_{22}$ are chosen in Table \ref{label 1} such that 
\begin{equation*}
\mathcal{R}_{21} \precsim m( \max_{k=1,2} \mathcal{L}_{k1})^{-1} 
\hspace{0.2cm}
\text{ and }
\hspace{0.2cm}
\mathcal{R}_{22} \precsim m^{2}(\max_{k=1,2} \mathcal{L}_{k2})^{-2}.
\end{equation*}
The inequality \eqref{eq:prop2end8} follows since we work under the assumption $\mathcal{R}_{21} \succsim \widebar{\Delta}_{N} \log(p)$; see the discussion following Proposition \ref{prop:supGhat2}.
\end{proof}

\begin{proof}[Proof of Proposition \ref{prop:supGhat3}]
To apply Lemma \ref{prop:supGhat}, we choose
\begin{equation} \label{eq:Ghat3fctOmega}
t_{j,r}(d) = \Big( \frac{j}{m} \Big)^{d-d_{0,r}} \lambda_{j}^{d_{0,r}} = \Big( \frac{j}{m} \Big)^{d} \lambda_{m}^{d_{0,r}}
\end{equation}
and
\begin{equation*}
\Omega=
\Theta_{1}=
\begin{cases}
\{ d ~|~ d_{0,r} - \frac{1}{2} + \Delta \leq d \leq \Delta_{2} \}, &\hspace{0.2cm} \text{ if } d_{0,r} \geq \Delta_{1} + \frac{1}{2}, \\
\{ d ~|~ \Delta_{1} \leq d \leq \Delta_{2} \}, &\hspace{0.2cm} \text{ if } d_{0,r} < \Delta_{1} + \frac{1}{2}
\end{cases}
\end{equation*}
in \eqref{eq:crucialterms}--\eqref{eq:crucialterms1}. 

\textit{Probabilistic part:}
It is enough to distinguish three cases
\begin{equation} \label{eq:sepProp3Ai}
\begin{aligned}
&\Prob \Big(\sup_{d_{r} \in \Theta_{1} } 
| \frac{1}{m} \sum_{j=1}^{m} \Big(\frac{j}{m}\Big)^{2d_{r}-2d_{0,r}} \lambda_{j}^{2d_{0,r}} (I_{X,rr}(\lambda_{j})-\E(I_{X,rr}(\lambda_{j}))) |  > \vertiii{G} \widetilde{\nu}_{2} \Big)
\\&\leq
\sum_{k=1}^{3} \Prob \Big(\sup_{d_{r} \in \Theta_{1} } 
| \frac{1}{m} \sum_{j=1}^{m} \Big(\frac{j}{m}\Big)^{2d_{r}-2d_{0,r}} \lambda_{j}^{2d_{0,r}} (I_{X,rr}(\lambda_{j})-\E(I_{X,rr}(\lambda_{j}))) | \mathds{1}_{A_{k}} > \vertiii{G} \widetilde{\nu}_{2} \Big) 
\end{aligned}
\end{equation}
with 
\begin{gather*}
A_{1}=\{d_{0,r} \geq \Delta_{1} +\frac{1}{2} , d_{0,r}-\frac{1}{2}+\Delta>0\}, \hspace{0.2cm} A_{2}=\{ d_{0,r} \geq \Delta_{1} +\frac{1}{2}, d_{0,r}-\frac{1}{2}+\Delta \leq 0 \}, 
\\
A_3=\{d_{0,r} < \Delta_{1} +\frac{1}{2} \}.
\end{gather*}
For each case, we apply Lemma \ref{prop:supGhat} and bound the respective quantities $L_{rr,i}$, $i=1,\dots,5$ in \eqref{eq:prop:supGhatLs} that then yield the desired result. We will show that, for $k=1,2,3$, 
\begin{equation} \label{eq:3LA1applied_mathcalL1k}
\begin{aligned}
&\Prob \Big(\sup_{d_{r} \in \Theta_{1} } 
| \frac{1}{m} \sum_{j=1}^{m} \Big(\frac{j}{m}\Big)^{2d_{r}-2d_{0,r}} \lambda_{j}^{2d_{0,r}} (I_{X,rr}(\lambda_{j})-\E(I_{X,rr}(\lambda_{j}))) | \mathds{1}_{A_{k}} > \vertiii{G}\widetilde{\nu}_{2} \Big)
\\&\leq
c_{1} \exp\Bigg( -c_{2} \min 
\Bigg\{ 
\frac{ \widetilde{\nu}_{2} m }{ \gamma^2 \mathcal{L}_{k1} },
\frac{ \widetilde{\nu}^2_{2} m^2 }{ \gamma^4 \mathcal{L}_{k2}^2 }
\Bigg\} \Bigg).
\end{aligned}
\end{equation}

\textit{Case $d_{0,r} \geq \Delta_{1} +\frac{1}{2}$ and $d_{0,r}-\frac{1}{2}+\Delta>0$:} We get
\begin{align}
L_{rr,1}
&= N^{2d_{0,r}} \left(
\max_{j = 1, \dots, m} \Big( \frac{j}{m} \Big)^{2d_{0,r}-1+2\Delta} \lambda_{m}^{2d_{0,r}} +
\sup_{d_{r} \in \Theta_{1}} \max_{j = 1, \dots, m} 2 |\log\Big( \frac{j}{m} \Big)| \Big( \frac{j}{m} \Big)^{2d_{r}} \lambda_{m}^{2d_{0,r}} 
\right)
\nonumber
\\&\leq c (1+\log( m)) 
\max_{j = 1, \dots, m} \Big( \frac{j}{m} \Big)^{2d_{0,r}-1+2\Delta} m^{2d_{0,r}} 
\nonumber
\\&\leq 
c \log( m) m^{2\Delta_{2}}=:\mathcal{L}_{11}, \label{eq:3L11}\\
L_{rr,i}
&= c^{2}_{r,i,N}
\left(
\max_{j = 1, \dots, m} \Big( \frac{j}{m} \Big)^{2d_{0,r}-1+2\Delta} \lambda_{m}^{2d_{0,r}} +
\sup_{d_{r} \in \Theta_{1}} \max_{j = 1, \dots, m} 2 |\log\Big( \frac{j}{m} \Big)| \Big( \frac{j}{m} \Big)^{2d_{r}} \lambda_{m}^{2d_{0,r}} 
\right)\nonumber
\\&\leq 
c (1+\log( m)) c^{2}_{r,i,N} \max_{j = 1, \dots, m} \Big( \frac{j}{m} \Big)^{2d_{0,r}-1+2\Delta} \lambda_{m}^{2d_{0,r}} \nonumber
\\&= c \log( m)
\begin{cases}
N^{\max\{2d_{0,r},\frac{1}{2}\}} \lambda_{m}^{2d_{0,r}},
\hspace{0.2cm} &\text{ if } d_{0,r} > \frac{1}{4}, \ i=2,
\\
N^{\max\{2d_{0,r},0\}} m^{\frac{1}{2}} \lambda_{m}^{2d_{0,r}},
\hspace{0.2cm} &\text{ if } d_{0,r} \leq \frac{1}{4}, \ i=4,
\end{cases}
\nonumber
\\&\leq
c \log( m) m^{2\widebar{\Delta}_{u}}  =:\mathcal{L}_{12}. \label{eq:3L21}
\end{align}
Then, applying Lemma \ref{prop:supGhat} with $i=2$ gives \eqref{eq:3LA1applied_mathcalL1k} with $k=1$.

\textit{Case $d_{0,r} \geq \Delta_{1} +\frac{1}{2}$ and $d_{0,r}-\frac{1}{2}+\Delta \leq 0$:} We get
\begin{align}
L_{rr,1}
&= N^{2d_{0,r}} \left(
\max_{j = 1, \dots, m} \Big( \frac{j}{m} \Big)^{2d_{0,r}-1+2\Delta} \lambda_{m}^{2d_{0,r}} +
\sup_{d_{r} \in \Theta_{1}} \max_{j = 1, \dots, m} 2 |\log\Big( \frac{j}{m} \Big)| \Big( \frac{j}{m} \Big)^{2d_{r}} \lambda_{m}^{2d_{0,r}} 
\right)\nonumber
\\&\leq 
c (1+\log( m)) 
\max_{j = 1, \dots, m} \Big( \frac{j}{m} \Big)^{2d_{0,r}-1+2\Delta} m^{2d_{0,r}}
\nonumber
\\&\leq 
c \log( m) m^{1-2 \Delta}=:\mathcal{L}_{21} \label{eq:3L12}
\end{align}
and 
\begin{equation}
\begin{aligned}
L_{rr,2}
&= 
\log(N)^{\frac{1}{2}} N^{\max\{2d_{0,r},\frac{1}{2}\}} 
\left(
\max_{j = 1, \dots, m} \Big( \frac{j}{m} \Big)^{2d_{0,r}-1+2\Delta} \lambda_{m}^{2d_{0,r}} +
\sup_{d_{r} \in \Theta_{1}} \max_{j = 1, \dots, m} 2 |\log\Big( \frac{j}{m} \Big)| \Big( \frac{j}{m} \Big)^{2d_{r}} \lambda_{m}^{2d_{0,r}} 
\right)
\\&\leq 
c (1+\log( m )) N^{\max\{2d_{0,r},\frac{1}{2}\}} 
\max_{j = 1, \dots, m} \Big( \frac{j}{m} \Big)^{2d_{0,r}-1+2\Delta} \lambda_{m}^{2d_{0,r}} 
\\&\leq
c \log( m ) \max\{m^{1-2\Delta}, N^{\frac{1}{2}} m^{1-2\Delta} N^{-2d_{0,r}} \}
\\&\leq
c \log( m ) \max\{m^{1-2\Delta}, N^{-\frac{1}{2}-2\Delta_{1}} m^{1-2\Delta} \}
=:\mathcal{L}_{22}.
\label{eq:3L22}
\end{aligned}
\end{equation}
Then, applying Lemma \ref{prop:supGhat} with $i=2,5$ gives \eqref{eq:3LA1applied_mathcalL1k} with $k=2$.

\textit{Case $d_{0,r} < \Delta_{1} +\frac{1}{2}$:}
In this case, $\Theta_{1}=\{d ~|~ \Delta_{1} \leq d \leq \Delta_{2} \}$.Then,
\begin{align}
L_{1,rr}
&= N^{\max\{2d_{0,r}, 0\}} 
\left(
\max_{j = 1, \dots, m} \Big( \frac{j}{m} \Big)^{2\Delta_{1}} \lambda_{m}^{2d_{0,r}} +
 \sup_{d_{r} \in \Theta_{1}} \max_{j = 1, \dots, m} 2 |\log\Big( \frac{j}{m} \Big)| \Big( \frac{j}{m} \Big)^{2d_r} \lambda_{m}^{2d_{0,r}} 
 \right)
 \nonumber
\\&\leq 
c (1 + \log( m )) 
N^{\max\{2d_{0,r}, 0\}} 
\max_{j = 1, \dots, m} \Big( \frac{j}{m} \Big)^{2\Delta_{1}} \lambda_{m}^{2d_{0,r}} 
\nonumber
\\&\leq 
c \log( m ) \max\{ m^{2d_{0,r}-2 \Delta_{1}}, m^{-2 \Delta_{1}} \lambda_{m}^{ 2 \Delta_{1}} \}
\nonumber
\\&\leq 
c \log( m ) \max\{ m^{2\widebar{\Delta}_{l,1}-2 \Delta_{1}}, N^{-2 \Delta_{1}} \}
=:\mathcal{L}_{31} \label{eq:3L13}
\end{align}
with $\widebar{\Delta}_{l,1}$ as in \eqref{eq:delta-u-l} 
and we distinguish further two cases: for $d_{0,r}>0$, 
\begin{align}
L_{rr,i}
&= 
c^{2}_{r,i,N}
\left(
\max_{j = 1, \dots, m} \Big( \frac{j}{m} \Big)^{2\Delta_{1}} \lambda_{m}^{2d_{0,r}} +
\sup_{d_{r} \in \Theta_{1}} \max_{j = 1, \dots, m} 2 |\log\Big( \frac{j}{m} \Big)| \Big( \frac{j}{m} \Big)^{2d_r} \lambda_{m}^{2d_{0,r}} 
\right) \nonumber
\\&\leq
c (1+\log( m)) 
\begin{cases}
N^{\max\{2d_{0,r},\frac{1}{2} \}} m^{-2\Delta_{1}} \lambda_{m}^{2d_{0,r}},  
& \hspace{0.2cm} \text{ if } d_{0,r} > \frac{1}{4}, \ i=2, \\
m^{\frac{1}{2}} N^{\max\{2d_{0,r}, 0 \}} m^{-2\Delta_{1}} \lambda_{m}^{2d_{0,r}},  
& \hspace{0.2cm} \text{ if } d_{0,r} \leq \frac{1}{4}, \ i=4, 
\end{cases} \nonumber
\\&\leq
c \log( m)
\begin{cases}
m^{2d_{0,r}-2 \Delta_{1}},  
& \hspace{0.2cm} \text{ if } d_{0,r} > \frac{1}{4}, \ i=2, \\
m^{\frac{1}{2}+ 2d_{0,r}-2\Delta_{1} },  
& \hspace{0.2cm} \text{ if } d_{0,r} \leq \frac{1}{4}, \ i=4, 
\end{cases} \nonumber
\\&\leq
c \log( m)
m^{2\widebar{\Delta}_{l,2}-2\Delta_{1}}
\label{eq:3L23aa}
\end{align}
with $\widebar{\Delta}_{l,2}$ as in \eqref{eq:delta-u-l} and for $d_{0,r}\leq 0$,
\begin{align}
L_{rr,5}
&= 
N^{\max\{2d_{0,r},0 \}} 
\left(
\Big( \sum_{j=1}^{m} \Big( \frac{j}{m} \Big)^{4\Delta_{1}} \lambda_{m}^{4d_{0,r}} \Big)^{\frac{1}{2}} + 
\sup_{d_{r} \in \Theta_{1}} 2 \Big( \sum_{j=1}^{m} |\log\Big( \frac{j}{m} \Big)| \Big( \frac{j}{m} \Big)^{4d_r} \lambda_{m}^{4d_{0,r}} \Big)^{\frac{1}{2}}
\right) \nonumber
\\&\leq
c (1+\log( m)) 
\Big( \sum_{j=1}^{m} \Big( \frac{j}{m} \Big)^{4\Delta_{1}} \lambda_{m}^{4d_{0,r}} \Big)^{\frac{1}{2}} \nonumber
\\&\leq
c \log( m) \lambda_{m}^{2d_{0,r}-2\Delta_{1}} \Big(\sum_{j=1}^{m} \lambda_{j}^{4\Delta_{1}} \Big)^{\frac{1}{2}} \nonumber
\\&\leq
c \log( m) \Big(\sum_{j=1}^{m} \lambda_{j}^{4\Delta_{1}} \Big)^{\frac{1}{2}}. \label{eq:3L23bb}
\end{align}
Finally, given \eqref{eq:3L23aa} and \eqref{eq:3L23bb}, we define
\begin{equation}
\mathcal{L}_{32} := 
c \log( m) \max\{
m^{2\widebar{\Delta}_{l,2}-2\Delta_{1}}, \Big(\sum_{j=1}^{m} \lambda_{j}^{4\Delta_{1}} \Big)^{\frac{1}{2}} \}.
\label{eq:3L23}
\end{equation}
Then, applying Lemma \ref{prop:supGhat} gives \eqref{eq:3LA1applied_mathcalL1k} with $k=3$.

\textit{Deterministic part:} 
Using  Lemma \ref{prop:bias} with \eqref{eq:Ghat3fctOmega}, we get
\begin{equation} \label{eq:ineqprop3det}
\begin{aligned}
&\sup_{ d_r \in \Theta_{1} } 
| \frac{1}{m} \sum_{j=1}^{m} \Big(\frac{j}{m}\Big)^{2d_r-2d_{0,r}} \lambda_{j}^{2d_{0,r}} (\E I_{X,rr}(\lambda_{j})-\lambda_{j}^{-2d_{0,r}}g_{0,r}) |
\\&\leq
\bm{c}_{G,1}
\frac{1}{2 \pi} \lambda_{m}^{2q} \frac{1}{2\widetilde{\Delta}_{r}} 
+
\frac{1}{2 \pi} m^{-2\widetilde{\Delta}_{r}} \Big(1+\frac{1}{1-2\widetilde{\Delta}_{r}} \Big) Q_{m} \leq \max_{r = 1,\dots,p} \mathcal{T}_{3}(\widetilde{\Delta}_{r}) = \mathcal{T}_{3}
\end{aligned}
\end{equation}
with $\widetilde{\Delta}_{r} = (d_{0,r}-\frac{1}{2}+\Delta) \mathds{1}_{\{d_{0,r} \geq \Delta_{1} +\frac{1}{2}\}} + \Delta_{1} \mathds{1}_{\{d_{0,r} < \Delta_{1} +\frac{1}{2}\}}$ and 
$\mathcal{T}_{3}(\widetilde{\Delta}_{r})$ is in Table \ref{label 2}. The inequality can be obtained by bounding the two terms in \eqref{eq:prop:biasbound} given \eqref{eq:Ghat3fctOmega} as follows.
%distinguishing between $d_{0,r} \geq \Delta_{1} +\frac{1}{2}$ and $d_{0,r} < \Delta_{1} +\frac{1}{2}$ which results in different choices of $\Theta_{1}$.
The first summand in \eqref{eq:prop:biasbound} can be bounded as
\begin{equation*}
\begin{aligned}
\sup_{d_r \in \Theta_{1} } \frac{1}{2 \pi m} \sum_{j=1}^{m} \Big( \frac{j}{m} \Big)^{2d_r} \lambda_{m}^{2d_{0,r}} \lambda_{j}^{2q-2d_{0,r}}
& \leq 
\frac{1}{2 \pi m} \sum_{j=1}^{m} \Big( \frac{j}{m} \Big)^{2(\widetilde{\Delta}_{r}-d_{0,r})} \lambda_{j}^{2q}
\\& \leq
\lambda_{m}^{2q} \frac{1}{2 \pi m} \sum_{j=1}^{m} \Big( \frac{j}{m} \Big)^{2(\widetilde{\Delta}_{r}-d_{0,r})}
\\&\leq 
\lambda_{m}^{2q} \frac{1}{2 \pi} \frac{1}{2\widetilde{\Delta}_{r}-2d_{0,r}+1} 
\end{aligned}
\end{equation*}
and the second as
\begin{equation*}
\begin{aligned}
\sup_{ d_{r} \in \Theta_{1} } \frac{1}{2 \pi m} \sum_{j=1}^{m} \Big( \frac{j}{m} \Big)^{2d_{r}} \lambda_{m}^{2d_{0,r}} j^{-1} \lambda_{j}^{-2d_{0,r}}
&\leq
\frac{1}{2 \pi m} \sum_{j=1}^{m} \Big( \frac{j}{m} \Big)^{2(\widetilde{\Delta}_{r}-d_{0,r})} \lambda_{j}^{2d_{0,r}} j^{-1} \lambda_{j}^{-2d_{0,r}}
\\&=
\frac{1}{2 \pi m} \sum_{j=1}^{m} \Big( \frac{j}{m} \Big)^{2(\widetilde{\Delta}_{r}-d_{0,r})} j^{-1}
\\&\leq
\frac{1}{2 \pi} m^{-1-2\widetilde{\Delta}_{r}+2d_{0,r}} \Big(1+ \frac{1}{2d_{0,r}-2\widetilde{\Delta}}_{r}  \Big).
\end{aligned}
\end{equation*}

Finally, we combine our results on the probabilistic and deterministic terms to obtain the statement of the proposition.
With our choice of $\nu_{2}$ in \eqref{eq:nu2}, observe that, with explanations given below,
\begin{align}
&\Prob \Big( \sup_{d_r \in \Theta_{1} } 
| \frac{1}{m} \sum_{j=1}^{m} \Big(\frac{j}{m}\Big)^{2d_r-2d_{0,r}} \lambda_{j}^{2d_{0,r}} (I_{X,rr}(\lambda_{j})-\lambda_{j}^{-2d_{0,r}}g_{0,r}) | > \nu_{2} \Big) \nonumber
\\&\leq
\Prob \Big(\sup_{d_r \in \Theta_{1} } 
| \frac{1}{m} \sum_{j=1}^{m} \Big(\frac{j}{m}\Big)^{2d_r-2d_{0,r}} \lambda_{j}^{2d_{0,r}} (I_{X,rr}(\lambda_{j})-\E(I_{X,rr}(\lambda_{j}))) | \nonumber
\\&\hspace{1cm}+
\sup_{d_{r} \in \Theta_{1} } 
| \frac{1}{m} \sum_{j=1}^{m} \Big(\frac{j}{m}\Big)^{2d_r-2d_{0,r}} \lambda_{j}^{2d_{0,r}} (\E(I_{X,rr}(\lambda_{j}))-\lambda_{j}^{-2d_{0,r}}g_{0,r}) | > \bm{C} \vertiii{G}
\sqrt{\frac{\log(p)}{\mathcal{R}_{3}}} 
+
\mathcal{T}_{3} \Big) \nonumber
%\label{eq:prop3end1}
\\&\leq
\Prob \Big(\sup_{d_r \in \Theta_{1} } 
| \frac{1}{m} \sum_{j=1}^{m} \Big(\frac{j}{m}\Big)^{2d_r-2d_{0,r}} \lambda_{j}^{2d_{0,r}} (I_{X,rr}(\lambda_{j})-\E(I_{X,rr}(\lambda_{j}))) | 
> \bm{C} \vertiii{G}
\sqrt{\frac{\log(p)}{\mathcal{R}_{3}}} \Big) \label{eq:prop3end2}
\\&\leq
\sum_{k=1}^{3} \Prob \Big( \sup_{d_r \in \Theta_{1} } 
| \frac{1}{m} \sum_{j=1}^{m} \Big(\frac{j}{m}\Big)^{2d_r-2d_{0,r}} \lambda_{j}^{2d_{0,r}} (I_{X,rr}(\lambda_{j})-\E(I_{X,rr}(\lambda_{j}))) | \mathds{1}_{A_{k}} > \bm{C} \vertiii{G}\sqrt{\frac{\log(p)}{\mathcal{R}_{3}}}  \Big) \label{eq:prop3end3}
\\&\leq
\sum_{k=1}^{3} c_{1} \exp\Bigg( -c_{2} \bm{C} \min 
\Bigg\{ 
\frac{ \sqrt{\frac{\log(p)}{\log(m)^{-1} \widebar{\Delta}^{-1}_{N} \mathcal{R}_{31}}} m }{ \gamma^2 \widebar{\Delta}_{N} \mathcal{L}_{k1} },
\frac{ \frac{\log(p)}{\log(m)^{-2} \widebar{\Delta}^{-2}_{N} \mathcal{R}_{32}} m^2 }{ \gamma^4 \widebar{\Delta}_{N}^{2} \mathcal{L}_{k2}^2 }
\Bigg\} \Bigg) \label{eq:prop3end4}
\\&\leq
c_{1} \exp\Bigg( -c_{2} \bm{C} \min 
\Bigg\{ 
\frac{ \sqrt{\frac{\log(p)}{\log(m)^{-1} \widebar{\Delta}_{N}^{-1} \mathcal{R}_{31}}} m }{ \gamma^2 \widebar{\Delta}_{N} \max_{k=1,2,3}\mathcal{L}_{k1} },
\frac{ \frac{\log(p)}{\log(m)^{-2} \widebar{\Delta}_{N}^{-2} \mathcal{R}_{32}} m^2 }{ \gamma^4 \widebar{\Delta}_{N}^{2} \max_{k=1,2,3}\mathcal{L}_{k2}^2 }
\Bigg\} \Bigg) \nonumber
%\label{eq:prop3end5}
\\&\leq
c_{1} \exp\Bigg( -c_{2} \bm{C} \min 
\Bigg\{ \sqrt{\frac{\log(p) \mathcal{R}_{31}}{\log(m)\widebar{\Delta}_{N} } }, \log(p) 
\Bigg\} \Bigg) \label{eq:prop3end7}
\\&\leq
c_{1} p^{ -c_{2} \bm{C} } \label{eq:prop3end8}
\end{align}
with $\mathcal{L}_{k1},\mathcal{L}_{k2}$, $k=1,2,3$ defined in \eqref{eq:3L11}, \eqref{eq:3L21}, \eqref{eq:3L12}, \eqref{eq:3L22} and \eqref{eq:3L13}, \eqref{eq:3L23}. The constants $c_{1},c_{2}$ are generic and might differ from line to line.
The inequality \eqref{eq:prop3end2} is due to \eqref{eq:ineqprop3det}. In \eqref{eq:prop3end3}, the probabilistic part is bounded as in \eqref{eq:sepProp3Ai}. Applying  \eqref{eq:3LA1applied_mathcalL1k} yields \eqref{eq:prop3end4}. 
For the inequality \eqref{eq:prop3end7}, note that $\mathcal{R}_{31}$ and $\mathcal{R}_{32}$ are chosen in Table \ref{label 1} such that 
\begin{equation*}
\log(m)^{-1}\mathcal{R}_{31} \precsim m( \max_{k=1,2,3} \mathcal{L}_{k1})^{-1}
\hspace{0.2cm}
\text{ and }
\hspace{0.2cm}
\log(m)^{-2}\mathcal{R}_{32} \precsim m^{2}( \max_{k=1,2,3} \mathcal{L}_{k2})^{-2}.
\end{equation*}
The inequality \eqref{eq:prop3end8} follows since we work under the assumption $\mathcal{R}_{31} \succsim \widebar{\Delta}_{N} \log(m)\log(p)$; see the discussion following Proposition \ref{prop:supGhat3}.
\end{proof}

\begin{proof}[Proof of Proposition \ref{prop:supGhat4}]
To apply Lemmas \ref{prop:supGhat} and \ref{prop:bias}, we choose
\begin{equation} \label{eq:Ghat4fctOmega}
t_{j,r}(d_{0,r}) 
= | l_{j} -1|^{\frac{1}{2}} \lambda_{j}^{d_{0,r}} \mathds{1}_{\{d_{0,r} \geq \Delta_{1} + \frac{1}{2} \}}
= | l_{j} -1|^{\frac{1}{2}} \Big(\frac{j}{\ell}\Big)^{d_{0,r}} \lambda_{\ell}^{d_{0,r}} \mathds{1}_{\{d_{0,r} \geq \Delta_{1} + \frac{1}{2} \}}
\end{equation}
in \eqref{eq:crucialterms}--\eqref{eq:crucialterms1} with $\ell$ and $l_{j}$ as in \eqref{eq:lj}. Here, we do not need a uniform bound. For this reason, it is not necessary to take the supremum over all admissible estimates of $D_{0}$ in Lemmas \ref{prop:supGhat} and \ref{prop:bias}. In particular, Lemma \ref{prop:supGhat} simplifies, since the derivatives in the respective bounds $L_{rr,i}$ in \eqref{eq:prop:supGhatLs} become zero with the choice \eqref{eq:Ghat4fctOmega}. 
Throughout the proof, we assume $d_{0,r} \geq \Delta_{1} + \frac{1}{2}$. 

\textit{Probabilistic part:}
Note that
\begin{equation*}
\begin{aligned}
\max_{j = 1, \dots, m} | l_{j} -1| \Big(\frac{j}{\ell}\Big)^{2d_{0,r}}
&=
\max_{j = 1, \dots, m} 
\begin{cases}
\Big| \Big(\frac{j}{\ell}\Big)^{2(-\frac{1}{2}+\Delta)} -1\Big| \Big(\frac{j}{\ell}\Big)^{2d_{0,r}}, \hspace{0.2cm} &1 \leq j \leq \ell, \\
\Big| \Big(\frac{j}{\ell}\Big)^{2(\Delta_{1}-d_{0,r})}-1\Big| \Big(\frac{j}{\ell}\Big)^{2d_{0,r}}, \hspace{0.2cm} &\ell < j \leq m,
\end{cases}
\\&=
\max_{j = 1, \dots, m} 
\begin{cases}
\Big| \Big(\frac{j}{\ell}\Big)^{2(d_{0,r}-\frac{1}{2}+\Delta)} - \Big(\frac{j}{\ell}\Big)^{2d_{0,r}} \Big|, \hspace{0.2cm} &1 \leq j \leq \ell, \\
\Big| \Big(\frac{j}{\ell}\Big)^{2\Delta_{1}} - \Big(\frac{j}{\ell}\Big)^{2d_{0,r}} \Big|, \hspace{0.2cm} &\ell < j \leq m,
\end{cases}
\\&\leq
\max\left\{ \ell^{-\min\{2(d_{0,r}-\frac{1}{2}+\Delta),0\}}, \Big(\frac{m}{\ell}\Big)^{2d_{0,r}} \right\}.
\end{aligned}
\end{equation*}
%Note that $d_{0,r}-\frac{1}{2}+\Delta \leq \Delta_{1}$ if $d_{0,r}\leq 0$ for $\Delta$ small.
Then, 
\begin{equation}
\begin{aligned}
L_{rr,1}
&= 
N^{\max\{2d_{0,r},0\}} \max_{j = 1, \dots, m} | \ell_{j} -1| \Big(\frac{j}{\ell}\Big)^{2d_{0,r}} \lambda_{\ell}^{2d_{0,r}} \\
&\leq 
c N^{2d_{0,r}} \max\left\{ \ell^{-\min\{2(d_{0,r}-\frac{1}{2}+\Delta),0\}}, \Big(\frac{m}{\ell}\Big)^{2d_{0,r}} \right\} \Big(\frac{\ell}{N}\Big)^{2d_{0,r}} 
\\&\leq 
c \max\{\ell^{1-2\Delta},m^{2\Delta_{2}}\} 
=: \mathcal{L}_{1} \label{eq:4L11}.
\end{aligned}
\end{equation}
Furthermore, 
\begin{align}
L_{rr,i}
&= c^{2}_{r,i,N} \max_{j = 1, \dots, m} | \ell_{j} -1| \Big(\frac{j}{\ell}\Big)^{2d_{0,r}} \lambda_{\ell}^{2d_{0,r}} \nonumber \\
&\leq c 
\begin{cases}
N^{2d_{0,r}} \max\left\{ \ell^{-\min\{2(d_{0,r}-\frac{1}{2}+\Delta),0\}}, \Big(\frac{m}{\ell}\Big)^{2d_{0,r}} \right\} \Big(\frac{\ell}{N}\Big)^{2d_{0,r}},
\hspace{0.2cm}& \text{ if } d_{0,r} > \frac{1}{4}, \ i=2,\\
m^{\frac{1}{2}} N^{2d_{0,r}} \max\left\{ \ell^{-\min\{2(d_{0,r}-\frac{1}{2}+\Delta),0\}}, \Big(\frac{m}{\ell}\Big)^{2d_{0,r}} \right\} \Big(\frac{\ell}{N}\Big)^{2d_{0,r}},
\hspace{0.2cm}& \text{ if } d_{0,r} \leq \frac{1}{4}, \ i=4,
\end{cases} \nonumber
\\&\leq 
c 
\begin{cases}
\max\{\ell^{1-2\Delta},m^{2\Delta_{2}}\},
\hspace{0.2cm}& \text{ if } d_{0,r} > \frac{1}{4}, \ i=2,\\
\max\{m^{\frac{1}{2}} \ell^{1-2\Delta},m^{\frac{1}{2} + 2d_{0,r} }\},
\hspace{0.2cm}& \text{ if } d_{0,r} \leq \frac{1}{4}, \ i=4,
\end{cases} \nonumber
\\&\leq 
c \max\{m^{\frac{1}{2}} \ell^{1-2\Delta}, m^{2\widebar{\Delta}_{u}}\} 
=: \mathcal{L}_{2}. \label{eq:4L21}
\end{align}

\textit{Deterministic part:} Using  Lemma \ref{prop:bias} with \eqref{eq:Ghat4fctOmega}, we get
\begin{equation} \label{eq:ineqprop4det}
\begin{aligned}
&| \frac{1}{m} \sum_{j=1}^{m} (l_{j}-1) \lambda_{j}^{2d_{0,r}} (\E I_{X,rr}(\lambda_{j})-\lambda_{j}^{-2d_{0,r}}g_{0,r})| 
\\&\leq
\bm{c}_{G,1} \frac{1}{2 \pi} \lambda_{m}^{2q} \Big( \frac{1}{2\Delta} + 1 \Big)
+
\frac{1}{2 \pi m} \Big( \ell^{1-2\Delta} \Big(1+ \frac{1}{1-2\Delta} \Big)
+ \log(m) \Big) Q_{m} = \mathcal{T}_{4},
\end{aligned}
\end{equation}
where $\mathcal{T}_{4}$ is in Table \ref{label 2} and the inequality can be obtained  by bounding the two terms in \eqref{eq:prop:biasbound} given \eqref{eq:Ghat4fctOmega}. as follows. For $d_{0,r} \geq \Delta_{1} + \frac{1}{2}$, the first summand in \eqref{eq:prop:biasbound} can be bounded as
\begin{equation*}
\begin{aligned}
&
\frac{1}{2 \pi m} \sum_{j=1}^{m} |l_{j}-1| \lambda_{j}^{2d_{0,r}} \lambda_{j}^{2q-2d_{0,r}}
\\&=
\frac{1}{2 \pi m} \Big( \sum_{j=1}^{\ell} | \Big(\frac{j}{\ell}\Big)^{-1+2\Delta} -1 | \lambda_{j}^{2q}
+ \sum_{j=\ell+1}^{m} | \Big(\frac{j}{\ell}\Big)^{2\Delta_{1}-2d_{0,r}} -1 | \lambda_{j}^{2q} \Big)
\\&\leq
\frac{1}{2 \pi m} \lambda_{m}^{2q} \Big( \int_{0}^{\ell} x^{-1+2\Delta} dx \ell^{1-2\Delta} + (m-\ell) \Big)
\\&\leq
\frac{1}{2 \pi m} \lambda_{m}^{2q} \Big( \frac{1}{2\Delta} \ell^{1-2\Delta}+(m-\ell) \Big)
\\&\leq
\frac{1}{2 \pi} \lambda_{m}^{2q} \Big( \frac{1}{2\Delta} + 1 \Big)
\end{aligned}
\end{equation*}
and the second as
\begin{align*}
&\frac{1}{2 \pi m} \sum_{j=1}^{m} |l_{j}-1| \lambda_{j}^{2d_{0,r}} j^{-1} \lambda_{j}^{-2d_{0,r}}
=
\frac{1}{2 \pi m} \sum_{j=1}^{m} |l_{j}-1| j^{-1}
\\&=
\frac{1}{2 \pi m} \Big( \sum_{j=1}^{\ell} j^{-1} | \Big(\frac{j}{\ell}\Big)^{-1+2\Delta} -1 | 
+ \sum_{j=\ell+1}^{m} j^{-1} | \Big(\frac{j}{\ell}\Big)^{2\Delta_{1}-2d_{0,r}} -1 | \Big)
\\&\leq
\frac{1}{2 \pi m} \Big( \sum_{j=1}^{\ell} j^{-1} \Big(\frac{j}{\ell}\Big)^{-1+2\Delta} 
+ \log(m) \Big)
\\&\leq
\frac{1}{2 \pi m} \Big( \ell^{1-2\Delta} \Big(1+ \frac{1}{1-2\Delta} \Big)
+ \log(m) \Big).
\end{align*}

Finally, we combine our results on the probabilistic and deterministic terms to obtain the statement of the proposition.
With our choice of $\nu_{3}$ in \eqref{eq:nu3} and for any $\bm{C} \geq 1$, observe that
\begin{align}
&\Prob \Big( | \frac{1}{m} \sum_{j=1}^{m} (l_{j}-1) \lambda_{j}^{2d_{0,r}} (I_{X,rr}(\lambda_{j})-\lambda_{j}^{-2d_{0,r}}g_{0,r})| \mathds{1}_{\{d_{0,r} \geq \Delta_{1} + \frac{1}{2}\}} > \nu_{3} \Big) \nonumber
\\&\leq
\Prob \Big( 
| \frac{1}{m} \sum_{j=1}^{m} (l_{j}-1) \lambda_{j}^{2d_{0,r}} (I_{X,rr}(\lambda_{j})-\E I_{X,rr}(\lambda_{j})) |
\mathds{1}_{\{d_{0,r} \geq \Delta_{1} + \frac{1}{2}\}} \nonumber
\\&\hspace{1cm}+
| \frac{1}{m} \sum_{j=1}^{m} (l_{j}-1) \lambda_{j}^{2d_{0,r}} (\E I_{X,rr}(\lambda_{j}) - \lambda_{j}^{-2d_{0,r}}g_{0,r})| \mathds{1}_{\{d_{0,r} \geq \Delta_{1} + \frac{1}{2}\}} > \bm{C} \vertiii{G}\sqrt{\frac{\log(p)}{\mathcal{R}_{4}}} + \mathcal{T}_{4} 
\Big) \nonumber
\\&\leq
\Prob \Big(
| \frac{1}{m} \sum_{j=1}^{m} (l_{j}-1) \lambda_{j}^{2d_{0,r}} (I_{X,rr}(\lambda_{j})-\lambda_{j}^{-2d_{0,r}}g_{0,r})| \mathds{1}_{\{d_{0,r} \geq \Delta_{1} + \frac{1}{2}\}} > 
\bm{C} \vertiii{G} \sqrt{\frac{\log(p)}{\mathcal{R}_{4}}} \Big) \label{eq:prop4end1}
\\&\leq
\exp\Bigg( -c_{2}\bm{C} \min 
\Bigg\{ 
\frac{ \sqrt{\frac{\log(p)}{ \widebar{\Delta}_{N}^{-1} \mathcal{R}_{41}}} m }{ \gamma^2 \widebar{\Delta}_{N} \mathcal{L}_{1} },
\frac{ \frac{\log(p)}{ \widebar{\Delta}_{N}^{-2} \mathcal{R}_{42}} m^2 }{ \gamma^4 \widebar{\Delta}_{N}^{2} \mathcal{L}_{2}^2 }
\Bigg\} \Bigg) \label{eq:prop4end4}
\\&\leq
c_{1} \exp\Bigg( -c_{2}\bm{C} \min 
\Bigg\{ \sqrt{ \frac{\log(p)}{\widebar{\Delta}_{N}} \mathcal{R}_{41} }, \log(p) 
\Bigg\} \Bigg) \label{eq:prop4end7}
\\&\leq
c_{1} p^{ -c_{2}\bm{C} } \label{eq:prop4end8}
\end{align}
with $\mathcal{L}_{1},\mathcal{L}_{2}$ defined in \eqref{eq:4L11} and \eqref{eq:4L21}. The constants $c_{1},c_{2}$ are generic and might differ from line to line.
The inequality \eqref{eq:prop4end1} is due to \eqref{eq:ineqprop4det}. Applying Lemma \ref{prop:supGhat} yields \eqref{eq:prop4end4}. 
For the inequality \eqref{eq:prop4end7}, note that $\mathcal{R}_{41}$ and $\mathcal{R}_{42}$ are chosen in Table \ref{label 1} such that $\mathcal{R}_{41} \precsim m \mathcal{L}_{1}^{-1}$ and $\mathcal{R}_{42} \precsim m^{2} \mathcal{L}_{2}^{-2}$.
The inequality \eqref{eq:prop4end8} follows since we work under the assumption $\mathcal{R}_{41} \succsim \widebar{\Delta}_{N} \log(p)$; see the discussion following Proposition \ref{prop:supGhat4}.
\end{proof}

\begin{proof}[Proof of Proposition \ref{prop:Op}]
To prove the consistency result in \eqref{eq:consFrob}, we use the inequalities on the probability of $\{ \| \widehat{G}(\widehat{D})-G_{0} \|_{\max} > \delta \}$ discussed in Section \ref{s:mainResults} and 
combine the concentration inequalities in Propositions \ref{prop:supGhat1}--\ref{prop:supGhat4}. We have
\begin{align}
&
\Prob( \| \widehat{G}(\widehat{D})-G_{0} \|_{\max} > \delta ) \nonumber
\\ & =
\Prob \bigg( 
\{ \| \widehat{G}(\widehat{D})-G_{0} \|_{\max} > \delta \} \cap
\Big(\{ \| \widehat{D}-D_{0} \|_{\max} \leq \varepsilon \} \cup \{ \| \widehat{D}-D_{0} \|_{\max} > \varepsilon \} \Big)
\bigg) \nonumber
\\ & \leq
\Prob( \{
\| \widehat{G}(\widehat{D})-\widetilde{G}(\widehat{D}) \|_{\max} > \delta/2 \} \cap \{ \| \widehat{D}-D_{0} \|_{\max} \leq \varepsilon \} ) \nonumber
\\ & \hspace{1cm} +
\Prob\bigg( \Big(\{
\| \widetilde{G}(\widehat{D})-G_{0} \|_{\max} > \delta/2 \} \cap \{ \| \widehat{D}-D_{0} \|_{\max} \leq \varepsilon \} \Big)
\cup
\{ \| \widehat{D}-D_{0} \|_{\max} > \varepsilon \} \bigg) \nonumber
\\ & \leq
\Prob( \{
\| \widehat{G}(\widehat{D})-\widetilde{G}(\widehat{D}) \|_{\max} > \delta/2 \} \cap \{ \| \widehat{D}-D_{0} \|_{\max} \leq \varepsilon \} )
+ 
\Prob( \| \widehat{D}-D_{0} \|_{\max} > \eta )
\label{eq:A1add}
\\ & \leq
\sum_{r,s=1}^{p} \Prob( \sup_{ D \in \Omega(\varepsilon) } 
| \widehat{G}_{rs}(D)-\widetilde{G}_{rs}(D) | > \delta/2 )
+ 
\Prob( \| \widehat{D}-D_{0} \|_{\max} > \eta ) \label{eq:A1add2}
\end{align}
with $\eta= \min\{ \varepsilon, \frac{\delta}{4} (\|G_{0}\| \log(N) \lambda_{m}^{-2\varepsilon} 
L(-\varepsilon))^{-1} \}$. The inequality \eqref{eq:A1add} follows by Lemma \ref{le:FrobeniusGtoD}. We proceed with bounding the second probability in \eqref{eq:A1add} before we conclude with utilizing Propositions \ref{prop:supGhat1}--\ref{prop:supGhat4} to get the statement of this proposition.

To find an upper bound on the second probability in \eqref{eq:A1add}, we follow the \cite{robinson1995gaussian} approach used to prove the classical consistency result in the univariate case. Thus, we write the set $\Theta=[\Delta_{1},\Delta_{2}]$ as a union $\Theta=\Theta_{1} \cup \Theta_{2}$ with $\Theta_{1}$ as in 
\eqref{eq:defTheta1} and $\Theta_{2}$ as in \eqref{eq:defTheta2}; see Remark \ref{re:nonuniformly}.
%$\Theta_{1} = \{ d ~|~ d_{0,r} - \frac{1}{2} + \Delta \leq d \leq \Delta_{2} \}$ and $\Theta_{2} = \{ d ~|~ \Delta_{1} \leq d < d_{0,r} - \frac{1}{2} + \Delta \}$ 
%for some $\Delta>0$ since the function $R_{r}(d)$ behaves nonuniformly around $d=d_{0,r}-\frac{1}{2}$. 
Denote further $S_{r}(d)=R_{r}(d)-R_{r}(d_{0,r})$ with $R_{r}$ in \eqref{eq:univLWd}. Then, with explanations given below, 
\begin{align}
&
\Prob( | \widehat{d}_{r}-d_{0,r} | > \eta ) \nonumber
\\ &\leq
\Prob( \inf_{ d \in N_{\eta}^{c} \cap \Theta_{1}} S_{r}(d) < 0)
+
\Prob( \inf_{ d \in \Theta_{2}} S_{r}(d) < 0) \label{eq:propOP3}
\\ &\leq
\Prob( | \widehat{g}_{r}(d_{0,r})-g_{0,r} | > \frac{1}{8} \eta^2 \mathcal{V}_{1}(m) g_{0,r} ) 
+
\Prob( \sup_{ d \in \Theta_{1} } | \widehat{h}_{r}(d)-\widetilde{h}_{r}(d) | > \frac{1}{4} \eta^2 \mathcal{V}_{1}(m) L(\Delta_{2}-\Delta_{1}) g_{0,r} )
\nonumber
\\& \hspace{1cm}+
\Prob( | \frac{1}{m} \sum_{j=1}^{m} (l_{j}-1) (\lambda_{j}^{2d_{0,r}}I_{X}(\lambda_{j})-g_{0,r}) | \mathds{1}_{\{d_{0,r} \geq \Delta_{1} + \frac{1}{2}\}}
> \mathcal{V}_{2}(m) g_{0,r} ),  \label{eq:propOPb}
\end{align}
where $N_{\eta} = \{ d ~|~ |d-d_{0,r}| \leq \eta \}$ and $N_{\eta}^{c}$ denotes the complement of $N_{\eta}$.
The inequality in \eqref{eq:propOP3} follows from (3.2) in \cite{robinson1995gaussian}. Lemmas \ref{le:infS1} and \ref{le:Sd2} give \eqref{eq:propOPb}.
\par
Finally, combining \eqref{eq:A1add2} and \eqref{eq:propOPb} and choosing $\delta$ and $\varepsilon$ according to \eqref{eq:deltaepsilon}, we can infer that there exist constants $c_{1},c_{2}$ such that for any $\bm{C} \geq 1$, 
\begin{align}
&
\Prob( \| \widehat{G}(\widehat{D})-G_{0} \|_{\max} > \delta ) \nonumber
\\ & \leq
\sum_{r,s=1}^{p} \Prob( \sup_{ D \in \Omega(\varepsilon) } 
| \widehat{G}_{rs}(D)-\widetilde{G}_{rs}(D) | > \delta/2 )
+ 
\sum_{r=1}^{p} \Prob( | \widehat{g}_{r}(d_{0,r})-g_{0,r} | > \frac{1}{8} \eta^2 \mathcal{V}_{1}(m) g_{0,r} ) \nonumber
\\& \hspace{1cm}+
\sum_{r=1}^{p} \Prob( \sup_{ d \in \Theta_{1} } | \widehat{h}_{r}(d)-\widetilde{h}_{r}(d) | > \frac{1}{4} \eta^2 \mathcal{V}_{1}(m) L(\Delta_{2}-\Delta_{1}) g_{0,r} ) \nonumber
\\& \hspace{2cm}+
\sum_{r=1}^{p} \Prob( | \frac{1}{m} \sum_{j=1}^{m} (l_{j}-1) (\lambda_{j}^{2d_{0,r}}I_{X}(\lambda_{j})-g_{0,r}) | \mathds{1}_{\{d_{0,r} \geq \Delta_{1} + \frac{1}{2}\}}
> \varepsilon^2 \mathcal{V}_{2}(m) g_{0,r} ),  \label{eq:epsincorporated}
\\ &\leq
\sum_{r,s=1}^{p} \Prob( \sup_{ D \in \Omega(\varepsilon) } 
| \widehat{G}_{rs}(D)-\widetilde{G}_{rs}(D) | > \nu ) 
+
\sum_{r=1}^{p} \Prob( | \widehat{g}_{r}(d_{0,r})-g_{0,r} | > \nu_{1} ) \nonumber
\\& \hspace{1cm}+
\sum_{r=1}^{p} \Prob( \sup_{ d \in \Theta_{1} } | \widehat{h}_{r}(d)-\widetilde{h}_{r}(d) | > \nu_{2} ) \nonumber
\\& \hspace{2cm}+
\sum_{r=1}^{p} \Prob( | \frac{1}{m} \sum_{j=1}^{m} (l_{j}-1) (\lambda_{j}^{2d_{0,r}}I_{X}(\lambda_{j})-g_{0,r}) | \mathds{1}_{\{d_{0,r} \geq \Delta_{1} + \frac{1}{2}\}}
> \nu_{3} ) \label{eq:afterepsincorporated}
\\ &\leq
c_{1} p^{2-c_{2} \bm{C} }, \label{eq:lastlast}
\end{align}
where \eqref{eq:epsincorporated} follows by \eqref{eq:propOPb} and since $\varepsilon \in (0, \frac{1}{2})$ is assumed. 
For \eqref{eq:afterepsincorporated}, recall our choice of $\delta=
\max \{ 
2 \nu,  \varepsilon 4 \|G_{0}\| \log(N) \lambda_{m}^{-2\varepsilon} L(-\varepsilon) \}$ in \eqref{eq:deltaepsilon}.
Then, $\delta/2$ in the first probability of \eqref{eq:epsincorporated} can be bounded from below by $\nu$. For the remaining three probabilities in \eqref{eq:epsincorporated}, note that $\eta$ satisfies
\begin{equation*}
\eta= \min\{ \varepsilon, \frac{\delta}{4} (\|G_{0}\| \log(N) \lambda_{m}^{-2\varepsilon} 
L(-\varepsilon))^{-1} \} \geq \varepsilon,
\end{equation*}
since $\delta \geq \varepsilon 4 \|G_{0}\| \log(N) \lambda_{m}^{-2\varepsilon} L(-\varepsilon)$. Note also that $\varepsilon = \max_{i=1,2,3} \eta_{i}$ in \eqref{eq:deltaepsilon} can be bounded from below by $\eta_{1}, \eta_{2}$ or $\eta_{3}$. Then, for example, the lower bound in the second probability in \eqref{eq:epsincorporated} satisfies
\begin{equation*}
\frac{1}{8} \eta^2 \mathcal{V}_{1}(m) g_{0,r} 
\geq
\frac{1}{8} \varepsilon^2 \mathcal{V}_{1}(m) g_{0,r}  
\geq
\frac{1}{8} \eta_{1}^2 \mathcal{V}_{1}(m) g_{0,r}  
\geq \nu_{1},
\end{equation*}
where the last inequality follows by our choice of $\eta_{1}$ in \eqref{eq:eta1eta2eta3}.
This implies the second term as a bound in \eqref{eq:afterepsincorporated}. The other two probabilities can be dealt with similarly. Furthermore,
\eqref{eq:lastlast} is a consequence of applying Propositions \ref{prop:supGhat1}--\ref{prop:supGhat4} to the individual probabilities in \eqref{eq:afterepsincorporated}.
\end{proof}

\begin{proof}[Proof of Proposition \ref{prop:Opspectral}]
The key in proving Proposition \ref{prop:Opspectral} is finding a concentration 
inequality on $\| \widehat{G}(\widehat{D})-G_{0} \|_{\max}$. Such a concentration inequality is provided in Proposition \ref{prop:Op}. Then, it is left to follow exactly the proof of Proposition 3.6 in \cite{Sun2018:LargeSpectral}.
\end{proof}

\begin{proof}[Proof of Proposition \ref{prop:OpPrecision}]
The proof follows the ideas used in the proof of Theorem 1 in \cite{rothman2008sparse}. However,
we need to modify their proof in order to achieve a non-asymptotic result.
With the objective function $\ell_{\rho}$ as in \eqref{eq:Prholasso}, let
\begin{equation*}
\begin{aligned}
Q(\widehat{D},P) 
&= 
\ell_{\rho}(\widehat{D},P) - \ell_{\rho}(\widehat{D},P_{0})
\\&=
-\frac{1}{m} \sum_{j=1}^{m} \log|\lambda_{j}^{\widehat{D}} P \lambda_{j}^{\widehat{D}}|
+\tr( \widehat{G}(\widehat{D})P) + \rho \| P \|_{1,off}
\\&\hspace{1cm} +
\frac{1}{m} \sum_{j=1}^{m} \log|\lambda_{j}^{\widehat{D}} P_{0} \lambda_{j}^{\widehat{D}}|
-\tr( \widehat{G}(\widehat{D})P_{0}) - \rho \| P_{0} \|_{1,off}
\\&=
\tr( \widehat{G}(\widehat{D}) (P-P_{0}) ) 
-  (\log| P | - \log| P_{0} |)
+ \rho (\| P \|_{1,off} - \| P_{0} \|_{1,off})
\end{aligned}
\end{equation*}
and set $Z(\Delta)=Q(\widehat{D},P_{0}+\Delta)$.
We further define the set $R_{\lambda}=\{ P=P^{*} ~|~ \| P-P_{0} \|_{F} > \lambda \}$. Then, 
\begin{align}
\Prob( \| \widehat{P}_{\rho}(\widehat{D})-P_{0} \|_{F} > \lambda )
&=
\Prob( \widehat{P}_{\rho}(\widehat{D}) \in  R_{\lambda})
\leq
\Prob( \inf_{P \in R_{\lambda}} \ell_{\rho}(\widehat{D},P) \leq \inf_{P \in R^{c}_{\lambda}} \ell_{\rho}(\widehat{D},P) ) \nonumber
\\&=
\Prob( \inf_{P \in R_{\lambda}} \ell_{\rho}(\widehat{D},P) - \ell_{\rho}(\widehat{D},P_{0}) \leq \inf_{P \in R^{c}_{\lambda}} \ell_{\rho}(\widehat{D},P) - \ell_{\rho}(\widehat{D},P_{0}) ) \nonumber
\\&=
\Prob( \inf_{P \in R_{\lambda}} Q(\widehat{D},P) \leq \inf_{P \in R^{c}_{\lambda}} Q(\widehat{D},P) ) \nonumber
\\&\leq
\Prob( \inf_{P \in R_{\lambda}} Q(\widehat{D},P) \leq 0 ) \label{eq:infsmaller}
\\& \leq
\Prob( \inf_{\Delta \in \Theta_{N}} Z(\Delta) \leq 0 ), \label{eq:convexity}
\end{align}
where \eqref{eq:infsmaller} follows since $P_{0} \in R^{c}_{\lambda}$ and \eqref{eq:convexity} with
\begin{equation} \label{eq:thetalambda}
\Theta_{N}=\{ \Delta=\Delta^{*} ~|~ \| \Delta \|_{F} = \lambda \} \hspace{0.2cm} \text{ and } \lambda=\frac{16}{\underline{k}^2} \sqrt{p+\operatorname{s}} \ \delta  
\end{equation}
follows since $Z(\cdot)$ is a convex function. 
Then, it remains to prove that $\inf_{\Delta \in \Theta_{N}} Z(\Delta) > 0$
with high probability.

From the objective function in \eqref{eq:Prholasso} recall that the $l_{1}$-norm $\| \cdot \|_{1,off}$ excludes the diagonal elements and introduce its 
counterpart $\| \cdot \|_{1,on}$. Analogously, we write $\| \cdot \|_{F,off}$ and $\| \cdot \|_{F,on}$ for the Frobenius norm of the off and on diagonal elements of a matrix, respectively. Recall the definition of the index set $\operatorname{S}$ in \eqref{as:cardianlityS} and define further $M_{\operatorname{S}} = (m_{rs} \mathds{1}_{\{(r,s) \in \operatorname{S}\}})_{r,s=1,\dots,p}$ for a matrix $M=(m_{rs})_{r,s=1,\dots,p}$ and an index set $\operatorname{S}$.
Then, with further explanations given below, 
\begin{align}
Z(\Delta)
&= 
\tr( \widehat{G}(\widehat{D}) \Delta ) 
-  (\log| \Delta+P_{0} | - \log| P_{0} |)
+ \rho (\| \Delta+P_{0} \|_{1,off} - \| P_{0} \|_{1,off}) \nonumber
\\&= 
\tr( (\widehat{G}(\widehat{D})-G_{0}) \Delta ) 
+  
\vecop(\Delta)' \int_{0}^{1} (1-v) (P_{0}+v\Delta)^{-1} \otimes (P_{0}+v\Delta)^{-1} dv \vecop(\Delta) \nonumber
\\&\hspace{1cm}
+ 
\rho (\| \Delta+P_{0} \|_{1,off} - \| P_{0} \|_{1,off}) 
\label{al:line1}
\\&\geq
\tr( (\widehat{G}(\widehat{D})-G_{0}) \Delta ) 
+  
\| \Delta \|_{F}^2 \frac{1}{8} \underline{k}^2
+ 
\rho (\| \Delta+P_{0} \|_{1,off} - \| P_{0} \|_{1,off}) 
\label{al:line2}
\\&\geq
\| \Delta \|_{F}^2 \frac{1}{8} \underline{k}^2
-
\delta \| \Delta \|_{1}
+ 
\rho (\| \Delta+P_{0} \|_{1,off} - \| P_{0} \|_{1,off}) 
\label{al:line3}
\\&\geq
\| \Delta \|_{F}^2 \frac{1}{8} \underline{k}^2
-
\delta \| \Delta \|_{1}
+
\rho (
- \| \Delta_{\operatorname{S}} \|_{1,off}  + \| \Delta_{\operatorname{S}^{c}} \|_{1,off})
\label{al:line4}
\\&\geq
\| \Delta \|_{F}^2 \frac{1}{8} \underline{k}^2
-
\delta (\| \Delta \|_{1,on}+\| \Delta_{\operatorname{S}} \|_{1,off}+\| \Delta_{\operatorname{S}^{c}} \|_{1,off}) \nonumber
\\&\hspace{1cm}
+ \rho ( -\| \Delta_{\operatorname{S}} \|_{1,off} + \| \Delta_{\operatorname{S}^{c}} \|_{1,off}) 
\nonumber
%\label{al:line5}
\\&\geq
\| \Delta \|_{F,on}^2 \frac{1}{8} \underline{k}^2
-
\delta \| \Delta \|_{1,on}
+
\| \Delta \|_{F,off}^2 \frac{1}{8} \underline{k}^2
-
2 \delta \| \Delta_{\operatorname{S}} \|_{1,off} 
\label{al:line6}
\\&\geq
\| \Delta \|_{F,on}^2 (
\frac{1}{8} \underline{k}^2
-
\delta \sqrt{p} \| \Delta \|_{F,on}^{-1})
+
\| \Delta \|_{F,off}^2 (\frac{1}{8} \underline{k}^2
-
2 \delta \sqrt{\operatorname{s}}\| \Delta \|_{F,off}^{-1}) 
\label{al:line7}
\\&\geq
\| \Delta \|_{F,on}^2 (
\frac{1}{8} \underline{k}^2
-
2 \delta \sqrt{p + \operatorname{s}} \| \Delta \|_{F,on}^{-1})
+
\| \Delta \|_{F,off}^2 (\frac{1}{8} \underline{k}^2
-
2 \delta \sqrt{p + \operatorname{s}}\| \Delta \|_{F,off}^{-1}) \nonumber
%\label{al:line8}
\\&\geq
\| \Delta \|_{F}^2 (
\frac{1}{8} \underline{k}^2
-
\sqrt{2} \delta \sqrt{p + \operatorname{s}} \| \Delta \|_{F}^{-1})
>0. \label{al:line8}
\end{align}
The equality \eqref{al:line1} follows by the Taylor expansion of the function $f(t)=\log|P_{0}+t\Delta|$, which gives
\begin{equation*}
\begin{aligned}
&
\log| \Delta+P_{0} | - \log| P_{0} |
\\&=
\tr( G_{0} \Delta ) -
\vecop(\Delta)' \int_{0}^{1} (1-v) (P_{0}+v\Delta)^{-1} \otimes (P_{0}+v\Delta)^{-1} dv \vecop(\Delta).
\end{aligned}
\end{equation*}
In \eqref{al:line2}, following \cite{rothman2008sparse}, we further bound the integral part from below by its minimal eigenvalue
\begin{equation*}
\begin{aligned}
&\lambda_{\min} \Big(\int_{0}^{1} (1-v) (P_{0}+v\Delta)^{-1} \otimes (P_{0}+v\Delta)^{-1} dv\Big)
\\&\geq
\frac{1}{2} \min_{0\leq v \leq 1} \lambda_{\min}^{2} ( P_{0}+v\Delta)^{-1}
\geq
\frac{1}{2} \min_{\|\Delta\|_{F} \leq \lambda} \lambda_{\min}^{2} ( P_{0}+\Delta)^{-1}
\geq 
\frac{1}{8} \underline{k}^2,
\end{aligned}
\end{equation*}
since, for $\|\Delta\|_{F} \leq \lambda$,
\begin{equation*}
\begin{aligned}
\lambda_{\min}^{2} ( P_{0}+\Delta)^{-1}
&=
\lambda_{\max}^{-2} ( P_{0}+\Delta)
\geq
(\| P_{0} \| + \| \Delta \|)^{-2}
\geq
(\| P_{0} \| + \| \Delta \|_{F} )^{-2} 
\\&\geq
(\| P_{0} \| + \lambda )^{-2} 
\geq
(2 \| P_{0} \| )^{-2} 
=
\frac{1}{4} \lambda_{\min}^{2}(G_{0}) 
\geq
\frac{1}{4} \underline{k}^2,
\end{aligned}
\end{equation*}
since $\lambda=\frac{16}{\underline{k}^2} \sqrt{p+\operatorname{s}} \leq \| P_{0} \|$ by the assumption of Proposition \ref{prop:OpPrecision}.
Applying Proposition \ref{prop:Op} gives inequality \eqref{al:line3} with probability at least $1-c_{1} p^{2-c_{2} \bm{C}}$.
The inequality \eqref{al:line4} follows since 
\begin{equation*}
\begin{aligned}
\| \Delta+P_{0} \|_{1,off} - \| P_{0} \|_{1,off}
&=
\| \Delta_{\operatorname{S}}+P_{0,\operatorname{S}} \|_{1,off} + \| \Delta_{\operatorname{S}^{c}} \|_{1,off} - \| P_{0} \|_{1,off}
\\&\geq
-\| \Delta_{\operatorname{S}} \|_{1,off}  + \| \Delta_{\operatorname{S}^{c}} \|_{1,off}.
\end{aligned}
\end{equation*}
Line \eqref{al:line6} follows since $\rho=\delta$. 
In \eqref{al:line7}, we used the Cauchy-Schwarz inequality to get $\| \Delta \|_{1,on} \leq \sqrt{p}\| \Delta \|_{F,on}$ and $\| \Delta_{\operatorname{S}} \|_{1,off} \leq \sqrt{\operatorname{s}}\| \Delta \|_{F,off}$. Then, in \eqref{al:line8} it remains to verify that the expression is positive, which is a consequence of the choice of $\lambda$ in \eqref{eq:thetalambda}.
\end{proof}

\begin{proof}[Proof of Lemma \ref{col:OperatorPrecisionmodified}]
Lemma \ref{col:OperatorPrecisionmodified} is the analogue of Corollary 1 in \cite{rothman2008sparse}. 
The proof follows similar arguments as the proof of Proposition \ref{prop:OpPrecision}, aiming to use Proposition \ref{prop:Op}. 

Recall that $\Gamma_{0} = W^{-1}_{0} G_{0} W^{-1}_{0}$ and $\widehat{\Gamma}(\widehat{D}) = \widehat{W}^{-1}(\widehat{D}) \widehat{G}(\widehat{D}) \widehat{W}^{-1}(\widehat{D})$. The first step is to prove an analogue of Proposition \ref{prop:Op} in terms of the coherence matrix $\Gamma_{0}$, that is, 
\begin{equation} \label{eq:statementcoherence}
\Prob( \| \widehat{\Gamma}(\widehat{D}) - \Gamma_{0} \|_{\max} >  \widetilde{\delta} )
\leq
c_{1} p^{2-c_{2}\bm{C}}
\end{equation}
with $\widetilde{\delta} = 3 \delta \max\left\{1, \frac{1}{\underline{k}^2} \right\}$ and $\delta$ as in Proposition \ref{prop:Op}.
In order to prove \eqref{eq:statementcoherence}, note that
\begin{align}
\| \widehat{\Gamma}(\widehat{D}) - \Gamma_{0} \|_{\max}
&=
\max_{r,s=1,\dots, p} \left| \frac{ \widehat{G}_{rs}(\widehat{D})}{\sqrt{\widehat{G}_{rr}(\widehat{D})\widehat{G}_{ss}(\widehat{D})}} - \frac{G_{0,rs}}{\sqrt{G_{0,rr}G_{0,ss}}} \right|
\nonumber
\\&\leq
\max_{r,s=1,\dots,p} \left| \frac{ \widehat{G}_{rs}(\widehat{D})}{\sqrt{\widehat{G}_{rr}(\widehat{D})\widehat{G}_{ss}(\widehat{D})}} - \frac{ \widehat{G}_{rs}(\widehat{D}) }{ \sqrt{G_{0,rr}G_{0,ss}}} \right|
+
\max_{r,s=1,\dots,p} \left| \frac{\widehat{G}_{rs}(\widehat{D}) - G_{0,rs}}{\sqrt{G_{0,rr}G_{0,ss}}} \right|
\nonumber
\\&\leq
\max_{r,s=1,\dots,p}
\left| 1 - \left( \frac{ \widehat{G}_{rr}(\widehat{D})\widehat{G}_{ss}(\widehat{D}) }{ G_{0,rr}G_{0,ss}} \right)^{\frac{1}{2}} \right|
+
\max_{r,s=1,\dots,p} \left| \frac{\widehat{G}_{rs}(\widehat{D}) - G_{0,rs}}{\sqrt{G_{0,rr}G_{0,ss}}} \right|
\nonumber
\\&\leq
\max_{r,s=1,\dots,p}
\left| 1 - \frac{ \widehat{G}_{rr}(\widehat{D})\widehat{G}_{ss}(\widehat{D}) }{ G_{0,rr}G_{0,ss}} \right|
+
\delta \frac{1}{ \underline{k} }
\label{al:pppppp1}
\\&=
\max_{r,s=1,\dots,p} (G_{0,rr}G_{0,ss})^{-1}
\left| \widehat{G}_{rr}(\widehat{D})\widehat{G}_{ss}(\widehat{D}) - G_{0,rr}G_{0,ss} \right|
+
\delta \frac{1}{ \underline{k} }
\nonumber
\\&\leq
\max_{r,s=1,\dots,p} \frac{1}{ \underline{k}^2 } 
\left| (\widehat{G}_{rr}(\widehat{D}) - G_{0,rr} + G_{0,rr})(\widehat{G}_{ss}(\widehat{D}) - G_{0,ss} + G_{0,ss}) - G_{0,rr}G_{0,ss} \right|
+
\delta \frac{1}{ \underline{k} }
\nonumber
\\&\leq
\frac{1}{ \underline{k}^2 } (
\| \widehat{G}(\widehat{D}) -G_{0} \|^2_{\max} + 2 \max_{r=1,\dots,p} G_{0,rr} \| \widehat{G}(\widehat{D}) -G_{0} \|_{\max} )
+
\delta \frac{1}{ \underline{k} }
\nonumber
%\label{al:pppppp2}
\\&\leq
\delta^2 \frac{1}{ \underline{k}^2 } 
+
2\delta \frac{1}{ \underline{k} }
\leq
3 \delta \max\left\{1, \frac{1}{\underline{k}^2} \right\},
\label{al:pppppp3}
\end{align}
where \eqref{al:pppppp1} follows since
\begin{align}
\max_{r,s=1,\dots,p} \left| \frac{\widehat{G}_{rs}(\widehat{D}) - G_{0,rs}}{\sqrt{G_{0,rr}G_{0,ss}}} \right|
\leq \| W_{0}^{-1} \|^2 \| \widehat{G}(\widehat{D}) - G_{0} \|_{\max} 
\leq \delta \frac{1}{ \underline{k} }
\end{align}
with probability at least $1-c_{1} p^{2-c_{2}\bm{C}}$ by Proposition \ref{prop:Op} and since $\| W_{0}^{-1} \| \leq 1/ \lambda^{\frac{1}{2}}_{\min}(G_{0}) \leq 1/ \underline{k}^{\frac{1}{2}}$ by Assumption \ref{ass:f0}. The inequality \eqref{al:pppppp3} is also due to Proposition \ref{prop:Op}. 
This concludes the proof of \eqref{eq:statementcoherence}.

Moving on to proving the actual statement of Lemma \ref{col:OperatorPrecisionmodified} we follow \cite{rothman2008sparse} and \cite{shu2019estimation}. The proof requires a slight modification of the proof of Proposition \ref{prop:OpPrecision}. We conduct the same steps but use the objective function $\ell^{\Gamma}_{\rho}$ in \eqref{eq:onjfunctionGamma}; set $\Delta = \widehat{K} - K_{0}$ and replace \eqref{eq:thetalambda} by
\begin{equation} \label{eq:thetalambdaCOR}
\Theta_{N}=\{ \Delta=\Delta^{*} ~|~ \| \Delta \|_{F} = \lambda \} \hspace{0.2cm} \text{ and } \lambda=48 \max\left\{1, \frac{1}{\underline{k}^4} \right\} \sqrt{\operatorname{s}} \ \delta.
\end{equation}
We omit all steps up to \eqref{al:line2}. Then,
\begin{align}
Z(\Delta)
&\geq
\tr( (\widehat{\Gamma}(\widehat{D})-\Gamma_{0}) \Delta ) 
+  
\| \Delta \|_{F}^2 \frac{1}{8} \underline{k}^2
+ 
\rho (\| \Delta+K_{0} \|_{1,off} - \| K_{0} \|_{1,off}) 
%\label{al:line2COR}
\nonumber
\\&\geq
\| \Delta \|_{F}^2 \frac{1}{8} \underline{k}^2
-
\widetilde{\delta} \| \Delta \|_{1}
+ 
\rho (\| \Delta+K_{0} \|_{1,off} - \| K_{0} \|_{1,off}) 
\label{al:line3COR}
\\&\geq
\| \Delta \|_{F}^2 \frac{1}{8} \underline{k}^2
-
\widetilde{\delta} \| \Delta \|_{1}
+
\rho (
- \| \Delta_{\operatorname{S}} \|_{1,off}  + \| \Delta_{\operatorname{S}^{c}} \|_{1,off})
\label{al:line4COR}
\\&\geq
\| \Delta \|_{F}^2 \frac{1}{8} \underline{k}^2
-
\widetilde{\delta} (\| \Delta \|_{1,on}+\| \Delta_{\operatorname{S}} \|_{1,off}+\| \Delta_{\operatorname{S}^{c}} \|_{1,off}) \nonumber
\\&\hspace{1cm}
+ \rho ( -\| \Delta_{\operatorname{S}} \|_{1,off} + \| \Delta_{\operatorname{S}^{c}} \|_{1,off}) 
\nonumber
%\label{al:line5}
\\&\geq
\| \Delta \|_{F,on}^2 \frac{1}{8} \underline{k}^2
-
\widetilde{\delta} \| \Delta \|_{1,on}
+
\| \Delta \|_{F,off}^2 \frac{1}{8} \underline{k}^2
-
2 \widetilde{\delta} \| \Delta_{\operatorname{S}} \|_{1,off} 
\label{al:line6COR}
\\&\geq
\| \Delta \|_{F,on}^2 (
\frac{1}{8} \underline{k}^2
-
\widetilde{\delta} \sqrt{s} \| \Delta \|_{F,on}^{-1})
+
\| \Delta \|_{F,off}^2 (\frac{1}{8} \underline{k}^2
-
2 \widetilde{\delta} \sqrt{\operatorname{s}}\| \Delta \|_{F,off}^{-1}) 
\label{al:line7COR}
\\&\geq
\| \Delta \|_{F,on}^2 (
\frac{1}{8} \underline{k}^2
-
2 \widetilde{\delta} \sqrt{\operatorname{s}} \| \Delta \|_{F,on}^{-1})
+
\| \Delta \|_{F,off}^2 (\frac{1}{8} \underline{k}^2
-
2 \widetilde{\delta} \sqrt{\operatorname{s}}\| \Delta \|_{F,off}^{-1}) 
\nonumber
%\label{al:line8}
\\&\geq
\| \Delta \|_{F}^2 (
\frac{1}{8} \underline{k}^2
-
\sqrt{2} \widetilde{\delta} \sqrt{\operatorname{s}} \| \Delta \|_{F}^{-1})
>0.
\label{al:line8COR}
\end{align}
Applying \eqref{eq:statementcoherence} gives inequality \eqref{al:line3COR} with probability at least $1-c_{1} p^{2-c_{2} \bm{C}}$, and \eqref{al:line4COR} follows by the same arguments as \eqref{al:line4}.
Line \eqref{al:line6COR} follows since $\rho=\widetilde{\delta} = 3 \delta \max\left\{1, \frac{1}{\underline{k}^2} \right\}$. 
In \eqref{al:line7COR}, we used the Cauchy-Schwarz inequality to get $\| \Delta \|_{1,on} \leq \sqrt{\operatorname{s}} \| \Delta \|_{F,on}$ and $\| \Delta_{\operatorname{S}} \|_{1,off} \leq \sqrt{\operatorname{s}}\| \Delta \|_{F,off}$. Then, in \eqref{al:line8COR} it remains to verify that the expression is positive, which is a consequence of the choice of $\lambda$ in \eqref{eq:thetalambdaCOR}.
\end{proof}

\begin{proof}[Proof of Proposition \ref{prop:OperatorPrecisionmodified}]
The modified graphical local Whittle estimator \eqref{eq:modGLW} requires a consistency result for the estimator of $W_{0}^{-1}$ which can be derived, with explanations given below, as follows
\begin{align}
&
\| \widehat{W}(\widehat{D})^{-1} - W_{0}^{-1} \|
=
\max_{r=1,\dots,p} \left| \widehat{G}^{-\frac{1}{2}}_{rr}(\widehat{D}) - G^{-\frac{1}{2}}_{0,rr} \right|
\nonumber
\\&=
\max_{r=1,\dots,p} \left| \widehat{G}^{-\frac{1}{2}}_{rr}(\widehat{D}) - G^{-\frac{1}{2}}_{0,rr} \right|
\mathds{1}_{ \{ | \widehat{G}^{-\frac{1}{2}}_{rr}(\widehat{D}) - G^{-\frac{1}{2}}_{0,rr} | \geq 1 \} }
+
\max_{r=1,\dots,p} \left| \widehat{G}^{-\frac{1}{2}}_{rr}(\widehat{D}) - G^{-\frac{1}{2}}_{0,rr} \right|
\mathds{1}_{ \{ | \widehat{G}^{-\frac{1}{2}}_{rr}(\widehat{D}) - G^{-\frac{1}{2}}_{0,rr} | < 1 \} }
\nonumber
\\&\leq
\max_{r=1,\dots,p} \left| \widehat{G}^{-\frac{1}{2}}_{rr}(\widehat{D}) - G^{-\frac{1}{2}}_{0,rr} \right|^2
+
\max_{r=1,\dots,p} \left| \widehat{G}^{-\frac{1}{2}}_{rr}(\widehat{D}) - G^{-\frac{1}{2}}_{0,rr} \right|
\mathds{1}_{ \{ | \widehat{G}^{-\frac{1}{2}}_{rr}(\widehat{D}) - G^{-\frac{1}{2}}_{0,rr} | < 1 \} }
\nonumber
\\&\leq
\max_{r =1,\dots,p} \left| \frac{\widehat{G}_{rr}(\widehat{D}) - G_{0,rr}}{ G_{0,rr} } \right|^2
\left| \widehat{G}^{-\frac{1}{2}}_{rr}(\widehat{D}) \right|^2
\nonumber
\\&\hspace{1cm}+
\max_{r =1,\dots,p} \left| \frac{\widehat{G}_{rr}(\widehat{D}) - G_{0,rr}}{ G_{0,rr} } \right|
\left| \widehat{G}^{-\frac{1}{2}}_{rr}(\widehat{D}) \right|
\mathds{1}_{ \{ | \widehat{G}^{-\frac{1}{2}}_{rr}(\widehat{D}) - G^{-\frac{1}{2}}_{0,rr} | < 1 \} }
\label{eq:lplplplplplp0}
\\&=
\max_{r =1,\dots,p} \left| \frac{\widehat{G}_{rr}(\widehat{D}) - G_{0,rr}}{ G_{0,rr} } \right|^2
\left| \widehat{G}_{rr}(\widehat{D}) -G_{0,rr} + G_{0,rr}\right|^{-1}
\nonumber
\\&\hspace{1cm}+
\max_{r =1,\dots,p} \left| \frac{\widehat{G}_{rr}(\widehat{D}) - G_{0,rr}}{ G_{0,rr} } \right|
\left| \widehat{G}^{-\frac{1}{2}}_{rr}(\widehat{D}) \right|
\mathds{1}_{ \{ | \widehat{G}^{-\frac{1}{2}}_{rr}(\widehat{D}) - G^{-\frac{1}{2}}_{0,rr} | < 1 \} }
\nonumber
\\&\leq
\max_{r =1,\dots,p} \left| \frac{\widehat{G}_{rr}(\widehat{D}) - G_{0,rr}}{ G_{0,rr} } \right|^2
\left| \widehat{G}_{rr}(\widehat{D}) -G_{0,rr} \right|^{-1} + 
\max_{r =1,\dots,p} \left| \frac{\widehat{G}_{rr}(\widehat{D}) - G_{0,rr}}{ G^{\frac{3}{2}}_{0,rr} } \right|^2
\nonumber
\\&\hspace{1cm}+
\max_{r =1,\dots,p} \left| \frac{\widehat{G}_{rr}(\widehat{D}) - G_{0,rr}}{ G_{0,rr} } \right|
\left| G^{-\frac{1}{2}}_{0,rr} + 1 \right|
\nonumber
\\&\leq
\left( \frac{1}{ \underline{k}^2 } + \frac{1}{ \underline{k}^3 } \right)
\max_{r =1,\dots,p} \left| \widehat{G}_{rr}(\widehat{D}) - G_{0,rr} \right|^2
+ 
\left( \frac{1}{ \underline{k}^{ \frac{3}{2} } } + \frac{1}{ \underline{k} } \right)
\max_{r =1,\dots,p} \left| \widehat{G}_{rr}(\widehat{D}) - G_{0,rr} \right| 
\label{eq:lplplplplplp01}
\\&\leq
%\max\left\{\delta \frac{1}{ \underline{k}^2 }, \delta^{\frac{1}{2}} \frac{1}{ \underline{k}^3 } \right\}
%\leq 
\delta 4 \max\left\{ \frac{1}{ \underline{k}^3 }, 1 \right\},
\label{eq:lplplplplplp1}
\end{align}
where \eqref{eq:lplplplplplp0} is due to
\begin{align}
\left| \widehat{G}^{-\frac{1}{2}}_{rr}(\widehat{D}) - G^{-\frac{1}{2}}_{0,rr} \right|
=
\left| 1 - \left( \frac{G_{0,rr}}{\widehat{G}_{rr}} \right)^{-\frac{1}{2}} \right|
\left| \widehat{G}^{-\frac{1}{2}}_{rr}(\widehat{D}) \right|
&\leq
\left| 1 - \left( \frac{G_{0,rr}}{\widehat{G}_{rr}} \right)^{-\frac{1}{2}} \right|
\left| 1 + \left( \frac{G_{0,rr}}{\widehat{G}_{rr}} \right)^{-\frac{1}{2}} \right|
\left| \widehat{G}^{-\frac{1}{2}}_{rr}(\widehat{D}) \right|
\nonumber
\\&=
\left| \frac{\widehat{G}_{rr}(\widehat{D}) - G_{0,rr}}{ G_{0,rr} } \right| 
\left| \widehat{G}^{-\frac{1}{2}}_{rr}(\widehat{D}) \right|,
\nonumber
\end{align}
and \eqref{eq:lplplplplplp01} follows since $\lambda_{\min}(G_{0}) \geq \underline{k}$ by Assumption \ref{ass:f0}.
Finally, \eqref{eq:lplplplplplp1} is an application of Proposition \ref{prop:Op}.
We can further infer
\begin{align} \label{eq:lplplplplp}
\| \widehat{W}(\widehat{D})^{-1} \|
\leq
\| \widehat{W}(\widehat{D})^{-1} - W_{0}^{-1} \| + \| W_{0}^{-1} \|
\leq
\delta 4 \max\left\{ \frac{1}{ \underline{k}^3 }, 1 \right\}
+
\frac{1}{\underline{k}^{\frac{1}{2}}}
\leq 
5 \max\left\{ \frac{1}{\underline{k}^{3}}, 1 \right\}.
\end{align}
Finally, with further details given below,
\begin{align}
&
\| \widehat{P}_{\rho}^{M}(\widehat{D}) - P_{0} \| \nonumber
\\&\leq
\| \widehat{K}(\widehat{D}) - K_{0} \|_{F} \big( \| \widehat{W}(\widehat{D})^{-1} - W_{0}^{-1} \|^2 + \| W_{0}^{-1} \| \| \widehat{W}(\widehat{D})^{-1} \| \big)
\nonumber
\\&\hspace{1cm}+
\| \widehat{W}(\widehat{D})^{-1} - W_{0}^{-1} \| \left( \big( \| \widehat{K}(\widehat{D}) - K_{0} \|_{F} + \| K_{0} \| \big) \| W_{0}^{-1} \| + \| K_{0} \| \| \widehat{W}(\widehat{D})^{-1} \| \right) \label{eq:sksksksksksk1}
\\&\leq 
48 \max\left\{\frac{1}{\underline{k}^3},1 \right\} \sqrt{\operatorname{s}} \ \delta \left( 16 \max\left\{ \frac{1}{ \underline{k}^6 }, 1 \right\} + 
5\max\left\{ \frac{1}{\underline{k}^4}, 1 \right\} \right)
\nonumber
\\&\hspace{1cm}+
4 \delta \max\left\{ \frac{1}{ \underline{k}^3 }, 1 \right\} \left( \big(48 \max\left\{ \frac{1}{\underline{k}^3}, 1 \right\} \sqrt{\operatorname{s}} \ \delta + \| K_{0} \| \big) \frac{1}{\underline{k}^{\frac{1}{2}}} +  \| K_{0} \|  5\max\left\{ \frac{1}{\underline{k}^3}, 1 \right\}
\right) \label{eq:sksksksksksk2}
\\&\leq 30 * 48 \max\left\{ \frac{1}{\underline{k}^9}, 1 \right\} \max\{1, \| K_{0} \| \} \sqrt{\operatorname{s}} \ \delta \nonumber
\end{align}
with probability at least $1-c_{1} p^{2-c_{2} \bm{C}}$. The relation \eqref{eq:sksksksksksk1} is due to (S.52) in \cite{shu2019estimation}. 
Then, combining Lemma \ref{col:OperatorPrecisionmodified}, \eqref{eq:lplplplplplp1} and \eqref{eq:lplplplplp}, 
%and $\| K_{0} \| \leq \| W_{0} \| \| P_{0} \| \| W_{0} \| \leq \| G_{0} \| / \underline{k}$, 
we get \eqref{eq:sksksksksksk2}.
\end{proof}

\begin{proof}[Proof of Proposition \ref{prop:OperatorCLIME}]
We follow the proof of Theorem 6 in \cite{cai2011constrained}. Note that Theorem 6 in \cite{cai2011constrained} only requires that the maximum norm of the difference between covariance matrix and its estimator is controlled. 
Given that our Proposition \ref{prop:Op} provides an analogue result in the spectral domain, namely
\begin{equation}
\| \widehat{G}(\widehat{D})-G_{0} \|_{\max} \leq \delta
\end{equation}
with probability at least $1-c_{1} p^{2-c_{2}\bm{C}}$, we will only provide some of the key steps in the proof and otherwise refer to Theorem 6 in \cite{cai2011constrained}.
Following the proof of Theorem 6 in \cite{cai2011constrained} up to the point where equation (13) is established, we get
\begin{equation} \label{eq:njebcjecbejcbejc}
\| \widehat{\Theta}_{\rho} -P_{0} \|_{\max} \leq 4 \delta \| P_{0} \|_{1}
\end{equation}
by setting $\Sigma_{0} = G_{0}$ and $\Omega_{0} = P_{0}$ in \cite{cai2011constrained}. Then, one can further infer that
\begin{align}
\| \widehat{P}^{C}_{\rho}(\widehat{D}) -P_{0} \| 
&\leq (1 + 2^{1-a} + 3^{1-a}) \| \widehat{\Theta}_{\rho} -P_{0} \|_{\max}^{1-a} \| P_{0} \|_{a}^{a} \label{al:popopopop1}
\\&\leq 6 (4 \delta \| P_{0} \|_{1})^{1-a}  \| P_{0} \|_{a}^{a}, \label{al:popopopop2}
\end{align}
where \eqref{al:popopopop1} follows by (27) in \cite{cai2011constrained} and \eqref{al:popopopop2} by applying \eqref{eq:njebcjecbejcbejc}.
\end{proof}

\begin{proof}[Proof of Proposition \ref{prop:revoverythreshold}]
We prove \eqref{eq:GMSC:TH1} and \eqref{eq:GMSC:TH2} separately.

\textit{Proof of \eqref{eq:GMSC:TH1}:}
By the definition of the thresholding operator \eqref{eq:thresholding}, we get
\begin{align*}
\{ (r,s) ~|~ T_{\rho}(\widehat{G}_{rs}(\widehat{D})) \neq 0 , G_{0,rs} = 0 \} 
&= 
\{ (r,s) ~|~ | \widehat{G}_{rs}(\widehat{D}) | \geq \rho , G_{0,rs} = 0 \}
\\&\subseteq
\{ (r,s) ~|~ | \widehat{G}_{rs}(\widehat{D}) - G_{0,rs} | \geq \rho \}.
\end{align*}
Therefore, 
\begin{align}
\Prob \left( \sum_{r,s =1}^{d} \mathds{1}_{\{ T_{\rho}(\widehat{G}_{rs}(\widehat{D})) \neq 0 , G_{0,rs} = 0 \} } >0 \right)
\leq
\Prob\left( \max_{r,s =1,\dots, d} | \widehat{G}_{rs}(\widehat{D}) - G_{0,rs} | > \rho \right)
\leq
c_{1} p^{2-c_{2}\bm{C}}
\end{align}
by Proposition \ref{prop:Op}.

\textit{Proof of \eqref{eq:GMSC:TH2}:} By the definition of the thresholding operator \eqref{eq:thresholding}, we get
\begin{align*}
&
\{ (r,s) ~|~ T_{\rho}(\widehat{G}_{rs}(\widehat{D})) \leq 0 , G_{0,rs} > 0 \text{ or } T_{\rho}(\widehat{G}_{rs}(\widehat{D})) \geq 0 , G_{0,rs} < 0 \} 
%&= 
%\{ (r,s) ~|~ | \widehat{G}_{rs}(\widehat{D}) | \geq \rho , G_{rs} = 0 \}
\\&\subseteq
\{ (r,s) ~|~ -\widehat{G}_{rs}(\widehat{D}) > \rho , G_{0,rs} > \tau \text{ or } \widehat{G}_{rs}(\widehat{D}) > \rho , - G_{0,rs} > \tau \} 
\\&\subseteq
\{ (r,s) ~|~ | \widehat{G}_{rs}(\widehat{D}) - G_{0,rs} | > \tau + \rho \}.
\end{align*}
Therefore, 
\begin{align*}
&\Prob \left( \sum_{r,s =1}^{d} \mathds{1}_{\{ T_{\rho}(\widehat{G}_{rs}(\widehat{D})) \leq 0 , G_{0,rs} > 0 \text{ or } T_{\rho}(\widehat{G}_{rs}(\widehat{D})) \geq 0 , G_{0,rs} < 0 \} } >0 \right)
\\&\leq
\Prob\left( \max_{r,s =1,\dots, d} | \widehat{G}_{rs}(\widehat{D}) - G_{0,rs} | > \tau + \rho \right)
\leq
c_{1} p^{2-c_{2}\bm{C}}
\end{align*}
by Proposition \ref{prop:Op}.
\end{proof}

\begin{proof}[Proof of Proposition \ref{prop:revoveryCLIME}]
We omit the proof and as it is similar to the proof of Proposition \ref{prop:revoverythreshold}. 
\end{proof}

\section{Some technical results and their proofs}
\label{s:proofsB}

Section \ref{se:B1} provides probabilistic bounds which are used to show that Lemmas \ref{prop:supGhat} and \ref{prop:bias} are sufficient to prove a concentration inequality on the deviation between $\widehat{G}(\widehat{D})$ and the true $G_{0}$.
Section \ref{se:B2} gives some non-asymptotic results on the bias of the periodogram used to prove Lemma \ref{prop:bias}.
Finally, Section \ref{se:B3} concerns results on the matrix norms of the covariance matrix used to prove Lemma \ref{prop:supGhat}.

\subsection{Probabilistic bounds} \label{se:B1}

We prove here some probabilistic bounds used to show that the proof of Proposition \ref{prop:Op} can be reduced to proving Propositions \ref{prop:supGhat1}--\ref{prop:supGhat4}. We continue using the notation of the proof of Proposition \ref{prop:Op}.

\begin{lemma} \label{le:infS1}
For $\eta \in (0, \frac{1}{2})$, let $N_{\eta} = \{ d ~|~ |d-d_{0,r}| \leq \eta \}$ and $\Theta_{1}$ as in \eqref{eq:defTheta1}. Then, 
\begin{equation*}
\begin{aligned}
\Prob( \inf_{ d \in N_{\eta}^{c} \cap \Theta_{1}} S_{r}(d) < 0)
& \leq
\Prob \Big( \frac{1}{2} L(\Delta_{2}-\Delta_{1}) g_{0,r} \eta^2 \mathcal{V}_{1}(m) < \sup_{d \in \Theta_{1}} |\widehat{h}_{r}(d) - \widetilde{h}_{r}(d) | \Big) \\
&\hspace{1cm}+
\Prob \Big( \frac{1}{2} g_{0,r} \eta^2 \mathcal{V}_{1}(m) < |\widehat{g}_{r}(d_{0,r}) - g_{0,r} | \Big),
\end{aligned}
\end{equation*}
where $S_{r}(d)=R_{r}(d)-R_{r}(d_{0,r})$ with $R_{r}$ in \eqref{eq:univLWd}, $L(\cdot)$ in \eqref{eq:Tl}, $\mathcal{V}_{1}(m)$ in \eqref{eq:nu1nu2} and 
\begin{equation} \label{eq:hrandhrtilde}
\widehat{h}_{r}(d)=
\frac{1}{m} \sum_{j=1}^{m} \Big(\frac{j}{m}\Big)^{2d-2d_{0,r}} \lambda_{j}^{2d_{0,r}} I_{X,rr}(\lambda_{j}),
\hspace{0.2cm}
\widetilde{h}_{r}(d)=
\frac{1}{m} \sum_{j=1}^{m} \Big(\frac{j}{m}\Big)^{2d-2d_{0,r}} g_{0,r}.
\end{equation}
\end{lemma}

\begin{proof}
Let $\widetilde{S}_{r}(d)=\widetilde{R}_{r}(d)-\widetilde{R}_{r}(d_{0,r})$
with $\widetilde{R}_{r}(d)=\frac{1}{m} \sum_{j=1}^{m} \log (\lambda_{j}^{-2d} \widetilde{g}_{r}(d))$ and $\widetilde{g}_{r}(d)$ defined below \eqref{eq:Gtilde}. Then,
\begin{align}
\Prob( \inf_{ d \in N_{\eta}^{c} \cap \Theta_{1}} S_{r}(d) < 0) \nonumber
%\\ & =
%\Prob( \widetilde{S}_{r}(d) \leq \widetilde{S}_{r}(d) - \inf_{ d \in N_{\eta}^{c} \cap \Theta_{1}} S_{r}(d) ) \nonumber
& \leq
\Prob(  \inf_{ N_{\eta}^{c} \cap \Theta_{1}} \widetilde{S}_{r}(d) < \sup_{d \in \Theta_{1}} | \widetilde{S}_{r}(d) - S_{r}(d) | ) \nonumber
\\ & \leq
\Prob( \eta^2 \mathcal{V}_{1}(m) < \sup_{d \in \Theta_{1}} | \widetilde{S}_{r}(d) - S_{r}(d) | ) \label{eq:b1}
\\ & \leq
\Prob \Big( \frac{1}{2} L(\Delta_{2}-\Delta_{1}) g_{0,r} \eta^2 \mathcal{V}_{1}(m) < \sup_{d \in \Theta_{1}} |\widehat{h}_{r}(d) - \widetilde{h}_{r}(d) | \Big) \nonumber
\\&\hspace{1cm}+
\Prob \Big( \frac{1}{2} g_{0,r} \eta^2 \mathcal{V}_{1}(m) < |\widehat{g}_{r}(d_{0,r}) - g_{0,r} | \Big), \label{eq:b2}
\end{align}
where \eqref{eq:b1} is proved in Lemma \ref{le:infS} and \eqref{eq:b2} in Lemma \ref{le:subS} which is applicable since $\eta^2 \mathcal{V}_{1}(m) < \frac{1}{4} \mathcal{V}_{1}(m) = \frac{1}{4} \frac{1}{18} ( 1-\frac{1}{m^2}) < 1$.
\end{proof}

\begin{lemma} \label{le:infS}
For $d \in N_{\eta}^{c} \cap \Theta_{1}$, $\eta \in (0, \frac{1}{2})$, the quantity $\widetilde{S}_{r}(d)=\widetilde{R}_{r}(d)-\widetilde{R}_{r}(d_{0,r})$ can be bounded from below as
\begin{equation*}
\widetilde{S}_{r}(d) \geq \eta^2 \mathcal{V}_{1}(m)
\end{equation*}
with
\begin{equation*}
\mathcal{V}_{1}(m)=\frac{1}{3} \frac{1}{m^4} \sum_{i,j=1}^{m} (i-j)^2
\end{equation*}
satisfying $\mathcal{V}_{1}(m) = \frac{1}{18} ( 1-\frac{1}{m^2}) \sim \frac{1}{18} $ as $m \to \infty$.
\end{lemma}

\begin{proof}
The expression $\widetilde{S}_{r}(d)$ can be rewritten as
\begin{equation*}
\begin{aligned}
\widetilde{S}_{r}(d)
& =
\frac{1}{m} \sum_{j=1}^{m} \log (\lambda_{j}^{-2d}) + \log(\widetilde{g}_{r}(d)) -
\frac{1}{m} \sum_{j=1}^{m} \log (\lambda_{j}^{-2d_{0,r}}) - \log(g_{0,r})
%\\ & =
%\frac{1}{m} \sum_{j=1}^{m} \log (\lambda_{j}^{2d_{0,r}-2d}) +
%\log \Big(\frac{1}{m} \sum_{j=1}^{m} \lambda_{j}^{2d-2d_{0,r}}\Big)
\\ & =
\log \Big( \frac{1}{m} \sum_{j=1}^{m} \lambda_{j}^{2d-2d_{0,r}} \Big) - 
\frac{1}{m} \sum_{j=1}^{m} \log (\lambda_{j}^{2d-2d_{0,r}})
\\ & =
\log \Big( \frac{1}{m} \sum_{j=1}^{m} j^{2d-2d_{0,r}} \Big) - 
\frac{1}{m} \sum_{j=1}^{m} \log (j^{2d-2d_{0,r}}) =: f_{m}(d-d_{0,r}).
\end{aligned}
\end{equation*}
By Lemma \ref{le:fmineq} below, we know that
\begin{equation*}
f_{m}(d-d_{0,r}) \geq \min\{ f_{m}(-\eta), f_{m}(\eta)\}.
\end{equation*}
\par
We further prove lower bounds for $f_{m}(\eta)$ and $f_{m}(-\eta)$.
The proof is based on a result due to \cite{mercer1999some}.
For the readers' convenience, we shortly repeat this result:
Let $J(x)$ be the smallest closed interval that contains some $x_{j}$, $j=1,\dots,m$ and $g_{1},g_{2}$ be two twice differentiable functions on $J(x)$ with continuous second derivatives $g_{1}'',g_{2}''$. Then,
\begin{equation} \label{eq:Mercer}
\frac{g_{1}\Big( \frac{1}{m}\sum_{j=1}^{m} x_{j}  \Big) - \frac{1}{m} \sum_{j=1}^{m}g_{1}(x_{i}) }{g_{2}\Big( \frac{1}{m}\sum_{j=1}^{m} x_{j}  \Big) - \frac{1}{m} \sum_{j=1}^{m}g_{2}(x_{i})}
= \frac{g_{1}''(\xi)}{g_{2}''(\xi)} \text{ for some } \xi \in J(x), 
\end{equation}
given $g_{2}''(\xi) \neq 0$; see p.\ 678 in \cite{mercer1999some}\footnote{The relation \eqref{eq:Mercer} is also akin to Cauchy's mean value theorem.}.
From here on, we consider $f_{m}(\eta)$ and $f_{m}(-\eta)$ separately in order to find lower bounds.

\textit{Bounding $f_{m}(\eta)$:}
In \eqref{al:applyMercer1} below, we apply \eqref{eq:Mercer} with $(x_{1},\dots,x_{m})=(1,2^{2\eta},\dots,m^{2\eta})$ and $g_{1}(x)=\log(x)$, $g_{2}(x)=x^2$. The second derivatives are respectively $g_{1}''(x)=-x^{-2}$ and $g_{2}''(x)=2$ such that $\frac{g_{1}''(x)}{g_{2}''(x)}=-\frac{x^{-2}}{2}$. Then, there is a $\xi \in [1,m^{2\eta}]$ such that
\begin{align}
f_{m}(\eta)
&=
\log \Big(\frac{1}{m} \sum_{j=1}^{m} j^{2\eta} \Big)
-
\frac{1}{m} \sum_{i=1}^{m} \log (i^{2\eta}) \nonumber
\\&=
-\frac{1}{2\xi^2} \Big( \Big(\frac{1}{m} \sum_{j=1}^{m} j^{2\eta} \Big)^2
-
\frac{1}{m} \sum_{j=1}^{m} j^{4\eta} 
\Big) \label{al:applyMercer1}
\\&\geq
\frac{1}{2m^{4\eta}} \Big(
\frac{1}{m} \sum_{j=1}^{m} j^{4\eta} -
\Big(\frac{1}{m} \sum_{j=1}^{m} j^{2\eta} \Big)^2
\Big)
\nonumber
\\&\geq
\frac{1}{m^4} \sum_{i,j=1}^{m} (i-j)^2 \eta^2 \label{al:last1} ,
\end{align}
where the last inequality \eqref{al:last1} is due to \eqref{al:mvt} and proved below.

\textit{Bounding $f_{m}(-\eta)$:}
In \eqref{al:applyMercer2} below, we apply \eqref{eq:Mercer} with $(x_{1},\dots,x_{m})=(1,2^{-2\eta},\dots,m^{-2\eta})$ and $g_{1}(x)=\log(x)$, $g_{2}(x)=x^{-2}$. The second derivatives are respectively $g_{1}''(x)=-x^{-2}$ and $g_{2}''(x)=6x^{-4}$ such that $\frac{g_{1}''(x)}{g_{2}''(x)}=-\frac{x^{2}}{6}$. Then, there is a $\xi \in [m^{-2\eta},1]$ such that
\begin{align}
f_{m}(-\eta)
&=
\log \Big(\frac{1}{m} \sum_{j=1}^{m} j^{-2\eta} \Big)
-
\frac{1}{m} \sum_{i=1}^{m} \log (i^{-2\eta}) \nonumber
\\&=
-\frac{\xi^2}{6} \Big( \Big(\frac{1}{m} \sum_{j=1}^{m} j^{-2\eta} \Big)^{-2}
-
\frac{1}{m} \sum_{j=1}^{m} j^{4\eta} 
\Big) \label{al:applyMercer2}
\\&\geq
\frac{m^{-4\eta}}{6} \Big( 
\frac{1}{m} \sum_{j=1}^{m} j^{4\eta} -
\Big(\frac{1}{m} \sum_{j=1}^{m} j^{-2\eta} \Big)^{-2} 
\Big)
\nonumber
\\&\geq
\frac{1}{6m^{4\eta}} \Big( 
\frac{1}{m} \sum_{j=1}^{m} j^{4\eta} -
\Big(\frac{1}{m} \sum_{j=1}^{m} j^{2\eta} \Big)^{2} 
\Big) \label{al:Jensen}
\\&\geq
\frac{1}{3} \frac{1}{m^4} \sum_{i,j=1}^{m} (i-j)^2 \eta^2 \label{al:last2} ,
\end{align}
where \eqref{al:Jensen} is due to Jensen's inequality and \eqref{al:last2} is proved in \eqref{al:mvt} below.
\par
For the inequality in \eqref{al:last1} and \eqref{al:last2}, note that for $d \in N_{\eta}^{c} \cap \Theta_{1}$ with $\eta \in (0,\frac{1}{2})$,
\begin{align}
\frac{1}{m^{4\eta}} \Big(\frac{1}{m} \sum_{j=1}^{m} j^{4\eta} -
\Big(\frac{1}{m} \sum_{j=1}^{m} j^{2\eta} \Big)^{2} \Big)
& = 
\frac{1}{2m^{4\eta}} \frac{1}{m^2} \sum_{i,j=1}^{m} (i^{2\eta} - j^{2\eta})^2 \nonumber
\\ & = 
\frac{1}{2m^{4\eta}} \frac{1}{m^2} \sum_{i,j=1}^{m} ((i - j) (2\eta) x_{ij}^{2\eta-1})^2 \label{al:mvt}
\\ & \geq 
2\frac{1}{m^4} \sum_{i,j=1}^{m} (i-j)^2 \eta^2 \nonumber ,
\end{align}
where $x_{ij} \in (i, j)$ in \eqref{al:mvt} after applying the mean value theorem. 
\end{proof}

\begin{lemma} \label{le:fmineq}
For $d \in N_{\eta}^{c} \cap \Theta_{1}$, the following relation holds, 
\begin{equation*} 
%\label{eq:fmineq}
f_{m}(d-d_{0,r}) \geq \min\{ f_{m}(-\eta), f_{m}(\eta)\}
\end{equation*}
with
\begin{equation*}
f_{m}(x)=
\log \Big(\frac{1}{m} \sum_{j=1}^{m} j^{2x} \Big)
-
\frac{1}{m} \sum_{j=1}^{m} \log (j^{2x}).
\end{equation*}
\end{lemma}

\begin{proof}
The first and second derivatives of the function $f_{m}$ are given by
\begin{equation*}
\begin{aligned}
\frac{\partial }{\partial x} f_{m}(x)
&= \Big( \frac{1}{m} \sum_{j=1}^{m} j^{2x} \Big)^{-1} 2 \frac{1}{m} \sum_{j=1}^{m} \log(j) j^{2x} - 2 \frac{1}{m} \sum_{j=1}^{m} \log(j),
\\
\frac{\partial^2}{\partial x^2} f_{m}(x)
&= -\Big( \frac{1}{m} \sum_{j=1}^{m} j^{2x} \Big)^{-2}  \Big(2 \frac{1}{m} \sum_{j=1}^{m} \log(j) j^{2x} \Big)^2 
+ 
\Big( \frac{1}{m} \sum_{j=1}^{m} j^{2x} \Big)^{-1} 4 \frac{1}{m} \sum_{j=1}^{m} (\log(j))^2 j^{2x}.
\end{aligned}
\end{equation*}
Furthermore, the second derivative satisfies $\frac{\partial^2}{\partial x^2} f_{m}(x)>0$ for all $x$ since
\begin{equation} \label{eq:CSineqfm}
\begin{aligned}
\Big( \frac{1}{m} \sum_{j=1}^{m} \log(j) j^{2x} \Big)^2 
<
\frac{1}{m} \sum_{j=1}^{m} (\log(j))^2 j^{2x} \frac{1}{m} \sum_{j=1}^{m} j^{2x}
\end{aligned}
\end{equation}
using Cauchy-Schwarz inequality. The two sides in \eqref{eq:CSineqfm} cannot be equal since the vectors\\
$(0, \log(2) 2^{x}, \dots, \log(m) m^{x})$ and $(1, 2^{x}, \dots, m^{x})$ are linearly independent. The facts that $f_{m}$ is non-negative (by Jensen's inequality), is zero at $x =0$ and has a positive second derivative prove our claim.
\end{proof}

\begin{lemma} \label{le:subS}
For $d \in \Theta_{1}$ and all $\nu \in (0,1)$, the inclusion
\begin{equation*}
\begin{aligned}
&\{ | \widetilde{S}_{r}(d) - S_{r}(d) | > \nu \} 
\\&\subseteq
\Big\{ |\widehat{h}_{r}(d) - \widetilde{h}_{r}(d) |  > \nu \frac{1}{2} L(\Delta_{2}-\Delta_{1}) g_{0,r}  \Big\} 
\cup
\Big\{ |\widehat{g}_{r}(d_{0,r}) - g_{0,r} |  > \nu \frac{1}{2} g_{0,r} \Big\}
\end{aligned}
\end{equation*}
is satisfied with $\widehat{h}_{r}(d)$, $\widetilde{h}_{r}(d)$ in \eqref{eq:hrandhrtilde} and $L(\cdot)$ in \eqref{eq:Tl}.
\end{lemma}

\begin{proof}
The distance between $\widetilde{S}_{r}(d)$ and $S_{r}(d)$ can be written as
\begin{align}
| \widetilde{S}_{r}(d) - S_{r}(d) |
& =
\Big| 
\frac{1}{m} \sum_{j=1}^{m} \log (\lambda_{j}^{-2d}) + \log(\widetilde{g}_{r}(d)) -
\frac{1}{m} \sum_{j=1}^{m} \log (\lambda_{j}^{-2d_{0,r}}) - \log(g_{0,r}) \nonumber
\\& \hspace{1cm}-
\Big(\frac{1}{m} \sum_{j=1}^{m} \log (\lambda_{j}^{-2d}) + \log(\widehat{g}_{r}(d)) -
\frac{1}{m} \sum_{j=1}^{m} \log (\lambda_{j}^{-2d_{0.r}}) -  \log(\widehat{g}_{r}(d_{0,r})) \Big)
\Big| \nonumber
\\ & =
| \log(\widetilde{g}_{r}(d)) - \log(g_{0,r}) - \log(\widehat{g}_{r}(d)) + \log(\widehat{g}_{r}(d_{0,r}))| \nonumber
\\ & \leq
| \log( \widehat{g}_{r}(d) (\widetilde{g}_{r}(d))^{-1} )| + |\log( \widehat{g}_{r}(d_{0,r}) g_{0,r}^{-1} ) ) |. \label{eq:logloggh}
\end{align}
We consider only $| \log( \widehat{g}_{r}(d) (\widetilde{g}_{r}(d))^{-1} )|$ in \eqref{eq:logloggh}, since $|\log( \widehat{g}_{r}(d_{0,r}) g_{0,r}^{-1} ) |$ can be treated analogously.
Then, with explanation given below, 
\begin{align}
&
\{ |\log( \widehat{g}_{r}(d) (\widetilde{g}_{r}(d))^{-1}) | > \nu \} \nonumber
\\&=
\{ | \log (\widehat{g}_{r}(d)) - \log (\widetilde{g}_{r}(d) )  | > \nu \} \nonumber
\\&=
\{ \log (\widehat{g}_{r}(d)) - \log (\widetilde{g}_{r}(d) )   > \nu \} 
\cup
\{ \log (\widetilde{g}_{r}(d)) - \log (\widehat{g}_{r}(d))   > \nu \} \nonumber
\\ & \subseteq
\Bigg\{ \frac{ \widehat{g}_{r}(d) - \widetilde{g}_{r}(d)  }{ \widetilde{g}_{r}(d) }  > \nu \Bigg\} 
\cup
\Bigg\{ \frac{ \widetilde{g}_{r}(d) - \widehat{g}_{r}(d) }{ \widehat{g}_{r}(d) }  > \nu \Bigg\} \label{eq:log(gg)1}
\\ & =
\Bigg\{ \frac{ \widehat{h}_{r}(d) - \widetilde{h}_{r}(d)  }{ \widetilde{h}_{r}(d) }  > \nu \Bigg\} 
\cup
\Bigg\{ \frac{ \widetilde{h}_{r}(d) - \widehat{h}_{r}(d) }{ \widehat{h}_{r}(d) }  > \nu \Bigg\} \label{eq:log(hh)}
\\ & \subseteq
\{ \widehat{h}_{r}(d) - \widetilde{h}_{r}(d)  > \nu L(\Delta_{2}-\Delta_{1}) g_{0,r} \}
\cup
\Big( \{ \widetilde{h}_{r}(d) - \widehat{h}_{r}(d)   > \nu \widehat{h}_{r}(d) \} \nonumber
\\ & \hspace{1cm} \cap
\Big(
\{ \widehat{h}_{r}(d) \geq  \frac{1}{2} L(\Delta_{2}-\Delta_{1}) g_{0,r} \}
\cup
\{ \widehat{h}_{r}(d) < \frac{1}{2} L(\Delta_{2}-\Delta_{1}) g_{0,r} \} \label{eq:log(gg)2}
\Big) \Big)
\\ & =
\{ \widehat{h}_{r}(d) - \widetilde{h}_{r}(d)  > \nu L(\Delta_{2}-\Delta_{1}) g_{0,r} \}
\cup
\{ \widetilde{h}_{r}(d) - \widehat{h}_{r}(d)   > \frac{1}{2} \nu L(\Delta_{2}-\Delta_{1}) g_{0,r} \} \nonumber
\\ & \hspace{1cm} \cup
\Big( 
\{ \widetilde{h}_{r}(d) - \widehat{h}_{r}(d) \geq \nu \widehat{h}_{r}(d) \}
\cap
\{ \widehat{h}_{r}(d) < \frac{1}{2} L(\Delta_{2}-\Delta_{1}) g_{0,r} \} \nonumber
\Big) \Big)
%\\ & \subseteq
%\{ \widehat{h}_{r}(d) - \widetilde{h}_{r}(d) > \nu L(\Delta_{2}-\Delta_{1}) g_{0,r} \}
%\cup
%\{ \widetilde{h}_{r}(d) - \widehat{h}_{r}(d)   > \frac{1}{2} \nu L(\Delta_{2}-\Delta_{1}) g_{0,r} \} \nonumber
%\\ & \hspace{1cm} \cup
%\Big(
%\{ \widetilde{h}_{r}(d) > (\nu+1) \widehat{h}_{r}(d) \}
%\cap
%\{ \widehat{h}_{r}(d) < \frac{1}{2} L(\Delta_{2}-\Delta_{1}) g_{0,r} \}
%\Big) \nonumber
\\ & \subseteq
\{ \widehat{h}_{r}(d) - \widetilde{h}_{r}(d) > \nu L(\Delta_{2}-\Delta_{1}) g_{0,r} \}
\cup
\{ \widetilde{h}_{r}(d) - \widehat{h}_{r}(d)   > \frac{1}{2} \nu L(\Delta_{2}-\Delta_{1}) g_{0,r} \} \nonumber
\\ & \hspace{1cm} \cup
\{ \widetilde{h}_{r}(d) < \frac{1}{2} L(\Delta_{2}-\Delta_{1}) g_{0,r} + \widetilde{h}_{r}(d) - \widehat{h}_{r}(d) \}
\nonumber
\\ & \subseteq
\{ |\widehat{h}_{r}(d) - \widetilde{h}_{r}(d) |  > \nu \frac{1}{2} L(\Delta_{2}-\Delta_{1}) g_{0,r}  \} 
\cup
\{ |\widehat{h}_{r}(d) - \widetilde{h}_{r}(d) |  > \frac{1}{2} L(\Delta_{2}-\Delta_{1}) g_{0,r}  \}
\nonumber
\\ & \subseteq
\{ |\widehat{h}_{r}(d) - \widetilde{h}_{r}(d) |  > \nu \frac{1}{2} L(\Delta_{2}-\Delta_{1}) g_{0,r}  \},\nonumber
\end{align}
where \eqref{eq:log(gg)1} follows by the mean value theorem. The equality in \eqref{eq:log(hh)} can be seen by the definitions of $\widehat{h}_{r}$ and $\widetilde{h}_{r}$ in \eqref{eq:hrandhrtilde} and noting that $\lambda_{m}^{2d-2d_{0,r}} \widehat{h}_{r}(d) = \widehat{g}_{r}(d)$ and $\lambda_{m}^{2d-2d_{0,r}} \widetilde{h}_{r}(d) = \widetilde{g}_{r}(d)$, The relation \eqref{eq:log(gg)2} is a consequence of the lower bound
\begin{equation*}
\widetilde{h}_{r}(d)
=
\frac{1}{m} \sum_{j=1}^{m} \Big(\frac{j}{m}\Big)^{2d-2d_{0,r}} g_{0.r}
\geq
\frac{1}{m} \sum_{j=1}^{m} \Big(\frac{j}{m}\Big)^{2(\Delta_{2}-\Delta_{1})} g_{0,r}
\geq
\int_{0}^{1} x^{2(\Delta_{2}-\Delta_{1})} dx g_{0,r} =: L(\Delta_{2}-\Delta_{1}) g_{0,r}
%\widetilde{h}_{r}(\eta).
\end{equation*}
since $x^{2d}$ is monotonically increasing.
\end{proof}

\begin{lemma} \label{le:Sd2}
The second probability in \eqref{eq:propOP3} can be bounded as
\begin{equation} \label{eq:RobSd2}
\Prob(\inf_{d \in \Theta_{2}} S_{r}(d) < 0)
\leq
\Prob\Big( \Big| \frac{1}{m} \sum_{j=1}^{m} (l_{j}-1) \Big( \Big(\frac{j}{\ell}\Big)^{2d_{0,r}}I_{X,rr}(\lambda_{j})-g_{0,r}\Big) \Big| \mathds{1}_{\{d_{0,r} \geq \Delta_{1} + \frac{1}{2}\}}
> g_{0,r} \mathcal{V}_{2}(m) \Big)
\end{equation}
with $l_{j}$ and $\ell$ as in \eqref{eq:lj} and 
\begin{equation*}
\mathcal{V}_{2}(m)=\frac{1}{m} \sum_{j=1}^{m} (l_{j}-1),
\end{equation*}
satisfying $\mathcal{V}_{2}(m) > 0$.
\end{lemma}

\begin{proof}
The proof of relation \eqref{eq:RobSd2} is part of the proof of Theorem 1 in \cite{robinson1995gaussian}. More precisely, one can infer the inequality \eqref{eq:RobSd2} from equation (3.21) in \cite{robinson1995gaussian}. For completeness, we sketch the proof.
Recall the definitions of $\ell$ and $l_{j}$ given in \eqref{eq:lj}. Then, 
\begin{align}
&
\Prob(\inf_{d \in \Theta_{2}} S_{r}(d) < 0) \nonumber
\\ & \leq
\Prob\Big( \frac{1}{m} \sum_{j=1}^{m} (l_{j}-1) j^{2d_{0,r}}I_{X,rr}(\lambda_{j}) \mathds{1}_{\{d_{0,r} \geq \Delta_{1} + \frac{1}{2}\}} > 0 \Big) \label{eq:robrep}
\\ & \leq
\Prob\Big( \Big| \frac{1}{m} \sum_{j=1}^{m} (l_{j}-1) \Big(\Big(\frac{j}{\ell}\Big)^{2d_{0,r}}I_{X,rr}(\lambda_{j})-g_{0,r}\Big) \Big| \mathds{1}_{\{d_{0,r} \geq \Delta_{1} + \frac{1}{2}\}}
> \frac{1}{m} \sum_{j=1}^{m} (l_{j}-1) g_{0,r} \Big), \label{eq:robrep1}
\end{align}
where \eqref{eq:robrep} is due to (3.21) in \cite{robinson1995gaussian} and \eqref{eq:robrep1} results from division with $\ell^{2d_{0,r}}$ and subtracting $\frac{1}{m} \sum_{j=1}^{m} (l_{j}-1) g_{0,r} $ on both sides in \eqref{eq:robrep}. 
Finally, $\mathcal{V}_{2}(m) > 0$ is proved in Lemma \ref{le:V2positive}.
%Finally, $\mathcal{V}_{2}(m) \sim \frac{1}{2 \exp(1) \Delta} -1$ as $m \to \infty$ is proved in the third equation on p.\ 1639 in \cite{robinson1995gaussian}.
\end{proof}

\begin{lemma} \label{le:FrobeniusGtoD}
Given $\| \widehat{D}-D_{0} \|_{\max} \leq \varepsilon$, the population quantity \eqref{eq:Gtilde} evaluated at $\widehat{D}$ can be bounded in terms of $\widehat{D}$ by
\begin{equation*}
\| \widetilde{G}(\widehat{D})-G_{0} \|_{\max}
\leq
2 \| \widehat{D}-D_{0} \|_{\max} \|G_{0}\| \log(N) \lambda_{m}^{-2\varepsilon} 
L(-\varepsilon) ,
\end{equation*}
where $L(\cdot)$ is defined in \eqref{eq:Tl}.
\end{lemma}

\begin{proof}
Note that $\widehat{d}_{r}-d_{0,r} \in [-\varepsilon,\varepsilon]$ for all $r =1, \dots, p$. Then, 
%Let $c$ be a generic constant that can change from line to line. Then,
\begin{align}
\| \widetilde{G}(\widehat{D})-G_{0} \|_{\max}
& =
\max_{r,s=1,\dots,p} \Big| \frac{1}{m} \sum_{j=1}^{m} (\lambda_{j}^{(\widehat{d}_{r}-d_{0,r})+(\widehat{d}_{s}-d_{0,s})}-1) G_{0,rs} \Big|
\nonumber
\\ & \leq
2 \max_{r=1,\dots,p} | \widehat{d}_{r}-d_{0,r}| \|G_{0}\| \frac{1}{m} \sum_{j=1}^{m} |\log(\lambda_{j})| \lambda_{j}^{-2\varepsilon} 
\label{eq:A1mono}
\\ & \leq
2 \| \widehat{D}-D_{0} \|_{\max} \|G_{0}\| \log(N) \lambda_{m}^{-2\varepsilon} 
\frac{1}{m} \sum_{j=1}^{m} \Big( \frac{j}{m} \Big)^{-2\varepsilon} \nonumber
\\ & \leq
2 \| \widehat{D}-D_{0} \|_{\max} \|G_{0}\| \log(N) \lambda_{m}^{-2\varepsilon} 
\int_{0}^{1} x^{-2\varepsilon} dx.
\label{eq:A3mono}
\end{align}
The relation \eqref{eq:A1mono} is a consequence of applying the mean value theorem while \eqref{eq:A3mono} follows since $x^{-2\varepsilon}$ is monotonically decreasing.
\end{proof}

\begin{lemma} \label{le:V2positive}
The quantity $\mathcal{V}_{2}(m)$ in \eqref{eq:nu1nu2} satisfies $\mathcal{V}_{2}(m) > 0$.
\end{lemma}

\begin{proof}
Recall $\mathcal{V}_{2}(m)$ in \eqref{eq:nu1nu2} and note that, with $l_{j}$ as in \eqref{eq:lj}, it can be written as 
\begin{align} \label{eq:V2twosummands}
\mathcal{V}_{2}(m)
=
\frac{1}{m} \sum_{j=1}^{m} (l_{j}-1)
=
\frac{1}{m} \sum_{j=1}^{\ell} \Big( \Big(\frac{j}{\ell}\Big)^{2(-\frac{1}{2}+\Delta)} - 1 \Big) 
+
\frac{1}{m} \sum_{j=\ell + 1}^{m} \Big( \Big(\frac{j}{\ell}\Big)^{2(\Delta_{1}-d_{0,r})} - 1 \Big).
\end{align}
We will prove that both summands are positive. The first one satisfies
\begin{align*}
\frac{1}{m} \sum_{j=1}^{\ell} \Big( \Big(\frac{j}{\ell}\Big)^{2(-\frac{1}{2}+\Delta)} - 1 \Big) > 0
\hspace{0.2cm}
\text{ if }
\hspace{0.2cm}
\frac{1}{\ell} \sum_{j=1}^{\ell} j^{2(-\frac{1}{2}+\Delta)} - \ell^{2(-\frac{1}{2}+\Delta)} > 0.
\end{align*}
Set $\alpha = 2(-\frac{1}{2}+\Delta)$ and note that $\alpha < 0$. Then, 
\begin{align}
\frac{1}{\ell} \sum_{j=1}^{\ell} j^{\alpha} - \ell^{\alpha}
&\geq 
\frac{1}{\ell} \sum_{j=1}^{\ell} j^{\alpha} - \Big( \frac{1}{\ell} \sum_{j=1}^{\ell} j \Big)^{\alpha} \label{al:pup11}
\\&
= \frac{\alpha(\alpha-1) \xi^{\alpha-2}}{2}
\Big( \frac{1}{\ell} \sum_{j=1}^{\ell} j^{2} - \Big( \frac{1}{\ell} \sum_{j=1}^{\ell} j \Big)^{2} \Big) \label{al:pup12}
\\&
\geq \frac{\alpha(\alpha-1) \ell^{\alpha-2}}{4}
\frac{1}{\ell^{2}} \sum_{i,j=1}^{\ell} (i-j)^{2}
> 0, \notag
\end{align}
where \eqref{al:pup11} follows since $\ell \geq \frac{\ell +1}{2}$ and \eqref{al:pup12} is due to the result stated in \eqref{eq:Mercer} with $g_{1}(x)=x^{\alpha}$, $g_{2}(x)=x^{2}$ and $\xi \in [1,\ell]$.
The second summand in \eqref{eq:V2twosummands} satisfies
\begin{align*}
\frac{1}{m} \sum_{j=\ell + 1}^{m} \Big( \Big(\frac{j}{\ell}\Big)^{2(\Delta_{1}-d_{0,r})} - 1 \Big) > 0
\hspace{0.2cm}
\text{ if }
\hspace{0.2cm}
\frac{1}{m-\ell} \sum_{j=\ell + 1}^{m} j^{2(\Delta_{1}-d_{0,r})} - \ell^{2(\Delta_{1}-d_{0,r})} > 0.
\end{align*}
Set $\beta = 2(\Delta_{1}-d_{0,r})$ and note that $\beta <0$. Then, 
\begin{align}
\frac{1}{m-\ell} \sum_{j=\ell + 1}^{m} j^{\beta} - \ell^{\beta}
&\geq 
\frac{1}{m-\ell} \sum_{j=\ell + 1}^{m} j^{\beta} - \Big( \frac{1}{m-\ell} \sum_{j=\ell + 1}^{m} j \Big)^{\beta} \label{al:pup21}
\\&
= \frac{\alpha(\beta-1) \xi^{\beta-2}}{2}
\Big( \frac{1}{m-\ell} \sum_{j=\ell + 1}^{m} j^{2} - \Big( \frac{1}{m-\ell} \sum_{j=\ell + 1}^{m} j \Big)^{2} \Big) \label{al:pup22}
\\&
\geq \frac{\beta(\beta-1) m^{\beta-2}}{4}
\frac{1}{(m-\ell)^2} \sum_{j=\ell + 1}^{m} (i-j)^{2}
> 0, \notag
\end{align}
where \eqref{al:pup21} follows since $3\ell > m+1$ which can be shown by induction principal. The inequality \eqref{al:pup22} follows by applying \eqref{eq:Mercer} with $g_{1}(x)=x^{\beta}$, $g_{2}(x)=x^{2}$ and $\xi \in [\ell+1,m]$.
\end{proof}

\subsection{Bound for periodogram bias} \label{se:B2}

We prove here the results used in the proof of Lemma \ref{prop:bias}.
\begin{lemma} \label{le:biasappB1}
Suppose Assumptions \ref{ass:f0} and \ref{ass:derivative}. Then, 
\begin{equation*} 
\Big| 
f_{rs}(\lambda_{j}) 
- \E (I_{X,rs}(\lambda_{j}) ) \Big|
\leq 
N^{-1} \lambda_{j}^{ -1-d_{0,r}-d_{0,s} }
\left(
\vertiii{G} 
\frac{72 (\cos(\lambda_{m}/2))^{-2} }{\pi(1+2\min\{\Delta_{1},-\Delta_{2}\})}
+ \bm{c}_{G,2} 4(2+ \log(m)) \right).
\end{equation*}
\end{lemma}

\begin{proof}
In order to bound the bias of the periodogram, we follow the ideas in the proof of Theorem 2 in \cite{robinson1995Log}. 
We adapt its arguments of asymptotic nature to find non-asymptotic bounds which still ensure that the bias is negligible.

The bias term of the periodogram can be written as
\begin{align}
\Big| 
f_{rs}(\lambda_{j}) 
- \E (I_{X,rs}(\lambda_{j}) ) \Big| 
=
\Big| 
\int_{-\pi}^{\pi} \Big( f_{rs}(\lambda) - f_{rs}(\lambda_{j}) \Big) K_{N}(\lambda - \lambda_{j}) d\lambda 
\Big|, \label{eq:det_integral_rep}
\end{align}
where
\begin{equation*}
K_{N}(\lambda) = \frac{1}{2\pi N} \left( \frac{ \sin(N \lambda / 2) }{\sin(\lambda/2)} \right)^2
\end{equation*}
is the Fej\'er kernel (see equation (4.3) in \cite{robinson1995Log}). The Fej\'er kernel can be expressed in terms of the Dirichlet kernel as
\begin{equation} \label{eq:def:KandD}
K_{N}(\lambda) = \frac{1}{2\pi N} | D_{N}(\lambda) |^{2}
\hspace{0.2cm}
\text{ with }
\hspace{0.2cm}
D_{N}(\lambda) = \sum_{n=1}^{N} e^{i n \lambda} = \frac{ \sin(N \lambda /2) }{\sin(\lambda/2)},
\end{equation}
which satisfies
\begin{equation} \label{eq:Dirichlet_inequality}
|D_{N}(\lambda)| \leq 2 |\lambda|^{-1}
\hspace{0.2cm}
\text{ for }
\hspace{0.2cm}
\lambda \in (-\pi,\pi)\backslash \{0\};
\end{equation}
see equation (4.7) in \cite{robinson1995Log}.

In the following analysis, we focus on the integral in \eqref{eq:det_integral_rep}
and separate the integration range as follows
\begin{equation} \label{eq:det_intervals}
\int_{-\pi}^{\pi} = 
\int_{-\pi}^{-\frac{\lambda_{j}}{2}} +
\int_{-\frac{\lambda_{j}}{2}}^{\frac{\lambda_{j}}{2}} +
\int_{\frac{\lambda_{j}}{2}}^{2\lambda_{j}} + 
\int_{2\lambda_{j}}^{\pi}.
\end{equation}
Note that the positive range of $(-\pi,\pi)$ is separated into one more interval than the negative range. The additional interval $(\frac{\lambda_{j}}{2}, 2\lambda_{j}]$ takes care of a potentially zero argument in $K_{N}$.   
Handling the intervals on the right-hand side of \eqref{eq:det_intervals} from left to right, we get
\begin{align}
&
\Big| 
\int_{-\pi}^{-\frac{\lambda_{j}}{2}} \Big( f_{rs}(\lambda) - f_{rs}(\lambda_{j}) \Big) K_{N}(\lambda - \lambda_{j}) d\lambda 
\Big| \nonumber
\\& \leq
\int_{\frac{\lambda_{j}}{2}}^{\pi} \Big( |f_{rs}(\lambda)| + |f_{rs}(\lambda_{j})| \Big) K_{N}(\lambda + \lambda_{j}) d\lambda  
 \nonumber
\\& \leq \vertiii{G}
\int_{\frac{\lambda_{j}}{2}}^{\pi} \Big( \lambda^{-d_{0,r}-d_{0,s}} + \lambda_{j}^{-d_{0,r}-d_{0,s}} \Big) K_{N}(\lambda + \lambda_{j}) d\lambda  
 \nonumber
\\& = \vertiii{G} \frac{1}{2\pi N} \left(
\int_{\frac{\lambda_{j}}{2}}^{\pi} \lambda^{-d_{0,r}-d_{0,s}} |D_{N}(\lambda + \lambda_{j})|^2 d\lambda  +
\lambda_{j}^{-d_{0,r}-d_{0,s}} \int_{\frac{\lambda_{j}}{2}}^{\pi} |D_{N}(\lambda + \lambda_{j})|^2 d\lambda  \right) \nonumber
\\& \leq \frac{1}{2\pi} \vertiii{G} N^{-1} \lambda_{j}^{ -1-d_{0,r}-d_{0,s} } (\cos(\lambda_{m}/2))^{-2} \left( 16 \frac{1}{1+2\Delta_{1}} + 8 \right), \label{eq:det_int11}
\end{align}
where \eqref{eq:det_int11} follows by Lemmas \ref{le:int_lambda_Dirichlet} and \ref{le:int_lambda_Dirichlet_02}.

%where \eqref{eq:det_int11} follows by applying \eqref{eq:Dirichlet_inequality}.
The integral centered around zero can be handled as
\begin{align}
&
\Big| 
\int_{-\frac{\lambda_{j}}{2}}^{\frac{\lambda_{j}}{2}} \Big( f_{rs}(\lambda) - f_{rs}(\lambda_{j}) \Big) K_{N}(\lambda - \lambda_{j}) d\lambda 
\Big| \nonumber
\\&\leq
\max_{ |\lambda | \leq \lambda_{j}/2}
K_{N}(\lambda - \lambda_{j})
\int_{-\frac{\lambda_{j}}{2}}^{\frac{\lambda_{j}}{2}} \Big| f_{rs}(\lambda) - f_{rs}(\lambda_{j}) \Big| d\lambda \nonumber
%\\& =
%\max_{ |\lambda | \leq \lambda_{j}/2}
%\frac{1}{2\pi N} | D_{N}(\lambda - \lambda_{j}) |^{2}
%\int_{-\frac{\lambda_{j}}{2}}^{\frac{\lambda_{j}}{2}} \Big| f_{rs}(\lambda) - f_{rs}(\lambda_{j}) \Big| d\lambda \nonumber
\\&\leq
\max_{ |\lambda | \leq \lambda_{j}/2}
\frac{1}{2\pi N} 4 | \lambda - \lambda_{j} |^{-2}
\int_{-\frac{\lambda_{j}}{2}}^{\frac{\lambda_{j}}{2}} \Big( | f_{rs}(\lambda) | + | f_{rs}(\lambda_{j}) | \Big) d\lambda \label{eq:det_int12} 
\\&\leq
\frac{16}{2 \pi N} | \lambda_{j} |^{-2} \vertiii{G} \Big(
\int_{-\frac{\lambda_{j}}{2}}^{\frac{\lambda_{j}}{2}} | \lambda|^{-d_{0,r}-d_{0,s}} d\lambda  + \lambda_{j}^{1-d_{0,r}-d_{0,s}} \Big) \nonumber
\\&\leq
\frac{16}{2\pi} \vertiii{G} N^{-1} \lambda_{j}^{-1-d_{0,r}-d_{0,s}} \left( 4 \frac{1}{1-2\Delta_{2}} + 1\right) , \label{eq:qqqq}
\end{align}
where \eqref{eq:det_int12} follows by applying \eqref{eq:def:KandD}--\eqref{eq:Dirichlet_inequality}.

For the next integral, applying the mean value theorem gives
\begin{align}
&
\Big| 
\int_{\frac{\lambda_{j}}{2}}^{2\lambda_{j}} \Big( f_{rs}(\lambda) - f_{rs}(\lambda_{j}) \Big) K_{N}(\lambda - \lambda_{j}) d\lambda 
\Big| \nonumber
\\& \leq
\max_{ \frac{\lambda_{j}}{2} < \lambda \leq 2\lambda_{j} } \left| \frac{\partial}{\partial \lambda} f_{rs}(\lambda) \right|
\int_{\frac{\lambda_{j}}{2}}^{2\lambda_{j}} | \lambda - \lambda_{j} | K_{N}(\lambda - \lambda_{j}) d\lambda \nonumber
\\& \leq
\bm{c}_{G,2} \lambda_{j}^{-1-d_{0,r}-d_{0,s}}
\int_{\frac{\lambda_{j}}{2}}^{2\lambda_{j}} | \lambda - \lambda_{j} | \frac{1}{2\pi N} | D_{N}(\lambda - \lambda_{j}) |^2 d\lambda \label{eq:det_dom_applydervative}
\\& \leq
\bm{c}_{G,2} \lambda_{j}^{-1-d_{0,r}-d_{0,s}}
\frac{1}{\pi N} \int_{\frac{\lambda_{j}}{2}}^{2\lambda_{j}} | D_{N}(\lambda - \lambda_{j}) | d\lambda \nonumber
\\& \leq
\bm{c}_{G,2} \lambda_{j}^{-1-d_{0,r}-d_{0,s}}
\frac{1}{\pi N} (4 (\pi +2)
+ 4 \log(j)) \label{eq:det_dom_eqlog}
\\& \leq
\bm{c}_{G,2} \lambda_{j}^{-1-d_{0,r}-d_{0,s}} N^{-1} 4(2+ \log(j)),  \label{eq:qqqqq}
\end{align}
where \eqref{eq:det_dom_applydervative} is a consequence of Assumption \ref{ass:derivative} and \eqref{eq:det_dom_eqlog} follows by Lemma \ref{eq:integral_Dirichlet}. 

The remaining integral can be bounded as
\begin{align}
&
\Big| 
\int_{2\lambda_{j}}^{\pi} \Big( f_{rs}(\lambda) - f_{rs}(\lambda_{j}) \Big) K_{N}(\lambda - \lambda_{j}) d\lambda 
\Big| \nonumber
\\& \leq
\int_{2\lambda_{j}}^{\pi} \Big( |f_{rs}(\lambda)| + |f_{rs}(\lambda_{j})| \Big) K_{N}(\lambda - \lambda_{j}) d\lambda  \nonumber
\\& \leq 
\vertiii{G}
\int_{2\lambda_{j}}^{\pi} \Big( \lambda^{-d_{0,r}-d_{0,s}} + \lambda_{j}^{-d_{0,r}-d_{0,s}} \Big) K_{N}(\lambda - \lambda_{j}) d\lambda  \nonumber
%\\& \leq 
%\vertiii{G} \left(
%\int_{2\lambda_{j}}^{\pi} \lambda^{-d_{0,r}-d_{0,s}} K_{N}(\lambda - \lambda_{j}) d\lambda  +
%\lambda_{j}^{-d_{0,r}-d_{0,s}} \int_{2\lambda_{j}}^{\pi} K_{N}(\lambda - \lambda_{j}) d\lambda  \right) \nonumber
\\& =
\frac{1}{2\pi N} \vertiii{G} \left(
\int_{2\lambda_{j}}^{\pi} \lambda^{-d_{0,r}-d_{0,s}} |D_{N}(\lambda - \lambda_{j})|^2 d\lambda  +
\lambda_{j}^{-d_{0,r}-d_{0,s}} \int_{2\lambda_{j}}^{\pi} |D_{N}(\lambda - \lambda_{j})|^2 d\lambda  \right) \nonumber
\\& \leq 
\frac{1}{2\pi} \vertiii{G} N^{-1} \lambda_{j}^{-1-d_{0,r}-d_{0,s}} \left(16 \frac{1}{1+2\Delta_{1}} + 8\right), \label{eq:det_int14}
\end{align}
where \eqref{eq:det_int14} follows by Lemmas \ref{le:int_lambda_Dirichlet} and \ref{le:int_lambda_Dirichlet_02}.

Finally, using the integral representation \eqref{eq:det_integral_rep} and combining the inequalities \eqref{eq:det_int11}, \eqref{eq:qqqq}, \eqref{eq:qqqqq} and \eqref{eq:det_int14} for the individual integrals in \eqref{eq:det_intervals} gives
\begin{align*}
&
\Big| 
\int_{-\pi}^{\pi} \Big( f_{rs}(\lambda) - f_{rs}(\lambda_{j}) \Big) K_{N}(\lambda - \lambda_{j}) d\lambda 
\Big| 
\\&\leq
\frac{1}{2\pi} \vertiii{G} N^{-1} \lambda_{j}^{ -1-d_{0,r}-d_{0,s} } (\cos(\lambda_{m}/2))^{-2} \left( 16 \frac{1}{1+2\Delta_{1}} + 8 \right)
\\ &\hspace{1cm} +
\frac{16}{2\pi} \vertiii{G} N^{-1} \lambda_{j}^{-1-d_{0,r}-d_{0,s}} \left( 4 \frac{1}{1-2\Delta_{2}} + 1\right)
\\ &\hspace{2cm} +
\bm{c}_{G,2} N^{-1} \lambda_{j}^{-1-d_{0,r}-d_{0,s}} 4(2+ \log(j))
+
\frac{1}{2\pi} \vertiii{G} N^{-1} \lambda_{j}^{-1-d_{0,r}-d_{0,s}} \left(16 \frac{1}{1+2\Delta_{1}} + 8\right)
\\&\leq
\frac{1}{\pi} \vertiii{G} N^{-1} \lambda_{j}^{ -1-d_{0,r}-d_{0,s} } \left( ((\cos(\lambda_{m}/2))^{-2} + 1)
\left( 8 \frac{1}{1+2 \Delta_{1}} + 4 \right) + 4 \left( 8 \frac{1}{1-2\Delta_{2}} + 2 \right) \right)
\\ &\hspace{1cm} +
\bm{c}_{G,2} N^{-1} \lambda_{j}^{-1-d_{0,r}-d_{0,s}} 4(2+ \log(j))
\\&\leq
\frac{1}{\pi} \vertiii{G} N^{-1} \lambda_{j}^{ -1-d_{0,r}-d_{0,s} } \left( (\cos(\lambda_{m}/2))^{-2}
24 \frac{1}{1+2 \Delta_{1}} + 48 \frac{1}{1-2\Delta_{2}} \right)
\\ &\hspace{1cm} +
\bm{c}_{G,2} N^{-1} \lambda_{j}^{-1-d_{0,r}-d_{0,s}} 4(2+ \log(j))
\\&\leq
\frac{1}{\pi} \vertiii{G} N^{-1} \lambda_{j}^{ -1-d_{0,r}-d_{0,s} } ((\cos(\lambda_{m}/2))^{-2} + 2)
24 \frac{1}{1+2\min\{\Delta_{1},-\Delta_{2}\}} 
\\ &\hspace{1cm} +
\bm{c}_{G,2} N^{-1} \lambda_{j}^{-1-d_{0,r}-d_{0,s}} 4(2+ \log(j))
\\&\leq
N^{-1} \lambda_{j}^{ -1-d_{0,r}-d_{0,s} }
\left(
\vertiii{G} 
\frac{72 (\cos(\lambda_{m}/2))^{-2} }{\pi(1+2\min\{\Delta_{1},-\Delta_{2}\})}
+ \bm{c}_{G,2} 4(2+ \log(m)) \right).
\end{align*}
\end{proof}

\begin{lemma} \label{le:int_lambda_Dirichlet}
With the Dirichlet kernel defined in \eqref{eq:def:KandD}, 
\begin{align}
\int_{\frac{\lambda_{j}}{2}}^{\pi} \lambda^{-d_{0,r}-d_{0,s}} |D_{N}(\lambda + \lambda_{j})|^2 d\lambda
&\leq 16 \frac{1}{1+2\Delta_{1}} \lambda_{j}^{-1-d_{0,r}-d_{0,s}} (\cos(\lambda_{m}/2))^{-2}, \label{le:int_lambda_Dirichlet1}
\\
\int_{2\lambda_{j}}^{\pi} \lambda^{-d_{0,r}-d_{0,s}} |D_{N}(\lambda - \lambda_{j})|^2 d\lambda
&\leq 16 \frac{1}{1+2\Delta_{1}} \lambda_{j}^{-1-d_{0,r}-d_{0,s}}. \label{le:int_lambda_Dirichlet2}
\end{align}
\end{lemma}

\begin{proof}
We prove the inequalities \eqref{le:int_lambda_Dirichlet1} and \eqref{le:int_lambda_Dirichlet2} separately. For \eqref{le:int_lambda_Dirichlet1},
\begin{align}
\int_{\frac{\lambda_{j}}{2}}^{\pi} \lambda^{-d_{0,r}-d_{0,s}} |D_{N}(\lambda + \lambda_{j})|^2 d\lambda  
&= 
\int_{\frac{\lambda_{j}}{2}}^{\pi} \lambda^{-d_{0,r}-d_{0,s}} |D_{N}(\lambda)|^2 \left( \frac{|D_{N}(\lambda + \lambda_{j})|}{|D_{N}(\lambda)|}\right)^2 d\lambda \nonumber
\\&= \sup_{\frac{\lambda_{j}}{2} \leq \lambda \leq \pi}  \left( \frac{|D_{N}(\lambda + \lambda_{j})|}{|D_{N}(\lambda)|}\right)^2
\int_{\frac{\lambda_{j}}{2}}^{\pi} \lambda^{-d_{0,r}-d_{0,s}} |D_{N}(\lambda)|^2 d\lambda \nonumber  
\\&\leq (\cos(\lambda_{m}/2))^{-2} 
\int_{\frac{\lambda_{j}}{2}}^{\pi} \lambda^{-d_{0,r}-d_{0,s}} |D_{N}(\lambda)|^2 d\lambda \label{eq:aaaa}  
\\&\leq 4 (\cos(\lambda_{m}/2))^{-2}
\int_{\frac{\lambda_{j}}{2}}^{\infty} \lambda^{-2-d_{0,r}-d_{0,s}} d\lambda \label{eq:bbbb}   
\\&\leq 4 (\cos(\lambda_{m}/2))^{-2} \frac{1}{1+2\Delta_{1}}
2^{1+d_{0,r}+d_{0,s}} \lambda_{j}^{-1-d_{0,r}-d_{0,s}} \nonumber 
\\&\leq16 (\cos(\lambda_{m}/2))^{-2} \frac{1}{1+2\Delta_{1}} \lambda_{j}^{-1-d_{0,r}-d_{0,s}}, \nonumber 
\end{align}
where \eqref{eq:aaaa} is proved in Lemma \ref{eq:sup_Dirichlet} and \eqref{eq:bbbb} follows by \eqref{eq:Dirichlet_inequality}. Similarly, 
\begin{align}
\int_{2\lambda_{j}}^{\pi} \lambda^{-d_{0,r}-d_{0,s}} |D_{N}(\lambda - \lambda_{j})|^2 d\lambda
&=
\int_{2\lambda_{j}}^{\pi} \lambda^{-d_{0,r}-d_{0,s}} |D_{N}(\lambda)|^2 \left( \frac{|D_{N}(\lambda - \lambda_{j})|}{|D_{N}(\lambda)|}\right)^2 d\lambda \nonumber
\\&\leq 
\sup_{2\lambda_{j} \leq \lambda \leq \pi} \left( \frac{|D_{N}(\lambda - \lambda_{j})|}{|D_{N}(\lambda)|}\right)^2
\int_{2\lambda_{j}}^{\pi} \lambda^{-d_{0,r}-d_{0,s}} |D_{N}(\lambda)|^2 d\lambda \nonumber
\\&\leq 
4 \int_{2\lambda_{j}}^{\pi} \lambda^{-d_{0,r}-d_{0,s}} |D_{N}(\lambda)|^2 d\lambda \label{eq:aaaa1111}
\\&\leq 
16 \int_{2\lambda_{j}}^{\infty} \lambda^{-2-d_{0,r}-d_{0,s}} d\lambda \label{eq:bbbb1111}
\\&\leq
16 \frac{1}{1+2\Delta_{1}} 2^{-1-d_{0,r}-d_{0,s}} \lambda_{j}^{-1-d_{0,r}-d_{0,s}} \nonumber
\\&\leq
16 \frac{1}{1+2\Delta_{1}} \lambda_{j}^{-1-d_{0,r}-d_{0,s}}, \nonumber
\end{align}
where \eqref{eq:aaaa1111} is proved in Lemma \ref{eq:sup_Dirichlet} and \eqref{eq:bbbb1111} follows by \eqref{eq:Dirichlet_inequality}. 
\end{proof}

\begin{lemma} \label{le:int_lambda_Dirichlet_02}
With the Dirichlet kernel defined in \eqref{eq:def:KandD}, 
\begin{align}
\int_{\frac{\lambda_{j}}{2}}^{\pi} |D_{N}(\lambda + \lambda_{j})|^2 d\lambda 
\leq 8 \lambda_{j}^{-1} (\cos(\lambda_{m}/2))^{-2}
\hspace{0.2cm} \text{ and } \hspace{0.2cm} 
\int_{2\lambda_{j}}^{\pi} |D_{N}(\lambda - \lambda_{j})|^2 d\lambda  
\leq 
8 \lambda_{j}^{-1}. \label{le:int_lambda_Dirichlet_02_1}
\end{align}
\end{lemma}

\begin{proof}
The first integral in \eqref{le:int_lambda_Dirichlet_02_1} can be bounded as
\begin{align}
\int_{\frac{\lambda_{j}}{2}}^{\pi} |D_{N}(\lambda + \lambda_{j})|^2 d\lambda 
&=
\int_{\frac{\lambda_{j}}{2}}^{\pi} |D_{N}(\lambda)|^2 \left( \frac{|D_{N}(\lambda + \lambda_{j})|}{|D_{N}(\lambda)|}\right)^2 d\lambda \nonumber
\\&\leq
\sup_{\frac{\lambda_{j}}{2} \leq \lambda \leq \pi} \left( \frac{|D_{N}(\lambda + \lambda_{j})|}{|D_{N}(\lambda)|}\right)^2 \int_{\frac{\lambda_{j}}{2}}^{\pi} |D_{N}(\lambda)|^2 d\lambda \nonumber  
\\&\leq (\cos(\lambda_{m}/2))^{-2}
\int_{\frac{\lambda_{j}}{2}}^{\pi} |D_{N}(\lambda)|^2 d\lambda \label{eq:aaaa1}   
\\&\leq 
4 (\cos(\lambda_{m}/2))^{-2} \int_{\frac{\lambda_{j}}{2}}^{\infty} |\lambda|^{-2} d\lambda 
\leq 
8 (\cos(\lambda_{m}/2))^{-2} |\lambda_{j}|^{-1}, \label{eq:bbbb1}     
\end{align}
where \eqref{eq:aaaa1} is proved in Lemma \ref{eq:sup_Dirichlet} and the inequality \eqref{eq:bbbb1} follows by \eqref{eq:Dirichlet_inequality}. 
Similarly, 
\begin{align}
\int_{2\lambda_{j}}^{\pi} |D_{N}(\lambda - \lambda_{j})|^2 d\lambda  
&=
\int_{2\lambda_{j}}^{\pi} |D_{N}(\lambda)|^2 \left( \frac{|D_{N}(\lambda - \lambda_{j})|}{|D_{N}(\lambda)|}\right)^2 d\lambda \nonumber
\\&\leq
\sup_{2\lambda_{j} \leq \lambda \leq \pi} \left( \frac{|D_{N}(\lambda - \lambda_{j})|}{|D_{N}(\lambda)|}\right)^2 
\int_{2\lambda_{j}}^{\pi} |D_{N}(\lambda)|^2 d\lambda \label{eq:aaaa2222} 
\\&\leq 
4 \int_{2\lambda_{j}}^{\pi} |D_{N}(\lambda)|^2 d\lambda  \nonumber
\\&\leq 
16 \int_{2\lambda_{j}}^{\infty} |\lambda|^{-2} d\lambda  
\leq 
8 |\lambda_{j}|^{-1}, \label{eq:bbbb2222}
\end{align}
where \eqref{eq:aaaa2222} is proved in Lemma \ref{eq:sup_Dirichlet} and the inequality \eqref{eq:bbbb2222} follows by \eqref{eq:Dirichlet_inequality}.
\end{proof}

\subsection{Bounds on covariance matrices norms} \label{se:B3}

This section collects our results with bounds on different matrix norms of the covariance matrix and their proofs.

\begin{lemma} \label{le: boundsforallSigma}
Let $\Sigma_{rr} = (\Sigma_{rr}(n-k))_{n,k=1,\dots,N} = \E (\mathcal{X} e_{r} (\mathcal{X} e_{r})')$ and $d_{0,r} \in [\Delta_{1},\Delta_{2}]$ in \eqref{eq:f}. Then, there exist constants $c_{1}, c_{2}$ such that the spectral and Frobenius norms of $\Sigma_{rr}$ can be bounded as
\begin{align*} 
\Vert \Sigma_{rr} \Vert &\leq c_{1} \vertiii{G} N^{\max\{2d_{0,r}, 0\}} \widebar{\Delta}_{N} 	, \\
\Vert \Sigma_{rr} \Vert_{F} &\leq c_{2} \vertiii{G} N^{\max\{2d_{0,r}, \frac{1}{2}\}}	 \widebar{\Delta}_{N}
\end{align*}
with $\widebar{\Delta}_{N}$ as in \eqref{eq:delta-0-1/4}.
\end{lemma}

\begin{proof}
The bounds follow from Lemmas \ref{le:ineqSpecSigmashort} and \ref{le:ineqSpecFrobSigma} below.
\end{proof}

\begin{lemma} \label{le:ineqSpecSigmashort}
Let $\Sigma_{rr} = (\Sigma_{rr}(n-k))_{n,k=1,\dots,N} = \E (\mathcal{X} e_{r} (\mathcal{X} e_{r})')$ and $d_{0,r}\leq 0$ in \eqref{eq:f}.
Then, there exists a constant $c$ such that the spectral and Frobenius norms of $\Sigma_{rr}$ can be bounded as
\begin{align} 
\Vert \Sigma_{rr} \Vert &\leq c \vertiii{G}	\label{eq:specnormSigmashort}, \\
\Vert \Sigma_{rr} \Vert_{F} &\leq c \vertiii{G} N^{\frac{1}{2}}	\label{eq:frobnormSigmashort}.
\end{align}
\end{lemma}

\begin{proof}
See Lemma C.5. in \cite{Sun2018:LargeSpectral} for the inequality \eqref{eq:specnormSigmashort}. The second inequality \eqref{eq:frobnormSigmashort} is a simple consequence of using the fact that $\| A \|_{F} \leq \sqrt{\rank(A)} \| A \|$ for a matrix $A$ and applying \eqref{eq:specnormSigmashort}.
\end{proof}

\begin{lemma} \label{le:ineqSpecFrobSigma}
Let $\Sigma_{rr} = (\Sigma_{rr}(n-k))_{n,k=1,\dots,N} = \E (\mathcal{X} e_{r} (\mathcal{X} e_{r})')$ and $d_{0,r}>0$ in \eqref{eq:f}.
Then, there exist constants $c_{1}, c_{2}$ such that the spectral and Frobenius norms of $\Sigma_{rr}$ can be bounded as
\begin{align} 
\Vert \Sigma_{rr} \Vert &\leq c_{1} \vertiii{G} N^{2d_{0,r}} \log(N),	\label{eq:specnormSigma}
\\
\Vert \Sigma_{rr} \Vert_{F} &\leq 
c_{2} \vertiii{G}
\begin{cases}
N^{\frac{1}{2}} \log(N)^{\frac{1}{2}}, \hspace{0.2cm}  &\text{ if } d_{0,r} \leq \frac{1}{4}, \\
N^{2 d_{0,r}} \log(N)^{\frac{1}{2}}, \hspace{0.2cm}  &\text{ if } d_{0,r} > \frac{1}{4}. \\
\end{cases}
\label{eq:frobnormSigma}
\end{align}
\end{lemma}

\begin{remark} \label{re:discussionlog}
We pause here to comment on Lemma \ref{le:ineqSpecFrobSigma}. In the light of asymptotic results on the autocovariances under long-range dependence, it might be surprising to the reader to find the logarithm in all our bounds opposed to only in the case $d_{0,r} = \frac{1}{4}$. This could certainly be avoided by bounding certain sums in our proofs by integrals depending on the memory parameter $d_{0,r}$. Then, one can replace the logarithm respectively by $\frac{1}{d_{0,r}}$ for the bound in \eqref{eq:specnormSigma} and $\frac{1}{\sqrt{4d_{0,r}-1}}\mathds{1}_{\{d_{0,r} \neq \frac{1}{4}\}}$ in \eqref{eq:frobnormSigma}. However, the memory parameter $d_{0,r}$ approaching either zero or $\frac{1}{4}$ results in potentially high constants. In contrast, the logarithm can be controlled in our high probability bounds.
\end{remark}

\begin{proof}[Proof of Lemma \ref{le:ineqSpecFrobSigma}]
The autocovariance matrix $\Sigma_{rr}=(\Sigma_{rr}(n-k))_{n,k=1,\dots,N}$ of the component series $\{X_{r,n}\}_{n \in \ZZ}$ can be represented in terms of the spectral density $f_{rr}(\omega)$ of the respective component series as
\begin{equation} \label{eq:specfrob1}
\begin{aligned}
\Sigma_{rr}
&=
\frac{1}{2 \pi} \int^{\pi}_{-\pi} f_{rr}(\omega) (e^{ i (n-k) \omega } )_{n,k=1,\dots,N} d \omega \\
&=
\frac{1}{2 \pi} \int^{\pi}_{-\pi} |\omega|^{-2d_{0,r}} G_{rr}(\omega) ( e^{ i (n-k) \omega } )_{n,k=1,\dots,N} d \omega.
\end{aligned}
\end{equation}
In the following, we prove the statements for the spectral and the Frobenius norm separately. 

\textit{Spectral norm:}
Using \eqref{eq:specfrob1}, we can bound the spectral norm of $\Sigma_{rr}$ as
\begin{align}
\Vert \Sigma_{rr} \Vert
&\leq
\esssup_{ \omega \in (-\pi,\pi)} |G_{rr}(\omega)|
\left\lVert
\frac{1}{2 \pi} \int^{\pi}_{-\pi} |\omega|^{-2d_{0,r}} (e^{ i (n-k) \omega })_{n,k=1,\dots,N} d \omega \right\rVert \label{eq:Sigmaeq1} \\
&\leq
\vertiii{G}
\left\lVert
\frac{1}{2 \pi} \int^{\pi}_{-\pi} | \exp( i \omega ) -1|^{-2d_{0,r}} (e^{ i (n-k) \omega })_{n,k=1,\dots,N} d \omega \right\rVert
\label{eq:Sigmaeq3},
\end{align}
where \eqref{eq:Sigmaeq1} follows since $(e^{ i (n-k) \omega })_{n,k=1,\dots,N}$ is positive semidefinite and because of $|e^{ix}-e^{iy}| \leq |x-y|$. 
In order to bound the spectral norm in \eqref{eq:Sigmaeq3}, we replace the matrix by an integral operator with piecewise constant kernel.
More specifically, define the integral operator $K_{k}: L^{2}(0,1) \to L^{2}(0,1)$ as
\begin{equation} \label{eq:integralop}
(K_{k}g)(x)=\int_{0}^{1} g(y) k_{N}(x,y) dy
\end{equation}
with kernel function 
\begin{equation*}
%\label{eq:kernelfct}
k_{N}(x,y) = \frac{1}{2 \pi} \int^{\pi}_{-\pi} | e^{ i \omega } -1|^{-2d_{0,r}} e^{ i ([Nx]-[Ny]) \omega } d \omega.
\end{equation*}
Then,
\begin{equation} \label{eq:offdiagnorm}
\left\lVert
\frac{1}{2 \pi} \int^{\pi}_{-\pi} | e^{ i \omega } -1|^{-2d_{0,r}} ( e^{ i (n-k) \omega } )_{n,k=1,\dots,N} d \omega \right\rVert
= 
N \Vert K_{k} \Vert_{op} 
\leq
c N^{2d_{0,r}} \log(N);
\end{equation}
see Lemma 4.1 in \cite{bottcher2010:weighted} for the equality in \eqref{eq:offdiagnorm}. The inequality stated in \eqref{eq:offdiagnorm} is proved in Lemma \ref{le:boundspecnormintegralop}.
Combining \eqref{eq:Sigmaeq3} and \eqref{eq:offdiagnorm} yields \eqref{eq:specnormSigma}.
\par
%We denote further $A_{N}(x,y) = \{[Nx]-[Ny] \neq 0 \} \cap \{d_{0,r} > \frac{1}{4}\}$ and its complement $A^{c}_{N}(x,y)$.
\textit{Frobenius norm:}
We deal with the Frobenius norm similarly as with the spectral norm as
\begin{align}
\Vert \Sigma_{rr} \Vert_{F}
&=
\Big(\frac{1}{(2 \pi)^2} \int^{\pi}_{-\pi} \int^{\pi}_{-\pi} G_{rr}(\omega_{1}) G_{rr}(\omega_{2}) |\omega_{1}\omega_{2}|^{-2d_{0,r}} 
\sum_{n,k =1}^{N} e^{ i (n-k) (\omega_{1}-\omega_{2}) } d \omega_{1} d\omega_{2} \Big)^{\frac{1}{2}} \nonumber \\
&=
\Big(\frac{1}{(2 \pi)^2} \int^{\pi}_{-\pi} \int^{\pi}_{-\pi} G_{rr}(\omega_{1}) G_{rr}(\omega_{2}) |\omega_{1}\omega_{2}|^{-2d_{0,r}} 
| \sum_{n=1}^{N} e^{ i n (\omega_{1}-\omega_{2}) } |^2 d \omega_{1} d\omega_{2}\Big)^{\frac{1}{2}} \nonumber \\
&\leq
\esssup_{ \omega \in (-\pi,\pi)} |G_{rr}(\omega)|
\Big(\frac{1}{(2 \pi)^2} \int^{\pi}_{-\pi} \int^{\pi}_{-\pi} |\omega_{1}\omega_{2}|^{-2d_{0,r}} 
| \sum_{n=1}^{N} e^{ i n (\omega_{1}-\omega_{2}) } |^2 d \omega_{1} d\omega_{2}\Big)^{\frac{1}{2}} \nonumber \\
&\leq
\vertiii{G}
\Big(\frac{1}{(2 \pi)^2} \int^{\pi}_{-\pi} \int^{\pi}_{-\pi} | \exp( i \omega_{1} ) -1|^{-2d_{0,r}} | \exp( i \omega_{2} ) -1|^{-2d_{0,r}}
| \sum_{n=1}^{N} e^{ i n (\omega_{1}-\omega_{2}) } |^2 d \omega_{1} d\omega_{2}\Big)^{\frac{1}{2}} \label{eq:specSigmaeq2} \\
&=
\vertiii{G}
\left\lVert
\frac{1}{2 \pi} \int^{\pi}_{-\pi} | \exp( i \omega ) -1|^{-2d_{0,r}} ( e^{ i (n-k) \omega } )_{n,k=1,\dots,N} d \omega \right\rVert_{F}
\label{eq:specSigmaeq3},
\end{align}
where \eqref{eq:specSigmaeq2} follows since $|e^{ix}-e^{iy}| \leq |x-y|$. Again, we apply Lemma 4.1 in \cite{bottcher2010:weighted} to replace the matrix in \eqref{eq:specSigmaeq3} by an integral operator
\begin{equation} \label{eq:frobnormkernelapprox}
\left\lVert
\frac{1}{2 \pi} \int^{\pi}_{-\pi} | \exp( i \omega ) -1|^{-2d_{0,r}} (\exp( i (n-k) \omega ))_{n,k=1,\dots,N} d \omega \right\rVert_{F}
= 
N \Vert K_{k} \Vert_{2}
\end{equation}
with integral operator $K_{k}$ as in \eqref{eq:integralop}. The equality follows by Lemma 4.1 in \cite{bottcher2010:weighted}.
Combining \eqref{eq:specSigmaeq3} and \eqref{eq:frobnormkernelapprox} with Lemma \ref{le:boundL2normintegralop} below gives
\begin{equation*}
\begin{aligned}
\Vert \Sigma_{rr} \Vert_{F} 
\leq
c
\begin{cases}
N^{\frac{1}{2}} \log(N)^{\frac{1}{2}}, \hspace{0.2cm}  &\text{ if } d_{0,r} \leq \frac{1}{4}, \\
N^{2 d_{0,r}} \log(N)^{\frac{1}{2}}, \hspace{0.2cm}  &\text{ if } d_{0,r} > \frac{1}{4}. \\
\end{cases}
\end{aligned}
\end{equation*}
\end{proof}

\begin{lemma} \label{le:boundspecnormintegralop}
For $d_{0,r}>0$, the operator norm of the integral operator $K_{k}$ defined in \eqref{eq:integralop} can be bounded as
\begin{equation}
N \Vert K_{k} \Vert_{op} 
\leq
c N^{2d_{0,r}} \log(N).
\end{equation}
\end{lemma}

\begin{proof}
In view of the definition of the integral operator norm (see the end of Section \ref{se:intro}), 
\begin{align}
\int_{0}^{1} \left| \int_{0}^{1} g(y) k_{N}(x,y) dy \right|^{2} dx \nonumber
&\leq
\int_{0}^{1} \left( \int_{0}^{1} |g(y)| |k_{N}(x,y)|^{\frac{1}{2}} |k_{N}(x,y)|^{\frac{1}{2}} dy \right)^{2} dx \nonumber
\\&\leq
\int_{0}^{1} \int_{0}^{1} |g(y)|^{2} |k_{N}(x,y)| dy \int_{0}^{1} |k_{N}(x,y)| dy dx \nonumber
\\&\leq
\sup_{x \in (0,1)}  \int_{0}^{1} |k_{N}(x,y)| dy \int_{0}^{1} \int_{0}^{1} |g(y)|^{2} |k_{N}(x,y)| dy dx \nonumber
\\&\leq
c N^{2d_{0,r}-1} \log(N) \sup_{y \in (0,1)}  \int_{0}^{1} |k_{N}(x,y)| dx
\int_{0}^{1} |g(y)|^{2} dy \label{eq:fff1}
\\&\leq
c (N^{2d_{0,r}-1} \log(N))^{2} \int_{0}^{1} |g(y)|^{2} dy \label{eq:fff},
\end{align}
where \eqref{eq:fff1} and \eqref{eq:fff} follow since, with explanations given below,
\begin{equation} \label{eq:wwwww}
\sup_{x \in (0,1)} \int_{0}^{1} |k_{N}(x,y)| dy
=
\sup_{y \in (0,1)} \int_{0}^{1} |k_{N}(x,y)| dx
\leq c N^{2d_{0,r}-1} (\log(N) + 1).
\end{equation}
In order to prove \eqref{eq:wwwww}, we consider $z_{N} \neq 0$ and $z_{N} = 0$ separately, with $z_{N} = [Nx]-[Ny]$. By Lemma \ref{le:approxL2}, we get 
\begin{align*}
\int_{0}^{1} | k_{N}(x,y) | \mathds{1}_{\{z_{N} \neq 0\}} dx
&\leq c
\int_{0}^{1} |[Nx]-[Ny]|^{2d_{0,r}-1} \mathds{1}_{\{[Nx]-[Ny] \neq 0\}} dx
\\&= c
\sum_{i=1}^{N} \int_{\frac{i-1}{N}}^{\frac{i}{N}} |i-1-[Ny]|^{2d_{0,r}-1} \mathds{1}_{\{i-1-[Ny] \neq 0\}} dx
\\&= c
\frac{1}{N} \sum_{i=1}^{N} |i-1-[Ny]|^{2d_{0,r}-1} \mathds{1}_{\{i-1-[Ny] \neq 0\}} 
\\&\leq c
\frac{1}{N} \sum_{i=1}^{N} \left| i-1-\Big[\frac{N-1}{2}\Big] \right|^{2d_{0,r}-1} \mathds{1}_{\{i-1-[\frac{N-1}{2}] \neq 0\}} 
\\& = c
\frac{1}{N} \sum_{i=-[\frac{N-1}{2}]}^{N-1-[\frac{N-1}{2}]} \left| i \right|^{2d_{0,r}-1} \mathds{1}_{\{i \neq 0\}} 
\\&\leq c
N^{2d_{0,r}-1} \frac{1}{N} \sum_{i=-[\frac{N-1}{2}]}^{N-1-[\frac{N-1}{2}]} \left| \frac{i}{N} \right|^{2d_{0,r}-1} \mathds{1}_{\{i \neq 0\}} 
\\&\leq c
N^{2d_{0,r}-1} \sum_{i=-[\frac{N-1}{2}]}^{N-1-[\frac{N-1}{2}]} |i|^{-1} \mathds{1}_{\{i \neq 0\}} 
\\&\leq c
N^{2d_{0,r}-1} \sum_{i=1}^{N} i^{-1} 
\leq
c N^{2d_{0,r}-1} (\log(N) + 1).
\end{align*}
Using \eqref{eq:ineqx=y} in Lemma \ref{le:approxL2}, for $z_{N}=0$, 
\begin{equation} \label{eq:int_equal}
\begin{aligned}
\int_{0}^{1} | k_{N}(x,y) |^{2} \mathds{1}_{\{z_{N} = 0\}} dx
&\leq
\left( \frac{\Gamma(1-2\Delta_{2})}{\Gamma^{2}(1-\Delta_{2})} \right)^{2}  \int_{0}^{1} \mathds{1}_{\{[Nx]-[Ny] = 0\}} dx
\\&=
\left( \frac{\Gamma(1-2\Delta_{2})}{\Gamma^{2}(1-\Delta_{2})} \right)^{2}  \int_{\frac{[Ny]}{N}}^{\frac{[Ny]+1}{N}} dx
\\&=
\left( \frac{\Gamma(1-2\Delta_{2})}{\Gamma^{2}(1-\Delta_{2})} \right)^{2}  \frac{1}{N}.
\end{aligned}
\end{equation}
\end{proof}

\begin{lemma} \label{le:boundL2normintegralop}
For $d_{0,r} > 0$, the $L^{2}$ norm of the integral operator $K_{k}$ defined in \eqref{eq:integralop} can be bounded as
\begin{equation}
N \Vert K_{k} \Vert_{2} 
\leq
c
\begin{cases}
N^{\frac{1}{2}} \log(N)^{\frac{1}{2}}, \hspace{0.2cm}  &\text{ if } d_{0,r} \leq \frac{1}{4}, \\
N^{2 d_{0,r}} \log(N)^{\frac{1}{2}}, \hspace{0.2cm}  &\text{ if } d_{0,r} > \frac{1}{4}. \\
\end{cases}
\end{equation}
\end{lemma}
%\frac{1}{\sqrt{4d_{0,r}-1}}

\begin{proof}
For $z_{N} = [Nx]-[Ny]$, we consider $z_{N} \neq 0$ and $z_{N} = 0$ separately. By Lemma \ref{le:approxL2}, we get 
\begin{align*}
&
\int_{0}^{1} \int_{0}^{1} | k_{N}(x,y) |^{2} \mathds{1}_{\{z_{N} \neq 0\}} dx dy
\\&\leq c
\int_{0}^{1} \int_{0}^{1} |[Nx]-[Ny]|^{4d_{0,r}-2} \mathds{1}_{\{[Nx]-[Ny] \neq 0\}} dx dy
\\&= c
\sum_{i=1}^{N} \sum_{j=1}^{N} \int_{\frac{i-1}{N}}^{\frac{i}{N}} \int_{\frac{j-1}{N}}^{\frac{j}{N}} |[Nx]-[Ny]|^{4d_{0,r}-2} \mathds{1}_{\{[Nx]-[Ny] \neq 0\}} dx dy
\\&= c
\frac{1}{N^{2}}\sum_{i=1}^{N} \sum_{j=1}^{N} | i-j|^{4d_{0,r}-2} \mathds{1}_{\{ i \neq j\}}
\\&=
c\frac{1}{N^{2}} \sum_{k=1}^{N} (N-k) k^{4d_{0,r}-2}
\\&\leq
c
\begin{cases}
 \frac{1}{N} \sum_{k=1}^{N} \left(1-\frac{k}{N}\right) k^{-1}, \hspace{0.2cm}  &\text{ if } d_{0,r} \leq \frac{1}{4}, \\
N^{4 d_{0,r}-2} \left(\frac{1}{N} \sum_{k=1}^{N} \left(\frac{k}{N}\right)^{4 d_{0,r}-2} - \frac{1}{N} \sum_{k=1}^{N} \left(\frac{k}{N}\right)^{4 d_{0,r}-1}\right) , \hspace{0.2cm}  &\text{ if } d_{0,r} > \frac{1}{4},
\end{cases}
\\&\leq
c
\begin{cases}
\frac{1}{N} \sum_{k=1}^{N} k^{-1}, \hspace{0.2cm}  &\text{ if } d_{0,r} \leq \frac{1}{4}, \\
N^{4 d_{0,r}-2} \sum_{k=1}^{N} k^{-1} , \hspace{0.2cm}  &\text{ if } d_{0,r} > \frac{1}{4},
\end{cases}
\\&\leq
c
\begin{cases}
\frac{1}{N} (\log(N)+1), \hspace{0.2cm}  &\text{ if } d_{0,r} \leq \frac{1}{4}, \\
N^{4 d_{0,r}-2} (\log(N)+1), \hspace{0.2cm}  &\text{ if } d_{0,r} > \frac{1}{4}.
\end{cases}
\end{align*}
Using \eqref{eq:ineqx=y} in Lemma \ref{le:approxL2} and the same arguments for $z_{N}=0$, 
\begin{align*}
\int_{0}^{1} \int_{0}^{1} | k_{N}(x,y) |^{2} \mathds{1}_{\{z_{N} = 0\}} dx dy
&\leq
\left( \frac{\Gamma(1-2\Delta_{2})}{\Gamma^{2}(1-\Delta_{2})} \right)^{2} \int_{0}^{1} \int_{0}^{1} \mathds{1}_{\{[Nx]-[Ny] = 0\}} dx dy
\\&=
\left( \frac{\Gamma(1-2\Delta_{2})}{\Gamma^{2}(1-\Delta_{2})} \right)^{2} \frac{1}{N}
\end{align*}
following \eqref{eq:int_equal}.
\end{proof}

\begin{lemma} \label{le:approxL2}
For $d_{0,r} > 0$, the function 
\begin{equation} \label{eq:kernelfct}
k_{N}(x,y) = \frac{1}{2 \pi} \int^{\pi}_{-\pi} | e^{ i \omega } -1|^{-2d_{0,r}} e^{ i ([Nx]-[Ny]) \omega } d \omega
\end{equation}
%evaluated at $z_{N}=[Nx]-[Ny]$ 
can be bounded as
\begin{equation} \label{eq:kernelappr}
0 < k_{N}(x,y) \leq c |z_{N}|^{2d_{0,r}-1} 
\hspace{0.2cm}
\text{ with }
\hspace{0.2cm}
z_{N}=[Nx]-[Ny].
\end{equation}
Furthermore,
\begin{equation} \label{eq:ineqx=y}
\frac{1}{2 \pi} \int^{\pi}_{-\pi} | e^{ i \omega } -1|^{-2d_{0,r}} d \omega
\leq
\frac{\Gamma(1-2\Delta_{2})}{\Gamma^{2}(1-\Delta_{2})}.
\end{equation}
%\label{eq:kerneltilde}
\end{lemma}

\begin{proof}
Set $z_{N}=[Nx]-[Ny]$. Then, the kernel \eqref{eq:kernelfct} satisfies
\begin{align}
k_{N}(x,y) 
&= 
\frac{1}{2 \pi} \int^{\pi}_{-\pi} | e^{ i \omega } -1|^{-2d_{0,r}} e^{ i ([Nx]-[Ny]) \omega } d \omega \nonumber \\
&=
\frac{1}{2 \pi} \int^{\pi}_{-\pi} | \exp( i \omega ) -1|^{-2d_{0,r}} \exp( i z_{N} \omega ) d \omega \nonumber \\
&=
(-1)^{|z_{N}|} \frac{\Gamma(1-2d_{0,r})}{\Gamma(1-|z_{N}|-d_{0,r})\Gamma(1+|z_{N}|-d_{0,r})} \label{eq:kernelappr1}\\
&=
\Gamma(1-2d_{0,r}) \frac{\sin(\pi d_{0,r})}{\pi(|z_{N}|+d_{0,r})} \frac{\Gamma(|z_{N}|+1+d_{0,r})}{\Gamma(|z_{N}|+1-d_{0,r})}  \label{eq:kernelappr2}.
\end{align}
See p. 665 in \cite{bottcher2007norms} for equations \eqref{eq:kernelappr1} and \eqref{eq:kernelappr2}. The last relation shows $k_{N}(x,y)  > 0$. We further have
\begin{align}
k_{N}(x,y) 
&<
\Gamma(1-2d_{0,r}) \frac{\sin(\pi d_{0,r})}{\pi(|z_{N}|+d_{0,r})}  \frac{(|z_{N}|+d_{0,r})^{2d_{0,r}}}{d_{0,r}^{2d_{0,r}}} \frac{\Gamma(d_{0,r}+1)}{\Gamma(d_{0,r}+1-2d_{0,r})} \label{eq:kernelappr3} \\
&\leq
c (|z_{N}|+d_{0,r})^{2d_{0,r}-1} \Big(\frac{d_{0,r}+1-2d_{0,r}}{d_{0,r}}\Big)^{2d_{0,r}}  \label{eq:kernelappr4} \\
&\leq
c |z_{N}|^{2d_{0,r}-1} 2.
\end{align}
%with $| \varepsilon_{N} | = | \varepsilon_{N}(x,y) | \leq 1$
The inequality \eqref{eq:kernelappr3} follows by using
\begin{equation} \label{eq:Qi1}
\frac{\Gamma(x+a)}{\Gamma(x+b)} < \frac{x^{a-b}}{x_{0}^{a-b}} \frac{\Gamma(x_{0}+a)}{\Gamma(x_{0}+b)}
\end{equation}
for $a>b \geq 0$ and $a+b \geq 1$ on $[x_{0},\infty)$ for any $x_{0}>0$; see (3.75) in \cite{Qi2012:Bounds2}. The inequality \eqref{eq:Qi1} is applied with $x=|z_{N}|+d_{0,r}$, $x_{0}=d_{0,r}$, $a=1$, $b=1-d_{0,r}$.
Furthermore, \eqref{eq:kernelappr4} follows from the inequality  
\begin{equation} \label{eq:Qi2}
\frac{\Gamma(x+a)}{x^{a}\Gamma(x)} \leq 1
\end{equation}
for $x>0$ and $a \in (0,1)$; see (2.2) in \cite{Qi2012:Bounds2}. The inequality \eqref{eq:Qi2} is applied with $x = 1-d_{0,r}$ and $a = 2d_{0,r}$.

For the integral in \eqref{eq:ineqx=y}, we get
\begin{equation*}
\frac{1}{2 \pi} \int^{\pi}_{-\pi} | \exp( i \omega ) -1|^{-2d_{0,r}} d \omega
=
\frac{\Gamma(1-2d_{0,r})}{\Gamma^{2}(1-d_{0,r})}
=: g(d_{0.r})
\leq
\frac{\Gamma(1-2\Delta_{2})}{\Gamma^{2}(1-\Delta_{2})},
\end{equation*}
where the equality follows by p. 665 in \cite{bottcher2007norms} as in \eqref{eq:kernelappr1} and the inequality since the function $g$ can be checked to be monotonically increasing in $d_{0,r}$.
\end{proof}

\section{Some additional technical results}
\label{s:appSTR}

This appendix concerns three different kinds of technical results, required to prove our main results. Section \ref{s:uniformCI} presents a uniform concentration inequality.
The two remaining Sections \ref{s:matrixnormine} and \ref{se:Dirichlet} give some inequalities on matrix norms and the Dirichlet kernel, respectively.

\subsection{Uniform concentration inequality}
\label{s:uniformCI}

For completeness, we give a slightly modified version of Theorem 1 in \cite{dicker2017} and explain the differences from the original formulation. 
\begin{theorem} \label{th:Dicker_Th1}
Let $0<R<\infty$ and $t_{1}(u),\dots , t_{m}(u)$ be real-valued on $[0,R]^{K} \subseteq \RR^{K}$ and differentiable on $(0,R)^{K} \subseteq \RR^{K}$ with bounded derivative.
%, satisfying 
%\begin{equation} \label{eq:con1Th1Dicker}
%\max_{i=1,\dots,m} |t_{i}(u_{1}) - t_{i}(u_{2})| \leq L \|u_{1}-u_{2} \| \text{ for all } u_{1}, u_{2} \in [0,R]^{K}
%\end{equation}
%for some constant $0 < L < \infty$. 
Define $T(u)=\diag(t_{1}(u),\dots,t_{m}(u))$ and $Q(u)=VT(u)V'$, where $V$ is a $p \times m$ matrix. Let 
$\varepsilon=(\varepsilon_{1}, \dots , \varepsilon_{N})'$, where $\varepsilon_{1}, \dots , \varepsilon_{N}$ are independent mean $0$ sub-Gaussian random variables satisfying
\begin{equation} \label{eq:ACsubGgamma}
\max_{i=1,\dots,N} \| \varepsilon_{i} \|_{\phi} \leq \gamma
\end{equation}
for some constant $\gamma \in (0, \infty)$. Then, there exists a constant $C \in (0, \infty)$ such that
\begin{equation} \label{eq:thDicker_concineq}
\Prob \Big( \sup_{u \in [0,R]^{K}} | \varepsilon' Q(u) \varepsilon - \E (\varepsilon' Q(u) \varepsilon)  | > \nu \Big)
\leq
C \exp\Bigg( - \frac{1}{C} \min 
\Bigg\{ 
\frac{ \nu }{ \gamma^2 \mathcal{T} },
\frac{ \nu^2 }{ \gamma^4 \mathcal{T}_{i}^2 }
\Bigg\} \Bigg), 
\end{equation} 
for $\nu^2 \geq C \gamma^4 \mathcal{T}^2_{i} K^2$, $i=1,2$, where
%$\nu^2 \geq C \gamma^4 \| V'V \|^2 K^3 R^2 L_{i}$, where
\begin{equation*}
%\label{eq:tautau1tau2}
\mathcal{T}=\| V'V \| L_{1},
\hspace{0.2cm}
\mathcal{T}_{1}=\| V'V \| L_{2},
\hspace{0.2cm}
\mathcal{T}_{2}=\| V'V \|_{F} L_{1}
\end{equation*}
with
\begin{equation} \label{eq:L1L2_in_thC1}
L_{1} = \| T(0) \|+RK^{\frac{1}{2}} \sup_{u \in [0,R]^{K}} \| \bm{T}(u) \|
\hspace{0.2cm}
\text{ and }
\hspace{0.2cm}
L_{2} = \| T(0) \|_{F}+RK^{\frac{1}{2}} \sup_{u \in [0,R]^{K}} \| \bm{T}(u) \|_{F},
\end{equation}
where $\bm{T}(u) = \diag( \| \nabla t_{1}(u) \|_{F}, \dots, \| \nabla t_{m}(u) \|_{F}) $.
\end{theorem}

The statement in Theorem \ref{th:Dicker_Th1} differs in two points from the original Theorem 1 stated in \cite{dicker2017}.
First, the condition that the functions $t_{i}(u), i=1 \dots, m$, are differentiable with bounded derivatives replaces a Lipschitz condition on the functions $t_{i}(u), i=1 \dots, m$; see (1) in \cite{dicker2017}. This assumption makes it more convenient to write our results.
Second, \cite{dicker2017} get $\mathcal{T}_{1}$ in the bound, which uses the fact that $\| Q(u) \|_{F} = \| VT(u)V' \|_{F} \leq \| VV' \| \| T(u) \|_{F}$; see Lemma \ref{le:normineq1}. In some situations, it turns out to be helpful to consider $\mathcal{T}_{2}$, which is a consequence of the inequality
\begin{equation*}
\| Q(u) \|_{F} = \| VT(u)V' \|_{F} \leq \| VV' \|_{F} \| T(u) \|;
\end{equation*}
see again Lemma \ref{le:normineq1}.
Following the proof of Theorem 1 in \cite{dicker2017}, those changes yield the concentration inequality in \eqref{eq:thDicker_concineq} with $i=2$.

We conclude with a remark which comments on a possibility to slightly generalize the results in Theorem \ref{th:Dicker_Th1}.

\begin{remark} \label{re:Th1_diag}
Theorem \ref{th:Dicker_Th1} is stated in terms of the matrix $Q(u)=VT(u)V'$ under the assumption that $T(u)$ is a diagonal matrix. The proof of Theorem \ref{th:Dicker_Th1} relies on a chaining technique with a subsequent application of the Hanson-Wright inequality \cite[Theorem 1.1]{rudelson2013hanson}. The Hanson-Wright inequality is applicable for arbitrary matrices. In particular, it does not require symmetricity or diagonality. However, in order to bound the spectral and Frobenius norms of $T(u) -T(u')$, $u,u' \in [0,R]^K$ in the proof of Theorem \ref{th:Dicker_Th1} in terms of the quantities in \eqref{eq:L1L2_in_thC1}, we need to impose some structural assumptions. As stated in Theorem \ref{th:Dicker_Th1}, \cite{dicker2017} required $T(u)$ to be diagonal. This assumption can be slightly relaxed by supposing that there exists a representation $T(u) = A' \widetilde{T}(u) B$ with $A,B \in \RR^{m \times m}$ unitary and independent of $u$ and $\widetilde{T}(u)$ diagonal since the Frobenius norm and the spectral norm are invariant under unitary transformation.
\end{remark}

\subsection{Matrix norm inequalities}
\label{s:matrixnormine}

For completeness, we collect here some results on matrix norms.

\begin{lemma} \label{le:normineq1}
Let $B$ be a positive semidefinite matrix. Then,
\begin{equation*}
\| ABA'\|_{F}
\leq \| B \| \| A'A\|_{F}
\hspace{0.2cm}
\text{ and }
\hspace{0.2cm}
\| ABA'\|_{F}
\leq \| A'A \| \| B \|_{F}.
\end{equation*}
\end{lemma}

\begin{proof}
The Frobenius norm can be rewritten as
\begin{equation*}
\begin{aligned}
\| ABA'\|_{F}^2
&= \tr( (ABA')' ABA')
= \tr( BA' ABA' A)
\\&
\leq \lambda_{\max}(B) \tr( BA' A A' A)
\\&
\leq \lambda_{\max}(B)^2 \| A'A \|_{F}^2
= \| B \|^2 \| A'A \|^{2}_{F};
\end{aligned}
\end{equation*}
see Theorem 1 in \cite{Fang} for the eigenvalue-trace inequality in the second line.
Similarly,
\begin{equation*}
\begin{aligned}
\| ABA'\|_{F}^2
&= \tr( (ABA')' ABA')
= \tr( A' ABA' A B)
\\&
\leq \lambda_{\max}(A' A) \tr( A' A B^2)
\\&
\leq \lambda_{\max}(A' A)^2 \| B \|_{F}^2
= \| A'A \|^2 \| B\|_{F}^2.
\end{aligned}
\end{equation*}
\end{proof}

\begin{lemma} \label{le:normineq2}
Let $M$ be a positive semidefinite block-matrix
\begin{equation*}
M=
\begin{pmatrix}
A & X\\
X' & B
\end{pmatrix}.
\end{equation*}
Then,
\begin{equation*}
\| M \| \leq \| A \| + \| B \|
\hspace{0.2cm}
\text{ and }
\hspace{0.2cm}
\| M \|_{F} \leq \| A \|_{F} + \| B \|_{F}.
\end{equation*}
\end{lemma}

\begin{proof}
The inequalities are consequences of Lemma 1.1 in \cite{BOURIN20121906} and the unitary invariance of
the spectral and the Frobenius norm.
\end{proof}

\subsection{Dirichlet kernel} \label{se:Dirichlet}
In this section, we present some results regarding the Dirichlet kernel in \eqref{eq:def:KandD}.
%Using the definition of the Dirichlet kernel in \eqref{eq:def:KandD}, inequality \eqref{eq:det_int11} can be proved by
\begin{lemma} \label{eq:integral_Dirichlet}
The Dirichlet kernel in \eqref{eq:def:KandD} satisfies
\begin{equation*}
\int_{\frac{\lambda_{j}}{2}}^{2\lambda_{j}} | D_{N}(\lambda - \lambda_{j}) | d\lambda
\leq
4 (\pi +2)
+ 4 \log(j).
\end{equation*}
\end{lemma}

\begin{proof}
Using the representation \eqref{eq:def:KandD} of the Dirichlet kernel, a series of inequalities lead\ to
\begin{align*}
\int_{\frac{\lambda_{j}}{2}}^{2\lambda_{j}} | D_{N}(\lambda - \lambda_{j}) | d\lambda
&=
\int_{\frac{\lambda_{j}}{2}}^{2\lambda_{j}} \left| \frac{ \sin(N(\lambda - \lambda_{j})/2) }{\sin((\lambda - \lambda_{j})/2)} \right| d\lambda
=
\int_{-\frac{\lambda_{j}}{2}}^{\lambda_{j}} \left| \frac{ \sin(N \lambda /2) }{\sin(\lambda/2)} \right| d\lambda
\\&=
2\int_{0}^{\lambda_{j}} \left| \frac{ \lambda }{\sin(\lambda/2)} \frac{ \sin(N \lambda /2) }{\lambda} \right| d\lambda
\leq
2 \frac{ \lambda_{j} }{\sin(\lambda_{j}/2)} \int_{0}^{\lambda_{j}} \left| \frac{ \sin(N \lambda /2) }{\lambda} \right| d\lambda
\\&=
2 \frac{ \lambda_{j} }{\sin(\lambda_{j}/2)} 
\sum_{i=0}^{j-1} \int_{\lambda_{i}}^{\lambda_{i+1}} \left| \frac{ \sin(N \lambda /2) }{\lambda} \right| d\lambda
\leq
2 \pi
\sum_{i=0}^{j-1} \int_{0}^{\lambda_{1}} \left| \frac{ \sin(N (\lambda + \lambda_{i}) /2) }{\lambda+\lambda_{i}} \right| d\lambda
\\&=
2 \pi
\Big( \int_{0}^{\lambda_{1}} \left| \frac{ \sin(N \lambda /2) }{\lambda} \right| d\lambda
+
\sum_{i=1}^{j-1} \int_{0}^{\lambda_{1}} \left| \frac{ \sin(N (\lambda + \lambda_{i}) /2) }{\lambda+\lambda_{i}} \right| d\lambda \Big)
\\&=
2 \pi
\Big( \int_{0}^{\pi} \left| \frac{ \sin(\lambda) }{\lambda} \right| d\lambda
+
\int_{0}^{\lambda_{1}} \sin(N \lambda /2) \sum_{i=1}^{j-1} \frac{ 1 }{\lambda+\lambda_{i}} d\lambda \Big)
\\&\leq
4 \pi
+ 2\pi
\int_{0}^{\lambda_{1}} \sin(N \lambda /2) d\lambda \sum_{i=1}^{j-1} \frac{ 1 }{\lambda_{i}} 
=
4 \pi
+ 2\pi \frac{4}{N} \sum_{i=1}^{j-1} \frac{ 1 }{\lambda_{i}} 
\\&\leq
4 (\pi +1)
+ 4 \sum_{i=2}^{j} \frac{ 1 }{i}
\leq
4 (\pi +1)
+ 4 \int_{1}^{j} \frac{ 1 }{x} dx
\\&\leq
4 (\pi +1)
+ 4 (\log(j) +1 )
\\&=
4 (\pi +2)
+ 4 \log(j).
\end{align*}
\end{proof}

\begin{lemma} \label{eq:sup_Dirichlet}
The Dirichlet kernel \eqref{eq:def:KandD} satisfies
\begin{equation} \label{eq:Dirichlet_relations}
\sup_{\frac{\lambda_{j}}{2} \leq \lambda \leq \pi} \left( \frac{|D_{N}(\lambda + \lambda_{j})|}{|D_{N}(\lambda)|}\right)^2 \leq (\cos(\lambda_{m}/2))^{-2}
\hspace{0.2cm}
\text{ and }
\hspace{0.2cm}
\sup_{2\lambda_{j} \leq \lambda \leq \pi} \left( \frac{|D_{N}(\lambda - \lambda_{j})|}{|D_{N}(\lambda)|}\right)^2 \leq 4.
\end{equation}
\end{lemma}

\begin{proof}
Though very similar, we prove the two inequalities in \eqref{eq:Dirichlet_relations} separately. Using the representation \eqref{eq:def:KandD} of the Dirichlet kernel leads to
\begin{align}
\sup_{\frac{\lambda_{j}}{2} \leq \lambda \leq \pi} \left( \frac{|D_{N}(\lambda + \lambda_{j})|}{|D_{N}(\lambda)|}\right)^2
&=
\sup_{\frac{\lambda_{j}}{2} \leq \lambda \leq \pi} \left( \frac{| \sin(N(\lambda + \lambda_{j})/2) \sin(\lambda/2) |}{|\sin(N\lambda/2)\sin((\lambda + \lambda_{j})/2)|}\right)^2 \nonumber
\\&=
\sup_{\frac{\lambda_{j}}{2} \leq \lambda \leq \pi} \left( \frac{| \sin(N \lambda/2) \sin(\lambda/2) |}{|\sin(N\lambda/2)\sin((\lambda + \lambda_{j})/2)|}\right)^2 \label{eq:ccc}
\\&=
\sup_{\frac{\lambda_{j}}{2}\leq \lambda \leq \pi} \left( \frac{| \sin(\lambda/2) |}{|\sin((\lambda + \lambda_{j})/2)|}\right)^2 \nonumber
=:
\sup_{\frac{\lambda_{j}}{2}\leq \lambda \leq \pi} f_{1}(\lambda) \nonumber
\\&=
\left( \frac{1}{|\sin(\pi/2+\lambda_{j}/2)|}\right)^2 \label{eq:aaa}
\leq 
\left( \frac{1}{|\sin((\pi+\lambda_{m})/2)|}\right)^2 \nonumber
\\&=
(\cos(\lambda_{m}/2))^{-2}, 
\end{align}
where \eqref{eq:ccc} follows since $|\sin(\lambda)|$ is $\pi$-periodic. The inequality \eqref{eq:aaa} follows since the the function $f_{1}(\lambda)$ can be checked to be monotonically increasing on the interval $(\frac{\lambda_{j}}{2}, \pi)$

Similarly, 
\begin{align}
\sup_{2\lambda_{j} \leq \lambda \leq \pi} \left( \frac{|D_{N}(\lambda - \lambda_{j})|}{|D_{N}(\lambda)|}\right)^2
&=
\sup_{2\lambda_{j} \leq \lambda \leq \pi} \left( \frac{| \sin(N(\lambda - \lambda_{j})/2) \sin(\lambda/2) |}{|\sin(N\lambda/2)\sin((\lambda - \lambda_{j})/2)|}\right)^2 \nonumber
\\&=
\sup_{2\lambda_{j} \leq \lambda \leq \pi} \left( \frac{| \sin(N \lambda/2) \sin(\lambda/2) |}{|\sin(N\lambda/2)\sin((\lambda - \lambda_{j})/2)|}\right)^2 \nonumber
\\&=
\sup_{2\lambda_{j} \leq \lambda \leq \pi} \left( \frac{| \sin(\lambda/2) |}{|\sin((\lambda - \lambda_{j})/2)|}\right)^2 \nonumber
=:
\sup_{2\lambda_{j} \leq \lambda \leq \pi} f_{2}(\lambda) \nonumber
\\&=
\left( \frac{| \sin(\lambda_{j}) |}{|\sin(\lambda_{j}/2)|}\right)^2 \leq 4, \label{eq:aaaaa} 
\end{align}
where the inequality \eqref{eq:aaaaa} follows since the function $f_{2}(\lambda)$ 
is monotonically decreasing on the interval $(2\lambda_{j}, \pi)$.
\end{proof}

\section{Results and proofs for linear processes}
\label{se:linearprocesses}

In this section, we provide an extension of our results for linear processes $X_{n} = \sum_{j \in \ZZ} \VPsi_{j}\varepsilon_{n-j}$ with $\sum_{j\in \ZZ} \| \VPsi_{j} \|_{F}^{2} < \infty$ and sub-Gaussian innovations. 

\begin{remark} \label{re:different_innovations}
It is certainly conceivable to consider other innovations such as sub-exponential or with finite fourth moments. 
Concentration inequalities for i.i.d.\ random vectors which are either sub-Gaussian, sub-exponential or have finite fourth moments are respectively available in Section 5.2.3 in \cite{rudelson2013hanson}, Lemma 8.3 in \cite{erdHos2012bulk} and Lemma 4.1 in  \cite{Sun2018:LargeSpectral}. 
However, our proofs require a uniform concentration inequality stated in Theorem \ref{th:Dicker_Th1}.
The original proof of Theorem \ref{th:Dicker_Th1} in \cite{dicker2017} involves chaining techniques and subsequent application of the Hanson-Wright inequality for sub-Gaussian processes. Replacing the Hanson-Wright inequality by the respective inequalities for sub-exponential variables or data with finite fourth moments, would lead to uniform concentration inequalities as well.
\end{remark}

For completeness, we restate Lemma \ref{prop:supGhat} here for linear processes. The result only differs from the one for Gaussian time series by a general constant $\gamma$ which is determined by the sub-Gaussian innovations of the linear process in \eqref{eq:ACsubGgamma}.
\begin{lemma} \label{prop:supGhat_linear}
Suppose $\{X_{n}\}_{n \in \ZZ}$ can be represented as a $p$-dimensional linear process with spectral density $f_{X}$ as in \eqref{eq:f}.
Then, there are positive constants $c_{1},c_{2}$ such that 
\begin{equation} \label{eq:concreal1}
\begin{aligned}
&\Prob \Big( \sup_{D \in \Omega } 
| \widehat{H}_{rs}(D) - \E \widehat{H}_{rs}(D) | > 
\vertiii{G} \nu \Big) \leq 
\mathcal{B}(r,s,i),
\hspace{0.2cm} i =2,\dots,5,
\end{aligned}
\end{equation}
for $\nu^2 \geq \gamma^4 L^2_{rs,i}/ (m^2 c_{2} )$, where
\begin{equation} \label{eq:mathcalB}
\mathcal{B}(r,s,i)=
c_{1} \exp\Bigg( -c_{2} \min 
\Bigg\{ 
\frac{ \nu m }{ \gamma^2 \widebar{\Delta}_{N} L_{rs,1} },
\frac{ \nu^2 m^2 }{ \gamma^4 \widebar{\Delta}_{N}^{2} L_{rs,i}^2 }
\Bigg\} \Bigg) 
\end{equation}
%, $a,b \in \{r,s\}$  
with $r \neq s$ if $i=3$ and $\Omega$ is given in \eqref{eq:OmegaAB}. The constant $\gamma$ bounds the sub-Gaussian norm as in \eqref{eq:ACsubGgamma}, $\widebar{\Delta}_{N}$ is defined in \eqref{eq:delta-0-1/4} and $L_{rs,i}$'s as in \eqref{eq:ACsubGgamma}.
\end{lemma}

\begin{proof}
The main idea is to consider a truncated version of the linear process and to rewrite it as a linear mapping of an i.i.d.\ vector. The truncated version allows us then to write the local Whittle estimator as a quadratic form in sub-Gaussian random variables. 

Introduce a truncated version of the linear process as
$X^{L}_{n} = \sum_{j = -L}^{L} \VPsi_{j}\varepsilon_{n-j}$. 
The truncated version can be expressed as a linear transformation of an i.i.d.\ vector. Let $\psi_{\cdot r,j}$ and $\psi_{r \cdot,j}$ denote the $r$th column and row of $\VPsi_{j}$ and write the data matrix of the truncated process as $\mathcal{X}_{L} = [X^{L}_{1} : \cdots : X^{L}_{N} ]$ with $X^{L}_{n} = (X^{L}_{1,n},\dots, X^{L}_{p,n})'$. Then, 
\begin{align} 
e_{r}' \mathcal{X}_{L}'
=
(X_{r,1}^{L}, \dots, X_{r,N}^{L})
&= 
\left( e_{r}' \sum_{j = -L}^{L} \VPsi_{j}\varepsilon_{1-j}, \dots, e_{r}' \sum_{j = -L}^{L} \VPsi_{j}\varepsilon_{N-j} \right) \nonumber
\\&= 
\left( \sum_{j = -L}^{L} \psi_{r \cdot,j} \varepsilon_{1-j}, \dots, \sum_{j = -L}^{L} \psi_{r \cdot,j} \varepsilon_{N-j} \right) \nonumber
\\&=
\begin{pmatrix}
\varepsilon_{N+L}', \dots, \varepsilon_{1}', \dots, \varepsilon_{1-L}'
\end{pmatrix}
\begin{pmatrix}
0 			& 0 			& \cdots 	& \psi'_{r\cdot,-L} \\
\vdots		& \vdots		& \iddots	& \vdots \\
0 			& \psi'_{r\cdot,-L}& 		& \vdots \\
\psi'_{r\cdot,-L}	& \vdots		&		& \psi'_{r\cdot,L} \\
\vdots		& \vdots		& \iddots	& 0 \\
\vdots		& \psi'_{r\cdot,L}&		& \vdots \\
\psi'_{r\cdot,L} 	& 0			& \cdots & 0
\end{pmatrix}
=: \mathcal{E}_{L}' A_{L,r}', \label{eq:defALr}
\end{align}
where $A_{L,r}'$ is a $(N+2L)p \times N$ matrix.
As in \eqref{eq:realim} in the proof for Gaussian random variables in Lemma \ref{prop:supGhat}, we separate $| \widehat{H}_{rs}(D) - \E \widehat{H}_{rs}(D) | $ into diagonal elements and real and imaginary parts of the off-diagonal elements.
In order to avoid replicating all steps from the proof of Lemma \ref{prop:supGhat}, we focus here on the real part of the off-diagonal terms. The proofs for the diagonal elements and the imaginary part of the off-diagonal elements can be adapted analogously and are omitted.

\textit{Real part (off-diagonal):}
For the truncated linear process, the real part of the off-diagonal elements $\Re(\widehat{H}_{rs}(D))$ can be written in terms of \eqref{eq:PeriodogramMatrix} evaluated at the periodogram of the truncated process $I_{X_{L}}(\lambda_{j})$ as
\begin{equation} \label{eq:Loff_diag_rew}
\begin{aligned}
\frac{1}{m}e_{r}' \sum_{j=1}^{m} t_{j}(D) \Re(I_{X_{L}}(\lambda_{j})) t_{j}(D) e_{s}
&=
\frac{1}{2 \pi m} e_{r}' \sum_{j=1}^{m} t_{j}(D) \mathcal{X}_{L}'
(C_{j}C_{j}' + S_{j}S_{j}')
\mathcal{X}_{L} t_{j}(D) e_{s}.
%\\&=
%e_{r}' \mathcal{X}_{L}' \frac{1}{2 \pi m} \sum_{j=1}^{m} t_{j,r}(d_{r}) 
%(C_{j}C_{j}' + S_{j}S_{j}')
%t_{j,r}(d_{s}) \mathcal{X}_{L} e_{s}.
\end{aligned}
\end{equation}
As in \eqref{eq:Hrr_rewritten}, we write \eqref{eq:Loff_diag_rew} as a quadratic form but now using
$(e_{r}'\mathcal{X}_{L}' ~ e_{s}'\mathcal{X}_{L}' )=\mathcal{E}'_{L} A_{L}'$, where $A'_{L} = (A_{L,r}' ~ A_{L,s}')$ and set
\begin{align} \label{eq:SIGMAL}
\Sigma_{L} = 
A_{L}A'_{L} = 
\begin{pmatrix}
A_{L,r} \\ 
A_{L,s}
\end{pmatrix}
(A_{L,r}' ~ A_{L,s}').
\end{align}
Due to Lemma \ref{le:relationsSigmaA}, $\Sigma_{L}$ can also be written as
\begin{equation*} 
%\label{eq:LSigma_rs}
\Sigma_{L} = 
\begin{pmatrix}
\Sigma_{L,rr} & \Sigma_{L,rs} \\
\Sigma_{L,sr} & \Sigma_{L,ss} 
\end{pmatrix}
\hspace{0.2cm}
\text{ with }
\hspace{0.2cm}
\Sigma_{L,rs} = (\Sigma_{L,rs}(n-k))_{n,k=1,\dots,N} = (\E X^{L}_{r,k} X^{L}_{s,n})_{n,k=1,\dots,N}.
\end{equation*}
We further recall from Lemma \ref{prop:supGhat} the $4m \times 2N$ matrix
\begin{equation*} 
\label{eq:Rtildem}
\widetilde{R}_{m} = 
\begin{pmatrix}
R_{m} & 0_{2m \times N} \\
0_{2m \times N} & R_{m}
\end{pmatrix}
\end{equation*}
with $R_{m}$ as in \eqref{eq:Rm} and the matrix $F(d_{r},d_{s})=\diag(T_{r}(d_{r}),T_{s}(d_{s}))$ with $T_{r}$ as in \eqref{eq:fctT2} and 
\begin{equation} \label{eq:fctT}
M_{m} =
\begin{pmatrix}
0_{m,m} &  I_{m} \\
0_{m,m} & 0_{m,m}
\end{pmatrix}.
\end{equation}
Write
\begin{equation} \label{eq:Lzzzzz}
\begin{aligned}
e_{r}' \mathcal{X}'_{L}
\sum_{j=1}^{m} t_{j,r}(d_{r}) 
(C_{j}C_{j}' + S_{j}S_{j}')
t_{j,r}(d_{s}) 
\mathcal{X}_{L} e_{s}
&=
(e_{r}'\mathcal{X}_{L}' ~ e_{s}'\mathcal{X}_{L}' )
\widetilde{R}_{m}' F(d_{r},d_{s}) M_{2m} 
F(d_{r},d_{s}) \widetilde{R}_{m}
\begin{pmatrix}
\mathcal{X}_{L} e_{r} \\ \mathcal{X}_{L} e_{s}
\end{pmatrix}
\\&= 
(e_{r}'\mathcal{X}_{L}' ~ e_{s}'\mathcal{X}_{L}' )
\widetilde{F}(d_{r},d_{s}) 
\begin{pmatrix}
\mathcal{X}_{L} e_{r} \\ \mathcal{X}_{L} e_{s}
\end{pmatrix}
\\&= 
\mathcal{E}_{L}' A_{L}' \widetilde{F}(d_{r},d_{s}) A_{L} \mathcal{E}_{L} 
\\&=
\mathcal{E}_{L}' R_{L}(d_{r},d_{s}) \mathcal{E}_{L}
\end{aligned}
\end{equation}
with
\begin{equation*}
R_{L}(d_{r},d_{s})=A_{L}' \widetilde{F}(d_{r},d_{s}) A_{L}
\hspace{0.2cm}
\text{ and }
\hspace{0.2cm}
\widetilde{F}(d_{r},d_{s}) = \widetilde{R}_{m}' F(d_{r},d_{s}) M_{2m} 
F(d_{r},d_{s}) \widetilde{R}_{m}.
\end{equation*}
In order to apply Theorem \ref{th:Dicker_Th1}, we further write
\begin{align}
&\Prob \Big( \sup_{ D \in \Omega }   | 
\sum_{j=1}^{m} e_{r}' (t_{j}(D) \mathcal{X}'_{L}
(C_{j}C_{j}' + S_{j}S_{j}')
\mathcal{X}_{L} t_{j}(D) - 
\E (t_{j}(D) \mathcal{X}'_{L}
(C_{j}C_{j}' + S_{j}S_{j}')
\mathcal{X}_{L} t_{j}(D) ) ) e_{s} | > \pi m \vertiii{G} \nu \Big) \nonumber
%\\&=
%\Prob \Big( \sup_{(d_{r},d_{s}) \in 
%[a_{r},b_{r}] \times [a_{s},b_{s}]}    | \sum_{j=1}^{m}
%e_{r}' (\mathcal{X}_{L}' t_{j,r}(d_{r}) 
%(C_{j}C_{j}' + S_{j}S_{j}')
%t_{j,r}(d_{s}) \mathcal{X}_{L} 
%\\&\hspace{4cm}- 
%\E ( \mathcal{X}_{L}' t_{j,r}(d_{r}) 
%(C_{j}C_{j}' + S_{j}S_{j}')
%t_{j,r}(d_{s}) \mathcal{X}_{L} ) ) e_{s} | > \pi m \vertiii{G} \nu \Big) \nonumber
\\&=
\Prob \Big( \sup_{(d_{r},d_{s}) \in 
[a_{r},b_{r}] \times [a_{s},b_{s}]} | \mathcal{E}_{L}' R_{L}(d_{r},d_{s}) \mathcal{E}_{L}  - 
\E (\mathcal{E}_{L}' R_{L}(d_{r},d_{s}) \mathcal{E}_{L} ) | > \pi m \vertiii{G} \nu \Big). \label{eq:Lreal_formthc1}
\end{align}
Note that the matrix $R_{L}(d_{r},d_{s})$ can be rewritten as
\begin{equation} \label{eq:LRrewrtitten}
\begin{aligned}
R_{L}(d_{r},d_{s})
&=
A_{L}' \mathcal{A}_{N,i}^{-1} \widetilde{R}_{m}' \mathcal{A}_{2m,i} F(d_{r},d_{s}) M_{2m} F(d_{r},d_{s})  \mathcal{A}_{2m,i} \widetilde{R}_{m} \mathcal{A}_{N,i}^{-1} A_{L}
\end{aligned}
\end{equation}
for $i=1,\dots,4$ with
\begin{equation} \label{eq:LmathcalAi}
\begin{aligned}
\mathcal{A}_{m,1}
&=
\diag(c_{r,1,N}I_{m},\widetilde{c}_{s,1,N}I_{m}),
\hspace{0.2cm}
\mathcal{A}_{m,2}
=
\diag(c_{r,2,N}I_{m},\widetilde{c}_{s,2,N}I_{m}),
\\
\mathcal{A}_{m,3}
&=
\diag(c_{r,3,N}I_{m},\widetilde{c}_{s,3,N}I_{m}),
\hspace{0.2cm}
\mathcal{A}_{m,4}
=
m^{\frac{1}{4}} \mathcal{A}_{m,1}.
\end{aligned}
\end{equation}
The matrices $\mathcal{A}_{N,i}$ in \eqref{eq:LRrewrtitten} are defined by replacing $I_{m}$'s in \eqref{eq:LmathcalAi} by $I_{N}$. %These matrices are to normalize $\Sigma^{\frac{1}{2}}_{L}$ in \eqref{eq:Rrewrtitten}.
\par
We continue to bound \eqref{eq:Lreal_formthc1} by applying Theorem \ref{th:Dicker_Th1}. In order to verify the applicability of Theorem \ref{th:Dicker_Th1}, note that the matrix $\mathcal{A}_{2m,i} F(d_{r},d_{s}) M_{2m} F(d_{r},d_{s})  \mathcal{A}_{2m,i} $ in \eqref{eq:LRrewrtitten}  is not diagonal but can be represented as a unitary transformation of a diagonal matrix as described in \eqref{eq:real_SVDTMT}.
For this reason, Theorem \ref{th:Dicker_Th1} remains applicable due to Remark \ref{re:Th1_diag}.
In \eqref{eq:Lreal_imCI1} below, we apply Theorem \ref{th:Dicker_Th1} with $K=2$ and $R = b_{r} - a_{r} \leq 1$ to obtain
\begin{equation} \label{eq:Lreal_imCI1}
\begin{aligned}
&\Prob \Big( \sup_{(d_{r},d_{s}) \in 
[a_{r},b_{r}] \times [a_{s},b_{s}]} | \mathcal{E}'_{L} R_{L}(d_{r},d_{s}) \mathcal{E}_{L}  - 
\E (\mathcal{E}'_{L} R_{L}(d_{r},d_{s}) \mathcal{E}_{L} ) | > \pi m \vertiii{G} \nu \Big)
\\&\leq
c_{1} \exp\Bigg( -c_{2} \min 
\Bigg\{ 
\frac{ \nu m \vertiii{G} }{ \gamma^2 \mathcal{T}^{L}_{1} },
\frac{ \nu^2 m^2 \vertiii{G}^2 }{ \gamma^4 (\mathcal{T}^{L}_{i})^{2} }
\Bigg\} \Bigg) 
%\\&\leq
%c_{1} \exp\Bigg( -c_{2} \min 
%\Bigg\{ 
%\frac{ \nu m }{ \gamma^2 \widebar{\Delta}_{N} L_{rs,1} },
%\frac{ \nu^2 m^2 }{ \gamma^4 \widebar{\Delta}_{N}^{2} L^2_{rs,i} }
%\Bigg\} \Bigg) \label{eq:real_imCI2}
\end{aligned}
\end{equation}
for $ \nu^2 \geq \gamma^4 (\mathcal{T}^{L}_{i})^{2}/(c_{2} m^2 \vertiii{G}^2 )$ and $i=2,\dots,5$ with
\begin{equation} \label{eq:Lreal_imTi}
\mathcal{T}^{L}_{1}=\| B^{L}_{m,1} \| L_{rs,1},
\hspace{0.2cm}
\mathcal{T}^{L}_{i}=\| B^{L}_{m,i} \|_{F} L_{rs,i},
\hspace{0.2cm}
\mathcal{T}^{L}_{5}=\| B^{L}_{m,1} \| L_{rs,5}
\end{equation}
for $i=2,\dots,4$, where the $L_{rs,i}$'s are given in \eqref{eq:prop:supGhatLs} and 
\begin{equation} \label{eq:Lreal_Bmi}
B^{L}_{m,i}= \widetilde{R}_{m} \mathcal{A}_{N,i}^{-1} \Sigma_{L} \mathcal{A}_{N,i}^{-1} \widetilde{R}_{m}'
\end{equation}
for $i=1,\dots,4$, which is due to the relation \eqref{eq:SIGMAL}. 
Based on \eqref{eq:Lzzzzz}, an equivalent formulation of \eqref{eq:Lreal_imCI1} can be given as
\begin{equation*} 
%\label{eq:Lreal_imCI1}
\begin{aligned}
&\Prob \Big( \sup_{(d_{r},d_{s}) \in 
[a_{r},b_{r}] \times [a_{s},b_{s}]} | e_{r}' \mathcal{X}'_{L} \widetilde{F}(d_{r},d_{s}) \mathcal{X}_{L} e_{s}  - 
\E (e_{r}' \mathcal{X}'_{L} \widetilde{F}(d_{r},d_{s}) \mathcal{X}_{L} e_{s} ) | > \pi m \vertiii{G} \nu \Big)
\\&\leq
c_{1} \exp\Bigg( -c_{2} \min 
\Bigg\{ 
\frac{ \nu m \vertiii{G} }{ \gamma^2 \mathcal{T}^{L}_{1} },
\frac{ \nu^2 m^2 \vertiii{G}^2 }{ \gamma^4 (\mathcal{T}^{L}_{i})^{2} }
\Bigg\} \Bigg).
\end{aligned}
\end{equation*}
Then, letting $L$ go to infinity on both sides, Lemmas \ref{le:LconvD} and \ref{le:convLofB} lead to
\begin{equation*} 
%\label{eq:Lreal_imCI1}
\begin{aligned}
&\Prob \Big( \sup_{(d_{r},d_{s}) \in 
[a_{r},b_{r}] \times [a_{s},b_{s}]} | e_{r}' \mathcal{X}' \widetilde{F}(d_{r},d_{s}) \mathcal{X} e_{s}  - 
\E (e_{r}' \mathcal{X}' \widetilde{F}(d_{r},d_{s}) \mathcal{X} e_{s} ) | > \pi m \vertiii{G} \nu \Big)
\\&\leq
c_{1} \exp\Bigg( -c_{2} \min 
\Bigg\{ 
\frac{ \nu m \vertiii{G} }{ \gamma^2 \mathcal{T}_{1} },
\frac{ \nu^2 m^2 \vertiii{G}^2 }{ \gamma^4 \mathcal{T}_{i}^{2} }
\Bigg\} \Bigg),
\end{aligned}
\end{equation*}
which coincides with \eqref{eq:real_imCI1}. Then, one can use the same arguments as in Lemma \ref{prop:supGhat} following \eqref{eq:real_imCI1}.
\end{proof}

The following lemma gathers some relationships between the matrices $A_{L,r}$ in \eqref{eq:defALr} and the autocovariance matrices of the truncated linear process. For completeness, we also state the representation of the autocovariance function for the original linear process.

\begin{lemma} \label{le:relationsSigmaA}
The autocovariance matrix and the coefficient matrices of the linear process relate as follows. For $k \leq n$,
\begin{equation} \label{eq:bcejhbcjhbdjhce}
\begin{gathered}
\Sigma_{L,rs}(n-k)  = \sum_{j=-L}^{L-(n-k)} \psi_{r \cdot, j+(n-k)} \psi_{s \cdot, j }',
\hspace{0.2cm}
\Sigma_{rs}(n-k)  = \sum_{j \in \ZZ} \psi_{r \cdot, j+(n-k)} \psi_{s \cdot, j }',
\\
A_{L,s}A'_{L,r}
= 
\Sigma_{L,rs}(k-n).
\end{gathered}
\end{equation}
Furthermore, $\Sigma_{L,rs}(k-n) = (\Sigma_{L,sr}(n-k))'$ and $\Sigma_{rs}(k-n) = (\Sigma_{sr}(n-k))'$.
%Note that $\E X_{r,k} X^{'}_{s,n} = (\E X_{s,n} X^{'}_{r,k})' $ and $\E X_{r,k} X^{'}_{s,n} = (\E X_{s,n} X^{'}_{r,k})' $. 
\end{lemma}

\begin{proof}
Recall $ X^{L}_{r} = (X_{r,1}^{L}, \dots, X_{r,N}^{L})'$ and suppose $k \leq n$. Then, 
\begin{align*}
\Sigma_{L,rs}(k-n) 
=
\E X^{L}_{r,k} X^{L'}_{s,n} 
&= 
\E \left( 
\sum_{j = -L}^{L} \psi_{r \cdot,j} \varepsilon_{n-j} 
\sum_{j = -L}^{L} \psi_{s \cdot,j} \varepsilon_{k-j} \right)
\\&= 
\sum_{j_{1} =n-L}^{n+L} \sum_{j_{2} = k-L}^{k+L} 
\psi_{r \cdot, n-j_{1}}
\E(\varepsilon_{j_{1}} \varepsilon_{j_{2}}')
\psi_{s \cdot, k-j_{2} }'
\\&= 
\sum_{j = \max\{k,n\} -L}^{\min\{k,n\}+L} 
\psi_{r \cdot, n-j}
\psi_{s \cdot, k-j}'
\\&= 
\sum_{j=-L}^{L-(n-k)} \psi_{r \cdot, j+(n-k)} \psi_{s \cdot, j }'.
\end{align*}
Analogously, for the non-truncated version of the linear process leading to the second relation in \eqref{eq:bcejhbcjhbdjhce}.
%\begin{align*}
%\Sigma_{rs}(k-n) 
%=
%\E X_{r,k} X^{'}_{s,n} 
%&= 
%\E \left( 
%\sum_{j = 0}^{\infty} \psi_{r \cdot,j} \varepsilon_{n-j}
%\sum_{j = 0}^{\infty} \psi_{s \cdot,j} \varepsilon_{k-j} \right)
%\\&= 
%\sum_{j_{1} = -\infty}^{k} \sum_{j_{2} =-\infty}^{n} 
%\psi_{r \cdot, k-j_{1}}
%\E(\varepsilon_{j_{1}} \varepsilon_{j_{2}}')
%\psi_{s \cdot, n-j_{2} }'
%\\&= 
%\sum_{j = -\infty}^{\min\{k,n\}} 
%\psi_{r \cdot, k-j}
%\psi_{s \cdot, n-j }'
%\\&= 
%\sum_{j=0}^{\infty} \psi_{r \cdot, j+(k-n)} \psi_{s \cdot, j }'.
%\end{align*}
Finally, the $kn$th element of $A_{L,r}A'_{L,s}$ satisfies
\begin{align*}
(A_{L,r}A'_{L,s})_{kn}
= 
\sum_{j=-L}^{L-(n-k)} \psi_{r \cdot, j+(n-k)} \psi_{s \cdot, j }'
= \Sigma_{L,rs}(n-k) .
\end{align*}
\end{proof}
The next three lemmas follow the ideas in \cite{Sun2018:LargeSpectral}, but also allow for long-range dependence and address that we need a uniform concentration inequality.
\begin{lemma} \label{le:convLofB}
Recall $B^{L}_{m,i}= \widetilde{R}_{m} \mathcal{A}_{N,i}^{-1} \Sigma_{L} \mathcal{A}_{N,i}^{-1} \widetilde{R}_{m}'$ in \eqref{eq:Lreal_Bmi} with $\Sigma_{L}$ as in \eqref{eq:SIGMAL}. Then, 
\begin{equation}
\| B^{L}_{m,i} \| \to \| B_{m,i} \|
\hspace{0.2cm}
\text{ and }
\hspace{0.2cm}
\| B^{L}_{m,i} \|_{F} \to \| B_{m,i} \|_{F},
\hspace{0.2cm}
\text{ as }
\hspace{0.2cm}
L \to \infty,
\end{equation}
with $B_{m,i}= \widetilde{R}_{m} \mathcal{A}_{N,i}^{-1} \Sigma \mathcal{A}_{N,i}^{-1} \widetilde{R}_{m}'$ as in \eqref{eq:real_Bmi}.
%\begin{equation}
%\| \Sigma^{L}_{rr} \| \to \| \Sigma_{rr} \| 
%\end{equation}
\end{lemma}

\begin{proof}
It is sufficient to prove $\| B^{L}_{m,i} - B_{m,i} \|_{F} \to 0$ as $L \to \infty$ since by triangle inequality
\begin{equation*}
\begin{gathered}
| \| B^{L}_{m,i} \|_{F} - \| B_{m,i} \|_{F} |
\leq
\| B^{L}_{m,i} - B_{m,i} \|_{F},
\\
| \| B^{L}_{m,i} \| - \| B_{m,i} \| |
\leq
\| B^{L}_{m,i} - B_{m,i} \|
\leq
\| B^{L}_{m,i} - B_{m,i} \|_{F}.
\end{gathered}
\end{equation*}
We can further reduce the problem as follows:
\begin{align}
\| B^{L}_{m,i} - B_{m,i} \|_{F}
&=
\| \widetilde{R}_{m} \mathcal{A}_{N,i}^{-1} (\Sigma_{L} - \Sigma) \mathcal{A}_{N,i}^{-1} \widetilde{R}_{m}' \|_{F} \label{al:wwww1}
\leq
\| \mathcal{A}_{N,i}^{-1} (\Sigma_{L} - \Sigma) \mathcal{A}_{N,i}^{-1} \|_{F}
\\&\leq
\| \mathcal{A}_{N,i}^{-2} \| \| \Sigma_{L} - \Sigma \|_{F}
\leq
\| \mathcal{A}_{N,i}^{-2} \| (\| \Sigma_{L,rr} - \Sigma_{rr} \|_{F} + \| \Sigma_{L,ss} - \Sigma_{ss} \|_{F}). \label{al:wwww2}
\end{align}
We used Lemma \ref{le:normineq1} and $\| \widetilde{R}_{m} \|^2 = 1$ in \eqref{al:wwww1}. The first inequality in \eqref{al:wwww2} also follows by Lemma \ref{le:normineq1}. Note that $\mathcal{A}_{N,i}$ are diagonal matrices which do not depend on $L$. The second inequality in \eqref{al:wwww2} is due to Lemma \ref{le:normineq2}. Finally, 
\begin{align}
&\| \Sigma_{L,rr} - \Sigma_{rr} \|^{2}_{F} \nonumber
\\&=
\| \left( \sum_{j=L-(n-k) + 1}^{\infty} \psi_{r \cdot, j+(n-k)} \psi_{r \cdot, j }' +
\sum_{j=-\infty}^{-L-1} \psi_{r \cdot, j+(n-k)} \psi_{r \cdot, j }'
 \right)_{k,n =1,\dots,N} \|^{2}_{F} \nonumber
\\&\leq N^{2}
\max_{k,n=1,\dots,N} | \sum_{i=L + 1}^{\infty} \psi_{r \cdot, i} \psi_{r \cdot, i-(n-k) }' |^2
+
N^{2}
\max_{k,n=1,\dots,N} | \sum_{i=L + 1}^{\infty} \psi_{r \cdot, -i + (n-k)} \psi_{r \cdot, -i }' |^2 \nonumber
\\&\leq N^{2} 
\max_{k,n=1,\dots,N} \left( | \sum_{i=L + 1}^{\infty} \| \psi_{r \cdot, i} \|_{F} \| \psi_{r \cdot, i-(n-k) } \|_{F} |^2
+
| \sum_{i=L + 1}^{\infty} \| \psi_{r \cdot, -i + (n-k)} \|_{F} \| \psi_{r \cdot, -i } \|_{F} |^2 \right) \label{al:kkkk1}
\\&\leq N^{2}
\max_{k,n=1,\dots,N} \Bigg( \sum_{i=L + 1}^{\infty} \| \psi_{r \cdot, i} \|^{2}_{F} \sum_{i=L + 1}^{\infty} \| \psi_{r \cdot, i-(n-k) } \|^2_{F} \nonumber
\\&\hspace{3cm}+
\sum_{i=L + 1}^{\infty} \| \psi_{r \cdot, -i + (n-k)} \|^{2}_{F} \sum_{i=L + 1}^{\infty} \| \psi_{r \cdot, -i } \|^2_{F} \Bigg) \label{al:kkkk2}
\to 0 
\text{ as } 
L \to \infty,
\end{align}
where both \eqref{al:kkkk1} and \eqref{al:kkkk2} follow by the Cauchy-Schwarz inequality.
\end{proof}

\begin{lemma} \label{le:XLX}
Recall the data matrices $\mathcal{X}_{L} = [X^{L}_{1} : \cdots : X^{L}_{N} ]$ and $\mathcal{X} = [X_{1} : \cdots : X_{N} ]$. Then,
\begin{equation*}
\E \| e_{r}' \mathcal{X}_{L}' - e_{r}' \mathcal{X'} \|^{2}_{F} \to 0, 
\hspace{0.2cm}
\text{ as }
\hspace{0.2cm}
L \to \infty.
\end{equation*}
\end{lemma}

\begin{proof}
The expected value can be explicitly calculated as
\begin{align*}
\E \| e_{r}' \mathcal{X}_{L}' - e_{r}' \mathcal{X'} \|^{2}_{F}
&=
\sum_{n =1}^{N} \E | X_{r,n}^{L} - X_{r,n} |^2 
=
\sum_{n =1}^{N} \E | \sum_{j = L+1}^{\infty} \psi_{r\cdot,j} \varepsilon_{n-j} + \sum_{j = - \infty}^{-L-1} \psi_{r\cdot,j} \varepsilon_{n-j} |^2 
\\&\leq
2 \sum_{n =1}^{N} \left( \E | \sum_{j = L+1}^{\infty} \psi_{r\cdot,j} \varepsilon_{n-j} |^2 + \E | \sum_{j = - \infty}^{-L-1} \psi_{r\cdot,j} \varepsilon_{n-j} |^2 \right)
%\\&=
%\E \sum_{r=1}^{p} \left( \sum_{j = L+1}^{\infty} \sum_{s=1}^{p} \psi_{rs,j} \varepsilon_{s,n-j} \right)^{2}
\\&=
2 \sum_{n =1}^{N} \sum_{j_{1},j_{2} = L+1}^{\infty} \left( \psi_{r\cdot,j_{1}} 
\E ( \varepsilon_{n-j_{1}} \varepsilon'_{n-j_{2}} ) \psi'_{r\cdot,j_{2}}
+
\psi_{r\cdot,-j_{1}} 
\E ( \varepsilon_{n+j_{1}} \varepsilon'_{n+j_{2}} ) \psi'_{r\cdot,-j_{2}} \right)
\\&=
2N \sum_{j = L+1}^{\infty} ( \psi_{r\cdot,j} \psi'_{r\cdot,j} + \psi_{r\cdot,-j}  \psi'_{r\cdot,-j} ) < \infty
\end{align*}
since $\sum_{j = 0}^{\infty} \| \VPsi_{j} \|^2_{F} < \infty$. Due to the summability of the last expression, letting $L \to \infty$ yields the desired convergence result.
\end{proof}

\begin{lemma} \label{le:LconvD}
The following convergence result holds:
\begin{equation*}
\begin{aligned}
&
\sup_{(d_{r},d_{s}) \in [a_{r},b_{r}] \times [a_{s},b_{s}]}
| e_{r}' \mathcal{X}'_{L} \widetilde{F}(d_{r},d_{s}) \mathcal{X}_{L} e_{s} %\nonumber
%\\& \phantom{=\Prob \Big( \sup_{d \in [a_{r},b_{b}] } | }
- 
\E ( e_{r}' \mathcal{X}'_{L} \widetilde{F}(d_{r},d_{s}) \mathcal{X}_{L} e_{s} ) |
\\ & \hspace{1cm}
\overset{\distr}{\to}
\sup_{(d_{r},d_{s}) \in [a_{r},b_{r}] \times [a_{s},b_{s}]}
| e_{r}' \mathcal{X}' \widetilde{F}(d_{r},d_{s}) \mathcal{X} e_{r} %\nonumber
%\\& \phantom{=\Prob \Big( \sup_{d \in [a_{r},b_{b}] } | }
- 
\E ( e_{r}' \mathcal{X}' \widetilde{F}(d_{r},d_{s}) \mathcal{X} e_{r}) |,
\hspace{0.2cm}
\text{ as }
\hspace{0.2cm}
L \to \infty.
\end{aligned}
\end{equation*}
\end{lemma}

\begin{proof}
Note that Lemma \ref{le:XLX} implies the convergence in probability,
\begin{equation*}
e_{r}' \mathcal{X}_{L}' \overset{\prob}{\to} e_{r}' \mathcal{X}'.
\end{equation*}
We write $C_{b}([a_{r},b_{r}] \times [a_{s},b_{s}], \RR^{N \times N})$ for the space of $\RR^{N \times N}$-valued, bounded and continuous functions on $[a_{r},b_{r}] \times [a_{s},b_{s}]$. The space is equipped with the uniform metric
\begin{equation*}
\| A - B \|_{\infty} = \sup_{(d_{r},d_{s}) \in [a_{r},b_{r}] \times [a_{s},b_{s}]} \| A(d_{r},d_{s}) - B(d_{r},d_{s}) \|_{F}.
\end{equation*}
Since the $\RR^{N} \to C_{b}([a_{r},b_{r}] \times [a_{s},b_{s}], \RR)$ function $x \mapsto x' A x$ is continuous, the continuous mapping theorem implies
\begin{equation} \label{eq:convprobL}
e_{r}' \mathcal{X}_{L}' \widetilde{F}(d_{r},d_{s}) \mathcal{X}_{L} e_{r}
\overset{\prob}{\to}
e_{r}' \mathcal{X}' \widetilde{F}(d_{r},d_{s}) \mathcal{X} e_{r}.
\end{equation}
Lemma \ref{le:XLX} also allows us to infer 
%$\E \| e_{r}' \mathcal{X}_{L}' - e_{r}' \mathcal{X'} \|_{1} \to 0$, as $L \to \infty$ which implies
\begin{align} \label{eq:convexpL}
\E(e_{r}' \mathcal{X}_{L}' \widetilde{F}(d_{r},d_{s}) \mathcal{X}_{L} e_{r} )
\to
\E(e_{r}' \mathcal{X}^{'} \widetilde{F}(d_{r},d_{s}) \mathcal{X} e_{r} )
\end{align}
since for any matrix function $A \in C_{b}([a_{r},b_{r}] \times [a_{s},b_{s}], \RR^{N \times N})$
\begin{align*}
&
\| \E(e_{r}' \mathcal{X}_{L}' A \mathcal{X}_{L} e_{r} )
-
\E(e_{r}' \mathcal{X}' A \mathcal{X} e_{r} ) \|_{\infty}
\\&\leq
\| \E(e_{r}' \mathcal{X}_{L}' A (\mathcal{X}_{L} e_{r}-\mathcal{X} e_{r}) ) \|_{\infty}
+
\| \E((e_{r}' \mathcal{X}' - e_{r}' \mathcal{X}_{L}') A \mathcal{X} e_{r} ) \|_{\infty}
\end{align*}
and
\begin{align*}
\| \E(e_{r}' \mathcal{X}_{L}' A (\mathcal{X}_{L} e_{r}-\mathcal{X} e_{r}) ) \|_{\infty}^2
\leq
\|A\|^2_{\infty} \E \| e_{r}' \mathcal{X}_{L}' \|_{F}^2
\E \| e_{r}' \mathcal{X}_{L}' - e_{r}' \mathcal{X}' \|_{F}^{2}
\to 0.
\end{align*}
Combining \eqref{eq:convprobL} and \eqref{eq:convexpL} yields
\begin{align*}
&
e_{r}' \mathcal{X}_{L}' \widetilde{F}(d_{r},d_{s}) \mathcal{X}_{L} e_{r} 
-
\E(e_{r}' \mathcal{X}_{L}' \widetilde{F}(d_{r},d_{s}) \mathcal{X}_{L} e_{r} )
\\ & \hspace{1cm}
\overset{\prob}{\to}
e_{r}' \mathcal{X}' \widetilde{F}(d_{r},d_{s}) \mathcal{X} e_{r} 
-
\E(e_{r}' \mathcal{X}' \widetilde{F}(d_{r},d_{s}) \mathcal{X} e_{r} ).
\end{align*}
Convergence in probability implies convergence in distribution. Finally, the continuous mapping theorem with $C_{b}([a_{r},b_{r}] \times [a_{s},b_{s}], \RR^{N \times N}) \to \RR$ and $A(d_{r},d_{s}) \mapsto \sup_{(d_{r},d_{s}) \in [a_{r},b_{r}] \times [a_{s},b_{s}]} | A(d_{r},d_{s}) |$ proves the claim of the lemma.
\end{proof}

\section{Complementary simulation results}
\label{app:table}

%We provide here a range of different simulation studies. While Section \ref{se:F1}
%\ref{appF:univsmulti}
%\ref{appF:rate}
%\ref{se:F4}
%\ref{se:F5}

\subsection{Performance measures}
\label{se:F1}
This section presents our simulation results for the thresholding and graphical local Whittle estimators discussed in Section \ref{s:sim}. A number of different measures are used to assess the performance of our estimators.

The first five measures in Tables \ref{tab:threshold} and \ref{tab:graphicalLW} concern whether the matrix entries are correctly estimated as zero. We use the standard true positive (TP), false negative (FN), false positive (FP) and true negative (TN) measures. For instance, TP quantifies how often the respective estimator correctly detected a nonzero matrix entry as nonzero. For the respective rates, we write TPR, FNR, TNR and FPR. The fifth metric we consider is the so-called precision calculated as TP/(TP+FP) which measures the estimator's accuracy in classifying entries as nonzero. The introduced measures perform well, even for small sample size $N$ and dimension $p$ and improve with increasing $N$ and $p$.

The last five measures in Tables \ref{tab:threshold} and \ref{tab:graphicalLW} consider different distance measures between our estimators and the true quantities. Besides the mean squared error for the memory parameters (MSE($D$)), we consider the distance measures (Frobenius and spectral norm) used in our main results to prove consistency of our estimators (Propositions \ref{prop:Opspectral} and \ref{prop:OpPrecision}). The results have been discussed in more detail in Section \ref{s:sim}.
Finally, Rel-Frobenius and -spectral denote the ratio between the distance of our sparse estimator and the distance between the associated nonsparse estimator.
For example, Rel-Frobenius for Tables \ref{tab:threshold} and \ref{tab:graphicalLW} are respectively defined as
\begin{equation} \label{appendix:relF}
\frac{\| T_\rho (\widehat G(\widehat D)) - G_0 \|_F}{ \| \widehat G(\widehat D) - G_0 \|_F }, 
\hspace{0.2cm}
\frac{\| \widehat{P}_\rho (\widehat D) - P_0 \|_F}{ \| \widehat P(\widehat D) - P_0 \|_F }, 
\end{equation}
where $\widehat G(\widehat D)$ is a nonsparse estimator of $G_{0}$ as defined in \eqref{eq:G(D)} which is the same as setting the threshold $\rho=0$ and analogously $P(\widehat D) := \widehat{P}_0 (\widehat D)$. Rel-spectral can be defined analogously by replacing the Frobenius norm with the spectral norm in \eqref{appendix:relF}.
In Tables \ref{tab:threshold} and \ref{tab:graphicalLW}, Rel-Frobenius and -spectral take values smaller than one, which indicates that the denominators in \eqref{appendix:relF} take values larger than the numerators, that is, the nonsparse estimators perform worse for our sparse DGPs.

\begin{landscape}
% Table generated by Excel2LaTeX from sheet 'Sheet2'
\begin{table}[htbp]
  \centering
   \small
    \begin{tabular}{ccccccccccc} \hline
    \multirow{2}[0]{*}{Models} & \multirow{2}[0]{*}{Measures} & \multicolumn{3}{c}{$p=20$} & \multicolumn{3}{c}{$p=40$} & \multicolumn{3}{c}{$p=60$} \\
          &       & $N=200$ & $N=400$ & $N=1000$ & $N=200$ & $N=400$ & $N=1000$ & $N=200$ & $N=400$ & $N=1000$ \\ \hline \hline

           \multirow{10}[0]{*}{thDGP1} & TPR   & 0.991 & 0.999 & 1.000 & 0.976 & 0.999 & 1.000 & 0.975 & 0.999 & 1.000 \\
          & FNR   & 0.009 & 0.001 & 0.000 & 0.024 & 0.001 & 0.000 & 0.025 & 0.001 & 0.000 \\
          & TNR   & 0.991 & 0.999 & 0.996 & 0.999 & 0.999 & 0.996 & 0.999 & 0.998 & 0.999 \\
          & FPR   & 0.009 & 0.001 & 0.004 & 0.001 & 0.001 & 0.004 & 0.001 & 0.002 & 0.001 \\
          & Precision & 0.899 & 0.984 & 0.949 & 0.973 & 0.964 & 0.900 & 0.960 & 0.936 & 0.963 \\
          & MSE($D$) & 0.095 & 0.050 & 0.022 & 0.193 & 0.102 & 0.045 & 0.295 & 0.155 & 0.069 \\
          & Frobenius & 0.174 & 0.110 & 0.088 & 0.228 & 0.171 & 0.139 & 0.294 & 0.228 & 0.146 \\
          & Spectral & 0.108 & 0.066 & 0.055 & 0.115 & 0.092 & 0.075 & 0.139 & 0.113 & 0.070 \\
          & Rel-Frobenius   & 0.379 & 0.318 & 0.367 & 0.248 & 0.247 & 0.293 & 0.213 & 0.219 & 0.204 \\
          & Rel-spectral  & 0.450 & 0.378 & 0.452 & 0.294 & 0.321 & 0.401 & 0.268 & 0.299 & 0.282 \\ \hline
          \multirow{10}[0]{*}{thDGP2} & TPR   & 0.968 & 0.996 & 1.000 & 0.947 & 0.994 & 1.000 & 0.936 & 0.990 & 1.000 \\
          & FNR   & 0.032 & 0.004 & 0.000 & 0.053 & 0.007 & 0.000 & 0.064 & 0.010 & 0.000 \\
          & TNR   & 0.997 & 0.998 & 0.999 & 0.999 & 0.999 & 1.000 & 1.000 & 1.000 & 1.000 \\
          & FPR   & 0.003 & 0.002 & 0.001 & 0.001 & 0.001 & 0.000 & 0.000 & 0.000 & 0.000 \\
          & Precision & 0.958 & 0.972 & 0.987 & 0.973 & 0.978 & 0.993 & 0.977 & 0.985 & 0.992 \\
          & MSE($D$) & 0.095 & 0.052 & 0.023 & 0.197 & 0.104 & 0.047 & 0.293 & 0.157 & 0.073 \\
          & Frobenius & 2.569 & 1.949 & 1.374 & 5.047 & 3.837 & 2.671 & 7.512 & 5.692 & 3.960 \\
          & Spectral & 1.253 & 0.911 & 0.606 & 1.997 & 1.430 & 0.926 & 2.629 & 1.846 & 1.191 \\
          & Rel-Frobenius   & 0.358 & 0.364 & 0.421 & 0.266 & 0.265 & 0.311 & 0.223 & 0.220 & 0.259 \\
          & Rel-spectral  & 0.403 & 0.404 & 0.497 & 0.303 & 0.297 & 0.365 & 0.247 & 0.250 & 0.298 \\ \hline
          \multirow{10}[0]{*}{thDGP3} & TPR   & 0.939 & 0.995 & 1.000 & 0.877 & 0.990 & 1.000 & 0.835 & 0.982 & 1.000 \\
          & FNR   & 0.061 & 0.005 & 0.000 & 0.123 & 0.010 & 0.000 & 0.165 & 0.018 & 0.000 \\
          & TNR   & 0.981 & 0.980 & 0.990 & 0.995 & 0.992 & 0.997 & 0.997 & 0.997 & 0.998 \\
          & FPR   & 0.019 & 0.020 & 0.010 & 0.005 & 0.008 & 0.003 & 0.003 & 0.003 & 0.002 \\
          & Precision & 0.900 & 0.900 & 0.945 & 0.930 & 0.913 & 0.969 & 0.941 & 0.943 & 0.962 \\
          & MSE($D$) & 0.106 & 0.064 & 0.034 & 0.215 & 0.128 & 0.068 & 0.323 & 0.192 & 0.100 \\
          & Frobenius & 2.666 & 2.211 & 1.817 & 4.943 & 3.965 & 3.073 & 7.231 & 5.695 & 4.302 \\
          & Spectral & 1.170 & 0.935 & 0.769 & 1.848 & 1.342 & 0.952 & 2.471 & 1.758 & 1.155 \\
          & Rel-Frobenius  & 0.690 & 0.725 & 0.803 & 0.554 & 0.576 & 0.665 & 0.484 & 0.489 & 0.585 \\
          & Rel-spectral  & 0.665 & 0.701 & 0.750 & 0.473 & 0.520 & 0.619 & 0.378 & 0.416 & 0.524 \\
          \hline
		\end{tabular}%
		\caption{Simulation results for the thresholding local Whittle estimator.} 	\label{tab:threshold}%
	\end{table}%
\end{landscape}

\clearpage
	
\begin{landscape}
\begin{table}[t!]
  \centering
  \small
	
    \begin{tabular}{ccccccccccc} \hline
    \multirow{2}[0]{*}{Models} & \multirow{2}[0]{*}{Measures} & \multicolumn{3}{c}{$p=20$} & \multicolumn{3}{c}{$p=40$} & \multicolumn{3}{c}{$p=60$} \\
          &       & $N=200$ & $N=400$ & $N=1000$ & $N=200$ & $N=400$ & $N=1000$ & $N=200$ & $N=400$ & $N=1000$  \\ \hline \hline
    
          \multirow{10}[0]{*}{DGP1} & TPR   & 0.957 & 0.994 & 1.000 & 0.943 & 0.980 & 1.000 & 0.940 & 0.973 & 1.000 \\
          & FNR   & 0.043 & 0.006 & 0.000 & 0.057 & 0.020 & 0.000 & 0.060 & 0.027 & 0.000 \\
          & TNR   & 1.000 & 1.000 & 1.000 & 0.912 & 1.000 & 1.000 & 0.887 & 0.942 & 1.000 \\
          & FPR   & 0.000 & 0.000 & 0.000 & 0.088 & 0.000 & 0.000 & 0.113 & 0.058 & 0.000 \\
          & Precision & 1.000 & 0.998 & 1.000 & 0.910 & 0.998 & 0.999 & 0.776 & 0.941 & 0.998 \\
          & MSE($D$) & 0.095 & 0.050 & 0.022 & 0.193 & 0.102 & 0.045 & 0.292 & 0.157 & 0.070 \\
          & Frobenius & 4.911 & 3.197 & 2.323 & 7.490 & 5.066 & 3.367 & 10.365 & 6.760 & 4.080 \\
          & Spectral & 3.006 & 1.592 & 1.137 & 3.892 & 2.411 & 1.335 & 4.498 & 3.446 & 1.407 \\
          & Rel-Frobenius & 0.184 & 0.209 & 0.258 & 0.065 & 0.112 & 0.159 & 0.012 & 0.064 & 0.108 \\
          & Rel-spectral & 0.167 & 0.177 & 0.238 & 0.053 & 0.103 & 0.139 & 0.007 & 0.063 & 0.089 \\ \hline
          \multirow{10}[0]{*}{DGP2} & TPR   & 0.935 & 0.989 & 1.000 & 0.956 & 0.988 & 1.000 & 0.969 & 0.985 & 1.000 \\
          & FNR   & 0.065 & 0.011 & 0.000 & 0.044 & 0.012 & 0.000 & 0.031 & 0.015 & 0.000 \\
          & TNR   & 1.000 & 1.000 & 1.000 & 0.782 & 1.000 & 1.000 & 0.556 & 0.958 & 1.000 \\
          & FPR   & 0.000 & 0.000 & 0.000 & 0.218 & 0.000 & 0.000 & 0.444 & 0.042 & 0.000 \\
          & Precision & 1.000 & 1.000 & 0.999 & 0.787 & 1.000 & 0.999 & 0.550 & 0.959 & 1.000 \\
          & MSE($D$) & 0.095 & 0.053 & 0.024 & 0.196 & 0.103 & 0.047 & 0.292 & 0.159 & 0.071 \\
          & Frobenius & 0.913 & 0.670 & 0.574 & 1.305 & 0.948 & 0.800 & 1.735 & 1.192 & 0.975 \\
          & Spectral & 0.586 & 0.366 & 0.325 & 0.660 & 0.414 & 0.354 & 0.775 & 0.495 & 0.370 \\
          & Rel-Frobenius & 0.219 & 0.273 & 0.386 & 0.070 & 0.134 & 0.236 & 0.014 & 0.072 & 0.163 \\ 
          & Rel-spectral & 0.217 & 0.259 & 0.426 & 0.058 & 0.114 & 0.236 & 0.009 & 0.059 & 0.152 \\ \hline
          \multirow{10}[0]{*}{DGP3} & TPR   & 0.350 & 0.380 & 0.659 & 0.534 & 0.361 & 0.599 & 0.648 & 0.410 & 0.539 \\
          & FNR   & 0.650 & 0.620 & 0.341 & 0.466 & 0.639 & 0.401 & 0.352 & 0.590 & 0.461 \\
          & TNR   & 1.000 & 1.000 & 0.950 & 0.714 & 1.000 & 0.987 & 0.540 & 0.915 & 0.996 \\
          & FPR   & 0.000 & 0.000 & 0.050 & 0.286 & 0.000 & 0.013 & 0.460 & 0.085 & 0.004 \\
          & Precision & 0.999 & 0.997 & 0.757 & 0.735 & 0.999 & 0.815 & 0.559 & 0.919 & 0.876 \\
          & MSE($D$) & 0.093 & 0.048 & 0.022 & 0.188 & 0.098 & 0.044 & 0.286 & 0.149 & 0.067 \\
          & Frobenius & 1.547 & 1.489 & 1.201 & 2.204 & 2.122 & 1.772 & 2.727 & 2.611 & 2.281 \\
          & Spectral & 0.608 & 0.582 & 0.489 & 0.638 & 0.600 & 0.523 & 0.668 & 0.612 & 0.545 \\
          & Rel-Frobenius & 0.380 & 0.615 & 0.790 & 0.126 & 0.306 & 0.523 & 0.024 & 0.161 & 0.383 \\
          & Rel-spectral & 0.235 & 0.421 & 0.649 & 0.060 & 0.166 & 0.352 & 0.008 & 0.075 & 0.223 \\
          
          \hline
             \end{tabular}%
		\caption{Simulation results for the graphical local Whittle estimator.}
		\label{tab:graphicalLW}%
	\end{table}%

\end{landscape}

\subsection{Illustration: Rate of convergence} \label{appF:rate}
To illustrate our non-asymptotic results numerically, consider the result in Corollary \ref{cor:Op}. Corollary \ref{cor:Op} is a consequence of our main result Proposition \ref{prop:Op} for only short- or long-range dependence, that is, the true memory parameters satisfy $D_{0} \succcurlyeq 0$. Then, using \eqref{eq:cor:Rs} with $\mathcal{R}_{1} \equiv \mathcal{R}_{11}$ and $m=N^{.8}$, we get
$$ \| \widehat G(\widehat D) - G_0 \|_{\rm max}  \approx \sqrt{\frac{\log p}{N^{.8 \times (1-2\Delta_2)}} }$$
with high probability. 
Taking the log-transformation gives
\begin{equation} \label{eq:xxxx}
\log (\| \widehat G(\widehat D) - G_0 \|_{\rm max}) \approx .5 \big( \log(\log(p)) - .8(1-2\Delta_2) \log N \big).
\end{equation}
We checked this relationship by considering the model thDGP3 with sample sizes $N = 200$, 400, 600, 800, 1000, 1200, 1400, 1600, 1800, 2000 and dimensions $p = 20, 40, 60, 80, 100, 120$.  

In Figure \ref{f:rate-N} we fixed the sample size as $N = 200, 400, 600$ and consider \eqref{eq:xxxx} as a function in $\log(\log(p))$. The expected slope is .5 in the log-log plot. Figure \ref{f:rate-N} shows that the estimated slope is close to .5.

In Figure \ref{f:rate-p} we fixed the dimension $p = 20, 30, 40$. In this setting, the expected slope for $\log N$  is $-.104$ calculated as
$$ .5 \times .8 (1-2 \Delta_2) = .4 \times (1 - 2 \max_{r =1, \dots,r} \widehat{d}_r) = .104, $$
where we chose $\Delta_2$ as the largest LRD parameter across all dimensions. In Figure \ref{f:rate-p}, it can be seen that the slope is close to $-.104$, as expected.

\begin{figure}[ht!]
	\centering
	\includegraphics[width=6.5 in, height=2.3 in]{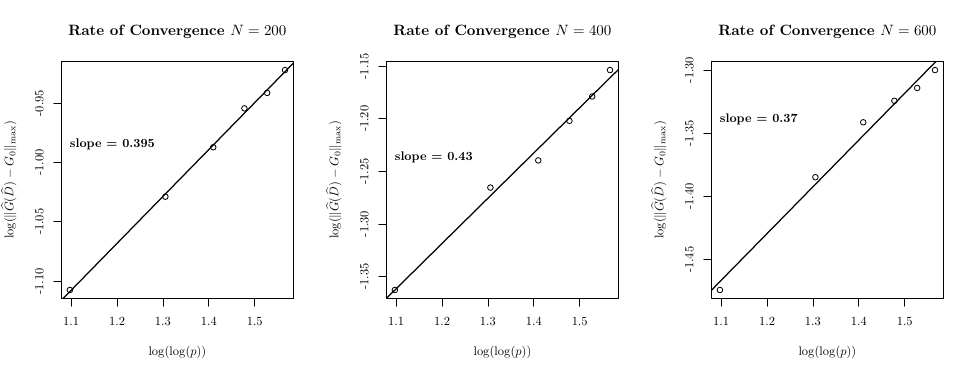}
	\vspace{-2 mm}
	\caption{The rate of convergence with fixed sample size $N$. Theoretically expected  slope is .5.} \label{f:rate-N}
\end{figure}

\begin{figure}[ht!]
	\centering
	\includegraphics[width=6.5 in, height=2.3 in]{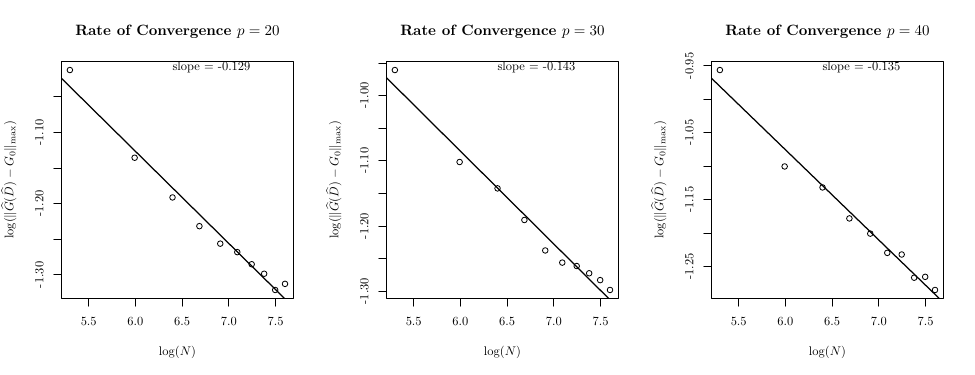}
	\vspace{-2 mm}
	\caption{The rate of convergence with fixed dimension $p$. Theoretically expected slope is $-.104$.} \label{f:rate-p}
\end{figure}

\subsection{Univariate versus multivariate estimation of the memory parameters} \label{appF:univsmulti}
From a theoretical perspective, Remark \ref{re:univd} argues that our proofs remain valid for multivariate estimation of $D_{0}$. However, an additional $p$ would appear in the bounds and significantly weaken the results.

In Table \ref{tab:addlabel1} below, we report on a small simulation study evaluating the efficiency of our proposed estimator compared to multivariate estimation  and also a shrinkage estimator. The shrinkage estimator is obtained by plugging the thresholded estimator $T_{\rho}(\widehat G (D))$ in \eqref{eq:thresholding} into the local Whittle estimation for $D_{0}$ in \eqref{eq:R(D)} and is labeled as ``Threshold" in Table \ref{tab:addlabel1}. We use the same DGPs as introduced in Section \ref{s:sim} with sample size $N=200$ and dimension $p=40$. 
For comparison we use the mean squared error of the estimated $D_{0}$ (MSE($D$)) and study the impact on $\widehat{P}(\widehat D)$ of the different estimation procedures by calculating the Frobenius and spectral distances to the true $P_{0}$. The computational time is measured as the duration time for five repetitions.  

\begin{table}[ht!]  \centering  
	\small
	\begin{tabular}{cccccc}    \hline
		Models & Methods & MSE($D$) & Frobenius & Spectral    & Time \\   \hline \hline
		 \multirow{3}[0]{*}{DGP1} & Univariate & 0.190 & 7.760 & 3.977  & 2 sec \\          & Multivariate & 0.190 & 7.745 & 3.964  & 8 sec \\          
		 & Threshold & 0.189 & 7.744 & 3.964 & 20 sec \\   \hline 
		 
		 \multirow{3}[0]{*}{DGP2} & Univariate & 0.196 & 1.413 & 0.707  & 2 sec \\          & Multivariate & 0.287 & 1.478 & 0.751  & 146 sec \\          
		 & Threshold & 0.196 & 1.449 & 0.717  & 270 sec \\ \hline    
		 
		\multirow{3}[0]{*}{DGP3} & Univariate & 0.186 & 2.206 & 0.642 & 2 sec \\          
		& Multivariate & 0.298 & 2.218 & 0.646  & 102 sec \\          
		& Threshold & 0.194 & 2.201 & 0.639  & 240 sec \\  \hline  \end{tabular}
	\caption{Performance measures for different estimation methods of the memory parameters $D_{0}$. }    
 \label{tab:addlabel1}
\end{table}%

As can be seen in the table, estimating the memory parameters $D_{0}$ univariately performs well compared to the other methods (Multivariate, Threshold).
Furthermore, the univariate estimation reduces computational time dramatically compared to multivariate or thresholded estimation. Note also that the mean squared error appears to be larger for multivariate estimation than for univariate estimation. One would expect the asymptotic variance of the multivariate estimators for $D_{0}$ to be smaller compared to the univariate case which is not reflected in the simulation results. We suspect that this may be due to numerical optimization issues with multivariate estimation. The Threshold method seems to perform slightly better than multivariate estimation. Since all our DGPs are highly sparse (see Figure \ref{f:graphical-DGP}), it may not come as a surprise that the univariate and thresholded estimators perform well.

\subsection{Modified precision matrix estimators and CLIME}
\label{se:F4}
In this section, we compare the performance of our graphical local Whittle estimator for the precision matrix with the alternative estimators presented in Section \ref{se:alternativeestimators}. The modified precision matrix estimators in Section \ref{se:alternativeestimators} are an estimator based on the coherence matrix (Section \ref{se:Spice}) and a local Whittle CLIME estimator (Section \ref{se:Clime}). 

For the coherence-based graphical local Whittle, we use Algorithm \ref{al:ALM} as introduced in Section \ref{ss:tuning-Graphical}.
Similar to the graphical local Whittle estimator, the local Whittle CLIME estimator can be computed by using an ADMM algorithm. \cite{wang2013large} proposed an algorithm for sparse inverse covariance matrix estimation.
We modify their Algorithm 1 to our setting, using it for estimation in the spectral domain.
The details can be found in Algorithm \ref{al:ADMM-CLIME}. For the algorithm, recall the function ${\rm shrink}(M,\nu) = \sign(M_{rs}) \max( |M_{rs}| - \nu, 0) $ for a matrix $M=(M_{rs})_{r,s =1, \dots, p}$ and some $\nu \geq 0$.

\begin{algorithm} \caption{Alternating direction method of multipliers (ADMM) algorithm for local Whittle CLIME estimator}
	\KwIn{Initial estimator $P^0 = {\rm diag}((\widehat G(\widehat D) + .1 I_p)^{-1})$, $V^0 = Y^0 = 0$, $\mu$, $\rho$.} \label{al:ADMM-CLIME}
	\KwOut{Sparse estimation of $P$.}
	{\bf Repeat until convergence:} \\
	\For{ $k=0, 1, \ldots$}{
		1. $P^{k+1} = {\rm shrink}(P^k - V^k, \mu)$. \\
		2. $U^{k+1} = \widehat G(\widehat D) P^{k+1} + Y^k$.\\
		3. $Z^{k+1} = Y^{k} + {\rm shrink}(I_p - U^{k+1}, \lambda)$. \\
		4. $Y^{k+1} = U^{k+1} - Z^{k+1}$. \\
		5. $V^{k+1} = \widehat G(\widehat D) (2Y^{k+1} - Y^k)(\rho \mu)$.
	}
\end{algorithm}

Comparison results are summarized in Tables \ref{t:DGP1-comparison}--\ref{t:DGP3-comparison}. Note that we also present the simulation results for our graphical local Whittle estimator (LW-GLASSO). All methods' performances improve with increasing sample size. The modified precision matrix estimator, labeled as LW-GLASSO (modified) in Tables \ref{t:DGP1-comparison}--\ref{t:DGP3-comparison} performs very similar to the classical LW-GLASSO. Our study does not reveal any significant differences between the two estimators. On the other hand, the local Whittle CLIME estimator tends to result in  higher TNR while the LW-GLASSO gives higher TPR. For example, in Table \ref{t:DGP3-comparison} when $p=60$ and $N=200$, LW-GLASSO gives .54 for TNR while LW-CLIME gives .999. It means that LW-GLASSO is good at finding a nonzero coefficient as nonzero, while LW-CLIME finds a zero coefficient as zero. 

\begin{landscape}
	% Table generated by Excel2LaTeX from sheet 'mean-med'
	\begin{table}[t!]
		\centering
		\small
		\begin{tabular}{ccccccccccc} \hline
			\multirow{2}[0]{*}{Methods} & \multirow{2}[0]{*}{Measures} & \multicolumn{3}{c}{$p=20$} & \multicolumn{3}{c}{$p=40$} & \multicolumn{3}{c}{$p=60$} \\
			&       & $N=200$ & $N=400$ & $N=1000$ & $N=200$ & $N=400$ & $N=1000$ & $N=200$ & $N=400$ & $N=1000$ \\ \hline \hline
    \multirow{9}[0]{*}{LW-GLASSO} & TPR   & 0.957 & 0.994 & 1.000 & 0.943 & 0.980 & 1.000 & 0.940 & 0.973 & 1.000 \\
& FNR   & 0.043 & 0.006 & 0.000 & 0.057 & 0.020 & 0.000 & 0.060 & 0.027 & 0.000 \\
& TNR   & 1.000 & 1.000 & 1.000 & 0.912 & 1.000 & 1.000 & 0.887 & 0.942 & 1.000 \\
& FPR   & 0.000 & 0.000 & 0.000 & 0.088 & 0.000 & 0.000 & 0.113 & 0.058 & 0.000 \\
& Precision & 1.000 & 0.998 & 1.000 & 0.910 & 0.998 & 0.999 & 0.776 & 0.941 & 0.998 \\
& Frobenius & 4.911 & 3.197 & 2.323 & 7.490 & 5.066 & 3.367 & 10.365 & 6.760 & 4.080 \\
& Spectral & 3.006 & 1.592 & 1.137 & 3.892 & 2.411 & 1.335 & 4.498 & 3.446 & 1.407 \\
& Rel-Frobenius & 0.184 & 0.209 & 0.258 & 0.065 & 0.112 & 0.159 & 0.012 & 0.064 & 0.108 \\
& Rel-spectral & 0.167 & 0.177 & 0.238 & 0.053 & 0.103 & 0.139 & 0.007 & 0.063 & 0.089 \\ \hline
\multicolumn{1}{c}{\multirow{9}[0]{*}{\shortstack[c]{LW-GLASSO \\ (modified)}}} & TPR   & 0.960 & 0.998 & 1.000 & 0.964 & 0.996 & 1.000 & 0.965 & 0.995 & 1.000 \\
& FNR   & 0.040 & 0.002 & 0.000 & 0.036 & 0.004 & 0.000 & 0.035 & 0.005 & 0.000 \\
& TNR   & 1.000 & 1.000 & 1.000 & 0.632 & 0.999 & 1.000 & 0.547 & 0.788 & 1.000 \\
& FPR   & 0.000 & 0.000 & 0.000 & 0.368 & 0.001 & 0.000 & 0.453 & 0.212 & 0.000 \\
& Precision & 1.000 & 1.000 & 0.999 & 0.639 & 0.999 & 1.000 & 0.539 & 0.793 & 1.000 \\
& Frobenius & 4.870 & 3.137 & 2.323 & 8.101 & 4.575 & 3.365 & 11.115 & 5.857 & 4.076 \\
& Spectral & 2.973 & 1.559 & 1.139 & 4.352 & 1.847 & 1.335 & 5.189 & 2.269 & 1.401 \\
& Rel-Frobenius & 0.183 & 0.207 & 0.258 & 0.072 & 0.103 & 0.159 & 0.014 & 0.056 & 0.108 \\
& Rel-spectral & 0.166 & 0.174 & 0.238 & 0.062 & 0.080 & 0.139 & 0.009 & 0.042 & 0.089 \\ \hline
\multirow{9}[0]{*}{LW-CLIME} & TPR   & 0.975 & 0.998 & 1.000 & 0.864 & 0.995 & 1.000 & 0.864 & 0.984 & 1.000 \\
& FNR   & 0.025 & 0.002 & 0.000 & 0.136 & 0.005 & 0.000 & 0.136 & 0.016 & 0.000 \\
& TNR   & 0.996 & 0.994 & 0.998 & 1.000 & 0.995 & 0.997 & 1.000 & 0.996 & 0.997 \\
& FPR   & 0.004 & 0.006 & 0.002 & 0.000 & 0.005 & 0.003 & 0.000 & 0.004 & 0.003 \\
& Precision & 0.956 & 0.935 & 0.972 & 0.998 & 0.887 & 0.930 & 1.000 & 0.875 & 0.899 \\
& Frobenius & 4.595 & 3.348 & 2.388 & 10.145 & 5.056 & 3.535 & 12.358 & 6.494 & 4.400 \\
& Spectral & 2.561 & 1.728 & 1.183 & 4.963 & 2.242 & 1.465 & 5.044 & 2.775 & 1.648 \\
& Rel-Frobenius & 0.176 & 0.219 & 0.265 & 0.085 & 0.114 & 0.167 & 0.013 & 0.061 & 0.116 \\
& Rel-spectral & 0.151 & 0.192 & 0.248 & 0.066 & 0.095 & 0.153 & 0.007 & 0.051 & 0.105 \\ \hline
		\end{tabular}%
		\caption{Simulation results for the modified graphical local Whittle and CLIME estimators with DGP1.}
		\label{t:DGP1-comparison}%
	\end{table}%
	
\end{landscape}

\clearpage

\begin{landscape}
	\begin{table}[t!]
		\centering
		\small
		\begin{tabular}{ccccccccccc} \hline
	\multirow{2}[0]{*}{Methods} & \multirow{2}[0]{*}{Measures} & \multicolumn{3}{c}{$p=20$} & \multicolumn{3}{c}{$p=40$} & \multicolumn{3}{c}{$p=60$} \\
	&       & $N=200$ & $N=400$ & $N=1000$ & $N=200$ & $N=400$ & $N=1000$ & $N=200$ & $N=400$ & $N=1000$ \\ \hline \hline
	
	    \multirow{9}[0]{*}{LW-GLASSO} & TPR   & 0.935 & 0.989 & 1.000 & 0.956 & 0.988 & 1.000 & 0.969 & 0.985 & 1.000 \\
	& FNR   & 0.065 & 0.011 & 0.000 & 0.044 & 0.012 & 0.000 & 0.031 & 0.015 & 0.000 \\
	& TNR   & 1.000 & 1.000 & 1.000 & 0.782 & 1.000 & 1.000 & 0.556 & 0.958 & 1.000 \\
	& FPR   & 0.000 & 0.000 & 0.000 & 0.218 & 0.000 & 0.000 & 0.444 & 0.042 & 0.000 \\
	& Precision & 1.000 & 1.000 & 0.999 & 0.787 & 1.000 & 0.999 & 0.550 & 0.959 & 1.000 \\
	& Frobenius & 0.913 & 0.670 & 0.574 & 1.305 & 0.948 & 0.800 & 1.735 & 1.192 & 0.975 \\
	& Spectral & 0.586 & 0.366 & 0.325 & 0.660 & 0.414 & 0.354 & 0.775 & 0.495 & 0.370 \\
	& Rel-Frobenius & 0.219 & 0.273 & 0.386 & 0.070 & 0.134 & 0.236 & 0.014 & 0.072 & 0.163 \\
	& Rel-spectral & 0.217 & 0.259 & 0.426 & 0.058 & 0.114 & 0.236 & 0.009 & 0.059 & 0.152 \\
	\hline
	\multicolumn{1}{c}{\multirow{9}[0]{*}{\shortstack[c]{LW-GLASSO \\ (modified)}}} & TPR   & 0.928 & 0.989 & 1.000 & 0.938 & 0.988 & 1.000 & 0.940 & 0.985 & 1.000 \\
	& FNR   & 0.073 & 0.011 & 0.000 & 0.062 & 0.012 & 0.000 & 0.060 & 0.015 & 0.000 \\
	& TNR   & 1.000 & 1.000 & 1.000 & 0.796 & 1.000 & 1.000 & 0.661 & 0.953 & 1.000 \\
	& FPR   & 0.000 & 0.000 & 0.000 & 0.204 & 0.000 & 0.000 & 0.339 & 0.047 & 0.000 \\
	& Precision & 1.000 & 1.000 & 0.999 & 0.802 & 1.000 & 0.999 & 0.667 & 0.954 & 1.000 \\
	& Frobenius & 0.926 & 0.670 & 0.574 & 1.383 & 0.948 & 0.800 & 1.824 & 1.203 & 0.975 \\
	& Spectral & 0.604 & 0.366 & 0.326 & 0.728 & 0.415 & 0.354 & 0.832 & 0.511 & 0.370 \\
	& Rel-Frobenius & 0.224 & 0.273 & 0.386 & 0.075 & 0.134 & 0.236 & 0.014 & 0.073 & 0.163 \\
	& Rel-spectral & 0.224 & 0.259 & 0.426 & 0.065 & 0.114 & 0.236 & 0.008 & 0.060 & 0.152 \\
	\hline
	\multirow{9}[0]{*}{LW-CLIME} & TPR   & 0.951 & 0.991 & 1.000 & 0.957 & 0.997 & 0.999 & 0.954 & 0.996 & 1.000 \\
	& FNR   & 0.049 & 0.009 & 0.000 & 0.043 & 0.003 & 0.001 & 0.046 & 0.004 & 0.000 \\
	& TNR   & 1.000 & 1.000 & 0.999 & 0.999 & 0.998 & 1.000 & 0.999 & 0.999 & 1.000 \\
	& FPR   & 0.000 & 0.000 & 0.001 & 0.001 & 0.002 & 0.000 & 0.001 & 0.001 & 0.000 \\
	& Precision & 0.995 & 0.998 & 0.986 & 0.966 & 0.953 & 0.997 & 0.969 & 0.944 & 0.997 \\
	& Frobenius & 0.870 & 0.673 & 0.580 & 1.280 & 0.976 & 0.804 & 1.588 & 1.238 & 0.978 \\
	& Spectral & 0.526 & 0.368 & 0.328 & 0.660 & 0.419 & 0.357 & 0.701 & 0.470 & 0.371 \\
	& Rel-Frobenius & 0.209 & 0.273 & 0.388 & 0.068 & 0.138 & 0.237 & 0.011 & 0.075 & 0.164 \\
	& Rel-spectral & 0.194 & 0.259 & 0.428 & 0.055 & 0.115 & 0.238 & 0.006 & 0.056 & 0.152 \\ \hline
		\end{tabular}%
		\caption{Simulation results for the modified graphical local Whittle and CLIME estimators with DGP2.}
	\label{t:DGP2-comparison}%
	
	\end{table}%
	
\end{landscape}

\clearpage

\begin{landscape}
	\begin{table}[t!]
		\centering
		\small
		\begin{tabular}{ccccccccccc} \hline
	\multirow{2}[0]{*}{Methods} & \multirow{2}[0]{*}{Measures} & \multicolumn{3}{c}{$p=20$} & \multicolumn{3}{c}{$p=40$} & \multicolumn{3}{c}{$p=60$} \\
	&       & $N=200$ & $N=400$ & $N=1000$ & $N=200$ & $N=400$ & $N=1000$ & $N=200$ & $N=400$ & $N=1000$ \\ \hline \hline
	
	    \multirow{9}[0]{*}{LW-GLASSO} & TPR   & 0.350 & 0.380 & 0.659 & 0.534 & 0.361 & 0.599 & 0.648 & 0.410 & 0.539 \\
	& FNR   & 0.650 & 0.620 & 0.341 & 0.466 & 0.639 & 0.401 & 0.352 & 0.590 & 0.461 \\
	& TNR   & 1.000 & 1.000 & 0.950 & 0.714 & 1.000 & 0.987 & 0.540 & 0.915 & 0.996 \\
	& FPR   & 0.000 & 0.000 & 0.050 & 0.286 & 0.000 & 0.013 & 0.460 & 0.085 & 0.004 \\
	& Precision & 0.999 & 0.997 & 0.757 & 0.735 & 0.999 & 0.815 & 0.559 & 0.919 & 0.876 \\
	& Frobenius & 1.547 & 1.489 & 1.201 & 2.204 & 2.122 & 1.772 & 2.727 & 2.611 & 2.281 \\
	& Spectral & 0.608 & 0.582 & 0.489 & 0.638 & 0.600 & 0.523 & 0.668 & 0.612 & 0.545 \\
	& Rel-Frobenius & 0.380 & 0.615 & 0.790 & 0.126 & 0.306 & 0.523 & 0.024 & 0.161 & 0.383 \\
	& Rel-spectral & 0.235 & 0.421 & 0.649 & 0.060 & 0.166 & 0.352 & 0.008 & 0.075 & 0.223 \\
	\hline
	\multicolumn{1}{c}{\multirow{9}[0]{*}{\shortstack[c]{LW-GLASSO \\ (modified)}}} & TPR   & 0.349 & 0.380 & 0.645 & 0.511 & 0.361 & 0.618 & 0.591 & 0.411 & 0.523 \\
	& FNR   & 0.651 & 0.620 & 0.355 & 0.489 & 0.639 & 0.382 & 0.409 & 0.589 & 0.477 \\
	& TNR   & 1.000 & 1.000 & 0.961 & 0.748 & 1.000 & 0.986 & 0.625 & 0.915 & 0.997 \\
	& FPR   & 0.000 & 0.000 & 0.039 & 0.252 & 0.000 & 0.014 & 0.375 & 0.085 & 0.003 \\
	& Precision & 1.000 & 0.996 & 0.786 & 0.766 & 0.999 & 0.804 & 0.642 & 0.919 & 0.900 \\
	& Frobenius & 1.548 & 1.491 & 1.222 & 2.202 & 2.124 & 1.752 & 2.710 & 2.611 & 2.311 \\
	& Spectral & 0.609 & 0.584 & 0.501 & 0.638 & 0.601 & 0.525 & 0.654 & 0.612 & 0.556 \\
	& Rel-Frobenius & 0.380 & 0.615 & 0.801 & 0.125 & 0.306 & 0.520 & 0.022 & 0.161 & 0.388 \\
	& Rel-spectral & 0.236 & 0.423 & 0.669 & 0.059 & 0.166 & 0.354 & 0.007 & 0.075 & 0.229 \\
	\hline
	\multirow{9}[0]{*}{LW-CLIME} & TPR   & 0.373 & 0.353 & 0.350 & 0.384 & 0.379 & 0.370 & 0.388 & 0.391 & 0.377 \\
	& FNR   & 0.627 & 0.647 & 0.650 & 0.616 & 0.621 & 0.630 & 0.612 & 0.609 & 0.623 \\
	& TNR   & 0.999 & 1.000 & 1.000 & 0.999 & 1.000 & 1.000 & 0.999 & 0.999 & 1.000 \\
	& FPR   & 0.001 & 0.000 & 0.000 & 0.001 & 0.000 & 0.000 & 0.001 & 0.001 & 0.000 \\
	& Precision & 0.992 & 0.999 & 1.000 & 0.962 & 0.999 & 0.998 & 0.938 & 0.980 & 0.999 \\
	& Frobenius & 1.540 & 1.509 & 1.484 & 2.197 & 2.107 & 2.074 & 2.714 & 2.578 & 2.528 \\
	& Spectral & 0.607 & 0.592 & 0.574 & 0.636 & 0.600 & 0.580 & 0.661 & 0.608 & 0.584 \\
	& Rel-Frobenius & 0.379 & 0.622 & 0.979 & 0.120 & 0.303 & 0.614 & 0.019 & 0.158 & 0.425 \\
	& Rel-spectral & 0.235 & 0.428 & 0.771 & 0.056 & 0.166 & 0.391 & 0.006 & 0.073 & 0.241 \\
	 \hline
	\end{tabular}%
	\caption{Simulation results for the modified graphical local Whittle and CLIME estimators with DGP3.}
	\label{t:DGP3-comparison}%
	\end{table}%
	
\end{landscape}
\subsection{Comparison to existing methods} \label{appF:comparisonSRD}
\label{se:F5}
In this section we emphasize the relevance of considering estimators which account for strong temporal correlation beyond short-range dependence. 
\cite{Sun2018:LargeSpectral} consider possibly high-dimensional time series under short-range dependence ($D_{0} \equiv 0$) and study estimation of the spectral density and its inverse. See also Section \ref{se:comparison} for a detailed comparison.
We conduct a simulation study with synthetic, long-range dependent data and apply the estimators proposed in \cite{Sun2018:LargeSpectral}. 
More precisely, we simulated long-range dependent data based on thDGP3 and DGP2 in Sections \ref{s:sub2Sim} and \ref{s:sub1Sim} following the sparsity pattern in Figure \ref{f:graphical-DGP}. We then applied the graphical LASSO and thresholded long-run variance estimators in \cite{Sun2018:LargeSpectral} to recover the sparsity patterns. As it can be seen from Figure  \ref{f:sparsity-simulation}, the estimators for short-range dependent models (right column) perform poorly and fail to find the true underlying zero coefficients. More detailed performance measures are provided in Table \ref{t:SRDmodels} indicating that the estimators for short-range dependent models perform poorly under long-range dependence.

\begin{figure}[t!]
	\centering
	\vspace{-1 mm}
	\includegraphics[width=4in, height=1.8 in]{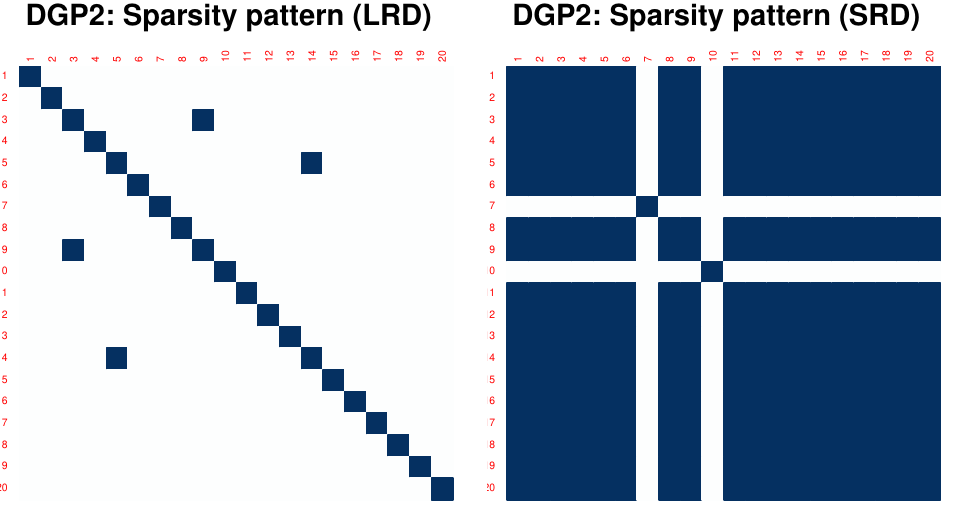}
	\includegraphics[width=4 in, height=1.8 in]{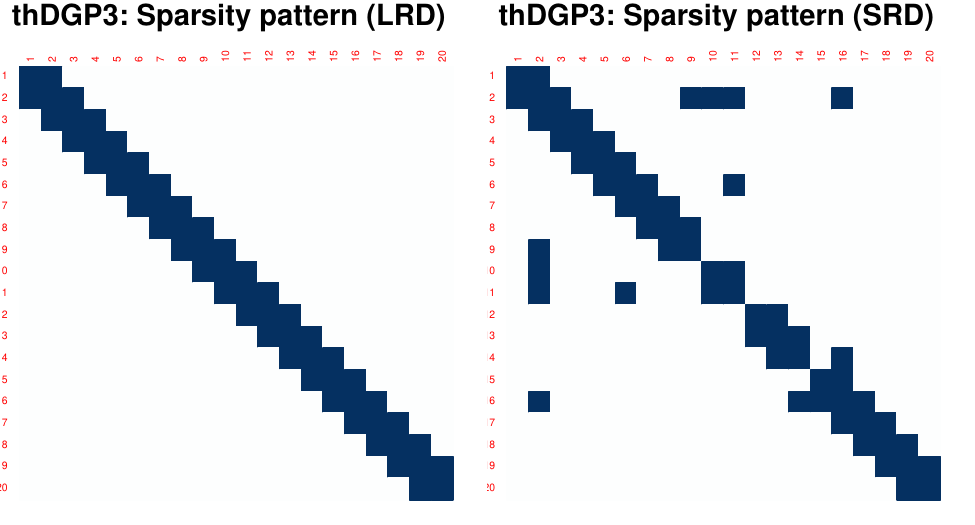}
	\vspace{-2 mm}
	\caption{Sparsity pattern when SRD model is used for LRD simulated data.} \label{f:sparsity-simulation}
\end{figure}

\begin{landscape}
\begin{table}[t!]
	\small
	\centering

	\begin{tabular}{cccccccccccc} \hline
		\multirow{2}[0]{*}{Models} & \multirow{2}[0]{*}{Methods} & \multirow{2}[0]{*}{Measures} & \multicolumn{3}{c}{$p=20$} & \multicolumn{3}{c}{$p=40$} & \multicolumn{3}{c}{$p=60$} \\
		&       &       & $N=200$ & $N=400$ & $N=1000$ & $N=200$ & $N=400$ & $N=1000$ & $N=200$ & $N=400$ & $N=1000$ \\ \hline \hline
		\multirow{10}[0]{*}{DGP2} & \multirow{5}[0]{*}{
			 \shortstack[c]{LW-GLASSO \\ (LRD)}} & TPR   & 0.969 & 0.998 & 1.000 & 0.980 & 0.998 & 1.000 & 0.989 & 0.996 & 1.000 \\
		&       & FNR   & 0.031 & 0.002 & 0.000 & 0.020 & 0.003 & 0.000 & 0.011 & 0.004 & 0.000 \\
		&       & TNR   & 1.000 & 0.999 & 0.998 & 0.654 & 1.000 & 0.999 & 0.325 & 0.940 & 1.000 \\
		&       & FPR   & 0.000 & 0.001 & 0.002 & 0.346 & 0.000 & 0.001 & 0.675 & 0.060 & 0.000 \\
		&       & Precision & 0.997 & 0.990 & 0.980 & 0.661 & 0.991 & 0.985 & 0.316 & 0.936 & 0.990 \\ 
		& \multicolumn{1}{c}{\multirow{5}[0]{*}{ \shortstack[c]{GLASSO\\(SRD)}}} & TPR   & 1.000 & 1.000 & 1.000 & 1.000 & 1.000 & 1.000 & 1.000 & 1.000 & 1.000 \\
		&       & FNR   & 0.000 & 0.000 & 0.000 & 0.000 & 0.000 & 0.000 & 0.000 & 0.000 & 0.000 \\
		&       & TNR   & 0.117 & 0.284 & 0.559 & 0.006 & 0.108 & 0.348 & 0.001 & 0.013 & 0.259 \\
		&       & FPR   & 0.883 & 0.716 & 0.441 & 0.994 & 0.892 & 0.652 & 0.999 & 0.987 & 0.741 \\
		&       & Precision & 0.068 & 0.085 & 0.138 & 0.030 & 0.034 & 0.046 & 0.020 & 0.020 & 0.027 \\ \hline
		\multirow{10}[0]{*}{thDGP3} & \multicolumn{1}{c}{\multirow{5}[0]{*}{ \shortstack[c]{Thresholding\\(LRD)}}} & TPR   & 0.938 & 0.996 & 1.000 & 0.884 & 0.992 & 1.000 & 0.843 & 0.987 & 1.000 \\
		&       & FNR   & 0.062 & 0.004 & 0.000 & 0.116 & 0.008 & 0.000 & 0.157 & 0.013 & 0.000 \\
		&       & TNR   & 0.981 & 0.981 & 0.989 & 0.994 & 0.993 & 0.997 & 0.997 & 0.996 & 0.998 \\
		&       & FPR   & 0.019 & 0.019 & 0.011 & 0.006 & 0.007 & 0.003 & 0.003 & 0.004 & 0.002 \\
		&       & Precision & 0.897 & 0.906 & 0.943 & 0.927 & 0.921 & 0.969 & 0.942 & 0.935 & 0.956 \\
		& \multicolumn{1}{c}{\multirow{5}[0]{*}{ \shortstack[c]{Thresholding\\(SRD)}}} & TPR   & 0.586 & 0.783 & 0.979 & 0.453 & 0.584 & 0.882 & 0.412 & 0.495 & 0.802 \\
		&       & FNR   & 0.414 & 0.217 & 0.021 & 0.547 & 0.416 & 0.118 & 0.588 & 0.505 & 0.198 \\
		&       & TNR   & 0.989 & 0.983 & 0.971 & 0.997 & 0.995 & 0.990 & 0.998 & 0.998 & 0.995 \\
		&       & FPR   & 0.011 & 0.017 & 0.029 & 0.003 & 0.005 & 0.010 & 0.002 & 0.002 & 0.005 \\
		&       & Precision & 0.914 & 0.898 & 0.858 & 0.924 & 0.916 & 0.882 & 0.935 & 0.940 & 0.896 \\ \hline
	\end{tabular}%
	\caption{Performance measures when SRD models are used for LRD models.} \label{t:SRDmodels}%
\end{table}%

\end{landscape}

\small
\bibliographystyle{plainnat}
\bibliography{LW_largedim_bib}

\flushleft
\begin{tabular}{lp{1 in}l}
Changryong Baek & & Marie-Christine D\"uker\\
Dept.\ of Statistics & & Dept.\ of Statistics and Data Science\\
Sungkyunkwan University & &  Cornell University\\
25-2, Sungkyunkwan-ro, Jongno-gu & & 129 Garden Ave, Comstock Hall\\
Seoul, 110-745, Korea & & Ithaca, NY 14850, USA\\
{\it crbaek@skku.edu} & & {\it duker@cornell.edu}\\
\end{tabular}

\flushleft
\begin{tabular}{lp{1 in}l}
Vladas Pipiras\\
Dept.\ of Statistics and Operations Research\\
UNC at Chapel Hill\\
CB\#3260, Hanes Hall\\
Chapel Hill, NC 27599, USA\\
{\it pipiras@email.unc.edu} \\
\end{tabular}

\end{document}